\documentclass[12pt]{amsart}

\usepackage{vmargin}
\usepackage{graphicx}
\usepackage{color}
\usepackage{transparent}
\input{amssym.def}
\input{amssym}
\usepackage{a4wide}
\usepackage{supertabular}
\usepackage{tikz-cd}
\usepackage{array}
\usepackage{multicol}
\usepackage{bbm}
\usepackage{bbold}
\usepackage{epstopdf}

\usepackage{import}
\setmarginsrb{2cm}{2cm}{2cm}{2.5cm}{75pt}{20pt}{20pt}{30mm}
\def\derpar#1#2{\frac{\partial#1}{\partial#2}}
\def\restr#1{\lfloor_{#1}}
\def\R{\mathbb R}
\def\N{\mathbb N}
\def\T{\mathbb T}

\def\Z{\mathbb Z}

\def\H{\mathcal H}

\def\F{\mathbb F}
\def\I{\mathbb I}

\def\S{\mathbb S}

\def\var{\varepsilon}

\def\pa{\partial}
\def\om{\omega}
\def\Om{\Omega}
\def\om{\omega}
\def\ov{\overline}
\def\cal{\mathcal}
\def\vphi{\varphi}
\def\hat{\widehat}

\def\tilde{\widetilde}

\def\Id{{\rm Id}\,}
\def\dps{\displaystyle}

\def\div{{\rm div}\,}
\def\vol{{\rm vol}\,}

\def\<{\langle}
\def\>{\rangle}

\def\II{{\rm I\!I}}

\def\Ric{{\rm Ric}}
\def\GTP{{\rm GTP}}
\def\LSI{{\rm LSI}}

\def\div{\operatorname{div}}
\def\TM{T\!M}
\def\TsM{T^*\!M}

\def\vol{{\rm vol}\,}

\def\TIJ{T_I^J}
\def\Tzeroun{T_0^1}

\def\Cnk#1#2{\small \left(\!\begin{array}{c} #1 
                \\ #2 \end{array} \!\right)} 

\def\med{\medskip}
\def\sm{\smallskip}
\def\bul{$\bullet$\ }

\def\begeq{\begin{equation}}
\def\endeq{\end{equation}}
\def\begar{\begin{eqnarray}}
\def\endar{\end{eqnarray}}
\def\begar*{\begin{eqnarray*}}
\def\endar*{\end{eqnarray*}}
\def\begal{\begin{align}}
\def\endal{\end{align}}
\def\begal*{\begin{align*}}
\def\endal*{\end{align*}}

\newtheorem{Thm}{Theorem}
\newtheorem{Lem}[Thm]{Lemma}

\newtheorem{Cor}[Thm]{Corollary}
\newtheorem{Prop}[Thm]{Proposition}

\numberwithin{equation}{section}
\numberwithin{Thm}{section}

\theoremstyle{definition}
\newtheorem{Def}[Thm]{Definition}
\newtheorem{Rk}[Thm]{Remark}
\newtheorem{Rks}[Thm]{Remarks}

\theoremstyle{remark}

\newtheorem{Ex}[Thm]{Example}

\newtheorem*{Thm*}{Theorem}
\newtheorem*{Lem*}{Lemma}
\newtheorem*{Conj*}{Conjecture}
\newtheorem*{Cor*}{Corollary}
\newtheorem*{Def*}{Definition}
\newtheorem*{Prop*}{Proposition}
\newtheorem*{Exo*}{Exercise}
\newtheorem*{Exs*}{Examples}
\newtheorem*{Ex*}{Example}
\newtheorem*{Rk*}{Remark}
\newtheorem*{Rks*}{Remarks}

\def\bibnotes{\medskip \noindent {\bf Bibliographical Notes }\sm \noindent}

\def\signcv{\vspace{1mm} \begin{center} {\sc C\'edric Villani\par\vspace{2mm}
Acad\'emie des Sciences\\
Universit\'e Claude Bernard Lyon I \\
Institut Camille Jordan\par
43 Bd du 11 Novembre 1918, 69100 Villeurbanne\par
FRANCE\par\vspace{3mm}
E-mail:} \tt{cv@cedricvillani.org} \end{center}}

\begin{document}

\title[Geometric kinetic Fokker--Planck equation]{On the kinetic Fokker--Planck equation in curved geometry}

\vspace*{-12mm}

\author{C. Villani}

\vspace*{-9mm}

\begin{abstract} These notes are devoted to the kinetic Fokker--Planck equation on the tangent bundle of a compact Riemannian manifold, with motivations coming from relativistic kinetic theory of gases (\`a la Debbasch) and geometric analysis (\`a la Bismut). They are developed in particular from collaboration with Fabrice Debbasch and Yann Ollivier. The presentation is mostly self-contained, intended for readers without much prior familiarity with geometric analysis. Once the kinetic equations are reviewed, I expand the basic toolbox from functional and spectral analysis. Then the main focus is on global hypoelliptic $C^\infty$ regularization and hypocoercive exponential equilibration, both of which are addressed both in $L^2$ and $L^1$ settings. Various types of interpolation inequalities, spectral gap and entropic inequalities are established and used. Several existing results in the fields are simplified and unified, along with new ones. Some directions of research are mentioned.
\end{abstract}

\maketitle

\vspace*{-15mm}

\tableofcontents

\vspace*{-10mm}

\bigskip
\pagebreak

These notes are a large expansion of the lecture which I gave at the research conference Festum Pi in Chania, Crete, in July 2025, on the kinetic Fokker--Planck equation on Riemannian manifolds. The initial plan for that study goes back to a 2008 project by Fabrice Debbasch, Yann Ollivier and myself; it was motivated both by physical models coming from general relativity, and the desire to further investigate the then-developing theory of hypocoercivity. The manuscript written at the time by the three of us had some shortcomings, and was left incomplete and sleeping for some fifteen years. Eventually it is on the occasion of the 2025 and 2026 Festum Pi Conferences that it was revived and improved into a coherent whole.

Thanks are due to the Mediterranean Agronomic Institute of Chania (MAICh), which once again provided the inspiring location and outstanding support for the festival. Part of this work was done as I was holding a joint chair between Universit\'e Claude Bernard Lyon 1 and Institut des hautes \'etudes scientifiques (IHES) in Bures-sur-Yvette, France. I received invaluable help from St\'ephane Mischler to catch up with recent literature on the kinetic Fokker--Planck equation, especially the bibliography of Section~3. Warm thanks are due to Fabrice Baudoin, Jean-Michel Bismut, Fran\c cois Bouchut and Gilles Lebeau for fruitful exchanges on this subject over the past twenty years. Further thanks to Sylvia Serfaty for pointing out Ref.~\cite{AAMN:KFP:24}, and to Giovanni Brigatti and Cl\'ement Mouhot for their related, beautiful contribution to the proceedings of Festum Pi 2024.

\sm



\section{Motivation}

The kinetic theory of gases and plasmas, founded by James Clerk Maxwell and Ludwig Boltzmann, is the study of evolving statistics of particles in the phase space of positions and velocities. One of its most important defining characteristics is the subtle interplay between transport processes mixing position and velocity variables, and interaction or diffusion processes acting in velocity variables. The analytical difficulties generated by this combination of conflicting structures, all throughout kinetic theory, range from tricky to formidable. This is particularly true of the nonlinear Boltzmann equation.

Among models with those features, the simplest ones -- and possibly the most important ones, even though not the most famous --  are linear kinetic Fokker--Planck equations, combining classical transport (with or without external forcing) and diffusion--friction in velocity space. Besides their own interest, they serve as preparation for more elaborate models. As a matter of fact, it was the most basic linear kinetic Fokker--Planck equation which lay at the foundation of H\"ormander's 1967 celebrated theorem on hypoelliptic regularisation; and the same model served as an entry door to the equilibration for various inhomogeneous kinetic equations, by Desvillettes and myself in the 2000's. Further, the spatially homogeneous Fokker--Planck equation plays an important role as an auxiliary diffusion equation in the study of various functional inequalities for the Boltzmann equation theory. All of this suggests that Fokker--Planck equations will continue to provide important lessons, not just for their own sake, but also for the more general picture of kinetic theory.

My memoir {\em Hypocoercivity}, initiated in 2003 and published in 2009, was developing a rather systematic study of certain quantitative functional methods for equilibration of degenerate diffusive equations. The word Hypocoercivity itself was borrowed from Thierry Gallay. The main focus of my memoir was equilibration, but I was also considering global regularity estimates, both in $L^2$ and $L^1$ context. Since then the tools therein were adapted to various models, and improved, by a number of authors. 

However, the theory was hardly developed in curved geometry, partly because of specific new difficulties, but also because the research community familiar with both kinetic theory and analysis on manifolds is extremely small. This is exactly the task that we assigned ourselves, with Fabrice Debbasch and Yann Ollivier, around 2008, inspired by models of mathematical physics coming from cosmology (Debbasch's model). By coincidence, at about the same time Jean-Michel Bismut and Gilles Lebeau were working on a kinetic Fokker--Planck equation on 1-forms (Bismut's diffusion), arising in Riemannian context to enlighten some problems from topology. So there was also the motivation to better understand and simplify their estimates.

I consider it extremely likely that the adaptation of the hypoelliptic and hypocoercive diffusion theory to curved geometry will be fruitful and insightful, even for those whose ultimate motivation is only the ``flat'' context. Think indeed of the extraordinary wealth of information which was obtained when the tools of spectral gap theory, logarithmic Sobolev inequalities and other functional inequalities were adapted from Euclidean or Gaussian setting to curved geometries.

Some of the main results in the present memoir can be summarised as follows: {\em If $M$ is a smooth compact Riemannian manifold of dimension $n$, then there are $s_0=s_0(n)$ and $\nu = \nu(M)>0$ such that if an initial probability density $f_0$ (or probability measure) is given on the tangent bundle $\TM$ with $\int f_0(x,v) |v|^{s_0}\,dx\,dv < +\infty$, then in the course of time the corresponding solution to the kinetic Fokker--Planck equation will converge like $O(e^{-\nu t})$ to the Gaussian equilibrium, in the sense of entropy.} Simple as that conclusion seems, getting there will employ a considerable amount of analysis.

The plan of these notes is as follows. Sections \ref{seckFP} to \ref{secmeet} set the scene with the main equations and problems. Sections \ref{secHC} to \ref{secspectral} are devoted to preliminaries and reminders from Riemannian geometry. Sections \ref{secfunctional} to \ref{secnash} develop the functional setup and basic functional inequalities, including Sobolev $L^2$-interpolation, Gaussian tail Poincar\'e inequality and anisotropic Nash interpolation. Sections \ref{seclebeau} to \ref{secpos} establish regularisation, localisation and positivity for the geometric kinetic Fokker--Planck operator and equation, recovering in particular Lebeau's maximal hypoellipticity and proving the smoothing effect of the evolution problem. Sections \ref{secineq}--\ref{sechypoco} and \ref{secL1conv}--\ref{secfisher} establish exponential equilibration for the evolution problem, in $L^2$ and $L^1$ with increasing geometry, while Section \ref{secqualspectr} localises the $L^2$-spectrum of the generator. Eventually a few conclusions are summarised.

Throughout the text I will use the classical notation for differential calculus in Euclidean space, including $\nabla$, $\pa_i$, $\nabla^k$, $\Delta$, etc. I write $\N = \{1, 2, 3, \ldots\}$ and $\N_0 = \{0,1,2,\ldots\}$. Most of the notation for Riemannian geometry and tensor calculus is introduced in Section \ref{secmeet}.

\bibnotes

The founding papers of kinetic theory are the works by Maxwell \cite{maxw:67} and Boltzmann \cite{boltz:weitere:72}, and the historical reference source is Boltzmann's treatise \cite{boltz:book}. A relatively recent survey of the theory of the Boltzmann equation is my 2002 monograph \cite{vill:handbook:02}; a more recent, although partial, review can be found in my 2025 Festum Pi lecture notes \cite{vill:fisher-Festum:25}. 

The founding paper for the modern theory of hypoellipticity is H\"ormander \cite{horm:hypo:67}. The modern study of hypocoercive equations was initiated by Desvillettes--Villani \cite{DV:FP:01,DV:boltz:05} with nonlinear techniques, and Helffer--Nier \cite{helfnier:witten:05} with linear techniques. Several results were improved and unified in my memoir \cite{vill:hypoco}. Further references will be provided in Section \ref{sechypo}.

A classical source on the Fokker--Planck equation is Risken \cite{risken:book}. Such equations were considered in a geometric context in Debbasch \cite{debbasch:diffcurved:04} (for models of general relativity and cosmology, in Lorentzian geometry) and Bismut \cite{bismut:hypoelliptic:05} (as an interpolation between Laplace equations and geodesic equations, to solve problems coming from infinite-dimensional de Rham and Hodge cohomology theories). Further references will be given in Sections \ref{seckFP} and \ref{secmeet}.

A rich panorama of functional inequalities in curved geometries can be found in the book by Bakry, Gentil and Ledoux \cite{BGL:book}.

\section{The kinetic Fokker--Planck equation} \label{seckFP}

Arguably the simplest nontrivial model of kinetic theory consists in supplementing free flow ($\ddot{x}=0$ in $\R^n$, Galilean motion) with a linear diffusion in velocity space: $\dot{x}$ evolving according to a Brownian motion, which is the same as the considered particle undergoing a white noise forcing. Let $(B_t)_{t\geq 0}$ the standard Brownian motion, with variance $nt$ (covariance matrix $tI_n$): so the stochastic differential equation for the position $x$ is 
\begeq\label{EDS}
\frac{d^2x}{dt^2} = \sqrt{\sigma}\, \frac{dB}{dt},
\endeq
where $\sigma>0$ is proportional to the variance of the forcing.
The associated linear diffusion equation for probability densities on $\R^n_x\times\R^n_v$ is called {\bf Kolmogorov's kinetic equation}
\begeq\label{K}
\derpar{f}{t} + v\cdot\nabla_x f = \theta \, \Delta_v f, \qquad \theta = \frac{\sigma}2,
\endeq
and admits an explicit Gaussian fundamental solution: If the initial conditions is a Dirac mass located at $(x_0,v_0)$, then the solution at $(t,x,v)$ reads 
\begin{multline} \label{FSK}
F\bigl(t,x,v|x_0,v_0\bigr) = (3\pi^2\theta)^{-n}\,t^{-2n}\,\exp 
\left [ \dps - \frac{1}{\pi^2\theta}\,
\left( \frac{3\, |x-(x_0+tv_0)|^2}{t^3} \right. \right. \\
\left. \left. - \frac{3\, [x-(x_0+t v_0)]\cdot (v-v_0)}{t^2}
+ \frac{|v-v_0|^2}{t} \right) \right].
\end{multline}

The model can be made more complex by adding a force field $F=F(t,x)$ and a dissipative mechanism, typically a linear friction, which is both the simplest physical model and the only one which maintains the Gaussian nature of velocity fluctuations. Then the equation of motion is
\begeq\label{stolambda}
\frac{d^2x}{dt^2} = \sqrt{\sigma}\, \frac{dB}{dt} + F(t,x) - \lambda\,\frac{dx}{dt},
\endeq
where $\lambda>0$ is the friction coefficient, and the associated diffusion equation is
\begeq\label{prekFP}
\derpar{f}{t} + v\cdot\nabla_x f +F(t,x)\cdot\nabla_v f = \frac{\sigma}2 \, \Delta_v f + \lambda \nabla_v\cdot (fv).
\endeq
Typically, this describes the action of a heat bath of minute molecules which on the one hand transfer energy to the considered particle and on the other hand slow it down when it accelerates too much. The equilibrium velocity distribution is obtained by solving $(\sigma/2)\,\Delta_v f + \lambda \nabla_v (fv) = 0$, whence $f$ is proportional to a Maxwell--Boltzmann statistical equilibrium of the form $e^{-|v|^2/(2kT)}$, where $k$ is Boltzmann's constant, if $T= \sigma/(2k\lambda)$. Identifying $T$ with the temperature of the heat bath (which at equilibrium is transferred to the particle), we arrive at the {\bf kinetic Fokker--Planck equation}
\begeq\label{kFP}
\derpar{f}{t} + v\cdot\nabla_x f +F(t,x)\cdot\nabla_v f =\lambda \bigl( kT \Delta_v f + \nabla_v\cdot (fv) \bigr).
\endeq
Here as in all the sequel, $f$ is the probability density for the considered particle in phase space; so  $\int f\,dx\,dv=1$.

For clarity, let us make the connection with more popular models set in position space, rather than position and velocity space. Of course the special case $\lambda=0$ leads back to the equation of classical mechanics (the statistical version of Newton's equation)
\begeq\label{kgeod}
\derpar{f}{t} + v\cdot\nabla_x f +F(t,x)\cdot\nabla_v f = 0.
\endeq
But now consider the situation in which $\lambda$ is extremely large, so that the motion is nearly static: $\ddot{x} \simeq 0$, still $\lambda \dot{x}$ is not necessarily small. If both $\sigma$ and $F$ are of the same order as $\lambda$, then we may neglect the left-hand side (and only the left-hand side) in \eqref{stolambda} and arrive at the classical first-order stochastic differential equation
\begeq\label{sto1}
\dot{x} = \frac{\sqrt{\sigma}}{\lambda} \frac{dB}{dt} + \frac{F(t,x)}{\lambda}.
\endeq
Then the associated probability density $\rho=\rho(t,x)$ satisfies the familiar drift-diffusion equation
\begeq\label{dd}
\derpar{\rho}{t} = \frac{\sigma}{2\lambda^2}\, \Delta \rho - \frac1{\lambda}\, \nabla\cdot (\rho F).
\endeq
In short, the limit $\lambda\to 0$ yields a conservative Hamiltonian flow in phase space \eqref{kgeod} and the limit $\lambda\to\infty$ yields a drift-diffusion equation in position space \eqref{dd}; if $F=0$ these two extremal cases are just the usual geodesic flow and the usual heat flow, respectively.
\sm

Back to the kinetic model. Two special cases of \eqref{kFP} are of particular interest:
\sm

\bul When $F(t,x) = -\omega x$ (harmonic confinement), then the coefficients in \eqref{kFP} are all linear in $(x,v)$ and the fundamental solution is still an explicit Gaussian -- not simple, but explicit. For instance, if $\lambda> 2\omega>0$, then it can be written as
\begeq\label{FSK2}
\frac1{(2\pi)^{n} \sqrt{\det A(t)}}\, \exp\biggl( -\frac12 \Bigl\< A^{-1}(t) \bigl(x-\ov{x}(t),v-\ov{v}(t)\bigr), \bigl(x-\ov{x}(t),v-\ov{v}(t)\bigr)\Bigr\>\biggr),
\endeq
where the $2n\times 2n$ positive symmetric matrix $A(t)$ is of the form
\begeq\label{AFSK}
A = \frac{\lambda kT}{\lambda^2- 4 \omega^2} \
\left[ \begin{array}{cc} a_{xx} I_n & a_{xv} I_n \\[2mm]
a_{xv} I_n & a_{vv} I_n
\end{array}\right],
\endeq
and
\[ a_{xx}(t) = \frac{\alpha^+ + \alpha^-}{\alpha^+ \alpha^-} 
+ \frac{4}{\alpha^++\alpha^-} \bigl( e^{-(\alpha^++\alpha^-)t}-1\bigr)
- \frac1{\alpha^+} e^{-2\alpha^+t} - \frac1{\alpha^-}e^{-2\alpha^-t}\]
\[ a_{xv} (t) = \bigl( e^{-\alpha^+t} - e^{-\alpha^-t}\bigr)^2\]
\[ a_{vv}(t) = \alpha^++\alpha^- + \frac{4 \alpha^+\alpha^-}{\alpha^++\alpha^-}
\bigl( e^{-(\alpha^++\alpha^-)t}-1\bigr) - \alpha^+e^{-2\alpha^+t} - \alpha^-e^{-2\alpha^-t},\]
and 
\begeq\label{mu+-}
\alpha^{\pm} = \frac12 \bigl( \lambda \pm \sqrt{\lambda^2-4\omega^2}\bigr),
\endeq
and the values $\ov{x}(t)$, $\ov{v}(t)$ are obtained from $x_0$, $v_0$ via
\[ \ov{x}(t) = \left( \frac{\alpha^+ e^{-\alpha^-t} - \alpha^- e^{-\alpha^+t}}{\alpha^+-\alpha^-}\right)\, x_0 
+ \left(\frac{e^{-\alpha^-t}-e^{-\alpha^+t}}{\alpha^+-\alpha^-}\right)\, v_0,\]
\[ \ov{v}(t) = \omega^2 \left( \frac{e^{-\alpha^+t}-e^{-\alpha^-t}}{\alpha^+-\alpha^-}\right) x_0 
+ \left(\frac{\alpha^+ e^{-\alpha^+t} - \alpha^- e^{-\alpha^-t}}{\alpha^+-\alpha^-}\right)\, v_0.\]

\bul More generally, when $F(t,x) = -\nabla V(x)$ (potential confinement), with $V$ not necessarily quadratic, then
\begeq\label{kFPV}
\derpar{f}{t} + v\cdot\nabla_x f -\nabla V(x)\cdot\nabla_v f =\lambda \Bigl( kT \Delta_v f + \nabla_v\cdot (fv) \Bigr).
\endeq
Then there is in general no explicit fundamental solution but there is still an explicit equilibrium distribution $\mu$ whose density is proportional to the negative exponential of the potential energy:
\begeq\label{finfty}
\mu(dx\,dv) = f_\infty(x,v)\,dx\,dv,\qquad
f_\infty(x,v) =\frac{e^{- \frac{1}{kT} \bigl(V(x) + \frac{|v|^2}2\bigr)}}{Z},
\endeq
with $Z$ a normalising constant. Then the relative density of $f$ with respect to $f_\infty$ satisfies the neat linear equation
\begeq \label{kFPh}
h = \frac{f}{f_\infty} \Longrightarrow \qquad
\derpar{h}{t} + v\cdot\nabla_x h - \nabla V(x) \cdot\nabla_v h 
= \lambda\bigl ( kT \Delta_v h - v\cdot\nabla_v h\bigr).
\endeq
Note that the differential operator on the left hand side ($v\cdot\nabla_x - \nabla V(x) \cdot\nabla_v$) is skew-symmetric in $L^2(\mu)$, while the one on the right hand side ($\lambda( kT \Delta_v - v\cdot\nabla_v )$) is symmetric. Moreover, there is a large family of Lyapunov functionals: if $C$ is any convex function on $\R_+$, then
\begeq\label{Lyapidentity} 
\frac{d}{dt} \iint_{\R^n\times\R^n} C\left( \frac{f}{f_\infty}\right) f_\infty\, dx\,dv
= - kT \iint C''\left(\frac{f}{f_\infty}\right) \left| \nabla_v \left(\frac{f}{f_\infty}\right) \right|^2\, f_\infty\, dx\,dv 
\leq 0.
\endeq
\sm

When $F$ does not derive from a potential, there may be a stationary solution but it does not qualify as a thermodynamical equilibrium. Interesting as it is, I will not consider this situation here.

Of particular interest in \eqref{Lyapidentity} are the cases $C(z) = (z\log z - z+1)$ (or $z\log z$), corresponding to Boltzmann's $H$ functional, or negative of the entropy, and $C(z) = (z-1)^2$, corresponding to fluctuations around $f_\infty$. Those two cases correspond respectively to the nonlinear problem and the linearisation problem in the theory of the Boltzmann equation. Considering the particle density, or the fluctuations around equilibrium, leads to two distinct settings, hereafter referred to as $L^1$ and $L^2$:
\sm

\bul {\em The $L^1$ problem:} Given an initial datum $f_0=f_0(x,v) \geq 0$ such that $\int f_0 =1$ (and possibly extra regularity conditions), establish uniform regularity bounds on $f$, the solution of \eqref{kFP}, and show that it converges in large time to the equilibrium distribution $f_\infty$.
\sm

\bul {\em The $L^2$ problem:} Given an initial datum $h_0=h_0(x,v)$ such that $\int h_0^2\,d\mu <\infty$ (and possibly extra regularity conditions), establish uniform regularity bounds on $h$, the solution of \eqref{kFPh}, and show that it converges in large time to the constant function $\int h_0\,d\mu$.
\sm

Note that the $L^2$ problem does not involve any sign condition on $h_0$. Further note that $\int h_0^2\, d\mu = \int (f_0^2/f_\infty)\,dx\,dv$, so the $L^2$ condition demands both more integrability and faster decay at infinity.
\sm

Let us introduce notation for the kinetic diffusion-friction operators appearing in those problems:
\begeq\label{calLf}
{\cal L}_v f = -\lambda \Bigl (kT \Delta_v f + \nabla_v \cdot (fv)\Bigr) 
= -\lambda \Bigl( kT \Delta_v f + v\cdot\nabla_v f + nf \Bigr)
\endeq
\begeq\label{Lh}
L_vh = -\lambda \Bigl (kT \Delta_v h - v\cdot \nabla_v h\Bigr).
\endeq
(Careful, these are the opposite of the respective generators, so ${\cal L}_v$ and $L_v$ are typically positive operators.)

A further extension is the tensor-valued diffusion equation, when the function of interest is not a distribution of mass, but a tensor-valued distribution; that is, a function $\F=\F(t,x,v)$ valued in vectors, or matrices, or higher-order tensors. Of course the Fokker--Planck equation can be defined for tensors as well as functions, using covariant derivation. But there is another ``natural'' evolution equation: the one which is satisfied by the tensor of $N$-derivatives of $f$ in \eqref{kFP}, or of $h$ in \eqref{kFPh}. Those equations differ only in the right-hand side. Fix $N\in\N$. Any term in the $N$-tensor of $(x,v)$ derivatives of a function $g$ at order $N$ is obtained by differentiating $I$ times in the $x$ variable (say in directions $i_1,\ldots i_I$) and $J$ times in the $v$ variable (say in directions $j_1,\ldots,j_J$) with $I+J=N$. It will be convenient to write
\[ \Bigl(\nabla_{x,v}^N u\Bigr)_{i_1 \ldots i_I}^{j_1\ldots j_J} 
= \nabla_{i_1\ldots i_I}^{j_1\ldots j_J}  u = \derpar{{}^I}{x_{i_1}\ldots \pa x_{i_I}}
\derpar{{}^J}{v_{j_1}\ldots \pa v_{j_J}} u. \]
Using the commutation of $\Delta_v$ with $\nabla_x$ and $\nabla_v$, and the identities
\[ (\nabla_x,\nabla_v)( v\cdot\nabla_x g) = \Bigl( (v\cdot\nabla_x)(\nabla_xg), (v\cdot\nabla_x)(\nabla_v g) + \nabla_x g\Bigr) \]
\[ (\nabla_x,\nabla_v) (v\cdot\nabla_v g) = \Bigl( (v\cdot\nabla_v)(\nabla_xg), (v\cdot\nabla_v)(\nabla_v g) + \nabla_v g\Bigr) \]
it is easy to arrive at the following expressions: 
\begeq\label{vDNg}
 \nabla_{i_1\ldots i_I}^{j_1\ldots j_J} (v\cdot\nabla_x g)
= (v\cdot\nabla_x) \bigl(\nabla_{i_1\ldots i_I}^{j_1\ldots j_J} g\bigr )
+ \sum_{\ell=1}^J \nabla_{i_1\ldots i_I j_\ell}^{j_1\ldots j_{\ell-1} j_{\ell+1}\ldots j_J} g
\endeq
\begeq\label{DNcalLg}
 \nabla_{i_1\ldots i_I}^{j_1\ldots j_J} ({\cal L}_v f)
 = {\cal L}_v \nabla_{i_1\ldots i_I}^{j_1\ldots j_J} f + \lambda J \nabla_{i_1\ldots i_I}^{j_1\ldots j_J} f
 \endeq
 \begeq\label{DNLg}
 \nabla_{i_1\ldots i_I}^{j_1\ldots j_J} (L_vh)
 = L_v \nabla_{i_1\ldots i_I}^{j_1\ldots j_J} h - \lambda J \nabla_{i_1\ldots i_I}^{j_1\ldots j_J} h
 \endeq
 Of course in the expression the terms are invariant under permutation, from Schwarz's lemma.

From this follows a natural generalisation of the geodesic derivation $v\cdot\nabla_x$ to $(I,J)$-tensors, not necessarily $N$-differentials:
\begeq\label{Xiflat}
\bigl( \Xi X)_{i_1\ldots i_I}^{j_1\ldots j_J}
= (v\cdot\nabla_x) X_{i_1\ldots i_I}^{j_1\ldots j_J}
+ \sum_{\ell=1}^J X_{i_1\ldots i_I j_\ell}^{j_1\ldots j_{\ell-1} j_{\ell+1}\ldots j_J}.
\endeq
(Again, this is not the same as letting $v\cdot\nabla_x$ act covariantly on the tensor.)
And as well there are natural generalisations of the operators ${\cal L}_v$ and $L_v$ which may be applied to $(I,J)$-tensors, not necessarily $N$-differentials:
\begeq\label{calLF}
\bigl( ({\cal L}_v)_I^J \F \bigr)_{i_1\ldots i_I}^{j_1\ldots j_J}
= {\cal L}_v(\F_{i_1\ldots i_I}^{j_1\ldots j_J}) - \lambda J \F_{i_1\ldots i_I}^{j_1\ldots j_J},
\endeq
\begeq\label{LX}
\bigl( (L_v)_I^J X \bigr)_{i_1\ldots i_I}^{j_1\ldots j_J}
= L_v(X_{i_1\ldots i_I}^{j_1\ldots j_J}) + \lambda J X_{i_1\ldots i_I}^{j_1\ldots j_J}.
\endeq
In alternative notation, $(L_v)_I^J X = -\lambda (kT \Delta_v - v\cdot\nabla_v - J)X$, where $J$ is the number of ``upper indices'' of $X$. 

Now we may consider the associated evolution problem on $N$-forms. The $L^1$ problem is not so natural, since the nonnegativity condition does not seem to make any sense, but a natural $L^2$ problem for tensors can be formulated, with the equation
\begeq\label{kFPX}
\pa_t X + \Xi X + (L_v)_I^J X =0.
\endeq

\bul {\em The $L^2$ problem for tensors:} Given an integer $N\geq 1$ and an initial tensor field $X_0=(X_0(x,v))_{i_1\ldots i_N}^{j_1\ldots j_J}$ indexed by all families $(i_1,\ldots i_I), (j_1,\ldots, j_J)$ such that $I+J=N$, invariant under permutations of either upper or lower indices, and such that $\int |X_0|^2\,d\mu <\infty$ (and possibly extra regularity conditions), establish uniform regularity bounds on $X$, the solution of \eqref{kFPX}, and study its long-time behavior.

\begin{Rk} Even if, for tensors, the $L^1$ problem seems less natural, \eqref{calLF} can be useful to study situations in which the decay at infinity is not so fast, so that the problem is set, say, in $L^2$ with Lebesgue measure rather than $L^2$ with Gaussian measure.
\end{Rk}

As a final remark for this introduction, if one is not interested in some asymptotic regime (small or large friction, small or large temperature), one may choose units of space and time in such a way that
\begeq\label{lambdanorm} \lambda = 1,\qquad kT = 1, \endeq
thereby reducing to the case where all coefficients in \eqref{kFP} or \eqref{kFPh} are unity, at the expense of possibly rescaling the force term $F$.

\bibnotes

Classical sources on the kinetic Fokker--Planck equation are the treatises by Chandrasekhar \cite{chandr:43}, where astrophysics plays a key role; and by Risken \cite{risken:book}, with a particular emphasis on modelling and computations. The explicit solution for linear forcing is in Section~10.2 of the latter reference. The kinetic Fokker--Planck equation is also often called the Kramers or Klein--Kramers or Kramers--Chandrasekhar or Ornstein--Uhlenbeck equation. The name Fokker--Planck equation is also often given to the diffusion-drift model \eqref{dd}, which itself is also called the Langevin equation. The terminology Fokker--Planck equation is also used sometimes for the model of kinetic collisions in plasmas which was established by Landau \cite{landau:Coulomb:36}, or for more general models of collisions in plasmas \cite{delcbers:book:94}. It is also common to dub Fokker--Planck equation any linear equation describing the evolution of the statistics of an ensemble of particles undergoing deterministic and stochastic processes; but then this is also called the Kolmogorov forward equation (thus, much more general than the particular Kolmogorov kinetic equation \eqref{K}). So any attempt of consistency in the terminology is certainly doomed in this large field, which comprises certainly hundreds of thousands of publications. 

Still, the literature devoted to the kinetic model \eqref{kFP} is minuscule compared to the huge literature devoted to the drift-diffusion (or Langevin) equation \eqref{dd}. Models described by the latter equation are immensely popular, in all areas of science, from chemistry to machine learning, especially when the force $F$ comes from a potential $V$; but from the point of view of classical mechanics one should recall that they come from a high-friction approximation.

As for the mathematical analysis of those diffusion equations, both \eqref{kFP} and \eqref{dd}, the $L^2$ problem has been much more studied than the $L^1$ problem, in particular because $L^2$ is suitable for spectral analysis and energy estimates, also because it naturally arises as the linearisation of various nonlinear problems around an equilibrium. But Boltzmann's entropy and $H$ Theorem are most naturally constructed and stated in the $L^1$ setting, also leading to a physically meaningful $L \log L$ estimate. In some sense the $L\log L$ and $L^2$ settings correspond to extremal cases of a whole class of problems, for functionals lying ``between'' the entropic and the quadratic nonlinearities; see e.g. Arnold, Markowich, Toscani and Unterreiter \cite{AMTU:FP:01} for \eqref{dd}. The $L^1$ problem for the general equation \eqref{kFP}, in terms of regularity and equilibration, was considered in \cite{vill:hypoco}.

The equation for tensors is considered by Bismut and Lebeau in the geometric case \cite{bismut:hypoelliptic:05,bismutlebeau:hypo:book}, but it already makes sense (and is already nontrivial) in Euclidean geometry, which is why I chose to present it at this early stage.

\section{Hypoellipticity and hypocoercivity} \label{sechypo}

The kinetic Fokker--Planck equation \eqref{kFP} displays diffusion only in the velocity variable, while the position variable appears only through the first-order operator $v\cdot\nabla_x - \nabla V(x)\cdot \nabla_v$. The latter operator is skew-symmetric in $L^2(dx\,dv)$ and in $L^2(\mu)$, and therefore is in itself associated neither with regularisation nor with equilibration; it is actually the generator of the classical Hamiltonian system with force potential $V$, which is conservative in any (mathematical) sense of the word. However, the fundamental solution \eqref{FSK} is positive everywhere and smooth in all variables $x$ and $v$ for $t>0$, with regularisation time scales $t^{1/2}$ in $v$ and $t^{3/2}$ in $x$. Also, as $t\to\infty$ the fundamental solution \eqref{FSK2} displays exponential convergence, as a time-dependent function of $(x,v)$, to the global equilibrium, and the rate of convergence is the lowest eigenvalue $\alpha^- = (\lambda - \sqrt{\lambda^2-4\omega^2})/2$. (When $\omega> \lambda/2$, the rate is just $\lambda/2$.)

These phenomena are not exceptional: They are observed with many diffusion generators $L$, when the truly diffusive part and the conservative part satisfy certain geometric conditions. By far the most famous such criterion was stated and proved by Lars H\"ormander in 1967, and is nearly optimal. It is expressed in terms of {\em commutators} and {\em Lie brackets}. 

Recall the basics of that formalism. A derivation $X$ on $\R^N$ is a local differential operator on functions of $x$, taking smooth functions to smooth functions, and such that $X(fg) = (Xf) g + f(Xg)$. To any derivation is associated a unique smooth vector field $a=(a_i(x))_{1\leq i\leq N}$ such that $X f = a\cdot\nabla f = \sum_i a_i(x)\, \pa f/\pa x_i$. If any two derivations $X,Y$ are given with respective vector fields $a$ and $b$, let $[X,Y] = XY - YX$ be their bracket, or commutator, then $[X,Y]$ is also a derivation, associated with the vector field $c_i = \sum_j (a_j \,\pa_j b_i - b_j\, \pa_j a_i)$, or by abuse of notation $c=[a,b]$. Thus smooth vector fields equipped with that bracket form a Lie algebra.

Now comes H\"ormander's criterion. Let $L = -( \sum X_i^2 + X_0 + m)$, where the $X_i$'s ($0\leq i\leq k$) are derivation operators in $\R^N$ associated with smooth vector fields $a_i$ respectively, and $m$ is a smooth function. If the Lie algebra generated by the $X_i$'s (that is, the linear space generated by successive brackets obtained from the $X_i$'s) span all directions, everywhere, then the resolution of the operator is regularising: that is, if $Lf$ is smooth then $f$ is smooth; and a solution of $\pa_t f + L f =0$, starting from an initial condition $f_0$, is smooth for any $t>0$ even if $f_0$ is not.

H\"ormander's theory comes with explicit estimates in Sobolev spaces, allowing to relate the Sobolev regularity of $u$ in a ball to that of $Lu$ in a larger ball, with a gain of regularity; and that gain is typically fractional, depending in particular on the number of commutators. Alternatively, for the evolution problem, it yields regularisation time scales.

Let us see how this works on \eqref{kFP}. In this case we work in Euclidean space $\R^n_x\times\R^n_v$, and there are $n+1$ derivations to be considered: 
\[ X_0 = v\cdot\nabla_x - \nabla V\cdot\nabla_v - v\cdot\nabla_v , \qquad 
X_i = \derpar{}{v_i}\qquad (1\leq i\leq n).\]
Equivalently,
\[ a_0(x,v) = [v, -\nabla V(x)-v],\qquad a_i(x,v) = [0, e_i],\]
where $(e_i)_{1\leq i\leq n}$ is the standard orthonormal basis of $\R^n$. Then for each $i$,
\[ [X_i, X_0] = \derpar{}{x_i} - \derpar{}{v_i}, \]
or equivalently $(e_i,-e_i)$, providing all the missing $x$-directions, and for each position $(x,v)$ in phase space. So only one commutator suffices. 

By convention a weight~1 is attributed to each $X_i$ appearing as a square in the expression of $L$, and by ``homogeneity'' a weight 2 will be attributed to $X_0$; then the weight of a commutator is the sum of the weights of the involved derivations (and the weight of a sum of operators is the maximum of the involved weights), and the weight of an operator is the minimal weight of an iterated commutator generating that operator. So in the case of the kinetic Fokker--Planck the weight of each $\pa_{v_i}$ is 1, the weight of $v\cdot\nabla_x - \nabla V\cdot\nabla_v - v\cdot\nabla_v$ is~2, and the weight of each $\pa_{x_i}$ is $1+2=3$. Those weights 1 for $v$-derivatives and 3 for $x$-derivatives provide a quantitative ``explanation'' for the regularisation rates in $v$ and $x$ appearing in \eqref{FSK}. By regularisation rates it is meant for instance that, locally in short time,
\[ \|\nabla^k_v f\|_{L^2(B)} = O \left( \frac{\|f_0\|_{L^2(B')}}{t^{k/2}}\right), \qquad
\|\nabla^k_x f\|_{L^2(B)} = O \left( \frac{\|f_0\|_{L^2(B')}}{t^{3k/2}}\right), \]
where $B$ is a ball in $\R^{2n}$ and $B'$ is a strictly larger ball. So the $v$-regularity appears just as fast as for the heat equation in $\R^n_v$, while the $x$-regularity has a rate (measured in the exponent of the typical regularisation time) 3 times slower.

Here is a heuristics to understand this factor 3 in terms of spreading singularities. Without loss of generality, fix dimension $n=1$. Consider the diffusive spreading of a singularity in $v$ (along a line parallel to the $v$-space) or in $x$ (along a line parallel to the $x$-space). For the $v$-singularity, diffusion in the velocity variable will act immediately to spread it in small time $t$, by a typical distance of $\sqrt{t}$. so the typical gradient will be of order $1/\sqrt{t}$ if the function itself is bounded. For the $x$-singularity, diffusion does not help, as it leaves the singularity invariant. But the transport operator will {\em twist} the direction of the singularity, by an angle $O(t)$ in small time. After that, the $v$-singularity will spread it by a distance $O(\sqrt{t})$ in the $v$ direction, and from the combination of both, the spreading in the $x$ direction will be about $\sqrt{t}\times (\sin t) \simeq t^{3/2}$. 

\begin{figure}
\def\svgwidth{0.5\textwidth}
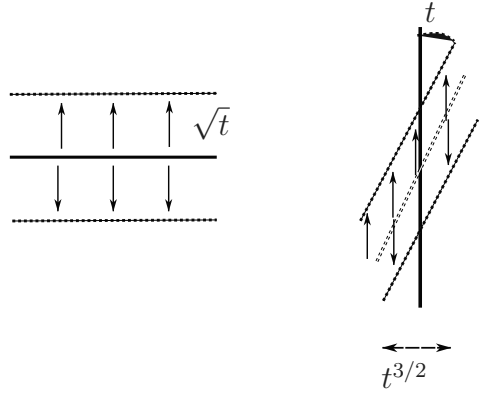

\caption{On the left, a singularity in velocity variable is spread by width $O(t^{1/2})$ through the action of the diffusion in velocity space. On the right, a singularity in position variable is spread by width $O(t^{3/2})$ by the combined action of transport twisting (or tilting) and vertical diffusion.}
\label{figcollision}

\end{figure}

After H\"ormander's pioneering work, there was an intense activity to better understand and refine his criterion, determine generally optimal exponents, classify operators in terms of the Lie algebra of derivations. Various tools of real and complex analysis were put to good use -- spectral theory, harmonic analysis, pseudodifferential calculus. Also a probabilistic version was developed. The main emphasis was on hypoelliptic diffusions of the form $\sum X_i^2$, rather than $X_0 + \sum X_i^2$.

In the 2000's, as the qualitative study of kinetic equations was gaining strength, the general formalism and techniques of hypoellipticity were carefully applied to the particular case of the kinetic Fokker--Planck equation \eqref{kFP} in Euclidean space, to get global estimates of regularisation and equilibration. Simultaneously, more elementary methods were devised to handle this specific equation and other models of interest, based on Sobolev norms including well-chosen lower-order terms. 

At the same time, the problem of equilibration for such equations was studied in its own right, leading to the concept of {\bf hypocoercivity}. By definition, a linear unbounded operator $L$ on a Hilbert space ${\cal H}$ is said to be hypocoercive on a subspace $\tilde{\cal H}$ (typically the orthogonal of the kernel of $L$) if for all $t\geq 0$, $\|e^{-tL}h_0\|\leq C e^{-\lambda t} \|h_0\|$, for all $h$ in $\tilde{\cal H}$. If $C=1$ {\em and} the norm is Hilbert, this is equivalent to coercivity; but if the norm is not Hilbert or if $C>1$, then this is a weaker concept. 

Hypoellipticity and hypocoercivity are often found together, but one can also have one without the other (and vice versa). There is no general criterion for hypocoercivity which would be as neat as H\"ormander's, but there is a general abstract frame as follows. Let $L =  \sum A_i^*A_i + B$ where $(A_i)_{1\leq i\leq k}$ and $B$ are arbitrary linear unbounded operators on a Hilbert space ${\cal H}$, and $B$ is antisymmetric; and let $C_{i,j}$ be successive commutators of $A_i$ with $B$ (that is, $C_{i,0} = A_i$, and $C_{i,j+1} = [C_{i,j},B]$). Then if 
\begeq\label{sumcoercive}
\sum_{i\leq k,\ j\leq N_c} C_{i,j}^*C_{i,j} \geq \kappa\, \Id \qquad \text{for some $\kappa>0$}
\endeq
and the commutators satisfy certain bounds (whose general form is omitted here), then $L$ is hypocoercive; and again, this can be quantified by differential equations in certain explicit Hilbert norms.

Here are two archetypal examples for \eqref{kFPh}, both assuming that $|\nabla^2V| = O (1+|\nabla V|)$ (that is, $V$ grows at most exponentially at infinity, in a sense):
\sm

\bul To prove instantaneous global regularisation from $L^2(\mu)$ to the Sobolev space $H^1(\mu)$, let
\begeq\label{Normh}
{\cal F}(t,h(t)) = \iint h^2\,d\mu + a t \iint |\nabla_vh|^2\,d\mu + 2bt^2 \iint \nabla_v h \cdot \nabla_x h \,d\mu + c t^3 \iint |\nabla_x h|^2\,d\mu
\endeq
with $a,c>0$ and $b\in\R$ such that $b^2<ac$ (so that for any $t$ the squared norm defined by ${\cal F}$ is comparable, from above and below, to the same expression with $b$ replaced by~0). Then for $a,b,c$ well chosen, 
\[ \frac{d}{dt} {\cal F}(t,h(t)) \leq 0,\]
and in particular
\begeq\label{xvappear} 
\|\nabla_v h\|_{L^2(\mu)} \leq \frac{C\,\|h_0\|_{L^2(\mu)}}{t^{1/2}}, \qquad
\|\nabla_x h\|_{L^2(\mu)} \leq \frac{C\,\|h_0\|_{L^2(\mu)}}{t^{3/2}}.
\endeq

\sm
\bul To prove equilibration in $L^2(\mu)$, consider
\begeq\label{equilibrh}
{\cal N}(h)^2 = \iint h^2\,d\mu + a \iint |\nabla_v^2h|^2\,d\mu + 2b \iint \nabla_v h \cdot \nabla_x h \,d\mu + c \iint |\nabla^xh|^2\,d\mu
\endeq
also with $a,c>0$ and $b\in\R$ such that $b^2<ac$ (so that ${\cal N}^2$ is comparable, from above and below, to the same expression with $b$ replaced by~0). Then for $a,b,c$ well chosen there is $\kappa>0$ such that
\begeq\label{exph} \iint h\,d\mu = 0 \Longrightarrow \qquad
\frac{d}{dt}  {\cal N}(h)^2 \leq - 2\kappa\, {\cal N}(h)^2.
\endeq

Of course, the combination of \eqref{xvappear} and \eqref{equilibrh} implies that $\|h(t) \| = O (\|h_0\|\,e^{-\kappa t})$. 
\sm

Reinforcing the assumptions on $V$, one can make the norms more demanding. There are also related techniques to handle the $L^1$ problem: For the regularisation, Lyapunov functionals cooked up with classical Sobolev spaces and interpolation arguments combining moments, $L^1$ and Sobolev norms; and for the equilibration, information-theoretical functionals combining Boltzmann and Fisher informations. I will come back to those issues later in the geometric context. 

For now, to conclude the presentation of the Euclidean setting, here are two typical global estimates, combining global regularisation and equilibration for the kinetic Fokker--Planck equation:
\sm

\bul If $|\nabla V|\to\infty$ at infinity and $|\nabla^j V| = O(|\nabla V|)$ for all $j\in\N$, then there is $\kappa>0$ such that for any $k\in\N$, there is $C_k>0$ such that solutions of \eqref{kFPh} satisfies
\begeq\label{eqhVkL2}
\|\nabla_v^{3k} h\|_{L^2(\mu)} + \|\nabla_x^k h\|_{L^2(\mu)}
\leq C_k\,\left(1 + \frac1{t^{3k}}\right)\, e^{-\kappa t}\, \|h_0\|_{L^2(\mu)}.
\endeq
\sm

\bul If $V$ is uniformly convex outside of a compact set, and $|\nabla^jV| = O(1)$ for all $j\geq 2$, and $E = \int f(x,v) [V(x)+|v|^2/2]\,dx\,dv <\infty$, then there is $\lambda>0$ such that for any $k\in\N$, there are $\alpha_k>0$, $\theta_k\in (0,1/2)$, $C_k>0$, only depending on $E$ and $k$, such that solutions of \eqref{kFPV} satisfy
\begeq\label{eqhVkL1}
\bigl\|\nabla_v^{3k} (f-f_\infty)\bigr \|_{L^2} +\bigl \|\nabla_x^k (f-f_\infty) \bigr\|_{L^2}
\leq C_k\,\left(1 + \frac1{t^{\alpha_k}}\right)\, e^{-\lambda t}\, \left(\iint f \log \frac{f}{f_\infty}\right)^{\theta_k}.
\endeq
Here the Lebesgue norms are with respect to the Lebesgue measure $dx\,dv$, so these are classical Sobolev spaces. Moreover, with a bit more work, one could replace the relative information $\int f\log (f/f_\infty)$ by the square of a weak distance between probability densities, such as an optimal transport (Wasserstein) distance.

\bibnotes

What I called here conservative/dissipative parts of the generator are sometimes called the reversible/irreversible parts, as in \cite[Chapter~10]{risken:book}.

For the Kolmogorov operator there are many ways to prove hypoellipticity; apart from the fundamental solution \eqref{K} this can be done by Fourier analysis as in Nier \cite[Section~1]{nier:pekin:06}; this is pedagogical but neither robust nor optimal.

The modern general theory of hypoellipticity was founded by H\"ormander \cite{horm:hypo:67}, with an amazing degree of maturity from the start. Soon after, J.~J. Kohn \cite{kohn:hypo:69}, Oleinik and Radkevi\v{c} \cite{oleinikradkevic:book} showed the robustness and flexibility of the then-young pseudodifferential calculus to establish hypoelliptic estimates in $L^2$. A series of papers by Folland, Stein and Rothschild, culminating in Rothschild--Stein \cite{rothschildstein:nilp:76}, adapted the machinery of harmonic analysis, singular operators, Calder\'on--Zygmund theory, to get optimal regularity estimates and to extend them from $L^2$ to $L^p$. Archetypal settings have been the Heisenberg group and more generally Carnot groups. Folland \cite{folland:75} reviews all that progress. H\"ormander himself addresses the topic in his famous treatise on linear differential operators \cite[Chapter 22]{hormander:LPDO:3}. This has turned into a huge area which I do not even attempt at describing. Modern surveys were written by Bramanti \cite{bramanti:invitation:book}, Bramanti and Brandolini \cite{bramantibrandolini:hormander:book}.

When nonlinear kinetic theory was developed in the nineties, in the wake of the works of DiPerna--Lions and others, robust hypoellipticity methods, adapted for problems with low regularity (sometimes aiming at gaining as tiny as just compactness), were obtained by combining velocity-averaging lemmas with energy estimates -- a method which was already in germ in H\"ormander's paper. See e.g. Lions~\cite{lions:landau:94} or Alexandre--Villani~\cite{AV:boltz:02} for applications to nonlinear collisional kinetic equations of hypoelliptic nature.

From the 2000's on, explicit estimates of global regularisation and equilibration for the kinetic Fokker--Planck equation, first in the works of Helffer--Nier \cite{helfnier:witten:05} and H\'erau--Nier \cite{heraunier:FP:04}, combining H\"ormander's method with tools from semiclassical limit and spectral or pseudospectral estimates. The work of Bismut and Lebeau \cite{bismutlebeau:hypo:book} was in the same vein.

The modern theory of hypocoercivity, based on Lyapunov functionals and norms, was born in the beginning of the 2000's. The word ``hypocoercivity'' itself was suggested to me in private discussion by Gallay in 2002. This was one of the main themes in my lecture at the 2006 International Conference of Mathematicians in Madrid \cite{vill:icm}, and it was also the main focus of my 2009 memoir \cite{vill:hypoco}, including estimates \eqref{eqhVkL2} and \eqref{eqhVkL1}, also showing how a number of examples from mathematical physics can be recast in this context.

One of my initial motivations was to understand in a more elementary way the methods of Helffer, Nier and H\'erau, and to improve their results. Also it was important, for further progress, to clearly distinguish the regularisation problem from the equilibration problem, and to distinguish those statements that were of abstract nature (that one may formulate in Hilbert spaces, for instance) from those which need real analysis. The use of mixed terms $\int \nabla_x h \cdot \nabla_v h$ drew inspiration from earlier works by Guo \cite{guo:landau:02} and Talay \cite{talay:stochHamilt:02}. The devising of a norm that is ``dynamically'' adapted to the evolution equation is also reminiscent of a classical proof of a classical result (recalled in the classical treatise by Arnold~\cite{arnold:EDO:mir}) in the theory of ordinary differential equations: {\em If a vector field has an equilibrium position for which all eigenvalues of the linearised equation have a negative real part, then the convergence near equilibrium is exponential.}

An important addition is provided by Dolbeault--Mouhot--Schmeiser \cite{DMS:hypo:15}, who develop a general construction of norms that are suitable for hypocoercive estimates, but do not involve any smoothness -- say, equivalent to the $L^2(\mu)$ norm. This is achieved typically by damping the derivatives with suitable regularising operator; it draws inspiration from H\'erau \cite{herau:linboltz:06}.

A notable functional reformulation of both global hypoellipticity and hypocoercivity, closer in spirit to H\"ormander's original paper, was achieved by Albritton--Armstrong--Mourrat--Novack \cite{AAMN:KFP:24} for the kinetic Fokker--Planck equation (in Euclidean geometry, possibly with boundaries); I shall come back to their work in Sections \ref{seclebeau} and \ref{sechypoco}.

Achleiter, Arnold, Mehrmann and Nigsch \cite{AAMN:hypo:25} study various reformulations of hypocoercivity in Hilbert spaces, obtaining necessary and sufficient conditions for bounded dissipative operators to be hypocoercive.

Here is a short list of further related work on hypocoercivity for various models of kinetic Fokker--Planck type:  Eckmann and Hairer \cite{eckmannhairer:uniqueness:01,eckmannhairer:hypo:03}, Rey-Bellet and Thomas \cite{reybelletthomas:anharmonic:00,reybelletthomas:exp:02} (for models arising in statistical mechanics, such as heat conduction), Mouhot and Neumann \cite{mouhotneumann:06}, various combinations of Bouin--Dolbeault--Lafl\`eche--Mouhot--Schmeiser \cite{BDL:frachypo:22,BDLS:hyposub:21,BDS:hypoweak:20,BM:hypofluid:22,DMS:hypo:15} (including fractional diffusions, weak confinement, fluid approximations, etc), various combinations of Cattiaux--Lebeau--Nasreddine--Puel \cite{CNP:heavycrit:19,LP:heavy:19,NP:heavy:15} (diffusion approximation in the case of heavy tails). I will also quote H\'erau \cite{herau:linboltz:06} for an early work on the linear Boltzmann model, Mischler--Qui\~{n}inao--Touboul \cite{MQT:NN:16} for a model of neural networks, Menegaki \cite{menegaki:anharmonic:20} for a new look at models of heat propagation... Nonlinear interaction models (Vlasov--Fokker--Planck kinetic equations) were considered by various authors, in particular Guillin--Liu--Wu--Zhang \cite{GLWZ:kFP:21}.

The introduction of powers of $t$ in the functional \eqref{Normh} is due to H\'erau \cite{herau:FP:07}; the original argument in \cite{vill:hypoco} was not using it and instead doing a bit more interpolation; both recipes are good to know.

The memoir \cite{vill:hypoco} also contained a general method for nonlinear equations with global smoothness bounds and a degenerate dissipativity, generalising and improving on earlier work with Desvillettes \cite{DV:FP:01,DV:boltz:05}, with applications to the nonlinear Boltzmann equation \cite[Part III]{vill:hypoco} and to the self-consistent Vlasov--Fokker--Planck equation with periodic interaction potential. It can only yield $O(t^{-\infty})$ decay rates, not exponential, and it does cover some situations where exponential decay is known to fail.

At least two other important classes of linear methods are worth mentioning and can often be mixed with the previously mentioned hypocoercivity point of view. The first one, unsurprisingly, is the spectral theory of linear operators. For the classical Fokker--Planck equation (without space variable), the spectrum is explicit and eigenfunctions are well-known; a linear treatment also applies to a number of variants, e.g. with fractional operators or weak confinement, as in works by Kavian, Mischler, Ndao, Tristani~\cite{KMN:FP:21,MT:FP:17,tristani:fractional:15}. When the spatial variable is there, the operator is of course no longer symmetric, but one can invoke the quantitative theory of semigroups for nonsymmetric generators, developed by Gualdani, Mischler and Mouhot for applications in kinetic theory (but not only) \cite{GMM:nonsym:17,MM:expslow:16}. A new twist in this line of ideas was introduced by Bernard--Fathi--L\'evitt--Stoltz \cite{BFLS:hypo:22}, working out hypocoercive estimates through the method of Schur complements, already known in spectral theory for other purposes.

The second one is the venerable theory of equilibration for recurring Markov processes, first explored by Wolfgang Doeblin and Theodore Harris, and then systematically developed in the nineties by Meyn and Tweedie \cite{meyntweedie:stability:92,meyntweedie:stability:93,meyntweedie:geometric:94,meyntweedie:MCbook}. Its two key ingredients are a strict positivity estimate and a confinement estimate. A pedagogical introduction is in Stroock \cite{stroock:markovbook}. Equations of degenerate Fokker--Planck type were considered in this style by Talay \cite{talay:stochHamilt:02}, then in a series of works by Hairer and Mattingly, see e.g.~\cite{HM:yetanother:11}; then precise results for the kinetic Fokker--Planck equations were obtained by Ca\~{n}izo, Cao, Evans and Yolda \cite{CCEY:harris:20,cao:kFP:21}; eventually optimised and simplified proofs were established by Ca\~{n}izo and Mischler \cite{canizomischler:harris:23}, who also provide a nice review and history of the field.

In relation with the latter mention, there is also a tradition to evaluate the distance to equilibrium through distances on probability measures, such as optimal transport (Wasserstein) distances, which are reviewed for instance in \cite[Chapter~6]{vill:oldnew} and well adapted to coupling techniques. In the context of hypocoercive estimates, there are several works making good use of them; see for instance Hairer--Mattingly \cite{HM:SGW:08} (for a stochastic model of fluid mechanics), Baudoin \cite{baudoin:wasshypo:16}, Bolley--Guillin--Malrieu \cite{BGM:VFP:10}, or Guillin--Le Bris--Monmarch\'e  \cite{GLBM:VFP:22} (the last two papers for the nonlinear Vlasov--Fokker--Planck equation under some assumptions on the coupling potential).

That list is just a sample. Bernard--Fathi--L\'evitt--Stoltz \cite{BFLS:hypo:22} provide a more complete attempt of references and description of methods. But as of 2026, a Google Scholar search on the keyword ``hypocoercivity'' lists more than 1000 contributions; so basically there is, by now, already no way to comprehensively review the literature on hypocoercivity, even though the word itself was introduced only two decades ago. Still, only a tiny proportion of those works incorporate Riemannian geometry; this will be the focus of the rest of this memoir.

\section{Kinetic theory meets Riemannian geometry} \label{secmeet}

The huge majority of studies on the kinetic Fokker--Planck equation are set in flat geometry: $\R^n$ or $\T^n = \R^n/\Z^n$. A number of them have geometric interpretation and intuition. There are however several motivations for going genuinely geometric, that is, to assume that the position space is a non-Euclidean geometry. Indeed:

(a) Some problems coming from mathematical physics involve both Riemannian (or Lorentzian) geometry and kinetic formalism; this includes in particular the study of matter in general relativity. Of particular interest in this respect is the relativistic kinetic Fokker--Planck equation introduced by Fabrice Debbasch, Kirone Mallick and Jean-Pierre Rivet.

(b) Geometric problems make constant use of geodesic flow and heat flow; when one wishes to combine both or interpolate between them, a kinetic Fokker--Planck equation naturally arises: recall indeed from \eqref{kgeod} and \eqref{dd} that both the geodesic and heat flows are two limit cases of the Fokker--Planck semigroup. In this vein, in a notable series of works from the 2000's, Jean-Michel Bismut showed how to address certain problems of topology by using a geometric version of the kinetic Fokker--Planck equation acting on 1-forms.

(c) Going geometric is also a way to test the robustness and optimality of tools coming from the ``flat'' theory.

(d) In the more demanding context of non-Euclidean geometries, one is naturally led to try and identify more elaborate estimates and points of view which may be useful also for kinetic equations set in Euclidean geometry.
\sm

The resulting equation will be called the {\bf geometric kinetic Fokker--Planck equation}; a precise formulation will be provided later. Before embarking on this study, I will mention two notable ``fixed points'' in the literature:
\sm

\bul H\"ormander's theory of local regularity is set in invariant terms, and clearly applies in a Riemannian as well as Euclidean geometry. It was by adapting these tools to global estimates that the first study of hypoellipticity and hypocoercivity was performed for the geometric kinetic Fokker--Planck equation, by Jean-Michel Bismut and Gilles Lebeau. This is for the $L^2$ problem.
\sm

\bul The Meyn--Tweedie theory of equilibration of Markov processes, which rests mainly on positivity and localisation, is very robust and goes over to the geometric context without notable difficulty. This is for the $L^1$ problem.
\sm

On the other hand, functional approaches to these issues do require some care in handling functionals inequalities, commutators, regularity estimates and so on. Commutators, spectral gap inequalities, classical Sobolev inequalities, logarithmic Sobolev inequalities, interpolation inequalities, Nash inequalities can all be put to good use in a Riemannian setting. This is precisely the task that Debbasch, Ollivier and myself assigned ourselves in 2008, letting our incomplete results circulate among experts. Since then only a handful of researchers have explored that path further: Fabrice Baudoin--Camille Tardif--J\"urgen Angst--Isma\"el Bailleul, Qiuyu Ren--ZhongKai Tao, Francis Nier--Xingfeng Sang--Francis White. Their results will be reviewed later, along the course of presentation of my own work, old and new, with Debbasch and Ollivier. Before that, in the rest of this section I will introduce the basic notation and concepts, and write the explicit form of the geometric kinetic Fokker--Planck equation.

In all the rest of these notes, the position space will be a smooth compact connected Riemannian manifold $M$ of dimension $n\geq 1$, equipped with a smooth Riemannian metric $g$. Compactness already allows for rich phenomena, and the unboundedndess of the tangent bundle will be a rich source of difficulties, even for a compact manifold. 

For any $x\in M$ I will write $T_xM$ for the tangent space at $x$, and $T_x^*M$ for its dual. For any two tangent vectors $u,v$ in $T_xM$, $\<u,v\>=\<u,v\>_x = g_x(u,v)$, $|v|^2=|v|_x^2 = g_x(v,v)$. Each $T_xM$ is a Euclidean space, and $g_x$ transforms 1-forms $p$ into vectors $v$ and vice versa, via
\[ \<p,w\>_x = \< v,w\>_x,\]
where brackets on the left-hand side are just evaluation, and on the right-hand side scalar product. Thus $g_x$ can be identified with a symmetric positive operator $T_xM\to T_x^*M$; its inverse will be denoted $g^*$. If $(g_{ij})$ are the components of $g$ then $(g^{ij})$ will stand for the components of $g^*$; so that the $n\times n$ matrix $(g^{ij})$ is the inverse of the matrix $(g_{ij})$, or $g^{ij} g_{jk} = \delta^i_k$, with implicit summation over the repeated index $j$. (Note: $\delta_k^\ell = \delta_{k\ell} = \delta^{k\ell}$ is always 1 if $k=\ell$ and 0 otherwise.)

The volume is $\vol(dx) = \sqrt{\det g_x}\, dx_1\ldots dx_n$, and this is also the $n$-dimensional Hausdorff measure on $M$. As for $T_xM$, it has a natural volume as a Euclidean space, $\vol_x(dv) = \sqrt{\det g_x}\, dv_1\ldots dv_n$.

Then $\TM = \cup_x (\{x\}\times T_xM)$ is the tangent bundle and $\TsM =\cup_x (\{x\}\times T^*_xM)$ the cotangent bundle; both are locally product. If $x^1,\ldots,x^n$ are local charts on $M$ (so that locally $x = \sum x^i e_i$), then locally $(\pa/\pa x^i)_{1\leq i\leq n}$ is a basis for each $T_xM$, and tangent vectors can be decomposed as $v=\sum v^i \pa/\pa x^i$. Then $(v^1,\ldots, v^n)$ are coordinates in $T_xM$, and tangent vectors are $\pa/\pa v^1,\ldots, \pa v^n$. I shall sometimes abbreviate $\pa/\pa x^i$ as $\pa_{x^i}$, or just $\pa_i$, and $\pa/\pa v^i$ as $\pa_{v^i}$. There is the dual basis $(dx^\ell)$ with
\[ \left\< dx^\ell, \derpar{}{x^k} \right\> = \delta_k^\ell.\]

Then one can define vector fields and their local coordinates, and higher order $(J,I)$-tensors ($I$-covariant, $J$-contravariant), which form a vector space $\TIJ$.
Following well-established convention, lower indices will be used for covariant components and upper indices for contravariant components; and indices of tensors will be raised and lowered through the action of the metric; and summation over repeated indices will be implicit. So for instance $Y^i = \sum_j g^{ij} Y_j = g^{ij} Y_ j$. If $X$ and $Y$ are two $(J,I)$-tensors, then their scalar product reads, in local coordinates,
\begeq\label{XY}
\<X,Y\> = \<X,Y\>_x = g^{i_1 i'_1} \ldots g^{i_I i'_I} g_{j_1j'_1}\ldots g_{j_Jj'_J} X_{i_1\ldots i_I}^{j_1\ldots j_J} Y_{i'_1\ldots i'_I}^{j'_1\ldots j'_J}.
\endeq
In particular, the squared operator norm of $X$ is given by
\begeq\label{XX}
|X|^2 =|X|_x^2 = g^{i_1 i'_1} \ldots g^{i_I i'_I} g_{j_1j'_1}\ldots g_{j_Jj'_J} X_{i_1\ldots i_I}^{j_1\ldots j_J} X_{i'_1\ldots i'_I}^{j'_1\ldots j'_J}.
\endeq
(Double bars will be reserved to functional norms.)

If $f$ is a function $M\to\R$, and a system of local coordinates $(x^1,\ldots,x^n)$ is given, then $\pa_i f = \pa f/\pa x^i$ is the partial derivative in direction $i$, and the differential is $df = \pa_i f dx^i$, while the gradient is $\nabla f = g^*(df) = g^{ij} \pa_j f \,\pa_i$. The divergence of a vector field is $\nabla\cdot X = \nabla_i X^i$, and equivalently $\nabla\cdot = -\nabla^*$ (adjoint in $L^2(\vol)$). It is also possible to take partial divergence of a tensor, along any component, say $(\nabla\cdot X)^{jk} = \nabla_i X^{ijk}$.

The covariant derivative of a vector field $X$ is given in coordinates by
\begeq\label{lXi}
\nabla_\ell X^i = (\nabla_\ell X)^i = \derpar{X^i}{x^\ell} + \Gamma^{i}_{j\ell} X^j
\endeq
(component $i$ of the covariant derivative of $X$ in the direction $\ell$), where $\Gamma_{ij}^k$ are the Christoffel symbols,
\begeq\label{christoffel}
\Gamma_{ij}^k (x) = \frac12 \Bigl( \pa_i g_{jm} + \pa_j g_{im} - \pa_m g_{ij}\Bigr) g^{km}.
\endeq
The Christoffel symbols satisfy several important identities, such as 
\begeq\label{idChrist}
\Gamma_{ij}^k = \Gamma_{ji}^k, \qquad \Gamma_{ij}^i = \frac12 \,\pa_j \log \det g.
\endeq

More generally, the covariant derivative of a $(J,I)$-tensor $X$ is determined by the formula
\begin{multline} \label{covX}
\nabla_\ell X_{i_1\ldots i_I}^{j_1\ldots j_J}
= \derpar{X_{i_1\ldots i_I}^{j_1\ldots j_J}}{x^\ell}
+ \Gamma_{s\ell}^{j_1} X_{i_1\ldots i_I}^{s j_2\ldots j_J} + \ldots + \Gamma_{s\ell}^{j_J} X_{i_1\ldots i_I}^{j_1 \ldots j_{J-1} s}\\
- \Gamma_{i_1 \ell}^{s} X_{s i_2\ldots i_I}^{j_1 \ldots j_J} - \ldots - \Gamma_{i_I\ell}^{s} X_{i_1\ldots i_{I-1} s}^{j_1 \ldots j_J}.
\end{multline}
When $w= w^k \pa_k$, $\nabla_w X =  w^k \nabla_kX$ is intrinsically defined, and in this way $\nabla X$ maps tangent vectors into $(J,I)$-tensors, thus it defines a $(J,I+1)$-tensor. An important particular case is $\nabla g = 0$ (invariance of the metric).

The geodesic flow is defined in $\TM$ by the differential equation $\nabla_{\dot{x}} \dot{x} = 0$ (the velocity remains constant in the direction of the motion), or equivalently, in coordinates,
\begeq\label{geodeq}
\ddot{x}^\ell + \Gamma_{ij}^\ell \dot{x}^i \dot{x}^j = 0 \qquad (1\leq \ell\leq n).
\endeq
(It is as if there were an extra force $-\Gamma \dot{x} \dot{x}$ acting on free particles; actually an effect of the choice of referential.)
I shall write $\xi$ for the associated vector field on $\TM$; this is the {\bf geodesic vector field, or geodesic derivation}; in coordinates
\begeq\label{xi}
\xi (x,v) = v^\ell \derpar{}{x^\ell} - \Gamma_{ij}^\ell v^i v^j \derpar{}{v^\ell}.
\endeq
The first part $v\cdot\nabla_x$ is the familiar Euclidean transport operator, while the second part involving the Christoffel symbols is an effect of non-Euclideanity. The flow determined by $\xi$ is also the Hamiltonian flow for the Hamiltonian function $H(x,v) = |v|^2/2$.

Being each $T_xM$ a Euclidean space, there is a natural gradient operator in $T_xM$, the {\bf vertical gradient}, and similarly a {\bf vertical Laplace operator} and a {\bf vertical divergence}:
\begeq\label{DeltaV}
\Delta_V X = g^{ij} \derpar{}{v^i} \derpar{X}{v^j}, \qquad \nabla_V\cdot X = \derpar{X^i}{v^i}.
\endeq
(Vertical here means that this is just in the fiber $T_xM$, $x$ being a parameter.) This formula works if $X$ is a function, or a tensor of any order.

There is a standard {\bf Gaussian distribution} with unit variance:
\begeq\label{mux}
\mu_x(dv) = \frac{e^{-\frac{|v|^2}{2}}}{(2\pi)^{n/2}}\,\vol_x(dv) = 
\frac{e^{-(g_{k\ell}v^k v^\ell)/2}}{(2\pi)^{n/2}}\, \sqrt{\det g}\, dv^1\ldots dv^n,
\endeq
note that $\mu_x[T_xM]=1$; 
and also a global Gaussian distribution on $\TM$:
\begeq\label{mu}
\mu(dx\,dv) = \frac{e^{-\frac{|v|^2}{2}}}{(2\pi)^{n/2}}\,\vol_x(dv)\, \frac{\vol(dx)}{\vol(M)}  = f_\infty(x,v)\, \vol(dx)\,\vol_x(dv),
\endeq
so that $\mu[\TM]=1$ as well.

Now at last we can write down the {\bf geometric kinetic Fokker--Planck equation}:
\begin{align}\label{eqf}
\derpar{f}{t} + \xi f & = \nabla_ V\cdot (\nabla_V f + fv) \\ \nonumber
& = \Delta_V f + v\cdot\nabla_V f + n f.
\end{align}
or for $h= f/f_\infty$: 
\begeq\label{eqh}
\derpar{h}{t} + \xi h =  \Delta_V h - v\cdot\nabla_V h.
\endeq
Note: the identity $\int (\Delta_V-v\nabla_V)h\, k\,d\mu = - \int \nabla_V h\cdot \nabla_V k\,d\mu$ shows that $-\Delta_V + v\cdot\nabla_V$ plays the same role in $L^2(\mu)$ as $-\Delta_V$ does in $L^2(dx\,dv)$, representing the Dirichlet form $\int | \nabla_V h|^2\,d\mu$.
This justifies the notation
\begeq\label{LDeltaV}
\Delta_V^\mu = \Delta_V - v\cdot\nabla_V.
\endeq

It remains to write down the problem for tensors. Let us do it only for the $L^2$ problem, that is, starting from \eqref{eqh}. As before, let us use lower indices for $x$-derivatives and upper indices for $v$-derivatives:
\begeq\label{Nablaij}
\nabla_{i_1\ldots i_I}^{j_1\ldots j_J} h =
g^{i_1i'_1}\ldots g^{i_I i'_I}\, \derpar{{}^I}{ x^{i'_1}\ldots \pa x^{i'_I}} 
\,\derpar{{}^J}{ v^{j_1}\ldots \pa v^{j_J}} \,h.
\endeq

The first step is to extend the geodesic operator from functions to tensors. The formula we would like to use is
\begeq\label{Xixi}
\Xi \nabla^N_{x,v} h = \nabla^N_{x,v} (\xi h).
\endeq
But contrary to the flat case, there is in general no closed formula of that type, because $\nabla^N_{x,v} (\xi h)$ might involve differentials of lower order. Still, at least for $N=1$, which is the first main case of interest, it does make sense, as shown by the formulas
\begeq\label{Xixifml}
\begin{cases}
\dps \derpar{}{x^r} (\xi h) = \xi \derpar{h}{x^r} + \bigl(\pa_r \Gamma_{ij}^\ell v^i v^j\bigr) \derpar{h}{v^\ell} \\[4mm]
\dps \derpar{}{v^s} (\xi h) = \xi \derpar{h}{v^s} + \derpar{h}{x^s} + \Gamma_{sj}^\ell v^j \derpar{h}{v^\ell}
+ \Gamma_{is}^\ell v^i \derpar{h}{v^\ell}
\end{cases}
\endeq
(here $\xi \derpar{h}{x^r}$, $\xi\derpar{h}{v^s}$ are interpreted in the naive sense, not in any covariant sense). A more intrinsic rewriting will be provided in the next section, through the notions of horizontal and vertical differentiations, see \eqref{Xiexpl} below. In any case there is a well-defined operator $\Xi$ on vector fields on $\TM$ such that
\begeq\label{Xixi1} \Xi \,\nabla_{x,v} h = \nabla_{x,v} (\xi h). \endeq
As for the right-hand side, formula \eqref{calLF} still works, in this section with the normalisation $\lambda=1$; and equation \eqref{kFPX} remains the same, recast here for vector fields:
\begeq\label{eqX}
\pa_t X + \Xi X  =(\Delta_V^\mu-P_V) X = (\Delta_V-v\cdot\nabla_V- P_V ) X,
\endeq
where the projector $P_V$ just keeps the vertical part: $P_V(\delta x, \delta v) = (0,\delta v)$.
(This is independent of the choice of charts.)
The operator $\Delta_V^\mu - P_V - \Xi$ is the hypoelliptic Laplacian on vector fields, sometimes called the {\bf Bismutian}.
I shall write ${\cal B}$ for its negative.

Of course the case of more general $N$-tensors is not over: it might be that for some special geometries and choice of index $(I,J)$ and initial condition $X_0$ the equation for $N$-tensors makes sense. But for the moment let us be content with the case $N=1$.

To summarise, there are three natural operators (negative of the respective generators):
\begeq
{\cal L} = \xi - \Delta_V - v\cdot\nabla_V - n  \qquad \text{ on $L^1(dx\,dv)$}
\endeq

\begeq
L = \xi - \Delta_V + v\cdot\nabla_V  \qquad  \text{on $L^2(\mu)$}
\endeq

\begeq
{\cal B} = \Xi - \Delta_V + v\cdot\nabla_V +P_V \qquad \text{on $L^2(\mu;\Tzeroun)$}.
\endeq

And here are the three basic settings in which to study the regularity and equilibration of the geometric Fokker--Planck equation:
\sm

\bul {\em $L^1$ problem:} Consider \eqref{eqf}, that is, $\pa_t f + {\cal L}f =0$, with a initial distribution $f_0\geq 0$ such that $\int f_0(x,v)\,dx\,dv = 1$.
\sm

\bul {\em $L^2$ problem:} Consider \eqref{eqh}, that is, $\pa_t h + L h =0$, with an initial function $h_0$ such that $\int h_0(x,v)^2\,d\mu(x,v) <\infty$.
\sm

\bul {\em $L^2$ problem for vector fields:} Consider \eqref{eqX}, that is, $\pa_t X + {\cal B}X =0$, with an initial vector field (or an initial $1$-form field) $X_0$ such that $\int |X_0(x,v)|^2\,d\mu(x,v) <\infty$.

\begin{Rks} 
\begin{itemize}
\item[(i)] Even if one considers, say, the $L^1$ problem for functions, one is naturally led to manipulate tensors, typically by considering higher-order derivatives of $f$.
\sm

\item[(ii)]
The geometric Fokker--Planck problem can be equivalently set on $\TM$, in terms of positions and velocities, or on $\TsM$, in terms of positions and impulsions; and the $L^2$ problem for vectors can be recast in terms of 1-forms. Those settings are strictly equivalent since the metric $g$ induces an isomorphism between $\TM$ and $\TsM$, via the formulas $p=gv$, $v=g^*p$. Working in $\TsM$ has some formal advantages: (a) the geodesic flow is set in the Hamiltonian formalism of conjugate variables; (b) $x$-derivatives come with lower indices and $v$-derivatives with upper indices, making the formalism somewhat simpler for the vector-valued (tensor-valued) problem. On the other hand, working in $\TM$ is arguably more intuitive, and going from one setting to the other is a dictionary. Both settings appear in the literature. All in all, I will stick to $\TM$ in these notes.
\sm

\item[(iii)] Recall the normalisation set at the level of \eqref{lambdanorm} in the whole space, up to a rescaling of the force; in the geometric context, this corresponds to rescaling the geometry of $M$ (multiplying the metric $g$ by a constant). In particular, un-normalising $\lambda$, there are two natural asymptotics: $\lambda\to\infty$ (geometry effects negligible compared to diffusion effects), a regime which formally reduces to the Brownian motion on the manifold $M$, by the same reasoning which led to \eqref{dd}; and $\lambda\to 0$ (diffusion negligible compared to the geometry), a regime which approaches the geodesic motion. 
\end{itemize}
\end{Rks}

\bibnotes

Kinetic models and relativity (special or general) were considered either by the kinetic community, or by the theoretical physics community, and with various points of view: see the reviews by Andr\'easson \cite{andreasson:einsteinvlasov:11}, Debbasch--Chevalier \cite{debbaschchevalier:relativistic:07}, Dunkel--H\"anggi \cite{dunkelhanggi:review:09}. Some of these models only introduce the Lorentzian invariance in an otherwise classical formalism, others are set in ``flat'' Minkowski space, others are genuinely Riemannian. 

The approach in Debbasch--Mallick--Rivet \cite{DMR:relativistic:97} is to adapt the classical diffusion equation, rather than the stochastic process, to the relativistic framework; this leads to the geometric kinetic Fokker--Planck equation (Ornstein--Uhlenbeck, in the terminology of that paper). Debbasch \cite{debbasch:diffcurved:04} further discusses the subject and notes that the kinetic point of view allows to resolve some fundamental difficulties in defining relativistic diffusion processes. Chevalier--Rivet \cite{CD:relat:08,CD:relatH:08}, Barbachoux--Debbasch--Rivet \cite{BDR:relatK:01,BDR:relat:01,DR:relat:98} push that study deeper and establish some basic qualitative properties, including an $H$-Theorem and limit regimes.

When only special relativity is considered and the geometry is just the flat Minkowski space, or Lorentzian geometry, the resulting kinetic Fokker--Planck equation is studied from the point of view of hypoellipticity and hypococercivity by Arnold--Toshpulatov \cite{arnoldtoshpulatov:25}.

The hypoelliptic Laplacian on vector fields was introduced by Bismut \cite{bismut:hypoelliptic:05} and studied by Bismut and Lebeau \cite{bismutlebeau:hypo:book}. A survey can be read in Bismut \cite{bismut:surveyhypo:08}. The terminology ``Bismutian'' was introduced by Lebeau \cite{lebeau:FP1:05}.

Nier \cite{nier:pekin:06} presents a course on pseudodifferential operators applied to hypoellipticity of Fokker--Planck operators on manifolds. He also considers the geometric Fokker--Planck equation on a manifold with boundary  \cite{nier:boundary:18} and extends the Bismut--Lebeau analysis to that setting.

Baudoin \cite{baudoin:BE+V} reworks hypocoercivity in $L^2$ and $L^1$ using the Bakry--\'Emery formalism based on the $\Gamma$ and $\Gamma_2$ calculus. In this way he obtains an alternative approach to hypocoercivity in $L^1$ (or rather in $L\log L$) for the Euclidean space with a confining potential.

Angst--Bailleul--Tardif \cite{ABT:brownian:15}, Baudoin--Tardif \cite{BT:foliations:18}, Ren--Tao \cite{RT:spectral:23} consider a variant in which the velocity is constrained to have unit norm. This is consistent since geodesic flow preserves the norm of velocity, provided that the vertical Euclidean drift-diffusion is replaced by a Laplace--Beltrami diffusion on the vertical unit sphere in each tangent space. Also in physics, kinetic models restricted to unit velocity are commonly used either as basic models or as simplifications or more complicated ones, for instance in the modelling of transport of photons or neutrons \cite{davison:book}. For that geometric kinetic Fokker--Planck equation with unit speed, the abovementioned authors establish rather sharp estimates of hypoellipticity and hypocoercivity in $L^2$. Nier--Sang--White \cite{NSW:25,NSW:grushin:25} consider the original geometric Fokker--Planck equation (there called geometric Kramers--Fokker--Planck), also in $L^2$.

For the choice of geometric setting: Debbasch--Mallick--Rivet, Bismut, Lebeau, Nier, Nier--Sang--White work in $\TsM$ \cite{bismut:hypoelliptic:05,bismutlebeau:hypo:book,debbasch:diffcurved:04,NSW:grushin:25}, while H\'erau--Nier (for the flat case), Angst--Bailleul--Tardif, Baudoin--Tardif, Ren--Tao work in $\TM$ \cite{ABT:brownian:15,BT:foliations:18,heraunier:FP:04,nier:boundary:18,NSW:25,NSW:grushin:25,RT:spectral:23}. My original preprint with Debbasch and Ollivier \cite{DOV:preprint} was set in $\TsM$; eventually I preferred to recast everything in $\TM$.

Standard references for analysis in manifolds are the books by Do Carmo \cite{docarmo:Riemann:book} and Gallot--Hulin--Lafontaine \cite{GHL:Riemann:book}. Berger \cite{berger:panoramic:book} provides a vast excursion across Riemannian geometry. I already recommended the treatise by Bakry--Gentil--Ledoux \cite{BGL:book} for functional inequalities with geometric content on manifolds.

\section{Horizontality and curvature} \label{secHC}

Analysis in $\TM$ will deal with functions and vector fields of $x$ and $v$, with constant need to invoke partial variations either in $x$ (horizontal) or in $v$ (vertical). The phase space has dimension $2n$, and horizontal variations should form an $n$-dimensional subspace, vertical variations likewise.

Vertical analysis in $\TM$ is conceptually easy: Each $(T_xM,g_x)$ is a Euclidean velocity space, so the notions of vertical differential, vertical gradient, vertical Laplace operator and so on, for functions as well as tensors, are the same as in $\R^n$, and $x$ appears as just a parameter. A tangent vector to the tangent space, say $(\delta x,\delta v)\in T_{(x,v)} (\TM)$ is vertical if and only if it is included in a fiber, that is $\delta x =0$ (verticality equation), and the vectors $\pa/\pa v^k$ will provide a basis for such vertical vectors. For the derivation one may use the symbols $\nabla_v$ and $\nabla_V$ interchangeably.

Horizontality is quite more subtle: The idea is that $\delta v=0$, but to appreciate the variation of $v$ as $x$ is moving, a connection has to be made between tangent spaces. The usual and most popular such notion is the Riemannian connection due to Tullio Levi-Civita, according to which the equation for horizontality of $(\delta x, \delta v)$ becomes
\begeq\label{eqvcst}
\delta v^k + \Gamma_{ij}^k v^i \delta x^j = 0.
\endeq
(When the Christoffel symbols vanish, this reduces to $\delta v =0$; and when $\delta x = v$, $\delta v = \dot{v}$ one recovers the equation of geodesics.) According to equation~\eqref{eqvcst}, a basis for horizontal variations is provided by
\begeq\label{basishoriz}
\left(\derpar{}{x^k}\right)_H = \derpar{}{x^k} - \Gamma_{ik}^\ell v^i \derpar{}{v^\ell}\qquad (1\leq k\leq n).
\endeq
Thus we have the notion of horizontal derivative in direction $k$:
\begeq\label{horizderk}
(\nabla_H)_k f = \derpar{f}{x^k} - \Gamma_{ik}^\ell v^i \derpar{f}{v^\ell}.
\endeq
Note that
\begeq\label{nablaHvm}
(\nabla_H)_kv^m = \Gamma_{k\ell}^m v^\ell.
\endeq

These formulas extend to the derivation of tensors (vertical and horizontal covariant derivatives):
\begeq\label{VHcov}
(\nabla_V)_k X = \derpar{X}{v^k}\qquad
(\nabla_H)_k X = \nabla_k X - \Gamma_{ik}^\ell v^i \derpar{X}{v^\ell}.
\endeq
(Recall, there is no need to invoke covariant derivatives for the vertical differentiation.) More generally, if a tangent vector $w= w^k\pa_k$ is given in $T_xM$,
\begeq\label{VHcovv}
(\nabla_V)_w X = w^k \derpar{X}{v^k}, \qquad
(\nabla_H)_w X = w^k \nabla_k X - \Gamma_{ik}^\ell v^i w^k \derpar{X}{v^\ell}.
\endeq

At this stage note that the geodesic derivation $\xi$ is nothing but
\begeq\label{xirevu}
\xi(x,v)  = v^k (\nabla_H)_k = (\nabla_H)_v = v\cdot\nabla_H.
\endeq
This formula remains unchanged when it is applied to tensors: this is the {\bf covariant geodesic differentiation}.

In the above, $\nabla_Vf$ and $\nabla_Hf$ are 1-forms, but they are in correspondence with vectors (vector fields) by the usual recipe
\[ (\nabla_V)^if = g^{ij}(\nabla_V)_j f\qquad (\nabla_H)^if = g^{ij} (\nabla_H)_jf.\]
There is consistency: In this way $\nabla_V f$ becomes a vertical vector and $\nabla_H f$ a horizontal vector. (One may also say that $\nabla_Hf$ is identified to a horizontal 1-form.) Moreover,
\[ (\nabla_V)_w X = g_{k\ell}\, w^k g^{\ell i} (\nabla_V)_i = w\cdot \nabla_V, \qquad 
(\nabla_H)_w X = g_{k\ell}\, w^k g^{\ell i} (\nabla_H)_i = w\cdot \nabla_H.\]

Similarly, if $X$ is a $(J,I)$-tensor, formulas \eqref{VHcov} or \eqref{VHcovv} define $(J,I+1)$-tensors $\nabla_VX$ and $\nabla_{H}X$, while $g^*\nabla_{V}X$ and $g^*\nabla_HX$ will be $(J+1,I)$-tensors. Also there are vertical and horizontal divergence operators, e.g. $\nabla_H\cdot X = (\nabla_H)_i X^i$, partial divergence operators, etc.

Associated to those operators are also natural notions of horizontal and vertical functions: A function $f=f(x,v)$ is said to be horizontal if $\nabla_V f =0$, and vertical if $\nabla_H f =0$. Here are the two basic examples:

\begin{Prop}[Basic invariant functions] \label{basicinv} Let $f$ be smooth on $\TM$, where $M$ is smooth and connected. Then\\
(i) If $f=f(x)$ then $\nabla_V f =0$, and conversely;

\noindent (ii) If $f=f(|v|)$; or more rigorously, if there is a function $\vphi$ such that $ f (x,v) = \vphi (g_x(v))$; then $\nabla_H f = 0$. In particular, $\xi f =0$.
\end{Prop}

\begin{proof}[Proof of Proposition \ref{basicinv}] Property (i) is obvious from the definitions. As for (ii), it follows from $\nabla_H \vphi (G) = \vphi'(G)\nabla_H G$ and $\nabla_H g_x(v)=0$: Indeed,
\begin{align*}
(\nabla_H)_k (g_{ij} v^i v^j) 
& = \pa_k \bigl(g_{ij}v^iv^j\bigr) - \Gamma_{k\ell}^r v^\ell \derpar{}{v^r} \bigl(g_{ij}v^iv^j\bigr)\\
& = (\pa_k g_{ij}) v^i v^j - \Gamma_{k\ell}^i g_{ij} v^\ell v^j - \Gamma_{k\ell}^j g_{ij} v^i v^\ell\\
& = \bigl( \pa_k g_{ij} - \Gamma_{ki}^s g_{sj} - \Gamma_{kj}^s g_{is}\bigr) v^i v^j,
\end{align*}
which vanishes since
\begin{align*} \pa_k g_{ij} - \Gamma_{kj}^s g_{is} - \Gamma_{ik}^s g_{sj}
& = \pa_k g_{ij} - g^{sr} \bigl( \pa_k g_{jr} + \pa_j g_{kr} - \pa_r g_{jk} \bigr) g_{is}\\
& = \pa_k g_{ij} - \delta_i^r (\pa_k g_{jr} + \pa_j g_{kr} - \pa_r g_{jk}) \\
& = \pa_k g_{ij} - \pa_k g_{ij} = 0.
\end{align*}
(Actually the connection was devised for that.)
\end{proof}

There is a notable difference between (i) and (ii) in Proposition \ref{basicinv}: the former considers any function of $x$, while the latter only deals with {\em radial} functions of $v$. It turns out that this is in the order of things: There is no such thing as a function ``depending only on $v$'', since, say $T_xM$ and $T_yM$ may have different parameterizations and different reference bases. Even if one considers smooth variations of frame along paths, in general going along one path or another will not yield the same results, due to curvature. A horizontal function is a function which is constant along any geodesic path in phase space: In presence of curvature, in general only radial functions qualify; in some sense the metric is the only invariant. For the same reasons, there is no general converse of (ii): it depends on the geometry whether there are horizontal functions which are not radially symmetric.

\begin{Ex} Consider $M=\T^d = \R^d/\Z^d$. Then $\TM \simeq \T^d\times\R^d$, coordinates $v^1,\ldots v^n$ makes sense globally, and any function $f=f(v)$ is left invariant by the geodesic flow $\dot{v} = 0$. On the contrary, consider $M=\S^2$ (the 2-dimensional sphere, seen as a model of the Earth), and let $N$ be the ``North Pole''. If a vector $v$ is chosen in $T_NM$ then go along the geodesic starting from $N$ with velocity $v$, until you hit the equator; then go along the equator for an angle $\theta\in [0,\pi]$; then go back geodesically to $N$. The resulting vector will still have the same norm as $v$, but it makes angle $\theta$ with $v$, and $\theta$ is arbitrary. Thus if $f$ is a horizontal function on $\S^2$ its value at $N$ can depend on $v$ only through $|v|$. In other words: The only horizontal functions on $\S^2$ are radially symmetric functions of the velocity, in contrast with $\T^2$.
\end{Ex}
 
So far it was mainly about consistent definitions, but now come the key properties, derived from the Riemannian structure, which will enormously facilitate the analysis in $\TM$:

\begin{Prop}[horizontal/vertical commutation] \label{propHV}
With notation \eqref{VHcovv}--\eqref{xirevu},

(i) $[\nabla_V, \xi] = \nabla_H$;

(ii) $[\nabla_V,\nabla_H] = 0$;

(iii) $[\nabla_V, v\cdot\nabla_V] = \nabla_V$.
\end{Prop}

\begin{Rks} 1. These identities hold true for $\nabla_V,\nabla_H,\xi, v\cdot\nabla_V$ acting covariantly on tensors, not just functions.

2. For the moment there is neither function space nor functional operator involved: these identities just express the result of a calculation performed on an arbitrary smooth function.

3. Either (i) or (ii) could have been taken as the basic identity motivating the definition of $\nabla_H$ (say, looking for $(\nabla_H)_j$ as a combination of $\pa_{x_j}$ and vertical derivations). 

4. Formula (iii) is just a Euclidean computation and its proof is immediate; but it is convenient to state it here among the basic commutation properties which hold true independently of the geometry.
\end{Rks}

\begin{proof}[Proof of Proposition \ref{propHV}]
As this is a warmup I will provide calculations for both (i) and (ii), but actually (i) follows from (ii) and \eqref{xirevu}.

Let us consider for instance the case of a $(1,1)$-tensor $X=(X_i^j)$, the general case being the same, just more cumbersome to write. Then
\[ (\nabla_V)_k X_i^j = \derpar{X_i^j}{v^k}\]
\begin{align*}
\xi X_i^j & = 
\left( v^\ell \nabla_\ell - \Gamma_{rs}^\ell v^r v^s\derpar{}{v^\ell}\right) X_i^j \\
& = v^\ell \derpar{X^i}{x^\ell}
-v^\ell \Gamma_{i\ell}^s X_s^j + v^\ell \Gamma_{s\ell}^j X_i^s
- \Gamma_{rs}^\ell v^r v^s \derpar{X_i^j}{v^\ell}.
\end{align*}
So
\begin{multline*} 
(\nabla_V)_k (\xi X_i^j) = \derpar{X_i^j}{x^k} + v^\ell \derpar{{}^2 X_i^j}{x^\ell\,\pa v^k}
-\Gamma_{ik}^s X_s^j - v^\ell \Gamma_{i\ell}^s \derpar{X_s^j}{v^k}\\
+ \Gamma_{sk}^j X_i^s + v^\ell \Gamma_{s\ell}^j \derpar{X_i^s}{v^k}
- 2 \Gamma_{ks}^\ell v^s \derpar{X_i^j}{v^\ell} - \Gamma_{rs}^\ell v^r v^s \derpar{{}^2 X_i^j}{v^\ell \,\pa v^k}.
\end{multline*}
Recalling that $(\nabla_V)_k X_i^j$ is a $(1,2)$-tensor,
\[
\xi (\nabla_V)_k X_i^j
= v^\ell \derpar{{}^2 X_i^j}{x^\ell\,\pa v^k} - v^\ell \Gamma_{i\ell}^s \derpar{X_s^j}{v^k} + v^\ell \Gamma_{s\ell}^j \derpar{X_i^s}{v^k} \\ 
- \Gamma_{rs}^\ell v^r v^s \derpar{{}^2 X_i^j}{v^k\,\pa v^\ell} - v^\ell \Gamma_{k\ell}^s \derpar{X_i^j}{v^s}.
\]
It follows that 
\[
[(\nabla_V)_k \xi] X_i^j = \derpar{X_i^j}{x^k} + \Gamma_{sk}^j X_i^s - \Gamma_{ik}^s X_s^j - \Gamma_{ks}^\ell v^s \derpar{X_i^j}{v^\ell} = (\nabla_H)_k X_i^j,
\]
which is Property (i). As for (ii), similar computations show that both $(\nabla_V)_q (\nabla_H)_r X_i^j$ and $(\nabla_H)_r (\nabla_V)_q X_i^j$ are equal to
\[ \derpar{{}^2 X_i^j}{x^r\,\pa v^q} + \Gamma_{sr}^j \derpar{X_i^s}{v^q}
- \Gamma_{ir}^s \derpar{X_s^j}{v^q} - \Gamma_{qr}^s \derpar{X_i^j}{v^s}
- \Gamma_{r\ell}^s v^\ell \derpar{{}^2 X_i^j}{v^s\,\pa v^q}.\]
\end{proof}

The next issue is about commutation of horizontal differentiations. Of course 
\[ \left[ (\nabla_V)_i, (\nabla_V)_j \right] = 0,\]
but in general
\[ \left[ (\nabla_H)_i, (\nabla_H)_j\right] \neq 0.\]
This is where curvature will appear. Let ${\cal R}$ be the {\bf Riemannian curvature tensor}: it is a $(1,3)$-tensor defined by
\[ {\cal R}(u,v) = \nabla_u\nabla_v - \nabla_v\nabla_u - \nabla_{[u,v]}.\]
So it is determined by its components $R_{kij}^\ell$ such that
\[ \bigl( \nabla_i \nabla_j - \nabla_j\nabla_i \bigr) \derpar{}{x^k} = R_{kij}^\ell \derpar{}{x^\ell}
= \left\< dx^\ell, {\cal R}(\pa_i,\pa_j) \pa_k \right\>. \]
By computation,
\begeq\label{RGamma}
R_{kij}^\ell = \bigl( \pa_i \Gamma_{jk}^\ell - \pa_j \Gamma_{ik}^\ell\bigr)
+ \bigl( \Gamma_{ir}^\ell \Gamma_{jk}^r - \Gamma_{jr}^\ell \Gamma_{ik}^r\bigr).
\endeq
Note: If $f$ is a function then the Hessian tensor of $f$, $\nabla^2 f$, has coordinates
\[ \nabla_i\nabla_j f = \pa^2_{ij} f - \Gamma_{ij}^s \pa_s f = \nabla_j\nabla_i f  \]
(since $\Gamma_{ij}^s= \Gamma_{ji}^s$); so it is only starting with third derivatives of a function that curvature effects will imply some noncommutativity.

If $X$ is an $(L,K)$-tensor then ${\cal R}X$ can be defined as a $(L,K+2)$-tensor by contracting on the indices $k$ and $\ell$:
\begin{align}\label{RX}
& ({\cal R}X)_{k_1\ldots k_Kij}^{\ell_1\ldots \ell_L}\\
\nonumber & \qquad = R_{sij}^{\ell_1} X_{k_1\ldots k_K}^{s \ell_2\ldots \ell_L} + \ldots + R_{sij}^{\ell_L} X_{k_1\ldots k_K}^{\ell_1\ldots \ell_{L-1}s} 
- R_{k_1 ij}^s X_{s k_2\ldots k_K}^{\ell_1\ldots \ell_L} - \ldots - R_{k_Kij}^s X_{k_1\ldots k_{K-1}s}^{\ell_1\ldots \ell_L} \\
\nonumber & \qquad = \bigl({\cal R} (\pa_i,\pa_j) X\bigr)_{k_1\ldots k_K}^{\ell_1\ldots \ell_L}.
\end{align}
(So $({\cal R}X)(\pa_i,\pa_j)$ is an $(L,K)$-tensor.)
For a function $f$ this reduces to ${\cal R}f=0$; for a vector field $X$ this gives $({\cal R}X)_{ij}^\ell = \<dx^\ell, {\cal R}(\pa_i,\pa_j)X\>$.

By lowering indices one can also define the fully covariant curvature tensor:
\begeq\label{Rsijr}
R_{sijr} = g_{\ell r} R_{sij}^\ell;
\endeq
in this setting $R$ is a $(0,4)$-tensor, satisfying nice algebraic symmetry identities:
\begeq\label{Ralg1}
R_{sijr} = R_{jrsi}
\endeq
\begeq\label{Ralg2}
R_{sjir} = -  R_{sijr} = - R_{rjis}
\endeq
\begeq \label{Ralg3}
R_{sij\ell} + R_{ijs\ell} + R_{jsi\ell} = 0.
\endeq
A connection with the elementary geometric notion of curvature is the following: if $u,v$ are two tangent vectors, then
\begeq\label{Rvuuv}
\< R(v,u)u,v\> = R_{kij\ell} u^k v^i u^j v^\ell = \sigma(P_{u,v}) \bigl( |u|^2 |v|^2 - \<u,v\>^2\bigr),
\endeq
where $P_{u,v}$ is the plane generated by $u,v$ in $T_xM$ and $\sigma(P)$ is the sectional curvature along the plane $P$ (Jacobian determinant of the Gauss normal map from $\exp(P)$ to the unit sphere).

Now comes the key property for commuting derivatives when working on $\TM$. Note that $v$ always denotes the velocity vector of the point $(x,v)\in\TM$ of interest (at which functions and tensors are evaluated).

\begin{Prop}[Commutators of horizontal derivatives] \label{propcomH}
With notation \eqref{xirevu}, \eqref{VHcovv} and \eqref{RGamma},

(i) $[\nabla_H,\nabla_H] = {\cal R} - \<{\cal R}v, \nabla_V\>$ 

\noindent in the sense that for any two vectors $u,w\in T_xM$ and any tensor $X$ on $\TM$,
\begeq\label{NNablaH} 
\bigl[ (\nabla_H)_u, (\nabla_H)_w\bigr] X = 
{\cal R}(u,w) X - \bigl\<{\cal R}(u,w) v, \nabla_V\> X
\endeq
when evaluated at $(x,v)\in\TM$; or more explicitly, for any indices $i,j$, $k_1,\ldots k_K$, $\ell_1,\ldots \ell_L$ in $\{1,\ldots, n\}$,
\begeq\label{NNablaHcoord}
\bigl[ (\nabla_H)_i, (\nabla_H)_j \bigr] X_{k_1\ldots k_K}^{\ell_1\ldots \ell_L}
= ({\cal R}X)_{k_1\ldots k_K ij}^{\ell_1\ldots\ell_L}
- R_{kij}^\ell v^k \derpar{X_{k_1\ldots k_K}^{\ell_1\ldots \ell_L}}{v^\ell}.
\endeq

(ii) $[\xi,\nabla_H] = {\cal R}(v,\cdot) - \<{\cal R}(v,\cdot)v, \nabla_V\>$

\noindent in the sense that for any vector $w\in T_xM$ and any tensor $X$ on $\TM$,
\begeq\label{xiNablaH} 
\bigl[ \xi, (\nabla_H)_w\bigr] X = 
{\cal R}(v,w) X - \bigl\<{\cal R}(v,w)v , \nabla_V\> X;
\endeq
or more explicitly, for any indices $j$, $k_1,\ldots k_K$, $\ell_1,\ldots \ell_L$ in $\{1,\ldots,n\}$,
\begeq\label{xiNablaHcoord}
\bigl[ \xi, (\nabla_H)_j \bigr] X_{k_1\ldots k_K}^{\ell_1\ldots \ell_L}
= ({\cal R}X)_{k_1\ldots k_K ij}^{\ell_1\ldots\ell_L} v^i
- R_{kij}^\ell v^iv^k \derpar{X_{k_1\ldots k_K}^{\ell_1\ldots \ell_L}}{v^\ell}.
\endeq
\end{Prop}

\begin{Rk} Notation $\<{\cal R}(u,w)v , \nabla_V\> X$ in \eqref{NNablaH} means that the tensor contraction involving ${\cal R}$ only bears on the $\pa/\pa v^\ell$, not on the indices of $X$. Any ambiguity is removed in \eqref{NNablaHcoord}. A convenient synthetic way to rewrite \eqref{xiNablaH} is
\begeq\label{xiNablaHr}
\bigl[ \nabla_H, \xi ] = -\Sigma(v,v) \nabla_V - \rho v,
\endeq
where $\Sigma$ is a bilinear operator on $\TM$, valued in tensors $\TM\to\TM$, and $\rho$ is a linear map on $\TM$, also valued in tensors $\TM\to\TM$:
 \begeq\label{Sigmarho}
 \begin{cases}
 \<\Sigma(v,v) u,w\> =-\<{\cal R}(v,u)v,w)\> = \< {\cal R}(u,v)v, w\> \\[2mm]
 (\rho v) z = {\cal R}(v,z).
 \end{cases}\endeq
\end{Rk}
Note that the term in $\rho$ appears only when \eqref{xiNablaHr} is applied to vector fields or higher order tensors, not functions.

Before proving Proposition \ref{propcomH}, here is a notable consequence: By (ii), for any function of $(x,v)$ we have
\[ 
\begin{cases}
\nabla_H (\xi f) = \xi \nabla_H f + \<{\cal R}(v,\cdot) v,\nabla_V f\> \\
\nabla_V (\xi f) = \xi \nabla_V f + \nabla_H f
\end{cases} \]
or more explicitly, writing $e_i=\pa_i$ for the $i$th tangent vector in the coordinate basis of $T_xM$,
\[ 
\begin{cases}
(\nabla_H)_i (\xi f) = \xi \nabla_H f + \<{\cal R}(v,e_i) v,\nabla_V f\> \\
(\nabla_V)_i (\xi f) = \xi (\nabla_V)_i f + (\nabla_H)_i f.
\end{cases}\]
This provides explicit formulas for the {\bf geodesic flow generator} $\Xi$ on vector fields which are not necessarily gradients (or 1-forms which are not necessarily differentials): If $X$ is such a vector field, decompose $X$ into a horizontal part $X_H$ and a vertical part $X_H$, then from \eqref{Xixi1}
\begeq\label{Xiexpl}
\Xi(X_H,X_V) = \Bigl( \xi X_H + \<{\cal R}(v,\cdot) v,X_V\>, \ 
\xi X_V + X_H\Bigr).
\endeq
Notice that due to the term $ \<{\cal R}(v,\cdot) v,X_V\>$, $\nabla\Xi (X_H,X_V)$ may not be a closed expression of $\nabla (X_V,X_H)$, a priori preventing  to write the geodesic flow generator on higher-order tensors. One can also write the explicit formula for ${\cal B}$: if $X=(X_H,X_V)$ then
\begeq\label{BXexpl}
\begin{cases}
({\cal B}X)_H = \xi X_H - \Sigma(v,v) X_V - \Delta_V^\mu X_H \\[2mm]
({\cal B}X)_V = \xi X_V + X_H - \Delta_V^\mu X_V + X_V
\end{cases}
\endeq
(where as usual $\Delta_V^\mu = \Delta_V - v\cdot\nabla_V$).

\begin{proof}[Proof of Proposition \ref{propcomH}] First note that (ii) follows from (i) and \eqref{xirevu}. So it all boils down to establish \eqref{NNablaHcoord}. Let us compute for instance with $X$ a $(1,1)$-tensor. From \eqref{VHcov} and \eqref{lXi},
\[
(\nabla_H)_j X_k^\ell = \derpar{X_k^\ell}{x^j} + \Gamma_{js}^\ell X_k^s - \Gamma_{jk}^s X_s^\ell 
- \Gamma_{js}^r v^s \derpar{X_k^\ell}{v^r}.
\]
Tediously iterating this formula (now with a $(1,2)$-tensor),
\begin{multline*}
(\nabla_H)_i (\nabla_H)_j X_k^\ell
= \derpar{{}^2 X_k^\ell}{x^i\,\pa x^j}
- \Gamma_{js}^r v^s \derpar{{}^2 X_k^\ell}{x^i\, \pa v^r}
- \Gamma_{ia}^b v^a \derpar{{}^2 X_k^\ell}{v^b\, \pa x^j}
+ \Gamma_{ia}^b \Gamma_{js}^r v^a v^s \derpar{{}^2 X_k^\ell}{v^b\, \pa v^r}\\
+ \Gamma_{js}^\ell \derpar{X_k^s}{x^i} - \Gamma_{jk}^s \derpar{X_s^\ell}{x^i} 
+ \Gamma_{ia}^\ell \derpar{X_k^a}{x^j} - \Gamma_{ik}^a \derpar{X_a^\ell}{x^j}
- \Gamma_{ij}^a \derpar{X_k^\ell}{x^a} \\
- \pa_i \Gamma_{js}^r v^s \derpar{X_k^\ell}{v^r}\\
- \Gamma_{ia}^\ell \Gamma_{js}^r v^s \derpar{X_k^a}{v^r}
+ \Gamma_{ik}^a \Gamma_{js}^r v^s \derpar{X_a^\ell}{v^r}
+ \Gamma_{ij}^a \Gamma_{ak}^r v^s \derpar{X_k^\ell}{v^r}
- \Gamma_{ia}^b \Gamma_{js}^\ell v^a \derpar{X_k^s}{v^b}
+ \Gamma_{ia}^b \Gamma_{jk}^s v^a \derpar{X_s^\ell}{v^b}
+ \Gamma_{ia}^b \Gamma_{jb}^r v^a \derpar{X_k^\ell}{v^r}\\
+ \pa_i \Gamma_{js}^\ell X_k^s -\pa_i \Gamma_{jk}^s X_s^\ell 
+ \Gamma_{ia}^\ell \Gamma_{js}^a X_k^s
- \Gamma_{ia}^\ell \Gamma_{jk}^s X_s^a - \Gamma_{ik}^a \Gamma_{js}^\ell X_a^s
+ \Gamma_{ik}^a \Gamma_{ja}^s X_s^\ell
- \Gamma_{ij}^a \Gamma_{as}^\ell X_k^s +\Gamma_{ij}^a \Gamma_{ak}^s X_s^\ell.
\end{multline*}
Using the symmetry of Christoffel symbols $\Gamma_{ij}^\ell = \Gamma_{ji}^\ell$ and formulas \eqref{RGamma},
\[ \Bigl( (\nabla_H)_i (\nabla_H)_j - (\nabla_H)_j (\nabla_H)_i\Bigr) X_k^\ell
= R_{sij}^\ell X_k^s - R_{kij}^s X_s^\ell - R_{sij}^r v^s \derpar{X_k^\ell}{v^r},\]
as desired.
\end{proof}

Propositions \ref{propHV} and \ref{propcomH}, together with formula \eqref{Xiexpl}, make it possible to compute commutators of $\Xi$ with differential operators, involving contractions and differentials of ${\cal R}$. For instance
\begeq\label{nablaVjR}
(\nabla_V)_j \bigl( R_{kmi}^\ell v^k v^m \bigr) = (R_{kji}^\ell + R_{jki}^\ell) v^k =: (-\Sigma' v)_{ij}^\ell,
\endeq
where $\Sigma'=\Sigma'(x)$ is a $(1,3)$-tensor field (like $\rho$) which can be synthetically written
\begeq\label{rho}
\Sigma': v\longmapsto - {\cal R}(v,\cdot) - {\cal R} (\cdot,\cdot) v.
\endeq
(It is the same as $\nabla_V \Sigma$, which explains the choice of notation.)
Also,
\begin{align} \label{nablaHjR}
(\nabla_H)_j \bigl(R_{kmi}^\ell v^k v^m\bigr) 
&= \bigl( \pa_j R_{kmi}^\ell - \Gamma_{ji}^s  R_{kms}^\ell+ \Gamma_{js}^\ell R_{kmi}^s \bigr) v^k v^m \\
\nonumber & =: \nabla_j {\cal R}^\ell (v,e_i) v.
\end{align}
Then by computation
\begin{multline}
[\Xi, (\nabla_H)_j ] (X_H, X_V)
= \biggl( \Bigl( \bigl[ \xi,(\nabla_H)_j\bigr] (X_H)_i  - \bigl\< (\nabla_H)_j {\cal R}(v,e_i)v, X_V \bigr\> \Bigr)_{1\leq i\leq n}, \\ 
\Bigl( \bigl[ \xi,(\nabla_H)_j\bigr] (X_V)_i\Bigr)_{1\leq i\leq n} \biggr)
\end{multline}
and
\begin{multline}
[(\nabla_V)_j,\xi ] (X_H, X_V)
= \biggl( \Bigl( \bigl[ (\nabla_H)_j,\xi\bigr] (X_H)_i  + \bigl\< (- \Sigma' v)_{ij}, X_V \bigr\> \Bigr)_{1\leq i\leq n}, \ 
\Bigl( \bigl[ (\nabla_V)_j,\xi\bigr] (X_V)_i\Bigr)_{1\leq i\leq n} \biggr).
\end{multline}
In a more compact form,
\begeq\label{Xicom}
\begin{cases}
[\nabla_H, \Xi ] =  -\Sigma(v,v)\nabla_V - \rho v + \bigl( \< \nabla{\cal R}(v,\cdot) v, P_V\>, 0 \bigr)\\
[\nabla_V, \Xi] = \nabla_H - \bigl( \<\Sigma' v, P_V\>, 0\bigr),
\end{cases}
\endeq
where $P_V (X_H,X_V) = X_V$ (just keep the vertical component). Complicated as they may seem, these expressions display the same dominant terms as commutators of $\xi$.

\bibnotes

Curvature is the central quantitative notion of non-Euclideanity and is largely treated in any exposition of Riemannian geometry, like those of Berger, DoCarmo or Gallot--Hulin--Lafontaine \cite{berger:panoramic:book,docarmo:Riemann:book,GHL:Riemann:book}. The notion of horizontal-vertical decomposition is less popular, but can also be found in those treatises. 

The presentation in this chapter is similar to Debbasch--Ollivier--Villani \cite{DOV:preprint}. Obviously, performing (sometimes heavy) computations and identifying results is not a fully satisfactory scheme of proof, but at least it avoids the use of a more difficult or more abstract formalism, and prepares ground for the very down-to-Earth computations which will inevitably appear in subsequent analysis.

An alternative treatment of horizontality and verticality goes through the {\bf Sasaki metric}, which is a canonical Riemannian structure on $\TM$: In the above notation, it is determined by the scalar products
\[ \left\< \left(\derpar{}{x^i}\right)_H, \left(\derpar{}{x^j}\right)_H \right\> = g_{ij},\qquad
\left\< \derpar{}{v^i}, \derpar{}{v^j} \right\> = g_{ij},\qquad
\left\< \left(\derpar{}{x^i}\right)_H, \derpar{}{v^j} \right\> = 0. \]

\section{Lebesgue and Laplace analysis on $\TM$}

Dealing with functions on $\TM$ requires convenient functional spaces to study integrability, localisation (presently meaning control at large velocities) and regularity. To begin with, we need Lebesgue $L^p$ spaces to evaluate the size of functions, and Laplace operators to navigate through regularity. 

As for Lebesgue spaces, there are two natural choices of reference measure: volume measure, or Gaussian (equilibrium) measure, and both are important. In coordinates and with temperature normalised to 1, recalling that both $\vol(dx)$ and $\vol_x(dv)$ are proportional to $\sqrt{\det g}$,
\begeq\label{volvol}
\vol(dx\,dv) = \vol(dx)\,\vol_x(dv) = (\det g)\,dx^1\ldots\,dx^n\,dv^1\,\ldots dv^n,
\endeq
often just denoted $dx\,dv$ by convenience, and
\begin{align}\label{muvol}
\mu(dx\,dv) & = \frac{e^{-\frac{|v|^2}{2}}}{(2\pi)^{n/2}\,\vol(M)}\,\vol(dx\,dv) \\ 
\nonumber & = \vol(M)^{-1}\, (2\pi)^{-n/2}\, e^{-\frac12 g_{k\ell} v^k v^\ell}\, (\det g) \,dx^1\ldots\,dx^n\,dv^1\,\ldots dv^n.
\end{align}
Then define Lebesgue spaces with or without Gaussian weight, for $1\leq p<\infty$:
\begeq\label{defleb}
\|f\|_{L^p} = \|f\|_{L^p(dx\,dv)} = \left( \iint |f(x,v)|^p\,\vol(dx\,dv) \right)^{1/p}
\endeq
\begeq\label{deflebmu}
\|f\|_{L^p(\mu)} = \|f\|_{L^p(d\mu)} = \left( \iint |f(x,v)|^p\,\mu(dx\,dv) \right)^{1/p},
\endeq
together with the usual modification for $L^\infty$. Let us also add the possibility of a polynomial weight in $|v|$:
\begeq\label{Lpk}
\|f\|_{L^p_\kappa} = \bigl\| (1+|v|^2)^{\kappa/2} f\bigr\|_{L^p_\kappa}\qquad
\|f\|_{L^p_\kappa(\mu)} = \bigl\| (1+|v|^2)^{\kappa/2} f\bigr\|_{L^p_\kappa(\mu)}
\endeq

Those notions easily extend to tensor-valued functions, through
\[ \|X\|_{L^p_\kappa} = \| |X| \|_{L^p_\kappa},\]
where $|X(x,v)|$ is the (operator) norm of $X(x,v)$, recall \eqref{XX}.

In these notes $M$ is assumed to be compact, so there is no need to localise in $x$, but for an unbounded manifold $M$, coming with a confining potential, one could introduce similar weights in space, using powers of the distance function (to a fixed reference point).

Now it would be a mistake to work out regularity with the naive variables $x,v$ which are not intrinsic and not well-adapted to the kinetic Fokker--Planck equation. Recall from Section \ref{secHC} that the horizontal and vertical differentiations provide a natural classification of differentiations in $\TM$; so the convenient notion of partial regularity is about horizontal and vertical variations, rather than variations in $x$ and $v$. For this we need horizontal and vertical formulas of integration by parts, in the form of divergence theorems.

The {\bf vertical divergence} $\div_V$ and {\bf horizontal covariant divergence} $\div_H$ are defined as follows: If $X$ is a vector field $\TM\to\TM$ (meaning: for any $x,v$, $X(x,v) \in T_xM$), then
\begeq\label{divV}
\div_V X = \nabla_V\cdot X =(\nabla_V)_i\cdot X^i,
\endeq
\begeq\label{divH}
\div_H X = \nabla_H\cdot  X = (\nabla_H)_i\cdot X^i.
\endeq
Those formulas extend to higher dimensional tensors in an obvious way.
In the sequel, all functions are smooth and behaviour at infinity is calibrated from the context; for instance, if a function $\vphi$ is integrated against a polynomally increasing expression, it will be assumed that $\vphi = O(1/|v|^s)$ for all $s>0$, as $|v|\to\infty$.

\begin{Prop}[Vertical and horizontal divergence formulas] \label{propdiv}

With notation \eqref{divV}--\eqref{divH},

(i) For any $x$ and any vector field $X:T_xM\to T_xM$, vanishing at infinity,
\begeq\label{intdivV} \int_{T_xM} (\nabla_V\cdot X) \,\vol_x(dv) = 0.
\endeq
As a consequence for any vector field $X:\TM\to\TM$, vanishing at infinity, and any function $\chi:M\to\R$,
\begeq\label{iintdivV}
\iint_{\TM} (\nabla_V\cdot X)(x,v)\, \chi(x)\,\vol(dx\,dv) = 0.
\endeq

(ii) For any vector field $X:\TM\to\TM$, vanishing at infinity, and any function $\vphi:\R\to\R$,
\begeq\label{iintdivH}
\iint_{\TM} (\nabla_H\cdot X)(x,v)\, \vphi(|v|^2)\,\vol(dx\,dv) = 0.
\endeq

(iii) If $X$ is a $(J,I)$ tensor $\TM\to\TM$ and $Y$ an $(I+1,J)$-tensor $\TM\to\TM$, both vanishing at infinity, then for any two functions $\chi:M\to\R$ and $\vphi:\R\to\R$,
\begeq\label{ibpV}
\iint \bigl\<\nabla_V X, Y\bigr\> \chi(x)\,\vol(dx\,dv) 
= - \iint \bigl\<X,\nabla_V\cdot Y\bigr\>\,\chi(x)\,\vol(dx\,dv),
\endeq
\begeq\label{ibpVmu}
\iint \bigl\<\nabla_V X, Y\bigr\>\, \chi(x)\,\mu(dx\,dv) 
= - \iint \bigl\<X,(\nabla_V-v)\cdot Y\bigr\>\,\chi(x)\,\mu(dx\,dv),
\endeq
\begeq\label{ibpH}
\iint \bigl\<\nabla_HX, Y\bigr\> \,\vphi(|v|^2)\,\vol(dx\,dv) = - \iint \bigl\<X, \nabla_H\cdot Y\bigr\> \,\vphi(|v|^2)\,\vol(dx\,dv).
\endeq

(iv)  If $X$ is a $(J,I)$ tensor $\TM\to\TM$ and $\vphi:\R\to\R$, then
\begeq\label{intdh}
\nabla_x \left( \int_{T_xM} X(x,v) \vphi(|v|^2)\,dv \right) = \int_{T_xM} (\nabla_H X)(x,v)\,\vphi(|v|^2)\,dv.
\endeq
In particular, if $F=F(x,v)$ is a function $\TM\to \R$ and $\vphi:\R\to\R$, then
\begeq\label{fmldivx}
\nabla_x \cdot \left(\int_{T_xM} F(x,v)\, v\, \vphi(|v|^2)\,dv\right) = \int_{T_xM} (\xi F)\,\vphi(|v|^2)\,dv.
\endeq

(v) The adjoint of $\nabla_V$ in $L^2(\vol)$ is $-\nabla_V\cdot$; the adjoint of $\nabla_V$ in $L^2(\mu)$ is $-(\nabla_V-v)\cdot$; the adjoint of $\nabla_H$, either in $L^2(\vol)$ or in $L^2(\mu)$, is $-\nabla_H\cdot$. All those Lebesgue spaces are to be understood as valued in $\TIJ$, the space of $(J,I)$-tensors, for arbitrary $J,I$ in $\N_0$. In short,
\begeq\label{shortnot}
\nabla_V^* = -\nabla_V\cdot,\qquad \nabla_V^{*\mu} = - (\nabla_V - v)\cdot \qquad \nabla_H^* = -\nabla_H\cdot.
\endeq
\end{Prop}

\begin{Rks} All those formulas extend to vector fields which may not be vanishing at infinity, using density arguments; e.g. vector fields with polynomial growth if the reference measure is $\mu$. Further note that the last formula can be applied in particular to $\vphi(q) = e^{-q/2}$, thus recovering an integration by parts for the measure $\mu$. 
\end{Rks}

\begin{proof}[Proof of Proposition \ref{propdiv}]
For (i): \eqref{intdivV} is just the usual divergence theorem in Euclidean space $T_xM$, and \eqref{iintdivV} follows upon $x$-integration.

For (ii): First, here is the expression of $\nabla_H\cdot X$ in coordinates:
\[ \nabla_H\cdot X = \derpar{X^i}{x^i} + \Gamma_{is}^i X^s - \Gamma_{ik}^\ell v^k \derpar{X^i}{v^\ell}. \]
Let $(O_\theta)$ be an atlas of $M$, so there are only finitely many open subsets $O_\theta$, and let $(\eta_\theta)$ a partition of unity subordinate to that atlas, so each $\eta_\theta$ is supported in $O_\theta$, and $\sum \eta_\theta = 1$. Writing 
\[ \iint \nabla_H\cdot X \vphi = \iint \nabla_H\cdot (X \sum_\theta \eta_\theta)\vphi  = \sum_\theta \iint \nabla_H\cdot (X\eta_\theta)\vphi,\] 
we see that we just need to prove the vanishing of the integral when $X$ is supported in $T O_\theta$, and for that we may use a fixed set of coordinates $x^1,\ldots x^n, v^1,\ldots v^n$. Thus the quantity to evaluate is
\begeq\label{integal0}
\iint \left( \derpar{X^i}{x^i} + \Gamma_{is}^i X^s - \Gamma_{ik}^\ell v^k \derpar{X^i}{v^\ell}\right)\, (\det g)\, \vphi \left(\sum g_{k\ell} v^k\,v^\ell \right)\, dx^1\,\ldots dx^n\,dv^1\ldots dv^n.
\endeq
Integrate by parts in $x^i$ for the first term inside brackets, and in $v^\ell$ for the third term: Then \eqref{integal0} becomes

\begin{multline}
\iint \Bigl[ - X^i \left(\frac{\pa_i \det g}{\det g}\, \vphi(|v|^2) + \pa_i g_{k\ell} \,v^k v^\ell \vphi'(|v|^2) \right) \\
 +X^s\, \Gamma_{is}^i \vphi(|v|^2) + X^i \left(\Gamma_{ik}^k \vphi(|v|^2) + 2 \Gamma_{ik}^\ell v^k g_{\ell r} v^r \vphi'(|v|^2)\right) \,(\det g)\Bigr]\,dx^1\ldots dv^n.
\end{multline}

On the one hand, the terms proportional to $\vphi(|v|^2)$ cancel out in view of \eqref{idChrist}. On the other hand, applying $g^*$ to \eqref{christoffel} and differentiating with respect to $x^i$ yields
\[ 2 \Gamma_{ik}^\ell \, g_{\ell r} v^k v^r = \pa_i g_{kr} v^k v^r, \]
upon which the terms proportional to $\vphi'(|v|^2)$ cancel out also.
This proves (ii).

Now for (iii): Formula~\eqref{ibpV} is a standard integration by parts, fiberwise, in Euclidean $T_xM$, or can be deduced as a consequence of \eqref{iintdivV} since 
\begeq\label{nablaVprod}
\nabla_V\cdot \<X,Y\> = \< \nabla_V\cdot X, Y\> + \<X, \nabla_V\cdot Y\>
\endeq 
(an easy generalisation of the Euclidean formula $\nabla \cdot (fX) = \nabla f \cdot X + f\nabla\cdot X)$. Similarly, \eqref{ibpH} will follow from \eqref{iintdivH} and
\begeq\label{nablaHprod}
\nabla_H\cdot \<X,Y\> = \< \nabla_H \cdot X, Y\> + \<X, \nabla_H\cdot Y\>,
\endeq
but the latter is more tricky than \eqref{nablaVprod} because of the need to use covariant derivatives. Again this can be checked through coordinates. Let us consider for instance $I=J=1$ and compute the left hand side of \eqref{nablaHprod}:
\begin{align*} 
\nabla_H\cdot \<X,Y\> & = \nabla_i \Bigl( X_{i_1}^{j_1} Y^{i i_1}_{j_1} \Bigr) \\
& = \bigl( \pa_i X_{i_1}^{j_1} \bigr) Y^{i i_1}_{j_1}  + 
X_{i_1}^{j_1} \pa_i \bigl( Y^{i i_1}_{j_1} \bigr) + \Gamma_{si}^i X_{i_1}^{j_1} Y^{s i_1}_{j_1} \\
& \qquad\qquad\qquad\qquad  - \Gamma_{ir}^s v^r \derpar{X_{i_1}^{j_1}}{v^s} Y^{ii_1}_{j_1} 
- \Gamma_{ir}^s v^r X_{i_1}^{j_1}\derpar{Y^{ii_1}_{j_1}}{v^s}.
\end{align*}
The first term on the right hand side of \eqref{nablaHprod} is
\[
 \< \nabla_H X, \cdot Y\> = \bigl( \pa_i X_{i_1}^{j_1} \bigr) Y^{i i_1}_{j_1} 
+ \Gamma_{is}^{j_1} X^s_{i_1} Y^{i i_1}_{j_1} - \Gamma_{ii_1}^s X_s^{j_1} Y^{i i_1}_{j_1} 
- \Gamma_{ir}^s v^r \derpar{X_{i_1}^{j_1}}{v^s} Y^{ii_1}_{j_1},
\]
while the second one is
\[
 \< X, \nabla_H\cdot Y\> = X_{i_1}^{j_1} \bigl( \pa_i Y^{i i_1}_{j_1} \bigr) \\
 + \Gamma_{si}^i X_{i_1}^{j_1} Y^{s i_1}_{j_1}  + \Gamma_{is}^{i_1} X_{i_1}^{j_1} Y^{i s}_{j_1} - \Gamma_{ij_1}^{s} X^{j_1}_{i_1} Y^{i i_1}_{s}  \\
- \Gamma_{ir}^s v^r X_{i_1}^{j_1} \derpar{Y^{ii_1}_{j_1}}{v^s},
\]
and this yields \eqref{nablaHprod} since the terms $\Gamma_{is}^{j_1} X^s_{i_1} Y^{i i_1}_{j_1} - \Gamma_{ii_1}^s X_s^{j_1} Y^{i i_1}_{j_1}$ cancel out upon addition. (Obviously this is no accident, the formula of covariant derivative was built to preserve scalar products of vectors, or more generally tensors.)
\sm

To prove (iv), test the equation against an arbitrary $Y(x)$, then by (iii),
\begin{align*}
\int_M \Bigl \< Y(x), \nabla_x \int_{T_xM}  &  X(x,v)\, \vphi(|v|)\,dv \Bigr\> \,\vol(dx) \\
&  = - \int_M  \left \< \nabla_x\cdot Y(x), \int_{T_xM} X(x,v)\, \vphi(|v|)\,dv \right\> \,\vol(dx) \\
& = - \int_M  \left \< \nabla_H\cdot Y(x), \int_{T_xM} X(x,v)\, \vphi(|v|)\,dv \right\> \,\vol(dx)\\
& = - \iint_{\TM} \< \nabla_H Y, X\> \,\vphi(|v|)\, \vol(dx)\,dv\\
& = \iint_{\TM} \< Y,\nabla_H X\>\, \vphi(|v|)\,\vol(dx)\,dv \\
& = \int_M \left \< Y(x), \int_{T_xM} \nabla_H X(x,v) \,\vphi(|v|)\,dv \right\> \,\vol(dx),
\end{align*}
which proves the claim. Then \eqref{fmldivx} is obtained thanks to \eqref{xirevu}.

Finally (v) is just a rewriting of (iii), which concludes the proof of Proposition \ref{propdiv}.
\end{proof}

Here is an important consequence of Proposition \ref{propdiv}:

\begin{Cor}[Conservativity of the geodesic derivation]
If $F:\TM\to\R$ is a smooth function vanishing at infinity, then
\[ \iint_{\TM} (\xi F) F = 0,\]
where the measure is either $\mu$ or Liouville. More generally, for any smooth function $\Phi:\R\to\R$, such that $\Phi'(F)$ vanishes at infinity,
\[ \iint_{\TM} (\xi F) \Phi'(F) = 0.\]
\end{Cor}

\begin{Rk} \label{rkXXnot0} From formula \eqref{Xiexpl}, $\int \<X,\Xi X\>_{L^2} $ might not be equal to~0, but it is still controlled:
\begeq\label{XXiX}
\begin{cases}
\dps \iint \left| \< X_H, (\Xi X)_H \> \right| \leq C \iint |X_H| \, |X_V| \, |v|^2\\[2mm]
\dps \iint \left| \<X_V, (\Xi X)_V\> \right | \leq C \iint |X_H|\, |X_V|\
\end{cases}
\endeq
(where the measure is either $\mu$ or Liouville). 
\end{Rk}

Next, proposition \ref{propdiv} makes it possible to introduce the vertical and horizontal Laplace operators:
\begeq\label{Laplaces} 
\Delta_V = -\nabla_V^*\nabla_V = \nabla_V\cdot \nabla_V
\qquad
\Delta_H = - \nabla_H^*\nabla_H = \nabla_H \cdot \nabla_H
\endeq
as well as the Gaussian vertical Laplace operator
\begeq\label{Laplacemu}
\Delta_V^\mu = -\nabla_V^{*\mu} \nabla_V = \Delta_V - v\cdot\nabla_V.
\endeq
The explicit expression of $\Delta_V$ and $\Delta_V^\mu$ is simple:
\begeq\label{coordDeltaV}
\Delta_V = g^{ij}\derpar{}{v^i} \derpar{}{v^j} \qquad
\Delta_V^\mu = g^{ij}\derpar{}{v^i} \derpar{}{v^j} -v^i\derpar{}{v^i}.
\endeq
This yields again the operator involved in the (geometric) kinetic Fokker--Planck equations \eqref{kFP} or \eqref{kFPh}; and this is the same for functions or tensors. The expression of $\Delta_H$ is more tricky:
\begeq\label{coordDeltaH}
\Delta_H = g^{ij} (\nabla_H)_i (\nabla_H)_j,
\endeq
and $(\nabla_H)_i (\nabla_H)_j$ was computed in the proof of Proposition \ref{propcomH}. Using symmetry $i\leftrightarrow j$ in \eqref{coordDeltaH} leads to minor simplifications. Evaluated on functions, the horizontal Laplace operator reads
\begin{multline} \label{DeltaHf}
\Delta_H f = g^{ij} \left[ \derpar{{}^2f}{x^i\, \pa x^j} - \Gamma_{ij}^\ell \derpar{f}{x^\ell} - 2 \Gamma_{jr}^\ell v^r \derpar{{}^2f}{x^i\, \pa v^\ell}  \right. \\
\left. + \bigl( - \pa_i \Gamma_{jk}^\ell v^k + \Gamma_{ir}^s \Gamma_{js}^\ell v^r + \Gamma_{ij}^a
\Gamma_{as}^\ell v^s \bigr) \derpar{f}{v^\ell} + \Gamma_{ir}^s \Gamma_{jk}^\ell v^r v^k \derpar{{}^2 f}{v^s\, \pa v^\ell} \right].
\end{multline}
(Of course this reduces to $\Delta_x f = g^{ij} (\derpar{{}^2f}{x^i\, \pa x^j} - \Gamma_{ij}^\ell \derpar{f}{x^\ell})$ if $f$ is horizontal.)
But applied to tensors $\Delta_H$ has a significantly more cumbersome expression, whose complexity grows with the number of indices involved; for instance to $(1,1)$-tensors:
\begin{multline}
\Delta_H X_k^\ell 
= g^{ij} \Bigl[ \derpar{{}^2 X_k^\ell}{x^i\, \pa x^j} - \Gamma_{ij}^s \derpar{X_k^\ell}{x^s} + 2 \Gamma_{js}^\ell \derpar{X_k^s}{x^i} - 2 \Gamma_{jk}^s \derpar{X_s^\ell}{x^i} - 2 \Gamma_{js}^r v^s \derpar{{}^2 X_k^\ell}{x^i\,\pa v^r}  \\
- \bigl( \pa_i \Gamma_{js}^r v^s - \Gamma_{ij}^a \Gamma_{ab}^r v^s\bigr) \derpar{X_k^\ell}{v^r} - 2 \Gamma_{ia}^\ell \Gamma_{js}^r v^s \derpar{X_k^a}{v^r} 
+ 2 \Gamma_{ik}^a \Gamma_{js}^r v^s \derpar{X_a^\ell}{v^r} + \Gamma_{ia}^b \Gamma_{jr}^s\, v^a v^s\derpar{{}^2X_k^\ell}{v^b\, \pa v^r}\\
+ \bigl( \pa_i \Gamma_{js}^\ell + \Gamma_{ia}^\ell \Gamma_{js}^a - \Gamma_{ij}^a \Gamma_{as}^\ell \bigr) _k^s - \bigl(\pa_i \Gamma_{jk}^s - \Gamma_{ik}^a \Gamma_{ja}^s - \Gamma_{ij}^a\Gamma_{ak}^s\bigr) X_s^\ell -\Gamma_{ia}^\ell \Gamma_{jk}^s X_s^a \Bigr] .
\end{multline}
(As usual $\Delta_H X_k^\ell$ means $(\Delta_H X)_k^\ell$, and not $\Delta_H (X_k^\ell)$ which is anyway meaningless.)

From \eqref{coordDeltaV} and \eqref{coordDeltaH} follows the expression of the {\bf total Laplacian}, which is the Laplace operator on $\TM$ endowed with the Sasaki metric:
\begeq\label{DDD} \Delta = \Delta_H+\Delta_V= g^{ij} (\nabla_H)_i (\nabla_H)_j + g^{ij}\derpar{}{v^i} \derpar{}{v^j}.
\endeq
Further, as a consequence of Proposition \ref{propHV} we have
\begeq\label{comDVDH}  [\Delta_H,\Delta_V] = 0, \qquad [\Delta_H,\Delta_V^\mu] =0,
\endeq
in the sense that $\Delta_H\Delta_V X = \Delta_V\Delta_H X$, whatever the smooth tensor $X$, and likewise when $\Delta_V$ is replaced by $\Delta_V^\mu$.

From the definition of these Laplace operators also follow the key Dirichlet identities
\begeq\label{DeltaVXX}
\begin{cases}
\dps \int \<\Delta_VX,X\> \vol(dx\,dv) = - \int |\nabla_V X|^ 2\,\vol(dx\,dv) \\[4mm]
\dps \int \<\Delta_V^\mu X,X\> \mu(dx\,dv) = - \int |\nabla_V X|^2\,\mu(dx\,dv)
\end{cases}
\endeq
\begeq\label{DeltaHXX}
\begin{cases} \dps \int \<\Delta_HX,X\> \vol(dx\,dv) = - \int |\nabla_H X|^ 2\,\vol(dx\,dv)\\[4mm]
\dps \int \<\Delta_H X,X\> \mu(dx\,dv) = - \int |\nabla_H X|^2\,\mu(dx\,dv).
\end{cases}
\endeq
In particular, $\Delta_V$ (resp. $\Delta_V^\mu$) is self-adjoint nonpositive in tensor-valued $L^2(\vol)$ (resp. $L^2(\mu)$), and $\Delta_H$ is self-adjoint nonpositive both in tensor-valued $L^2(\vol)$ and tensor-valued $L^2(\mu)$. Those statements should be understood in the sense of linear operators defined on a $L^2$-dense set of smooth functions, say $C^\infty$ functions in $x$ and $v$ with fast decay as $|v|\to\infty$.

These Laplace operators also inherit some basic invariances from those of the derivations:

\begin{Prop}[Basic invariance properties of the Laplace operators] \label{basicLapinv}

(i) If $f=f(x)$ then $\Delta_V f =0$; and more generally, for any function $u$, one has $\Delta_V(uf) = (\Delta_Vu) f$.

(ii) If $f=\vphi(|v|^2)$ is radially symmetric, then $\Delta_H f =0$; and more generally, for any function $u$, one has $\Delta_H(uf) = (\Delta_H u) f$;

(iii) If $f$ is smooth and rapidly decaying as $|v|\to\infty$, and $\vphi$ is any smooth function $\R\to\R$ with at most polynomial growth, then
\begeq\label{Deltaint}
\Delta_x \int_{T_xM} f(x,v)\, \vphi(|v|^2)\,\vol_x(dv) = \int_{T_xM} \Delta_H f(x,v)\,\vphi(|v|^2)\,\vol_x(dv),
\endeq
where $\Delta_x$ stands for the classical Laplace operator on $M$; and likewise
\begeq\label{Deltamuint}
\Delta_x \int_{T_xM} f(x,v)\, \vphi(|v|^2)\,\mu_x(dv) = \int_{T_xM} \Delta_H f(x,v)\,\vphi(|v|^2)\,\mu_x(dv).
\endeq
\end{Prop}

\begin{proof}[Proof of Proposition \ref{basicLapinv}]
(i) and (ii) follow at once from Proposition \ref{basicinv}. As for (iii), it is best seen by duality: If $h = h(x)$ is a smooth test-function, then by formal self-adjointness of $\Delta_x$ in $L^2(M)$, of $\Delta_H$ in $L^2(\TM)$ and Proposition \ref{basicinv} again,
\begin{align*}
\int \Bigl ( \int f(x,v)\, \vphi(|v|^2)\,  & \vol_x(dv)\Bigr) \Delta_x h(x)\, \vol(dx) \\
& = \iint f(x,v) \,\vphi(|v|^2) \Delta_x h(x)\,\vol(dx)\, \vol_x(dv)\\
& =  \iint f(x,v)\, \vphi(|v|^2) \Delta_H h(x)\,\vol(dx)\, \vol_x(dv)\\
& =  \iint \Delta_H \bigl( f(x,v) \,\vphi(|v|^2)\bigr) h(x)\,\vol(dx)\, \vol_x(dv)\\
& =  \iint (\Delta_H f)(x,v) \,\vphi(|v|^2)\, h(x)\,\vol(dx)\, \vol_x(dv)\\
& =  \int \left(\int (\Delta_H f)(x,v)\,\vol_x(dv)\right) \vphi(|v|^2)\, h(x)\,\vol(dx).
\end{align*}
\end{proof}

One cannot in general say more. A well-known property of the Euclidean Laplace operator in $\R^n$ is that it commutes with isometries: for any isometry $R$, $\Delta (f\circ R) = (\Delta f)\circ R$. Likewise, $\Delta_H$ and $\Delta_V$ will be invariant under any horizontal and vertical isometry, respectively; but in general there are just hardly any such isometries. For horizontal ones this was to be expected, since the horizontality equation is so constraining. As for vertical isometries, one could first think of something like $R(x,v) = (x,\sigma(v))$, where $\sigma$ is a linear isometry. But that is a mistake: one cannot speak of $\sigma$ independently of $x$ since it will depend on the choice of parameterization of $T_xM$ (except for very particular choices like $\sigma(v) = -v$). Then actually one should write $R(x,v) = (x,\sigma_x(v))$, and this induce variations also in $x$, not just in $v$, and the equations for $R$ in general have no solution.

\bibnotes

Divergence formulas are classical and a cornerstone of Riemannian analysis in general. The horizontal, vertical and total Laplace operators appear in various sources, e.g. \cite{GHL:Riemann:book}.

\section{Spectral radioscopy} \label{secspectral}

Ever since Fourier, the analysis of linear partial differential equations has gone hand in hand with spectral theory, be it in classical or quantum mathematical physics, chaotic dynamics, geometry, probability, and so on. To manipulate $\Delta_H$ and $\Delta_V$ let us submit them to spectral investigation, on the one hand to get a better understanding of their properties, on the other hand to complete the proof of  their crucial commutation property. On the latter issue, it seems that it was already established in \eqref{comDVDH}, but this was as an identity for operators acting on smooth functions. To apply functional analysis here, in particular the codiagonalisation theorem for commuting unbounded operators, a more demanding notion of commutation is needed, namely the commutation of spectral projections. In finite dimension the commutation of symmetric operators is equivalent to the commutation of their spectral projections, but in infinite dimension this is not always the case. So let us go spectral for a little while. I will start with elementary reminders, and at times go into lengthy subtleties just to check that eventually things work out well (Proposition \ref{propdhdxcom}). The reader who is ready to trust that, can skip this section, save for the spectral analysis of $\Delta_v^\mu$ (Subsection~\ref{subDeltaVmu}).

The simplest case $M=\T^n = \R^n/\Z^n$ already involves three iconic cases of spectral analysis: $\Delta_x$ on $\T^n$; $\Delta_v$ on $\R^n$; $\Delta_V^\mu$ on $(\R^n,\mu)$. Here are the basic facts and their partial generalisations.

\subsection{\underline{$\Delta$ on $\T^n$}}

The spectrum of $-\Delta$ on $L^2(\T^n)$ is made of just eigenvalues, and normalised eigenfunctions are
\[
\begin{cases}
c_0(x) = 1 \\ \\
\dps c_k(x) = \frac1{\sqrt{2}}\, \cos(2\pi k\cdot x), \quad s_k(x) = \frac1{\sqrt{2}} \sin(2\pi k\cdot x). 
\end{cases} \]
Then $-\Delta c_k = 4 \pi^2 |k|^2 c_k$, $-\Delta s_k = 4\pi^2 |k|^2 s_k$, with 
\[ |k|^2 = \sum_{i=1}^n k_i^2,\]
so the spectrum is
\[\sigma(-\Delta) = \bigl\{ 4\pi^2 |k|^2; \qquad k\in \Z^n \bigr\}. \]
Further, counting with muiltiplicity, the number $N(\alpha)$ of eigenvalues at most $\alpha^2$ is the number of vectors $k\in\Z^n$ with $|k|\leq \alpha/(2\pi)$, that is the number of points with integer coordinates in the ball $B(0,\alpha/(2\pi))$. In particular,
\[ N(\alpha) \simeq |B^n| \left( \frac{\alpha}{2\pi}\right)^n \qquad \text{ as $\alpha\to\infty$}. \]
Actually there is enormous multiplicity, so eigenfunctions may be much more complicated than just the functions $c_k, s_k$ and their qualitative behaviour is still an active topic of research, in relation with ``quantum chaos''.

\subsection{\underline{$\Delta$ on $M$}}

More generally, if $M$ is a compact $n$-dimensional manifold, then the spectrum of $-\Delta$ is discrete and the number of eigenvalues (with multiplicity) no less than $\alpha^2$ is 
\[ N(\alpha) \simeq  \vol(M)\, |B_n|\, \left(\frac{\alpha}{2\pi}\right)^n \qquad \text{as $\alpha\to\infty$},\]
where $\vol(M)$ is the volume of $M$. Moreover the eigenfunctions $\vphi_\alpha$ for 
\[ -\Delta\vphi_\alpha = \alpha^2 \vphi_\alpha \]
satisfy
\[ \|\vphi_\alpha\|_{L^\infty(M)} = O \left(\alpha^{\frac{n-1}{2}}\right).\]
These are {\bf Weyl's estimates}.

\subsection{\underline{$\Delta$ on $\R^n$}}

Here the variable will be denoted $v$. The spectrum of $-\Delta$ in $L^2(\R^n)$ is just continuous spectrum, 
\[ \sigma(-\Delta) = \R_+,\] 
and there are generalised eigenfunctions (only generalised because they do not belong in $L^2(\R^n)$)
\[ c_\eta(v) = \cos (2\pi\eta\cdot v),\qquad s_\eta(v) = \sin (2\pi\eta\cdot v)\]
with $\eta\in\R^n$ ($\eta\neq 0$ for $s$-functions). The corresponding eigenvalues are $\lambda = 4 \pi^2 |\eta|^2$, so
\[ |\eta| = \frac{\sqrt{\lambda}}{2\pi}. \]
Fourier's inversion formula, in the form
\begin{align*} 
f(v) & = \int_{\R^n} e^{2i\pi\eta\cdot v} \hat{f}(\eta)\,d\eta\\
& = \iint_{\R^n\times\R^n} e^{2 i \pi \eta\cdot (v-w)} f(w)\,dw\,d\eta
\end{align*}
leads to
\begin{align*}
f & = \int_{\R^n} (c_\eta\otimes c_\eta + s_\eta\otimes s_\eta)\,f\,d\eta \\
& = \iint_{\S^{n-1}\otimes\R_+} \bigl( c_{\alpha\sigma}\otimes c_{\alpha\sigma} + s_{\alpha\sigma} \otimes s_{\alpha\sigma} \bigr)\, f\, \alpha^{n-1}\,d\alpha\,d\sigma\\
& = \int_{\R_+} \frac{\lambda^{\frac{n-1}2}}{(2\pi)^{n-1}}\,\frac{d\lambda}{4\pi\sqrt{\lambda}}\,
\int_{\S^{n-1}} \bigl ( c_{\sqrt{\lambda}\sigma} \otimes c_{\sqrt{\lambda}\sigma} + s_{\sqrt{\lambda}\sigma} \otimes s_{\sqrt{\lambda}\sigma}\bigr)\,f\,d\sigma.
\end{align*}
So the spectral measure of $-\Delta$ is
\begeq\label{smDelta}
P(d\lambda) = \left(\frac{\lambda}{2\pi}\right)^{\frac{n}2}
\left( \int_{\S^{n-1}} \bigl ( c_{\sqrt{\lambda}\sigma} \otimes c_{\sqrt{\lambda}\sigma} + s_{\sqrt{\lambda}\sigma} \otimes s_{\sqrt{\lambda}\sigma}\bigr)\, d\sigma\right)\,\frac{d\lambda}{2\lambda}.
\endeq
While the integrand does not define a projection in $L^2(\R^n)$, any average $\int \vphi(\lambda)\,P(d\lambda)$, for, say, $\vphi\in L^\infty(\R_+)$, does.

\subsection{\underline{$\Delta^\mu$ on $(\R^n,\mu)$}} \label{subDeltaVmu}

Recall that
\[ \mu(dv) = \frac{e^{-\frac{|v|^2}{2}}}{(2\pi)^{n/2}}\,dv \]
(no $x$ variable for the moment) and
\[ \Delta^\mu h = \Delta_v h - v\cdot\nabla_v h.\]
The spectral problem is well-known and the spectrum is purely discrete:
\begeq\label{sigmaDeltamu}
 \sigma(-\Delta^\mu) = \N_0 = \{0, 1, 2, \ldots\}.
 \endeq
An eigenbasis is provided by normalised multidimensional Hermite polynomials: for $m\in\N_0$ this is
\begin{align} \label{Hermite}
\vphi_m(v) & = \frac{(-1)^m}{\sqrt{m!}} e^{\frac{|v|^2}{2}}\, \left(\frac{d}{dv}\right)^m \, e^{-\frac{|v|^2}{2}}\\
\nonumber & = \frac1{\sqrt{m!}}{\cal H}_m(v), \qquad m \in \N_0^n,
\end{align}
and ${\cal H}_m$ is associated with the eigenvalue $m$.
The first few (not normalised) Hermite polynomials are
\begin{align*} {\cal H}_0(X) & = 1 \\
{\cal H}_1(X) & = X\\
{\cal H}_2(X) & = X^2-1\\
{\cal H}_3(X) & = X^3-3X\\
{\cal H}_4(X) & = X^4 - 6 X^2 + 3\\
{\cal H}_5(X) & = X^5 - 10 X^3 + 15 X.
\end{align*}
The Hermite polynomials satisfy some neat algebraic identities, such as
\begeq\label{Herm0}
\int_\R {\cal H}_m\,d\mu = m!
\endeq
\begeq\label{Herm1}
{\cal H}'_m = m {\cal H}_{m-1}
\endeq
\begeq\label{Herm2}
X {\cal H}_m = {\cal H}_{m+1} + m {\cal H}_{m-1}
\endeq
and by iteration
\begeq\label{Herm3}
X^2 {\cal H}_m = {\cal H}_{m+2} + (2m+1) {\cal H}_m + m(m-1) {\cal H}_{m-2}.
\endeq

In dimension $n$, formula \eqref{Hermite} remains true, but now
\[ (-1)^m = (-1)^{m_1}\ldots (-1)^{m_n}\ ,\qquad m! = m_1!\ldots m_n!\ , \qquad \left(\frac{d}{dv}\right)^m = 
\left(\frac{d}{dv_1}\right)^{m_1}\ldots \left(\frac{d}{dv_n}\right)^{m_n},\]
and the eigenvalue is
\[ s = |m| = m_1+\ldots + m_n.\]
Moreover, if $s$ is given, the dimensionality $N$ of the eigenspace ${\cal E}_s$ is the number of ways to decompose $s$ into such a sum, that is
\[ N(s,n) = \Cnk{s+n-1}{n-1}.\]
Then the spectral projection $P_s$ on ${\cal E}_s$ can be written
\begeq\label{Ps}
P_s = \sum_{s\in \N_0} \left( \sum_{|m|=s} \vphi_m\otimes \vphi_m\right)\,\delta_s,
\endeq
where 
\begeq\label{phim} \vphi_m = \frac1{\sqrt{m!}} {\cal H}_m \qquad (m\in\N_0^n) \endeq
are the normalised Hermite polynomials.

\subsection{\underline{Spectral analysis on $\TM$: preliminaries and $\Delta_V$}}

Now turn to the study of $\Delta_H$ and $\Delta_V$ on $\TM$. The easiest situation is $M=\T^n$, then $\TM = \T_x^n\times \R^n_v$, $\Delta_H = \Delta_x$, $\Delta_V = \Delta_v$, $\Delta_V^\mu = \Delta_v^\mu$. So $\sigma(-\Delta_H) = \sigma(-\Delta_x)$ in $L^2(T\T^n)$, but now all multiplicities are infinite, since, for any eigenfunction $\vphi$ of $-\Delta_x$, the function $\vphi(x)\,a(v)$ is an eigenfunction of $-\Delta_H$, for any $a\in L^2(\R^n)$ (be it for the $dv$ or for the $d\mu$ measure). Likewise, $\sigma(-\Delta_V) = \R_+$ with purely essential spectrum, and $\sigma(-\Delta_V^\mu) = \N_0$, with infinite multiplicities (choosing now $a(v)$ to be an eigenfunction of $-\Delta_V^\mu$ and $\vphi$ arbitrary).

Nice and easy, but in general there is no way to globally define $\Delta_x$ on $\TM$, and that was one motivation for $\Delta_H$. And as for the vertical operators, even if, say, $\Delta_V^\mu$ is just a way to globally define $\Delta_v^\mu$, the particular eigenfunctions we would like to use will not be globally defined. True, for each $x$, $-\Delta_V^\mu$ acting on $L^2(T_xM, \mu_x)$ will be similar to $-\Delta_v^\mu$ over $L^2(\R^n,\mu)$, with distinguished eigenbasis $\vphi_{m,x}$ obtained by normalising Hermite polynomials. But this goes through an arbitrary correspondence between $T_xM$ and $\R^n$; and in general, for a given $m$ it will be impossible to make that correspondence in a global, consistent way as $x$ varies.

Let me be more specific with a most simple example: $M=\S^2$, $s=1$. The multiplicity is 2 and normalised eigenfunctions of $-\Delta_v^\mu$ are $v_1,v_2$. Defining them globally amounts to choosing a smooth orthonormal basis $(e_1(x),e_2(x))$ of $T_x\S^2$. But that is impossible: the hairy sphere theorem guarantees that there is no way to globally define a smooth unit tangent vector field $e_1$ on $\S^2$. 

To circumvent that obstruction, use local coordinates for $T_xM$ to define $\vphi_m$ locally. In the sphere example, away from the North and South poles one may define a unit vector field $e_1 = (e_1^1, e_1^2,0)\in \R^3$ with $|e_1|=1$, then define $e_2$ as the vector product of $e_1$ with the normal vector; repeat the procedure in the neighbourhood of the North and South poles. Thus there are open sets $(O_\theta)_{1\leq\theta\leq 3}$ covering $\S^2$ and orthonormal bases $(e_{\theta,1}(x),e_{\theta,2}(x))$. Let $(\chi_\theta(x))_{1\leq\theta\leq 3}$ be supported in $O_\theta$ and valued in $[0,1]$ in such a way that $\sum \chi_\theta=1$; one can isometrically embed $L^2(\TM,\mu)$ into $\H = \prod_{\theta} L^2({\rm T} O_\theta,\mu_\theta)$ through $I: h\longmapsto (h\restr{O_\theta})$ and $\mu_\theta = \chi_\theta\,\mu$. Similarly, the operator $\Delta_V$ is identified to the tensor product of all $\Delta_V$'s acting on the tangent bundle of each $O_\theta$.

Also on each $O_\theta$, use a chart $\Psi_\theta:U_\theta\to O_\theta$, and identify $L^2(\TM)$ with $\prod L^2(U_\theta\times \R^n,\mu_\theta)$ through 
\[ h \longmapsto \bigl( (h\chi_\theta) \circ \Psi_\theta \bigr)_{\theta}, \qquad \mu_\theta = (\Psi_\theta^{-1})_\#\mu.\]
Now for each $\theta$ one may define ${\cal H}_m$ and $\vphi_m$. If $h\in {\cal E}_1$, then $h\chi_\theta\in {\cal E}_1$ for each $\theta$, and for $x$ in $O_\theta$ one may uniquely decompose $h(x,\cdot )$ into a combination of $\vphi_m(x,\cdot)$, with coefficients depending smoothly on $h$ by uniqueness of the decomposition and smoothness of the $\vphi_m$. 

More generally, for any compact Riemannian manifold $M$ one may introduce a covering $(O_\theta)_{\theta\in \Theta}$, where $\Theta$ is some finite set, by open sets $O_\theta$, each of which has its orthonormal basis $(e_{\theta,i})_{1\leq i\leq i}$, written $(e_{\theta,i}^r)_{1\leq i\leq n, 1\leq r\leq n}$;
a subordinate smooth partition of unity $(\chi_\theta)_{\theta\in \Theta}$, and for any $m\in \N_0^m$ one may define the eigenfunction $\vphi_m$, not as a function on $\TM$, but as a vector valued function
\[ \Phi_m(x,v) = \frac1{\sqrt{\vol(M)}} \Bigl(\chi_\theta(x)\, \vphi_m (e_{\theta,1}(x)\cdot v, \ldots e_{\theta,n}(x)\cdot v) \Bigr)_{\theta\in\Theta}\]
which may be written synthetically
\[ \Phi_m(x,v) = \frac1{\sqrt{\vol(M)}} \bigl (\chi_\theta(x)\, \vphi_m(E_\theta(x)\cdot v)\bigr)_{\theta\in\Theta}.\]
For any $s$ and any $\theta$ there are $N(s,n)$ such functions; one can see this collection as a set of $N(s,n)$ functions valued in $\R^\Theta$ and they constitute a local orthonormal basis of the $s$-eigenspace of $-\Delta_V^\mu$ in $\H$, still denoted ${\cal E}_s$. Note that a priori $\Phi_m(x,v)$ is not in the image of $I$ (does not correspond to a definite function on $\TM$, but to a collection of functions, each of them defined on the tangent bundle of $O_\theta$). In this way any $h:\TM\to\R$ belonging to ${\cal E}_s$ may be written as
\[ h(x,v) = \sum_\theta \sum_{|m|=s} \chi_\theta(x) \,\beta_{\theta,m}(x)\, \vphi_{\theta,m}(x,v),\]
where the coefficients $\beta_{\theta,m}(x)$ are uniquely defined and as smooth as $h$.

Of course there is an arbitrariness here in the choice of the eigenbasis, which does not leave the coefficients $\beta_m$ invariant as it will change the eigenfunctions $\vphi_m$. For example, consider $s=2$, $n=2$; under the change of orthonormal basis
\[ (e_1,e_2)\longmapsto \left(\frac{e_1+e_2}{\sqrt{2}} , \frac{e_1-e_2}{\sqrt{2}}\right)\]
the function $v_1 v_2$ becomes $(1/2)[(v_1^2-1) + (v_2^2-1)]$. In other words $\vphi_{(1,1)}$ is transformed into $(\vphi_{(0,2)} + \vphi_{(2,0)})/\sqrt{2}$.

One may regain intrinsic nature by abandoning the idea of defining $\vphi_m$ for each $m$, but instead considering the projector $P_s$ on ${\cal E}_s$. In the example of $\S^2$, $P_1$ is $v_1\otimes v_1 + v_2\otimes v_2$, or more rigorously
\[ P_1 = (e_1(x)\cdot v) \otimes (e_1(x)\cdot v) + (e_2(x)\cdot v)\otimes (e_2(x)\cdot v).\]
This is the orthogonal projector on the vector space generated by $\{v_1,v_2\}$ or equivalently on the vector space of linear functions of $v$, which is also characterised as the 1-eigenspace of $-\Delta_V^\mu$. That projector is intrinsically defined even if $e_1,e_2$ are not. 

More generally, for any $M$, for any $s\in\N_0$, the $x$-dependent orthogonal projection on ${\cal E}_s$, namely
\[ P_{x,s} = \sum_{|m|=s} \frac1{m!} \bigl[ {\cal H}_m (e_1\cdot v,\ldots e_n\cdot v)\bigr]\otimes_v
\bigl[ {\cal H}_m (e_1\cdot v,\ldots,e_n\cdot v) \bigr]\]
varies smoothly in $x$ (its kernel is a smooth function of $x$ and $v$). Here I wrote $k(x,v)\otimes_v k(x,v)$ as a shorthand for
\[ f \longmapsto \left[ \int k(x,w)\,f(x,w)\,\mu_x(dw) \right]\, k(x,v).\]
In other words, writing $\vphi_m(x,v)$ requires an arbitrary choice of local orthonormal basis; but $P_{x,s}=\sum \vphi_m(x,v)\otimes \vphi_m(x,v)$ is intrinsically defined and $x$ appears just as a parameter. This gives a description of all eigenfunctions and spectral projectors of $-\Delta_V^\mu$ -- just the same as for $-\Delta^\mu$, only with the $x$ variable added:
\[ -\Delta_V^\mu = \sum_{s\in\N_0} s\, P_{x,s}.\]
There is no free lunch, and to express $P_{x,s}$ in ``concrete'' calculations one will be led to choose a chart, and an orthonormal basis of $T_xM$ inducing a basis for (symmetric) operators in $T_xM$.

Finally, a similar procedure can be done for $-\Delta_V$, which actually can be obtained from $-\Delta_V^\mu$ through a limiting procedure. Indeed, upon rescaling one may replace $e^{-|v|^2/2}$ by $\mu_\var(dv) = e^{-\var|v|^2/2}$, then $-\Delta_V^\mu$ by $-\Delta_V - \var v\cdot\nabla_V$, and in the limit $\var\to 0$ this converges to just $-\Delta_V$, and the associated spectral measure will also converge weakly.

\subsection{\underline{Spectral analysis of $\Delta_H$}}

Now comes the more tricky analysis of $\Delta_H$. Let us do it in $L^2(\mu)$, keeping in mind that it will be the same, up to rescaling, as in $L^2(\mu_\var)$, and that the analysis of $\Delta_H$ in $L^2(dx\,dv)$ will come as a limit case.

As a first observation, $\sigma(-\Delta_x)\subset \sigma(-\Delta_H)$: Indeed, whenever $h=h(x)$ is an eigenfunction of $-\Delta_x$ with eigenvalue $-\lambda$, then in view of Proposition \ref{basicLapinv} the same is true for $-\Delta_H$. 

As a partial converse, if $h=h(x,v)$ is an eigenfunction of $-\Delta_H$, then $f(x)=\int h(x,v)\,\vphi(|v|^2)\,dv$ also satisfies $-\Delta_x f = \lambda f$. But this does not prove that $f$ is an eigenfunction of $-\Delta_x$, since $\int h(x,v)\vphi(|v|^2)\,dv$ might be zero, whatever the choice of $\vphi$ (for instance if $h(x,-v)=-h(x,v)$). So a more precise study is needed.

Since $\Delta_H$ commutes as a differential operator with $\Delta_V$, and since the eigenfunctions of $\Delta_V$ are smooth, $\Delta_H$ leaves any ${\cal E}_s$ invariant, and $-P_s \Delta_H = -\Delta_H P_s$. Thus the spectral decomposition of $-\Delta_H$ is the sum of all spectral decompositions of the symmetric operators $-\Delta_H P_s = P_s (-\Delta_H) P_s$ (converging as a nonnegative sum).

To determine the spectrum of $-\Delta_H P_s$ one has to see whether the equation
\begeq\label{DeltaHlf} 
(-\Delta_H - \lambda) f = h
\endeq
is uniquely solvable or not in ${\cal E}_s$. We shall see that this leads to a discrete spectrum with finite dimensional multiplicity. (This does not exclude a priori that there is infinite multiplicity or continuous spectrum when summing over $s\in\N_0$.)

To solve \eqref{DeltaHlf} localise in charts with an atlas $(O_\theta)_{\theta\in \Theta}$ and a partition of unity $(\chi_\theta)_{\theta\in\Theta}$; choose an orthonormal basis $(e_{\theta,i}(x))_{1\leq i\leq n}$ in each $O_\theta$; then there are the induced eigenfunctions $\vphi_{\theta,m}(x,v)$. Any computation related to $\Delta_H$ or to the associated quadratic form $\int |\nabla_Hf|^2$ thus reduces in a chart, for given $\theta$, to a computation on a function 
\begeq\label{fHm} f(x,v) = \sum_m \alpha_m(x)\, \vphi_m(E(x)\cdot v).
\endeq
In the sequel, such a context is fixed, so $\vphi_m(x,v)$ stands for $\vphi_m(E(x)\cdot v)$. Any local computation such as $\Delta_H f$ can be made in such a chart. Any global quantity such as $\int |\nabla_H f|^2\,d\mu$ requires to first turn to intrinsic quantities, such as $|\nabla_H f|^2 = \sum_\theta (\chi_\theta(x) |\nabla_Hf|^2)$, where each term in the sum can be computed through local coordinates.

From the stability of ${\cal E}_s$ under $\nabla_H$ we know that there are coefficients $G_{qm}^\ell(x)$ and $D_m^\ell(x)$ ($1\leq q\leq n$; $\ell, m \in \N_0^n$; $|\ell|=|m|=s$), such that
\begeq\label{Gqml}
(\nabla_H)_q \vphi_m = \sum_\ell G_{qm}^\ell\,\vphi_\ell
\endeq
Then $(\nabla_H)_q f = \nabla_q \alpha_m\, \vphi_m + \alpha_m\, (\nabla_H)_q \vphi_m$ becomes
\begeq\label{nablaHqf} 
(\nabla_H)_q f = \bigl( \nabla_q \alpha_\ell \delta^\ell_m + \alpha_\ell G_{q\ell}^m\bigr) \,\vphi_m
\endeq
(implicit summation over $m$ and $\ell$).
Likewise, $-\Delta_H f = -\Delta\alpha_m\vphi_m - 2 \nabla\alpha_m\cdot\nabla_H\vphi_m - \alpha_m \Delta_H\vphi_m$ becomes
\begeq\label{DeltaHqf}
-\Delta_Hf = \Bigl( -\Delta \alpha_\ell \delta_\ell^m - 2 g^{pq} \pa_p \alpha_\ell\, G_{q\ell}^m - \alpha_\ell\, D_\ell^m\Bigr)\vphi_m.
\endeq

Let us see how to compute $G_{q\ell}^m$ and $D_\ell^m$. If $(e_i)$ is the basis, define coefficients $\om_{qi}^j(x)$ by the identity
\begeq\label{basemobile}
\nabla_q e_i = \om_{qi}^j e_j.
\endeq
From the identity $\nabla g(e_i,e_j) = 0$ it follows that
\[ \om_{qi}^j = - \om_{qj}^i.\]
Furthermore,
\begin{align*} \Delta e_i & = g^{pq} \nabla_q \nabla_p e_i\\
& = g^{pq} \Bigl( \nabla_p \om_{qi}^j + \om_{qi}^k\, \om_{pk}^j \Bigr) e_j.
\end{align*}
Using the formula for horizontal derivation,
\begeq\label{nhqeiv}
(\nabla_H)_q (e_i\cdot v) = (\nabla_q e_i)\cdot v.
\endeq
Combining this with chain rule and \eqref{Herm1},
\[ (\nabla_H)_q \vphi_m(x,v) = 
\frac1{\sqrt{m!}} \sum_i m_i {\cal H}_{m-1_i} (E\cdot v)\, \om_{qi}^j(x)\,(v\cdot e_j).\]
The contribution of $j=i$ vanishes since $\om_{qi}^i=0$. Using \eqref{Herm2} this leads to
\begeq\label{itremainsHi}
 (\nabla_H)_q \vphi_m(x,v) = \frac1{\sqrt{m!}} \sum_{i\neq j} {\cal H}_{m-1_i+1_j}\, \om_{qi}^j
 \endeq
($1\leq i,j,q\leq n$, $m\in\N_0^n$, $|m|=s$).
Here $m-1_i+1_j$ is the $n$-tuple $(m_k)_{1\leq k\leq n}$ of integers which is equal to $m_k$ for any $k\neq i,j$, to $m_i-1$ for $k=i$ (except if $m_i=0$, in which case it is~0), and to $m_j+1$ for $k=j$. Recalling the relation between $\vphi_m$ and ${\cal H}_m$, identity \eqref{Hermite}, the right hand side of \eqref{itremainsHi} is $\sqrt{m_i (m_j+1)} \,\vphi_{m-1_i+1_j}\,\om_{qi}^j$, whence
\begeq\label{Gql}
G_{qm}^\ell = \sqrt{m_i (m_j+1)}\, 1_{\ell = m-1_i+1_j}\, \om_{qi}^j.
\endeq
In particular, 
\[ G_{q\ell}^m = - G_{qm}^\ell.\]

As an example, consider the case $n=2$, then $m=(m_1,m_2)$ can be identified with just $m_1$ ranging from $0$ to $s$. The matrix $\om_{qi}^j$ for $1\leq i,j\leq 2$ is the $2\times 2$ matrix 
\[ \left ( \begin{array}{cc} 0 & \om_q \\ -\om_q & 0 \end{array} \right), \]
where $\om_q=\<\nabla_q e_1,e_2\>$. Then for fixed $q$ the coefficients of the $G_{qm}^\ell$ matrix are $\om_q \sqrt{m_1(s+1-m_1)}$ just above the diagonal, and $-\om_q  \sqrt{m_1(s+1-m_1)}$ just below the diagonal.

Repeating the construction,
\begin{align*}
\nabla_{pq} \vphi_m = \sum_{a\neq b, i\neq j} 
\sqrt{m_i (m_j+1)}\, \vphi_{m-1_i-1_a +1_j+1_b}\, \om_{pa}^b\, \om_{qi}^j \\
+ \sum_{ij} \sqrt{m_i(m_j+1)}\, \vphi_{m-1_i+1_j}\, \nabla_p\om_{qi}^j.
\end{align*}
This leads to an explicit expression of $D_m^\ell$ involving $g^{pq} \om_{pa}^b \om_{qi}^j $ and
\[ \Om_i^j = g^{pq} \nabla_p \om_{qi}^j = \<\Delta e_i, e_j\>.\]
Note that
\[ \frac12(D_\ell^m - D_m^\ell) = - g^{pq} G_{qm}^\ell.\]

Reasoning in terms of quadratic forms, with $f'= \sum \alpha'_m \vphi_m$, after rearrangement this yields
\begeq\label{qf} \nabla_H f\cdot \nabla_H f'= 
g^{pq} (\delta_{m\ell} \pa_p - G_{pm}^\ell) \alpha_\ell \, (\delta_{m\ell}\pa_p - G_{qm}^{\ell'}) \alpha'_{\ell'}
- g^{pq} G_{pm}^\ell G_{qm}^{\ell'} \, \alpha_\ell\, \alpha'_{\ell'}
\endeq

For instance, still for $n=2$ and for given $s$ identifying $(m_1,m_2)$ with $m_1$, we find that $g^{pq}G_{pm}^\ell G_{qm}^{\ell'}$
is proportional to 
\[ g^{pq}\om_p\om_q = g^{pq} \<\nabla_p e_1,e_2\> \<\nabla_q e_1,e_2\> \]
and to the matrix $A^{\ell \ell'}$ which is
\[\begin{cases}
\begin{array}{cc} (\ell_1+1)(s-\ell_1) + (\ell_1-1)(s+2-\ell_1) & \qquad \text{ if $\ell=\ell'$}\\
- (\ell_1-1) (s+2-\ell_1) & \qquad \text{if $\ell_1 = \ell'_1+2$}\\
- (\ell_1+1) (s-\ell_1) & \qquad \text{if $\ell_1 = \ell'_1-2$}\\
0 &\qquad \text{otherwise}. \end{array}
\end{cases} \]

To summarise: There is an explicitable recipe transforming the Dirichlet form $\<\nabla_H f,  \nabla_H f'\>$ on ${\cal E}_s$ into a compact perturbation of the Dirichlet form $\<\nabla\alpha, \nabla\alpha'\>= \sum_\ell \<\nabla \alpha_\ell,\nabla \alpha'_\ell\>$, with lower order terms added inside the derivation operator, and a 0-order term added to the quadratic form. Once that reduction is done, one can extract successive eigenvalues of $-\Delta_H$ on ${\cal E}_s$ by the usual procedure of minimisation of the quadratic form under orthogonality conditions. In this way one can find inductively a sequence of nonnegative eigenvalues with finite multiplicity, thus reducing the quest of the spectrum to a sequence of finite-dimensional problems. A most important consequence is the commutation of the spectral projections associated to $\Delta_H$ and $\Delta_V$, for each ${\cal E}_s$ and thus eventually on the whole Hilbert space. 

\subsection{\underline{Tensor-valued spectral analysis}}

All the analysis in this section (spectra, spectral projections, commutation) goes through for functions valued in $\TIJ$.
\med

To summarise:

\begin{Prop}\label{propdhdxcom}
The unbounded linear operators $\Delta_V$ and $\Delta_H$ commute, either on $L^2(\TM)$ or $L^2(\TM,\mu)$, acting either on fuctions or tensors, in the sense that the spectral projections of $\Delta_V$ and those of $\Delta_H$ commute.
\end{Prop}

\bibnotes

The spectral study of Laplace operators on manifolds is addressed in many good sources such as Zelditch \cite{zelditch:laplacian:book} and Sogge \cite{sogge:laplacian:book}. The general theory of orthonormal polynomials is considered for instance in Nikiforov--Ouvarov \cite{niki:book}. Good information about Hermite polynomials can be found everywhere, including Wikipedia. It is more difficult to pinpoint an elementary source for the spectral study of $\Delta_H$, which is why I spent energy reproving some material which will be considered classical by experts. The description of $\Delta_H$ is still the subject of ongoing research, as pointed out to me by Jean-Michel Bismut.

The spectral theorem and spectral projections, and the notion of commuting unbounded operators, are treated for instance in the classic treatise by Reed--Simon \cite[Vol.~I, Chap.~7]{reedsimon:MMMP}. My thanks go to St\'ephane Attal for some explanations in this field.

\section{Functional analysis on $\TM$} \label{secfunctional}

For regularity of functions on $\TM$ there are two approaches. One is to localise via charts (if $(O_\theta)$ is a given atlas for $M$ then $O_\theta\times\R^n$ can be used as an atlas for $\TM$) and define functional spaces in each chart (through ``flat'' norms), then stitch those spaces together using a partition of unity adapted to the atlas. The other one is to work out regularity globally; of course, localisation will always be needed to compute in coordinates, but the definitions of the norms and thus their value will be independent of the choice of atlas. Both approaches are locally equivalent, but the treatment of large velocities may differ. The second approach would also adapt more naturally to unbounded manifolds, and is more suitable to handling the intrinsic operators appearing in partial differential equations; so it will be the preferred approach.

From the study of Section \ref{secspectral} one can construct operators by combining $\Delta_H$ and $\Delta_V$, and take advantage of their commutation to define functional calculus. Indeed, given any $I,J$, let ${\cal H} = L^2(\vol; \TIJ)$; then there is a joint spectral decomposition, in the form of a measure $\Pi (d\lambda_H\,d\lambda_V)$ on $\R^2$ (actually supported in $\R_+^2$) and valued in the space of orthogonal projections in $\TIJ$, such that
\begeq\label{spectralHV} 
-\Delta_H = \int_{\R^2} \lambda_H\, \Pi(d\lambda_H\,d\lambda_V), \qquad
- \Delta_V = \int_{\R^2} \lambda_V\,\Pi(d\lambda_H\,d\lambda_V)
\endeq
and 
\[ \Id = \int_{\R^2} \Pi(d\lambda_H\,d\lambda_V).\]
Then if $h$ is any nonnegative or bounded measurable function $\R^2\to \R$, define
\begeq\label{hDD}
h(-\Delta_H,-\Delta_V) = \int_{\R^2} h(\lambda_H,\lambda_V)\,\Pi(d\lambda_H\,d\lambda_V).
\endeq
In particular, working with fractional power laws defines partial regularity, either horizontal or vertical, and there is a scale of associated $L^2$-Sobolev spaces, and also weighted Sobolev spaces, with two regularity indices and one localisation index: for all $\alpha,\beta,\kappa$ in $\R$,
\begeq\label{fracSob}
\|f\|_{H^{\alpha,\beta}} = \Bigl\| (1-\Delta_H)^{\alpha/2} (1-\Delta_V)^{\beta/2} X\Bigr\|_{L^2(\vol)},
\endeq
and more generally
\begeq\label{fracwSob}
\|X\|_{H^{\alpha,\beta}_\kappa} = \Bigl\| (1-\Delta_H)^{\alpha/2} (1-\Delta_V)^{\beta/2} \bigl( X(1+|v|^2)^{\kappa/2} \bigr) \Bigr\|_{L^2(\vol)}.
\endeq
The exponents $\alpha,\beta,\kappa$ will be called respectively the {\bf horizontal regularity index}, {\bf vertical regularity index} and {\bf kinetic weight}. 
Further let
\begeq\label{Halphaxv}
H^\alpha_\kappa = H^{\alpha,0}_\kappa\cap H^{0,\alpha}_\kappa
\endeq
be the Sobolev space of $L^2$ functions with (horizontal or vertical) derivatives of order $\alpha$ in $L^2$.
This defines a decent regularity scale, because  $(1-\Delta_H)^{\alpha/2}$ commutes with both $(1-\Delta_V)^{\beta/2}$ and $(1+|v|^2)^{\kappa/2}$,
and $(1-\Delta_V)^{\beta/2}$ almost commutes with $(1+|v|^2)^{\kappa/2}$, in the sense that the commutator involves less regularity and faster decay (except possibly if both $\beta$ and $\kappa$ are negative, a situation which will not occur here). Here is an equivalent norm, ``in Littlewood--Paley style''. Consider a smooth partition of unity of the form
\begeq\label{smooth1}
1 = \chi_0(r) + \sum_{\ell\in\N} \chi \left(\frac{r}{2^\ell}\right) \qquad (r\geq 0),
\endeq
where $\chi_0$ is smooth and compactly supported, and $\chi$ is smooth and compactly supported away from the origin; and for $\ell\in\N$ let $\chi_\ell(r) = \chi(r/2^\ell)$; so $\chi_\ell$ is a way to focus on $r\simeq 2^\ell$. Then define
\begeq\label{fracwSob'}
\|f\|'_{H^{\alpha,\beta}_\kappa} = \left(\sum_{\ell\in\N_0} A^{2\ell\kappa} \|f \chi_\ell\|^2_{H^{\alpha,\beta}} \right)^{1/2}.
\endeq
Then this defines a norm equivalent to \eqref{fracwSob} (the constant $A>1$ may be chosen in a convenient way). Also ``homogeneous'' seminorms may be defined in the usual way:
\[ \|f\|_{\dot{H}^{\alpha,\beta}} = \Bigl\| (-\Delta_H)^{\alpha/2} (-\Delta_V)^{\beta/2} f \Bigr\|_{L^2}, \qquad \text{etc.} \]

Similarly consider $\Delta_H^\mu$ for $\Delta_H$ in $L^2(\mu)$ (formally similar to $\Delta_H$, but in a different Hilbert space), $\Delta_V^\mu$ also in $L^2(\mu)$, then there is a joint spectral decomposition
\begeq\label{spectralHVmu} 
-\Delta_H^\mu = \int_{\R^2} \lambda_H\, \Pi^\mu(d\lambda_H\,d\lambda_V), \qquad
- \Delta_V^\mu = \int_{\R^2} \lambda_V\,\Pi^\mu(d\lambda_H\,d\lambda_V)
\endeq
with 
\[ \Id_{L^2(\mu)} = \int_{\R^2} \Pi^\mu(d\lambda_H\,d\lambda_V),\]
and there is a functional calculus $h(-\Delta_H^\mu, -\Delta_V^\mu)$, etc.

\begeq\label{fracwSobmu}
\|f\|_{H^{\alpha,\beta}_\kappa(\mu)} = \Bigl\| (1-\Delta_H)^{\alpha/2} (1-\Delta^\mu_V)^{\beta/2} \bigl( X(1+|v|^2)^{\kappa/2} \bigr) \Bigr\|_{L^2(\mu)}.
\endeq

Homogeneous Sobolev spaces are defined in a similar way with $1-\Delta_H$ and $1-\Delta_V$ replaced by $-\Delta_H$ and $-\Delta_V$; for instance
\begeq\label{homsob}
\|f\|_{\dot{H}^{\alpha,\beta}_\kappa(\mu)} = \Bigl\| (-\Delta_H)^{\alpha/2} (-\Delta^\mu_V)^{\beta/2} \bigl( X(1+|v|^2)^{\kappa/2} \bigr) \Bigr\|_{L^2(\mu)}.
\endeq
Then 
\[ \|f\|_{H^{\alpha,\beta}_\kappa(\mu)} \simeq \|f\|_{\dot{H}^{\alpha,\beta}_\kappa(\mu)}  + \|f\|_{L^2_\kappa(\mu)}, \ \qquad etc. \]

Classical fractional $L^p$-Sobolev spaces are defined similarly: say
\begeq\label{fracwSobp}
\|f\|_{W^{(\alpha,\beta), p}_\kappa} = \Bigl\| (1-\Delta_H)^{\alpha/2} (1-\Delta_V)^{\beta/2} \bigl( X(1+|v|^2)^{\kappa/2} \bigr) \Bigr\|_{L^p(\vol)}.
\endeq
In this case the spectral analysis becomes much more tricky, but at least interpolation will work just the same as in Euclidean geometry;
and for integers $\alpha,\beta$ we are back to the usual Sobolev spaces. Also functional calculus can be defined by completion from a dense set of smooth functions. All in all, classical tools from ``flat'' functional analysis extend to this geometric setting.
\med

From calculus rules, commutator formulas, and elementary functional analysis it is not hard to provide functional estimates of the operators at play: $\xi, \nabla_V, \nabla_H$, the associated Laplace operators, and their commutators. The next proposition gathers some of those, many of which follow the intuition.

\begin{Prop}[Some regularity bounds] \label{propboundsxi}
With notation \eqref{VHcov} \eqref{xirevu}, \eqref{Xiexpl}, acting on functions or tensors,
\sm

(i) $\nabla_H$ loses one horizontal derivative, $\Delta_H$ loses two horizontal derivatives: For any $k\in\N_0^n$, $s\in\R$, $\alpha,\beta,\kappa\in\R$, there is $C>0$ such that for all functions $f$,
\[ \|\nabla_H^k (1-\Delta_H)^{s/2} f \|_{H^{\alpha,\beta}_\kappa} \leq C \|f\|_{H^{\alpha+|k|+s,\beta}_\kappa}.\]
The same estimate holds true if $f$ is replaced by a tensor of arbitrary order; or if spaces $H^\sigma_\kappa$ are replaced by $W^{\sigma,p}_\kappa$, $1<p<\infty$; or if the volume measure is replaced by $\mu$.
\sm

(ii) $\nabla_V$ loses one vertical derivative, $\Delta_V$ loses two vertical derivatives: For any $k\in\N_0^n$, $s\in\R$, $\alpha,\beta,\kappa\in\R$ ($\kappa\geq 0$ if $s<0$), there is $C>0$ such that for all functions $f$,
\[ \|\nabla_V^k (1-\Delta_V)^{s/2} f\|_{H^{\alpha,\beta}_\kappa} \leq C \|f\|_{H^{\alpha,\beta+|k|+s}_\kappa}.\]
The same estimate holds true if $f$ is replaced by a tensor of arbitrary order; or if spaces $H^\sigma_\kappa$ are replaced by $W^{\sigma,p}_\kappa$, $1<p<\infty$; or if the volume measure is replaced by $\mu$, provided that $\Delta_V$ is replaced by $\Delta_V^\mu$.
\sm

(iii) If the smooth function $\vphi:\TM\to\R$ is of order $r\in\R$, in the sense that $\nabla_H^k\nabla_V^\ell \vphi(x,v) = O(\<v\>^{r-|\ell|})$ for all $k,\ell$ in $N_0^n$, $\<v\> = \sqrt{1+|v|^2}$, then multiplication by $\vphi$ loses at most $r$ powers of $|v|$: There is $C>0$ such that for all functions $f$,
\[ \|\vphi\, f\|_{H^{\alpha,\beta}_\kappa} \leq C \|f\|_{H^{\alpha,\beta}_{\kappa+r}}.\]
The same estimate holds true if $f$ is replaced by a tensor of arbitrary order; or if spaces $H^\sigma_\kappa$ are replaced by $W^{\sigma,p}_\kappa$, $1<p<\infty$; or if the volume measure is replaced by $\mu$.
\sm

(iv) The geodesic generator loses one horizontal derivative and one power of $|v|$: For any $\alpha,\beta,\kappa\in \R$, there is $C>0$ such that for all functions $f$,
\begeq\label{xibound}
\|\xi f\|_{H^{\alpha,\beta}_\kappa} \leq C\,\| f\|_{H^{\alpha+1,\beta}_{\kappa+1}}.
\endeq
The same estimate holds true if $f$ is replaced by a tensor of arbitrary order; or if spaces $H^\sigma_\kappa$ are replaced by $W^{\sigma,p}_\kappa$, $1<p<\infty$; or if the volume measure is replaced by $\mu$.
\sm

(v) Commuting the geodesic operator with $s$ vertical derivatives loses $s-1$ vertical derivatives and 1 horizontal: For any $k\in\N_0^n$ ($|k|\geq 1$), $\alpha,\beta,\kappa\in\R$, $s\in\R$, there is $C>0$ such that for all functions $f$,
\begeq\label{xiVbound}
\bigl\|[\nabla_V^k,\xi] f\bigr\|_{H^{\alpha,\beta}_\kappa} \leq C\,\| f\|_{H^{\alpha+1,\beta+|k|-1}_{\kappa}};\qquad
\bigl\| \bigl[(1-\Delta_V)^{s/2},\xi\bigr] f\bigr\|_{H^{\alpha,\beta}_\kappa}
\leq C\,\| f\|_{H^{\alpha+1,\beta+s-1}_{\kappa}};
\endeq
or more generally
\begeq\label{moregeneralxiVbound}
\|[\nabla_V^k (1-\Delta_V)^{s/2},\xi] f\|_{H^{\alpha,\beta}_\kappa} \leq C\,\| f\|_{H^{\alpha+1,\beta+|k|+s-1}_{\kappa}}.
\endeq
The same estimate holds true if $f$ is replaced by a tensor of arbitrary order; or if spaces $H^\sigma_\kappa$ are replaced by $W^{\sigma,p}_\kappa$, $1<p<\infty$; or if the volume measure is replaced by $\mu$, provided that $\Delta_V$ is replaced by $\Delta_V^\mu$.
\sm

(vi) Commuting the geodesic operator with $s$ horizontal derivatives loses $s-1$ horizontal derivatives and 1 vertical, and 2 powers of $|v|$: For any $k\in\N_0^n$, $\alpha,\beta,\kappa\in\R$, $s\in\R$, there is $C>0$ such that for all functions $f$,
\begeq\label{xiHbound}
\bigl\|[\nabla_H^k,\xi] f\bigr\|_{H^{\alpha,\beta}_\kappa} \leq C\,\| f\|_{H^{\alpha+|k|-1,\beta+1}_{\kappa+2}};\qquad
\bigl\|[(1-\Delta_H)^{s/2},\xi] f\bigr\|_{H^{\alpha,\beta}_\kappa} \leq C\,\| f\|_{H^{\alpha+s-1,\beta+1}_{\kappa+2}};
\endeq
or more generally
\begeq\label{moregeneralxiHbound}
\bigl\| [\nabla_H^k (1-\Delta_H)^{s/2},\xi] f\bigr\|_{H^{\alpha,\beta}_\kappa} \leq C\,\| f\|_{H^{\alpha+|k|+s-1,\beta+1}_{\kappa+2}};
\endeq
The same estimate holds true if $f$ is replaced by a tensor of arbitrary order; or if spaces $H^\sigma_\kappa$ are replaced by $W^{\sigma,p}_\kappa$, $1<p<\infty$.
\sm

(vii) Same estimates hold true if $\xi$, the geodesic generator, is replaced by $\Xi$, the geodesic generator on vector fields, up to extra contributions losing 1 horizontal derivative less and one power of $|v|$ more: for arbitrary vector fields $X$,
\begeq\label{Xibound}
\|\Xi X\|_{H^{\alpha,\beta}_\kappa} \leq C\,\bigl( \|X\|_{H^{\alpha+1,\beta}_{\kappa+1}} + \|X\|_{H^{\alpha,\beta}_{\kappa+2}}\bigr);
\endeq
\begeq\label{moregeneralXiVbound}
\bigl\|[\nabla_V^k (1-\Delta_V)^{s/2},\Xi] X\bigr\|_{H^{\alpha,\beta}_\kappa} \leq C\,\bigl( \| X\|_{H^{\alpha+1,\beta+|k|+s-1}_{\kappa}}
+ \|X\|_{H^{\alpha,\beta+|k|+s-1}_{\kappa+1}}\bigr);
\endeq
\begeq\label{moregeneralXiHbound}
\bigl\| [\nabla_H^k (1-\Delta_H)^{s/2},\Xi] X\bigr\|_{H^{\alpha,\beta}_\kappa} \leq C\, \| X\|_{H^{\alpha+|k|+s-1,\beta+1}_{\kappa+2}}.
\endeq
The same estimate holds true if $X$ is replaced by a tensor of arbitrary order; or if spaces $H^{\alpha,\beta}_\kappa$ are replaced by $W^{(\alpha,\beta),p}_\kappa$, $1<p<\infty$; or if the volume measure is replaced by $\mu$, provided that $\Delta_V$ is replaced by $\Delta_V^\mu$.
\sm

(viii) The difference of $[\nabla_V,\Xi]$ and $\nabla_H$ loses at most one vertical moment: For any tensor $X$,
\begeq\label{Xidiffboundnabla}
\Bigl \| \bigr ( [\nabla_V,\Xi] - \nabla_H\bigr) X\Bigr \|_{H^{\alpha,\beta}_\kappa} \leq C\,\| X\|_{H^{\alpha,\beta}_{\kappa+1}}.
\endeq
The same estimate holds true if spaces $H^{\alpha,\beta}_\kappa$ are replaced by $W^{(\alpha,\beta),p}_\kappa$, $1<p<\infty$; or if the volume measure is replaced by $\mu$, provided that $\Delta_V$ is replaced by $\Delta_V^\mu$.
\end{Prop}
\sm

\begin{Rk} Working with naive regularity spaces in $x,v$ rather than with horizontal and vertical variations, from \eqref{xi} one would have guessed that $\xi$ loses 1 derivative in $v$ and 2 powers of $|v|$; the whole structure of estimates would be different.
\end{Rk}

\begin{proof}[Sketch of proof of Proposition~\ref{propboundsxi}]
Estimates (i) and (ii) are immediate. Estimate (iii) is conveniently obtained from formula \eqref{fracwSob'}. (If the power of $-\Delta_V$ is negative then the decay will not be better than a certain power of $\<v\>$, which is the reason to restrict then to nonnegative $\kappa$.) Estimate (iv) follows from \eqref{xirevu}. Estimates (v) (vi), (vii) are similar, let us consider for instance (vi), and for notational simplicity consider only functions and assume $n=1$, so $k\in\N$. The case $k=1$ is deduced from Proposition \ref{propcomH}(ii) and the boundedness of ${\cal R}$, for all $\alpha,\beta,\kappa$ in $\R$. Then for higher $k\in\N$ this is due to
\[ [\nabla_H^k,\xi] = \sum_{j=0}^{k-1} \nabla_H^{j}[\nabla_H,\xi]\nabla_H^{k-1-j}. \]
From this follow the bounds on $[(1-\Delta_V),\xi]$; and then similarly on $[(1-\Delta_V)^k,\xi]$ for all $k\in\N$, so for all $s\in 2\N_0$; and then by complex interpolation for all powers $s\geq 0$. To prove it for negative $s$, consider $\tau = -s>0$ and write
\[ Z = (1-\Delta_H)^{-\tau/2} \xi X - \xi (1-\Delta_H)^{-\tau/2}X, \qquad
Y = (1-\Delta_H)^{-\tau/2}X,\]
so 
\[ (1-\Delta_H)^{\tau/2} Z = \xi (1-\Delta_H)^{\tau/2}Y - (1-\Delta_H)^{\tau/2}\xi Y,\]
thus from the case $s>0$
\[ \bigl\| (1-\Delta_H)^{\tau/2} Z \bigr \|_{H^{\alpha,\beta}_\kappa}
\leq C \bigl\| Y \|_{H^{\alpha+\tau-1,\beta}_{\kappa+2}} = C \bigl\| X \|_{H^{\alpha-1,\beta}_{\kappa+2}},\]
so 
\[  \bigl\| Z \bigr \|_{H^{\alpha+\tau,\beta}_\kappa} \leq C \bigl\| X \|_{H^{\alpha-1,\beta}_{\kappa+2}}.\]
Since $\alpha$ is arbitrary, upon replacing $\alpha$ by $\alpha-\tau$ this is
\[ \bigl\| Z \bigr \|_{H^{\alpha,\beta}_\kappa} \leq C \bigl\| X \|_{H^{\alpha-\tau-1,\beta}_{\kappa+2}},\]
which is the desired result. 
\end{proof}

The last ingredient in this section is a Sobolev embedding relating Hilbertian to pointwise estimates.

\begin{Prop}[crude Sobolev embedding in $\TM$] \label{propsobemb}
If $f\in H^{n+1}_{n+1}(\TM)$ then $f\in C(\TM)$; moreover, for any $k,\ell\in\N_0^n$,
\begeq\label{ineqsobemb} (1+|v|^\kappa) |\nabla_H^k\nabla_V^\ell f|  \leq C(k,\ell,\kappa,M) \, \|f\|_{H^{|k|+|\ell|+n+1}_{\kappa+n+1}}. \endeq
\end{Prop}

\begin{Rk} This inequality is clearly not optimal, but will suffice for future goals in these notes. I~trust that a more careful, intrinsic treatment can take the right-hand side into $\|f\|_{H^{\sigma}_{\kappa}}$, for any $\sigma>n$, just as the Sobolev embedding in $\R^{2n}$.
\end{Rk}

\begin{proof}[Sketch of proof of Proposition \ref{propsobemb}]
It suffices to prove the result for $\kappa=0$, $k,\ell =0$.
Introduce an atlas $(O_\theta)_{\theta\in\Theta}$ on $M$, and a subordiate smooth partition of unity $(\chi_\theta)_{\theta\in\Theta}$. Expressing $f=\sum f(x,v) \chi_\theta(x)$, it suffices to prove \eqref{ineqsobemb} when $f$ is supported in $T O_\theta$; there we can work with coordinates $x$ and $v$. Then from the expressions of $\nabla_H$, one can relate the intrinsic Sobolev norm (horizontal and vertical variations) with the nonintrinsic one (variations in $x$ and $v$): with transparent notation,
\[ \|f\|_{(H^{n+1})_{x,v}} \leq C\, \|f\|_{(H^{n+1}_{n+1})_{H,V}}.\]
On the other hand, by classical Sobolev embedding in $TO_\theta\subset\R^{2n}$,
\[ \|f\|_{C(TO_\theta)}\leq C \, \|f\|_{(H^{n+1})_{x,v}}.\]
(Any exponent strictly bigger than $n$ would do.)
This completes the proof.
\end{proof}

\bibnotes

This section introduced classical functional spaces which are used in many references and textbooks. Sobolev spaces and embeddings can be found everywhere, including Brezis \cite{brezis:AF}. Sobolev spaces on manifolds are in Hebey \cite{hebey:manifolds:96}, but defined and studied on $M$ rather than $\TM$ (of course one can always consider $\TM$ as the manifold, but then there is no specific treatment of velocities). Lunardi \cite{lunardi:interpolation} provides an introduction to interpolation theory. About the smooth partition of unity, classically used in frequency space for pseudo-differential operators and Littlewood--Paley theory, this can be found e.g. in Frazier--Jawerth--Weiss~\cite{FJW:book}.

For kinetic theory set in Euclidean context, weighted Sobolev spaces with differing regularity indices in $x$ and $v$, and powers of $|v|$, have become commonly used since the 2000's, see e.g. works by Guo, Mouhot and myself~\cite{guo:landau:02,MV:landau,vill:hypoco}. But even long before that, a number of works were working out fractional Sobolev regularity in $v$ and $x$ either in the context of H\"ormander's theory \cite{horm:hypo:67,rothschildstein:nilp:76} or in that of fractional regularity of kinetic velocity averages \`a la DiPerna--Lions--Meyer~\cite{DPLM:average:91}.

In the context of the geometric kinetic Fokker--Planck equation, the globally defined fractional Sobolev spaces with localisation, formula~\eqref{fracwSob}, were apparently introduced by Lebeau in his collaboration with Bismut~\cite{bismutlebeau:hypo:book,lebeau:FP1:05,lebeau:FP2:07}, together with functional calculus and also tools from pseudo-differential calculus.

In the Sobolev embedding of Proposition \ref{propsobemb} there is no specific treatment of velocities, so another, certainly sharper way to obtain it would consist in working in $\TM$ equipped with the Sasaki metric, check basic bounds on the Ricci curvature and volume of balls there, and apply Sobolev embedding theorems in the style of Varopoulos \cite{varopoulos:heatk:89}.

\section{$L^2$ interpolation inequalities} \label{secL2interp}

The scene is ready and I will now start to juggle between weighted $L^2$-Sobolev spaces. For a start, there are the classical interpolation inequalities, in which one may play simultaneously on all indices: For simplicity they will be stated here only for nonnegative kinetic weight $\kappa$.

\begin{Prop}[Sobolev interpolation] \label{propinterp}
If $\alpha_0,\alpha1,\beta_0,\beta_1\in\R$, $\kappa_0,\kappa_1\geq 0$ and $0\leq \theta\leq 1$ then there is $C>0$ such that for any function $f:\TM\to\R$,
\begeq\label{interpmoments} \|f\|_{H^{\alpha,\beta}_\kappa} \leq C \|f\|_{H^{\alpha_0,\beta_0}_{\kappa_0}}^{1-\theta}  \|f\|_{H^{\alpha_1,\beta_1}_{\kappa_1}}^\theta,
\endeq
where
\[ \alpha = (1-\theta)\alpha_0 + \theta \alpha_1,\qquad \beta = (1-\theta)\beta_0 + \theta \beta_1, \qquad \kappa = (1-\theta)\kappa_0 + \theta \kappa_1.\]
The same estimates hold true if functions are replaced by tensors of arbitrary order; or if the Liouville measure is replaced by the Gaussian measure.
\end{Prop}

The next important fact, more subtle, is that the Gaussian measure allows to trade moments for vertical regularity. This property, true for the Gaussian measure in $\R^n$, readily adapts to functions on the tangent bundle.

\begin{Prop}[Trading moments for regularity with Gaussian weight] \label{propgradmom}
For any $\alpha,\beta,\kappa\in\R$ ($\kappa\geq 0$ if $\beta<0$) and any $s>0$, there is $C>0$ such that for all tensors $X$ on $\TM$,
\begeq\label{ineqregv}
\|X\|_{H^{\alpha,\beta}_{\kappa+s}(\mu)} \leq C \|X\|_{H^{\alpha,\beta+s}_\kappa (\mu)}.
\endeq
\end{Prop}

The following immediate corollary states that Gaussian divergence loses one derivative, as one would have expected. Recall that
\begeq\label{divgauss}
\nabla_V^{\ast \mu} = - (\nabla_V-v)\cdot 
\endeq
is the adjoint of the vertical gradient in $L^2(\mu)$.

\begin{Cor}[Gaussian divergence loses one derivative] \label{divloss}
With the same notation as in Proposition \ref{propgradmom},
\begeq\label{ineqdivg}
\|\nabla_V^{\ast\mu} X\|_{H^{\alpha,\beta}_{\kappa}(\mu)} \leq C \|X\|_{H^{\alpha,\beta+1}_\kappa (\mu)}.
\endeq
\end{Cor}

\begin{Rk} Of course, from \eqref{ineqdivg} and the first part of Proposition \ref{propboundsxi}(ii) one recovers the already known
$\|\Delta_V^\mu X\|_{H^{\alpha,\beta}_\kappa(\mu)} \leq C \|X\|_{H^{\alpha,\beta+2}_\kappa(\mu)}$.
\end{Rk}

Another important corollary of Proposition \ref{propgradmom} together with Proposition \ref{propboundsxi}(iv) states that the geodesic flow loses at most one horizontal and one vertical derivatives:

\begin{Cor}[$L^2$ estimate of the geodesic flow] \label{L2geod}
With the same notation as in Proposition \ref{propgradmom},
\[ \|\xi X\|_{L^2(\mu)} \leq C \bigl( \|\nabla_V\nabla_H X\|_{L^2(\mu)} + \|\nabla_H X\|_{L^2(\mu)}\bigr) \leq  C  \|X\|_{H^{1,1}(\mu)}\]
and more generally
\begeq
\|\xi X\|_{H^{\alpha,\beta}_{\kappa}(\mu)} \leq C \|X\|_{H^{\alpha+1,\beta+1}_\kappa (\mu)}.
\endeq
\end{Cor}

\begin{proof}[Proof of Proposition \ref{ineqregv}]
It suffices to consider $\alpha=\beta=\kappa=0$ and $s=1$. 
Upon integration in $x$, the requested estimate will follow from the similar property for functions of just $v\in\R^n$. 
So the goal is
\begeq\label{firsttreat}
\int_{\R^n} |v|^2 h(v)^2\, e^{-\frac{|v|^2}{2}}\,dv \leq C \left( \int_{\R^n} |\nabla_v h(v)|^2\, e^{-\frac{|v|^2}{2}}\,dv + \int_{\R^n} h(v)^2\, e^{-\frac{|v|^2}{2}}\,dv\right).
\endeq
When \eqref{firsttreat} is established, it holds on each $T_xM$, and $x$-integration will yield
\[ \|h\|_{L^2_1(\mu)} \leq C \|h\|_{H^{0,1}(\mu)}, \]
as desired.

To prove \eqref{firsttreat}, use integration by parts:
\begin{align*}
\int h(v)^2 |v|^2 e^{-\frac{|v|^2}{2}}\,dv 
& = - \int h(v)^2 v\cdot \nabla_v \bigl( e^{-\frac{|v|^2}{2}}\bigr)\,dv \\
& = \int \nabla_v\cdot \bigl[ h(v)^2 v\bigr]\, e^{-\frac{|v|^2}{2}}\,dv \\
& = 2 \int h(v) \nabla_v h(v)\cdot v\, e^{-\frac{|v|^2}{2}}\,dv + n \int h(v)^2 e^{-\frac{|v|^2}{2}}\,dv\\
& \leq 2 \left( \int h(v)^2 |v|^2 \, e^{-\frac{|v|^2}{2}} \,dv\right)^{1/2} \left(\int|\nabla_v h(v)|^2\, e^{-\frac{|v|^2}{2}}\,dv\right)^{1/2} + n \int h(v)^2 e^{-\frac{|v|^2}{2}}\,dv,
\end{align*}
and the conclusion follows from Young's inequality $2ab\leq a^2/2 + 2b^2$.

The proof goes through with a tensor $X$ rather than a function $h$, either by direct adaptation, or reasoning by duality to reduce to scalar-valued functions.

Now consider the general case when $\alpha,\beta,\kappa$ are arbitrary. Replacing $X$ by $(1-\Delta_H)^{\alpha/2}X$ one may assume that $\alpha=0$. Integrating the inequality in $x$ reduces to prove it just in Euclidean Gaussian space. By interpolation one may assume that $s$ and $\alpha$ are even integers. By iterating the inequality, one may assume that $s=2$. Then what is needed is (with $a=\beta/2\in \N_0$)
\[ \Bigl\| (1-\Delta_V^\mu)^a \bigl( X(1+|v|^2)^{\kappa/2-1} (1+|v|^2) \bigr) \Bigr\|_{L^2(\mu)}
\leq C \Bigl\| (1-\Delta_V^\mu)^{a+1} \bigl( X(1+|v|^2)^{\kappa/2-1} \bigr) \Bigr\|_{L^2(\mu)}.\]
Since $\Delta^\mu_V(1+|v|^2) = n - |v|^2$, when $a\geq 0$, after systematically integrating by parts the left-hand side is controlled by a combination of $\int |\nabla_V^k X|^2 (1+|v|^2)^k\,d\mu$ for $k\leq a$; but then the same argument as above (iterated) yields the result. This works whatever the sign of $\kappa$.

If on the other hand $a<0$, assuming $\kappa>0$ one just needs consider the case in which $\kappa$ is an even integer; choosing $\kappa=0$, then it suffices to replicate the estimate for $\kappa=0$. So it all amounts to
\[ \int \Bigl| (1-\Delta_v^\mu)^{-1} \bigl[ X(1+|v|^2)\bigr] \Bigr|^2 \, e^{-\frac{|v|^2}{2}}\,dv
\leq C \int |X|^2\, e^{-\frac{|v|^2}{2}}\,dv. \]
It suffices to do it when $X$ is a real-valued function. Using $|v|^2= \sum v_i^2$ and symmetry, eventually it all boils down to
\begeq\label{allboilsdown}
\int_{\R^n} \Bigl| (1-\Delta_v^\mu)^{-1} \bigl[ f(v) v_1^2\bigr] \Bigr|^2 \, \mu(dv)
\leq C \int_{\R^n} f(v)^2\, \mu(dv).
\endeq

To prove \eqref{allboilsdown} one can use the spectral decomposition: Let $(\alpha_m)$ be the coefficients of $f$ in the basis of normalised Hermite polynomials, as in \eqref{Hermite}:
\[ f(v) = \sum_{m\in\N_0^n} \alpha_m \vphi_m(v). \]
Write $\ov{m} = (m_2,\ldots,m_n)$. It follows from \eqref{Hermite} and \eqref{Herm3} that
\[ v_1^2 \vphi_m(v) = \sqrt{(m_1+1)(m_1+2)}\, \vphi_{m_1+2,\ov{m}} + (2m+1) \vphi_m + \sqrt{m_1(m_1-1)}\, \vphi_{m_1-2,\ov{m}}.\]
Recall that the eigenvalue associated with $\vphi_m$ is $s=|m|$. So
\begin{multline*}
(1-\Delta_V^\mu)^{-1} (hv_1^2) \\
= \sum_{m\in\N_0^n} \alpha_m
\left( \frac{\sqrt{(m_1+1)(m_1+2)}}{1+|m|} \, \vphi_{m_1+2,\ov{m}}
+ \frac{(2m_1+1)}{1+|m|}\,\vphi_m + \frac{\sqrt{m_1(m_1-1)}1_{m_1\geq 2}}{1+|m|} \vphi_{m_1-2,\ov{m}}\right).
\end{multline*}
Changing summation indices, this is the same as
\begin{multline*} \sum_{m\in\N_0^n}
\left[ \alpha_{m_1-2,\ov{m}} \left( \sqrt{m_1 (m_1-1)}{|m|-1}\right) 1_{m_1\geq 2} 
+ \alpha_m \left(\frac{2m_1+1}{|m|+1} \right) \right. \\ \left. + \alpha_{m_1-2,\ov{m}} \left(\frac{\sqrt{(m_1+1)(m_1+2)}}{|m|+3}\right)\right]\,\vphi_m.
\end{multline*}
Since the basis $(\vphi_m)$ is orthonormal,
\begin{align*}
& \bigl\|(1-\Delta_V^\mu)^{-1} ( f v_1^2) \bigr\|_{L^2(\mu)}^2\\
& = \sum_{m\in \N_0^n} 
\left[ \alpha_{m_1-2,\ov{m}} \left( \frac{\sqrt{m_1 (m_1-1)}}{|m|-1}\right) 1_{m_1\geq 2} 
+ \alpha_m \left(\frac{2m_1+1}{|m|+1} \right) + \alpha_{m_1-2,\ov{m}} \left(\frac{\sqrt{(m_1+1)(m_1+2)}}{|m|+3}\right)\right]^2\\
& \leq 3 \sum_m \bigl( 2\times 1_{m_1\geq 2}\, \alpha_{m_1-2,\ov{m}}^2 + 4 \alpha_m^2 + \alpha_{m_1-2,\ov{m}}^2\bigr)\\
& \leq 21 \sum_m \alpha_m^2 = 21 \|f\|_{L^2(\mu)}^2,
\end{align*}
as needed. This concludes the argument also for negative regularity indices $\beta$.
\end{proof}

Combining Propositions \ref{propboundsxi} and \ref{propgradmom} further yields

\begin{Prop}[Weightless Gaussian estimates of commutators] \label{propxiboundnomom}
With the same notation as in Proposition \ref{propboundsxi}, for all $\alpha,\beta \in \R$, $k\in\N_0^n$, $s\in\R$, $\kappa\in\R$, 
\sm

(i) There is $C>0$ such that for all tensors $X$,
\begeq\label{xiboundnomom}
\bigl \||[\nabla_H^k (1-\Delta_H)^{s/2},\xi] X\bigr \|_{H^{\alpha,\beta}_\kappa(\mu)} \leq C\,\| X\|_{H^{\alpha+|k|+s-1,\beta+3}_\kappa(\mu)};
\endeq
and the same estimate holds if $\xi$ is replaced by $\Xi$, the geodesic generator on vector fields.
\sm

(ii) There is $C>0$ such that for all tensors $X$,
\begeq\label{Xiboundnabla}
\Bigl \|| \bigr|( [\nabla_V, \Xi] - \nabla_H\bigr) X\Bigr \|_{H^{\alpha,\beta}_\kappa(\mu)} \leq C\,\| X\|_{H^{\alpha,\beta+1}_\kappa(\mu)}.
\endeq
\end{Prop}

\begin{proof}[Proof of Proposition \ref{propxiboundnomom}]
First consider $\xi$. If $\beta\geq 0$ this follows directly from Propositions \ref{propboundsxi} and \ref{propgradmom}. 
For $\beta<0$ let us commute the equation with $(\Delta^\mu_V)^N$: if $f = (\Delta^\mu_V)^N h$ then
\begeq\label{aftercomN}
(\Delta^\mu_V)^N \bigl[\xi,(1-\Delta_H)^s \nabla_H^k\bigr] h
= \bigl[\xi, (1-\Delta_H)^s \nabla^k\bigr] f +\bigl [ (\Delta_V^\mu)^N, [\xi, (1-\Delta_H)^s\nabla_H^k]\bigr]h. 
\endeq
Now from the commutation of $\Delta_V^\mu$ with $\Delta_H$ and $\nabla_H$,
\[  \bigl[ (\Delta_V^\mu)^N, [\xi, (1-\Delta_H)^s\nabla_H^k]\bigr] = \bigl[ (1-\Delta_H)^s\nabla_H^k, [(\Delta_V^\mu)^N, \xi]
\bigr].\]
But $v$ and $v\cdot\nabla_V$ both commute with $\nabla_H$. And according to Propositions \ref{propHV} and \ref{propcomH}, in $L^2(\mu)$
\[
[\nabla_V, [\nabla_H,\xi]] = O(1) + O (|v|\nabla_V),
\]
combining this with \eqref{ineqregv},
\[
[\nabla_V, [\nabla_H,\xi]] = O(1) + O (\nabla_V^2),
\]
More generally, by induction and reasoning first for integer powers, then by interpolation,
\[ \Bigl\| \bigl[ ( 1-\Delta_V^\mu)^N, [ (1-\Delta_H)^{s/2} \nabla_H^k, \xi] \bigr]  h\Bigr\|_{H^{\alpha,\beta}(\mu)}
\leq C \|h\|_{H^{\alpha+|k|+s-1,\beta+2N+1}}.\]
Thus the ``error'' coming from the last term in \eqref{aftercomN} is of order $\|f\|_{H^{\alpha+|k|+s-2,\beta+1}}$ and therefore controlled by the right-hand side in \eqref{propxiboundnomom}. This proves (i) for the operator $\xi$. With $\xi$ replaced by $\Xi$, the reasoning is similar, and extra terms appear to be of lower order.

Finally, for (ii) it suffices to use \eqref{Xicom} (or Proposition \ref{propboundsxi}(viii)) and Proposition \ref{propgradmom}.
\end{proof}

The third main result in this section is a mixed horizontal/vertical interpolation inequality which allows to reduce to either pure horizontal or pure vertical derivatives.

\begin{Prop}[Mixed horizontal-vertical $L^2$ interpolation inequality] \label{propmixinterp}
Let $\alpha,\beta,\kappa\geq 0$. Then for any $\alpha',\beta'\geq 0$ with $\frac{\alpha}{\alpha'} + \frac{\beta}{\beta'} \leq 1$,
there is a constant $C$ such that for any tensor $X$,
\begin{align}\label{interpaabb}
\|X\|_{H^{\alpha,\beta}_\kappa}
& \leq C \|X\|_{H^{\alpha',0}_\kappa}^{\frac{\alpha}{\alpha'}} \,\|X\|_{H^{0,\beta'}_\kappa}^{\frac{\beta}{\beta'}} \,\|X\|_{L^2_\kappa}^\theta\\
& \leq C \Bigl( \|X\|_{H^{\alpha',0}_\kappa} + \|X\|_{H^{0,\beta'}_\kappa} \Bigr)^{1-\theta} \|X\|_{L^2_\kappa}^\theta,
\end{align}
with
\begeq\label{thetaab}
\theta = 1- \left(\frac{\alpha}{\alpha'} + \frac{\beta}{\beta'}\right).
\endeq
A similar result holds if Sobolev spaces $H^{\alpha,\beta}_\kappa$ are replaced by homogeneous Sobolev spaces $\dot{H}^{\alpha,\beta}_\kappa$; or by Gaussian Sobolev spaces $H^{\alpha,\beta}_\kappa(\mu)$.
More generally, if $(\alpha,\beta)$ belongs in the triangle whose vertices are $(0,0)$, $(\alpha_0,\beta_0)$, $(\alpha_1,\beta_1)$ in $\R^2$,
then
\begin{align}\label{interpa0a1b0b1}
\|X\|_{H^{\alpha,\beta}_\kappa}
& \leq C \|X\|_{H^{\alpha_0,\beta_0}_\kappa}^{\theta_0} \,\|X\|_{H^{\alpha_1,\beta_1}_\kappa}^{\theta_1} \,\|X\|_{L^2_\kappa}^\theta\\
& \leq C \Bigl( \|X\|_{H^{\alpha_0,\beta_0}_\kappa} + \|X\|_{H^{\alpha_1,\beta_1}_\kappa} \Bigr)^{1-\theta} \|X\|_{L^2_\kappa}^\theta,
\end{align}
where $\theta_0$,$\theta_1$,$\theta$ are such that
\[ \alpha = \theta_0\alpha_0 + \theta_1\alpha_1, \qquad \beta =\theta_0\beta_0+\theta_1\beta_1,\qquad \theta = 1-(\theta_0+\theta_1).\]
\end{Prop}

\begin{Rk} In practise, to apply Proposition \ref{propmixinterp}, just balance the horizontal and vertical regularity indices to get the right exponents.
\end{Rk}

\begin{proof}[Proof of Proposition \ref{propmixinterp}]
It suffices to treat $\gamma=0$. The proof is based on the codiagonalisation formula \eqref{spectralHV} (or \eqref{spectralHVmu} for Gaussian weight). For a start let us consider \eqref{interpaabb}.
Applying successively H\"older's inequality with conjugate exponents
$(\frac1{1-\theta}, \frac1{\theta})$ and Young's inequality with
exponents $(\frac{(1-\theta)\alpha'}{\alpha}, \frac{(1-\theta)\beta'}{\beta})$,
\begin{align*}
& \iint \bigl |(1-\Delta_H)^{\alpha/2} (1-\Delta_V)^{\beta/2} X|^2  = \int (1+\lambda_H)^{\alpha/2}\,(1+\lambda_V)^{\beta/2}\, \bigl\< \Pi(d\lambda_H\,d\lambda_V) X, X\bigr\> \\
& \leq \left( \int (1+\lambda_H)^{\frac{\alpha}{2(1-\theta)}} (1+\lambda_V)^{\frac{\beta}{2(1-\theta)}}
\, \bigl\<\Pi(d\lambda_H\,d\lambda_V) X, X\bigr\>\right)^{1-\theta}
\left( \int \bigl\< \Pi(d\lambda_H\,d\lambda_V) X, X\bigr\>\right)^\theta\\
& \leq C(\alpha,\beta,\alpha',\beta')
\left(\int ((1+\lambda_H)^{\alpha'/2}\,
\bigl\<\Pi(d\lambda_H\,d\lambda_V) X, X\bigr\>\right)^{\frac{\alpha'}{\alpha}} 
\left(\int ((1+\lambda_V)^{\beta'/2}\,
\bigl\<\Pi(d\lambda_H\,d\lambda_V) X, X\bigr\>\right)^{\frac{\beta'}{\beta}}\, \\
& \qquad\qquad\qquad\qquad\qquad\qquad\qquad\qquad\qquad\qquad\qquad\qquad\qquad\qquad\qquad  \left(\int\bigl\<\Pi(d\lambda_H\,d\lambda_V) X, X\bigr\>\right)^{\theta}\\
& \leq C(\alpha,\beta,\alpha',\beta')
\left(\int ((1+\lambda_H)^{\alpha'/2} + (1+\lambda_V)^{\beta'/2})\,
\bigl\<\Pi(d\lambda_H\,d\lambda_V) X, X\bigr\>\right)^{1-\theta}
\left(\int\bigl\<\Pi(d\lambda_H\,d\lambda_V) X, X\bigr\>\right)^{\theta}\\
& = C(\alpha,\beta,\alpha',\beta')\, \Bigl( \bigl\|(1-\Delta_H)^{\alpha'/2} X\bigr\|_{L^2}^2 +
\bigl\|(1-\Delta_V)^{\beta'/2} X\bigr\|_{L^2}^2 \Bigr)^{1-\theta}\, \|X\|_{L^2}^{2\theta},
\end{align*}
and \eqref{interpaabb} follows. The proof of \eqref{interpa0a1b0b1} is quite similar, writing $(\alpha,\beta)$ as a barycenter of $(\alpha_0,\beta_0)$, $(\alpha_1,\beta_1)$ and $(0,0)$.
\end{proof}

\begin{Rk} Both the moment/derivative trade, Proposition \ref{propgradmom}, and the mixed interpolation inequality, Proposition \ref{propmixinterp}, are $L^2$ effects: they do not hold pointwise, so it is easy to construct counterexamples showing that they fail if $L^2$-Sobolev spaces are replaced by, say, $W^{\sigma,\infty}$ spaces. So one can guess that they do not hold in $W^{\sigma,p}$ either, at least with the same exponents. However it is still possible to interpolate between the $L^2$ estimate and the trivial estimate either in $L^1$ or in $L^\infty$, to find partial bound in Sobolev spaces $W^{\sigma,p}$ rather than $H^\sigma$. For instance
\begeq\label{Wptrade}
\|f\|_{W^{(\alpha,\beta),p}_{\kappa+s}(\mu)} \leq C \|f\|_{W^{(\alpha,\beta+s\theta),p}_{\kappa(1-\theta)} (\mu)}\qquad\qquad
 \theta = \frac2{\max(p,p')}.
\endeq
\end{Rk}

\bibnotes

Proposition \ref{propgradmom} is well-known in Euclidean space, at least for $\beta\geq 0$, and Proposition \ref{propmixinterp} is probably well-known too; both were applied in my memoir \cite{vill:hypoco} to study the global regularity of the kinetic Fokker--Planck equation in Euclidean space. The adaptation to geometric context is from Debbasch--Ollivier--Villani~\cite{DOV:preprint}.

\section{Gaussian tail Poincar\'e inequalities}

Proposition \ref{propgradmom} will recur several times in the present work, and deserves a name:

\begin{Def}[Gaussian tail Poincar\'e estimate] \label{defGTP}
A measure $\nu$ on $\TM$ is said to satisfy a Gaussian tail Poincar\'e estimate with constant $K>0$, $\GTP(K)$, if for all $C^1$ functions $\vphi:\TM\to\R$,
\begeq\label{GTPineq}
\iint_{\TM} |v|^2 \vphi(x,v)^2\, \nu(dx\,dv) \leq \frac1{K} \left[ \iint_{\TM} |\nabla_V\vphi(x,v)|^2\, \nu (dx\,dv) + \iint_{\TM} \vphi(x,v)^2\,\nu(dx\,dv)\right].
\endeq
\end{Def}

\begin{Rks} \begin{itemize}
\item[(i)] This inequality is only about the tail behaviour of $\nu$ as $|v|\to\infty$, it says nothing about the spectral properties of $\nu$. Choosing $\nu = F(|v|) e^{-|v|^2}\,dv\,dx$ with $F$ vanishing near $|v|=0$ shows that $\nu$ may satisfy \eqref{GTPineq} but no Poincar\'e (spectral gap) inequality.

\item[(ii)] Choosing $\vphi(v) = w(x) \max (e^{\var |v|^2/2}, A)$ with $0<\var<K/2$, writing down \eqref{GTPineq}, simplifying and letting $A\to\infty$, it is easy to conclude that $\int e^{\var |v|^2/2} \nu(dv|x)$ is finite, and uniformly bounded in $x$, where $\nu(dv|x)$ stands for probability measure obtained by conditioning $\nu$ with respect to $x$. In other words, \eqref{GTPineq} is possible only if $\nu$ decays at least like a Gaussian in the $v$ variable, uniformly in $x$.

\item[(iii)] Choose $\nu(dx\,dv) = F(|v|)\,dv$ and let $F$ be a combination of Gaussians: $F(v)=e^{-\alpha|v|^2/2}$ when $|v| \in I_m = [4m-1, 4m+1]$ and $F(v)=e^{-\beta|v|^2/2}$ when $|v|\in I'_m= [4m+1,4m+3]$, $m\in\N$, where $\alpha<\beta$. Let $\vphi$ be a mollified step function, equal to~1 on $I_m$ for some $m$, to~0 away from $I_m$, and whose entire variation is contained in $I'_{m-1}$ and $I'_{m}$. If $R=4m$, then as $R\to\infty$ the left hand side in \eqref{GTPineq} looks like $R^2 e^{-\alpha R^2/2}$ and the right hand side like $R^2 e^{-\beta R^2/2}$; so the inequality cannot be true as $R\to\infty$ (recall that $\alpha<\beta$). This shows that no bound bearing only on the value of the density of $\nu$ can be sufficient for the Gaussian tail Poincar\'e inequality. In effect, \eqref{GTPineq} appears to be a {\bf first-order condition}, like the usual Poincar\'e inequality.

\item[(iv)] Applying this inequality to $\vphi(x,v) = |X(x,v)|$, where $X$ is a tensor, and using $|\nabla_V|X|| \leq |\nabla_VX|$, shows that \eqref{GTPineq} applies the same for any tensor in place of $\vphi$, save possibly for the multiplication by a constant depending on $n$.
\end{itemize}
\end{Rks}

To summarise: If the usual Poincar\'e inequality requires at least log linear decay (say $\exp(-|v|)$ at infinity), the Gaussian tail Poincar\'e inequality requires something like a log quadratic decay at infinity (Remark (ii) above), and any criterion has to involve the density but also its first order variation (Remark (iii)). The following criterion, of uniform log concavity type, will identify a large class of measures satisfying the Gaussian tail Poincar\'e inequality.

\begin{Prop}[Gaussian tail Poincar\'e from uniform log concavity] \label{GTPulc} Let $\nu (dx\,dv) = f(x,v)\,dx\,dv$ with
\begeq\label{condlogconc}
\forall (x,v)\in\TM,\qquad v\cdot\nabla_V \log f(x,v) \leq -\alpha|v|^2 + A
\endeq
for some $\alpha,A>0$. Then $\nu$ satisfies $\GTP(K)$ for some $K=K(n,\alpha,A)>0$.
\end{Prop}

\begin{proof}[Proof of Proposition \ref{GTPulc}]
Let $\gamma(v) = (2\pi)^{-n/2}\, e^{-|v|^2/2}$ and $h=f/\gamma$. Then
\begin{align} \label{vnvlh}
v\cdot\nabla_v\log h & = v\cdot\nabla_V \log f - v \cdot\nabla_V \log \gamma \\ \nonumber
& = v\cdot\nabla_V \log f + |v|^2\\ \nonumber
& \leq (1-\alpha)|v|^2 + A.
\end{align}
Then, with the reference measure $dx\,dv$, using $\nabla\gamma = -v\gamma$,
\begin{align*}
\iint |v|^2 \vphi^2\, f & = \iint |v|^2 \vphi^2\, h\gamma \\
& = \iint \vphi^2 h v\cdot (v\gamma) \\
& = - \iint \vphi^2 hv\cdot\nabla_V\gamma \\
& = \iint (v\cdot\nabla_V\vphi^2) h\gamma + \iint \vphi^2 v\cdot\nabla_V h \gamma + n \iint \vphi^2 h\gamma \\
& = 2 \iint (\vphi v\cdot\nabla_V\vphi) h\gamma + \iint \vphi^2 v\cdot\nabla_V (\log h) h\gamma + n\iint \vphi^2 h\gamma\\
& \leq 2 \sqrt{ \iint |v|^2 \vphi^2 h\gamma} \sqrt{\iint |\nabla_V\vphi|^2 h\gamma} 
+ (1-\alpha) \iint \vphi^2 |v|^2 h\gamma + (A+n) \iint \vphi^2 h\gamma\\
& \leq \left(1-\frac{\alpha}2\right) \iint \vphi^2 |v|^2 h\gamma
+ \frac2{\alpha} \iint |\nabla_V\vphi|^2 h\gamma + (A+n) \iint \vphi^2 h\gamma,
\end{align*}
where Young's inequality was used. Thus, recalling $h\gamma=f$,
\[ \iint |v|^2 \vphi^2\, f \leq \frac2{\alpha} \left[ \frac2{\alpha} \iint |\nabla_V\vphi|^2 f + (A+n) \iint \vphi^2 f\right]. \]
\end{proof}

\bibnotes

I am not aware of earlier references for Inequality \ref{GTPineq}, but it has certainly been treated before. A much more studied problem is  the usual Poincar\'e inequality which will come in Proposition \ref{poincare}. For the latter, a rather general criterion is $|\nabla\log f|^2/2 - \Delta \log f \to +\infty$ at infinity, which allows for essentially any log superlinear decay \cite[Theorem A.1]{vill:hypoco}.

\section{Anisotropic Nash-type interpolation inequality} \label{secnash}

When the reference measure is Gaussian, Proposition \ref{propgradmom} allows to control moments by vertical derivatives.
But when the reference is the Liouville measure $dx\,dv$, things are not so easy.
However, one can still control weighted derivatives by interpolating with higher order regularity and higher-order $L^1$ moments.
This strategy is reminiscent of Nash's inequality, which controls $L^2$ norm through higher regularity and $L^1$ norm: say, in $\R^n$,
\begeq\label{orignash} \|f\|_{L^2(\R^n)} \leq C\, \|f\|_{L^1(\R^n)}^\theta\, \|\nabla f\|_{L^2(\R^n)}^{1-\theta}, \qquad \theta = \frac2{n+2}.\endeq
The challenge now is to add kinetic weights, and replace the single variable $x$ by the pair of variables $(x,v)$. 
The resulting estimate is ``anisotropic'' in the sense that the horizontal and vertical regularity indices do not necessarily coincide. (It could also be dubbed ``kinetic Nash interpolation inequality''.)

\begin{Thm}[Anisotropic Nash-type inequality on the tangent bundle] \label{propnashinterp}
Let $\alpha,\beta,\alpha',\beta'$ be nonnegative integers such that
\[ \frac{\alpha}{\alpha'} + \frac{\beta}{\beta'} <1, \qquad \alpha'<\beta',\qquad
\alpha+\beta < \beta',\]
and let
\begeq\label{thetabar} \ov{\theta} = \frac{1 - \left(\frac{\alpha}{\alpha'} + \frac{\beta}{\beta'}\right)}
{1 + \frac{n}{2} \left(\frac{1}{\alpha'} + \frac1{\beta'}\right)}.
\endeq
Then there is $\underline{\theta}<\ov{\theta}$ such that for any $\theta \in (\underline{\theta},\ov{\theta})$,
any nonnegative integer $\kappa$, there are $\sigma>0$ and $C>0$ such that for any function $f:\TM\to\R_+$ with $\iint f=1$,
\begeq\label{ineqnashinterp}
\|f\|_{H^{\alpha,\beta}_\kappa} \leq C\, \|f\|_{L^1_\sigma} \, \bigl(\|f\|_{H^{\alpha',0}} + \|f\|_{H^{0,\beta'}}\bigr)^{1-\theta}. 
\endeq
Here the moment index $\sigma$ only depends on the dimension and regularity involved, as well as $\theta$:
\[ \sigma = \sigma(n,\alpha,\alpha',\beta,\beta',\kappa,\theta,\underline{\theta},\ov{\theta}), \]
and the constant $C$ only depends on $M,\alpha,\alpha',\beta,\beta',\kappa,\theta,\sigma$.
\end{Thm}

\begin{Rks} \label{rknashplat} 
\begin{itemize}
\item[(i)] The exponent $\ov{\theta}$ is natural in this problem and plausibly optimal, playing the role of $2/(n+2)$ in the classical Nash inequality. In the case $M=\R^n$, $\kappa=0$ one can prove actually
\begeq\label{nashplat}
\|f\|_{H^{\alpha,\beta}} \leq C \,\|f\|_{L^1}^{\ov{\theta}} \, \bigl(\|f\|_{H^{\alpha',0}} + \|f\|_{H^{0,\beta'}}\bigr)^{1-\ov{\theta}}. 
\endeq
Theorem \ref{propnashinterp} fails to achieve this neat bound, but shows that one can get arbitrarily close to it (in terms of $\theta$),
if one is ready to use arbitrarily large moments ($\sigma\to\infty$ as $\theta\to\ov{\theta}$).
\sm

\item[(ii)] Estimate \eqref{ineqnashinterp} should be improved into a homogeneous inequality, without the condition $\int f =1$. From the proof one can easily improve the $L^1$ factor $\|f\|_{L^1_\sigma}$ into $\|f\|_{L^1_\sigma}^\nu$ for some $\nu>0$ (possibly $\nu=\theta$), or replace the right hand side by a sum of homogeneous terms with various exponents $\theta$; but keeping track of those exponents is cumbersome and would not at this stage yield any notable improvement.
\sm 

\item[(iii)] Theorem \ref{propnashinterp} is stated and proven with integer indices $\alpha,\alpha',\beta,\beta'$. By interpolation this range can be widened. However, due to the constraints in the assumptions, it is not clear that the whole range of noninteger indices can be achieved. This is in contrast with the Euclidean inequality \eqref{nashplat} which can be proven for integer as well as noninteger indices. In the sequel this will not be a limitation, as I shall only need integer exponents for applications of Theorem \ref{propnashinterp}, for two different sets of parameters: (a) $\alpha=m-1$, $\alpha'=m$, $\beta=1$, $\beta'=3m$; (b) $\alpha=0$, $\alpha'=m$, $\beta=3m$, $\beta'=3m+1$ ($m\in\N$ in both cases). But for the sake of consistency, and future use, it is desirable to remove these restrictions.
\sm

\item[(iv)] A key to adapt the proof of \eqref{nashplat} and make Theorem \ref{propnashinterp} neater would be a spectral estimate of the form
\begeq\label{jointPin}
 \Bigl\| \Pi \bigl[\lambda_H \leq A, \, \lambda_V\leq B\bigr ] \Bigr\|_{L^1\to L^2} = O(A^{n/4} B^{N/4}),
 \endeq
where $\Pi$ is the spectral measure appearing in \eqref{spectralHV}. In the case of $\R^n$ this is true with $N=n$ exactly, and that is sufficient to arrive at \eqref{nashplat} with exponent $\ov{\theta}$. Considering that the analysis of $\Delta_V$ at any $x$ is the same as that of $\Delta$ on $\R^n$, and that Weyl's estimates on the spectral measure for $\Delta_x$ depend only on dimension and volume, it is not absurd to conjecture \eqref{jointPin}, also with $N=n$; this would be a kind of {\em tangent bundle variant of Weyl's estimates}. It may also be that there are geometric subtleties. Unable to rely on \eqref{jointPin} at this stage, my argument towards \eqref{ineqnashinterp} will instead be definitely extrinsic, using charts.
\sm

\item[(v)] For sure there are also more general inequalities with mixed terms in the right hand side of \eqref{ineqnashinterp}, similar to \eqref{interpa0a1b0b1}.
\end{itemize}
\end{Rks}

I will start the proof of Theorem \ref{propnashinterp} by a couple of lemmas set in ``flat'' geometry. In the sequel, $D_v= \sqrt{-\Delta_v}$, $D_x = \sqrt{-\Delta_x}$, and I will omit the volume measure $dx\,dv$.

\begin{Lem} \label{lem1nash}
Let $U$ be a bounded smooth open subset of $\R^n$ and $f=f(x,v)\geq 0$ be a nonnegative real-valued function on $U\times\R^n$, vanishing for $x$ near $\pa U$.
If $\alpha,\alpha',\beta,\beta'$ are nonnegative real numbers with $\alpha'>\alpha$, $\beta'>\beta$,
and $\kappa$ is a nonnegative integer then
\begeq\label{Dxlem1} 
\iint |D_x^\alpha f|^2\, (1+|v|^2)^\kappa
\leq C \left( \iint f(x,v)\, (1+|v|^\sigma)\right)^{2\delta} \left( 1+ \iint f^2 + \iint |D_x^{\alpha'} f|^2 + \iint |D_v^{\beta'}f|^2 \right)^{1-\delta},
\endeq
\begeq\label{Dplem1} 
\iint |D_v^\beta f|^2\, (1+|v|^2)^\kappa
\leq C \left( \iint f(x,v)\, (1+|v|^\sigma)\right)^{2\lambda}  \left(1+ \iint f^2 + \iint |D_x^{\alpha'} f|^2 + \iint |D_v^{\beta'}f|^2 \right)^{1-\lambda}
\endeq
for $\sigma$ large enough; and
\[ \delta = \left( 1- \frac{\alpha}{\alpha'}\right)\,
\left[ \frac{\min(\alpha',\beta')}{n+\min(\alpha',\beta')}\right],\qquad
\lambda = \left( 1- \frac{\beta}{\beta'}\right)\,
\left[ \frac{\min(\alpha',\beta')}{n+\min(\alpha',\beta')}\right].\]
In particular, $\delta,\lambda\longrightarrow 0$ as $\alpha\to\alpha'$, $\beta\to\beta'$.
\end{Lem}

\begin{Lem} \label{lem2nash}
Let $U$ be a bounded smooth open subset of $\R^n$ and $f=f(x,v)\geq 0$ be a nonnegative real-valued function
on $U\times\R^n$, vanishing for $x$ near $\pa U$. If $\alpha,\alpha',\beta,\beta'$ are nonnegative real numbers
satisfying $\alpha/\alpha' + \beta/\beta' <1$
and $\kappa$ is a nonnegative integer then
\begeq\label{ineq2nash} \iint |D_x^\alpha\, D_v^\beta f|^2\, (1+|v|^2)^\kappa
\leq C \left( \iint f(x,v)\, (1+|v|^\sigma)\right) \left( 1+ \iint f^2 + \iint |D_x^{\alpha'}f|^2 + \iint |D_v^{\beta'}f|^2\right)^{1-\theta},
\endeq
where $\theta$ can be chosen arbitrarily close to, but less than $\ov{\theta}$ appearing in \eqref{thetabar}, and $\sigma>0$ is large enough.
\end{Lem}

\begin{proof}[Proof of Lemma \ref{lem1nash}]
I shall only establish \eqref{Dxlem1}; a similar reasoning will work for \eqref{Dplem1}.

First, by interpolation in weighted Sobolev spaces,
\begeq\label{step1lem1}
\iint |D_x^\alpha f|^2 (1+|v|^2)^{\kappa}
\leq C \left( \iint |D_x^{\alpha'} f|^2 \right)^{\frac{\alpha}{\alpha'}}
\left( \iint f^2\, (1+|v|^2)^{\kappa'} \right)^{1-\frac{\alpha}{\alpha'}},
\endeq
where $\kappa' = \kappa/(1-\alpha/\alpha')$. Next, by H\"older's inequality, for any $\lambda>1$,
\begeq\label{step2lem1}
\iint f^2\, (1+|v|^2)^{\kappa'}
\leq \left( \iint f\, (1+|v|^2)^{\kappa'\lambda}\right)^{\frac1{\lambda}}
\left(\iint f^{\left(2-\frac1{\lambda}\right)\frac{\lambda}{\lambda-1}}\right)^{1-\frac1{\lambda}}.
\endeq
Let $\gamma= \min(\alpha',\beta')$. If $\gamma<n$ then choose $\lambda$ such that $(2-\frac1{\lambda})(\frac{\lambda}{\lambda-1}) = \frac{4n}{2n-2\gamma}$,
i.e. $\lambda = (2n+2\gamma)/(4\gamma)$. Apply the fractional Sobolev inequality
in $U\times\R^n\subset \R^{2n}$, noting that (by interpolation) $\int |D_x^{\gamma}f|^2 + \int |D_v^{\gamma}f|^2 + \int f^2$ is equivalent to $\int |D_{x,v}^\gamma f|^2+\int f^2$:
\begin{align} \label{newfracSob}
\left(\iint f^{\frac{4n}{2n-2\gamma}}\right)^{\frac{2n-2\gamma}{4n}} 
& \leq C \left(\iint f^2 + \iint |D_x^{\gamma}f|^2 + \iint |D_v^{\gamma}f|^2 \right)^{\frac12}\\
& \nonumber \leq C' \left(\iint f^2 + \iint |D_x^{\alpha'}f|^2 + \iint |D_v^{\beta'}f|^2 \right)^{\frac12},
\end{align}
where the last inequality is also by interpolation, but this time only for powers of $D_x$ or powers of $D_v$.
If $\gamma\geq n$ the Sobolev embedding is better, and the same estimate will work with any $\lambda>1$.
Finally insert this in \eqref{step2lem1}, then in \eqref{step1lem1}; this gives the result with $\sigma = 2\kappa'/\lambda$.
\end{proof}

\begin{proof}[Proof of Lemma \ref{lem2nash}]
First consider the case $\beta=0$. 

Let $\tilde{f}(\xi,\eta)$ stand for the Fourier transform of $f$ in both variables $x$ and $v$.
(So $\xi$ is conjugate to $x$ and $\eta$ to $v$.) By a classical estimate, for any $\gamma\in\N_0^n$,
\begeq\label{dgF} |\nabla_v^\gamma \tilde{f}| \leq C \iint f\, |v|^{|\gamma|} \,dv\,dx.
\endeq
(Recall $f\geq 0$.)

By the Fourier inversion formula, the left-hand side of \eqref{ineq2nash} is bounded by
\[ C \left(\iint |\xi|^{2\alpha} \, |(1-\Delta_\eta)^{\kappa} \tilde{f}|^2\,d\xi\,d\eta
+ \iint |\xi|^{2\alpha}\,|\tilde{f}|^2\,d\xi\,d\eta\right).\]
Divide this integral into three parts: (a) $|\xi|\leq R$, $|\eta|\leq S$,
(b) $|\xi|\leq R$, $|\eta|> S$, (c) $|\xi|>R$. By \eqref{dgF} and Lemma \ref{lem1nash},
\begin{align}\label{bd1n}
\iint_{|\xi|\leq R,\ |\eta|\leq S}
|\xi|^{2\alpha}\, |(1-\Delta_\eta)^{\kappa} \tilde{f}|^2\,d\xi\,d\eta  &
\leq C\, \left( \iint f\, (1+|v|^2)^\kappa\,dv\,dx\right)^2 \iint_{|\xi|\leq R,\ |\eta|\leq S} |\xi|^{2\alpha} \,d\xi\,d\eta \\
\nonumber & \leq C\, R^{n+2\alpha} S^n \left( \iint f\, (1+|v|^2)^\kappa\,dv\,dx\right)^2;
\end{align}
\begin{align} \label{bd2n}
\iint_{|\xi|\leq R, \, |\eta|> S}
|\xi|^{2\alpha}\, |(1-\Delta_\eta)^{\kappa}  & \tilde{f}|^2\,d\xi\,d\eta
\leq C \, \frac{R^{2\alpha}}{S^{2\nu}}
\iint |\eta|^{2\nu}\, |(1-\Delta_\eta)^{\kappa}  \tilde{f}|^2\,d\xi\,d\eta \\
& \leq C \,\frac{R^{2\alpha}}{S^{2\nu}} \iint (1+|v|^{2\gamma})\, |D_v^\nu f|^2  \nonumber\\
& \leq C \,\frac{R^{2\alpha}}{S^{2\nu}} \left(\iint f (1+|v|^2)^\sigma\right)^{1-\delta} \left( \iint f^2 
+ \iint |D_x^{\alpha'} f|^2 + \iint |D_v^{\beta'} f|^2 \right)^{1-\delta}, \nonumber
\end{align}
\[ \delta = \left( 1-\frac{\nu}{\beta'}\right)\, 
\left(\frac{\min(\alpha',\beta')}{n + \min(\alpha',\beta')}\right);\]
and finally,
\begin{align} \label{bd3n}
\iint_{|\xi|> R} |\xi|^{2\alpha}\, |(1-\Delta_\eta)^{\kappa}  \tilde{f}|^2\,& d\xi\,d\eta \\ \nonumber
& \leq \frac{C}{R^{2(\tau-\alpha)}}
\iint |\xi|^{2\tau}\, |(1-\Delta_\eta)^{\kappa}  \tilde{f}|^2\,d\xi\,d\eta \nonumber \\
& \leq \frac{C}{R^{2(\tau-\alpha)}} \left(\iint (1+|v|^2)^{\kappa}\, |D_x^\tau f|^2 + \iint f^2 \right)\nonumber\\
& \leq \frac{C}{R^{2(\tau-\alpha)}}\left(\iint f (1+|v|^\sigma)\right)^{\lambda} \left( 1 + \iint f^2 + \iint |D_x^{\alpha'} f|^2 + \iint |D_v^{\beta'}f|^2\right)^{1-\lambda}, \nonumber
\end{align}
\[ \lambda = \left( 1 - \frac{\tau}{\alpha'}\right)\, 
\left(\frac{\min(\alpha',\beta')}{n+\min(\alpha',\beta')}\right).\]

Adding up \eqref{bd1n}, \eqref{bd2n}, \eqref{bd3n}, recalling that $\int f (1+|v|^\sigma) \leq C \int f(1+|v|^{\sigma'})$ for $\sigma'\geq \sigma$, leads to a bound like
\begeq\label{bd4n}
C \left(1+ \iint f\, (1+|v|^\sigma)\,dv\,dx \right) \left( R^{n+2\alpha} S^n + \frac{R^{2\alpha}}{S^{2\nu}}\, 
W^{\frac{n+\nu\alpha'/\beta'}{n+\alpha'}} 
+ \frac1{R^{2(\tau-\alpha)}}W^{\frac{n+\tau\beta'/\alpha'}{n+\beta'}} 
\right)
\endeq
where
\[ W:= 1 + \iint f^2 + \iint |D_x^{\alpha'} f|^2 + \iint |D_v^{\beta'}f|^2,\]
and $C$ is a  large constant.
An elementary estimate shows that for any positive real numbers $R,S,W$
and any positive exponents $a,b,d,e,f,h,i,j$,
\[ \inf_{R,S>0} \Bigl(
R^a\, S^b + \frac{R^d}{S^e}\, W^f + \frac1{R^h}\,W^i\Bigr)
\leq C\, W^{i- \frac{h(i-\frac{bf}{b+e})}{a+b(\frac{d-a}{b+e})+h}}.\]
Applying this to \eqref{bd4n} we find in the end
\[ \iint |\xi|^{2\alpha}\,|(1-\Delta_\eta)^{\kappa} \tilde{f}|^2 \leq
C \left(\iint f (1+|v|^\sigma) \right) \left( 1 + \iint f^2 + \iint |D_x^{\alpha'} f|^2 + \iint |D_x^{\beta'}f|^2 \right)^{1-\theta},\]
where $\sigma$ is large enough, $\nu,\tau$ are subject to the restrictions $0\leq\nu<\beta'$, $\alpha\leq\tau<\alpha'$, and
\[ \theta = \lambda + \frac{2(\tau-\alpha)\left( (1-\lambda)  - \frac{n(1-\delta)}{n+2\nu}\right)}
{n+2\alpha - \frac{n^2}{n+2\nu} + 2 (\tau-\alpha)}.\]
It follows that
\[ \iint |D_x^\alpha f|^2 \leq
C \left( 1 + \iint f^2 + \iint |D_x^{\alpha'} f|^2 + \iint |D_x^{\beta'}f|^2 \right)^{1-\theta},\]
As $\tau\to \alpha'$ and $\nu\to \beta'$, we have $\delta\to 0$, $\lambda\to 0$,
and 
\begeq\label{thetalimab}
\theta \to \frac{1-\frac{\alpha}{\alpha'}}{1+\frac{n}{2} \left(\frac1{\alpha'}+\frac1{\beta'}\right)}
=: \ov{\theta}(\alpha,\alpha',0,\beta').
\endeq

This concludes the case $\beta=0$. The case $\alpha=0$ is similar, up to nonessential
lower order terms, and yields the limiting exponent
\[ \ov{\theta}(0,\alpha',\beta,\beta') = 
\frac{1-\frac{\beta}{\beta'}}{1+\frac{n}{2} \left(\frac1{\alpha'}+\frac1{\beta'}\right)}.\]

Finally, the general case is deduced from the cases $\alpha=0$ and $\beta=0$,
as follows. First choose $r$ and $s$ such that $r^{-1}+s^{-1}=1$, $r^{-1} \geq \alpha/\alpha'$,
$s^{-1}\geq \beta/\beta'$; then use H\"older's inequality with the conjugate exponents $r$ and $s$:
\[ \iint |\xi|^{2\alpha}\,|\eta|^{2\beta} \, |(1-\Delta_\eta)^{\kappa} \tilde{f}|^2\,d\xi\,d\eta
\leq \left( \iint |\xi|^{2\alpha r}\, |(1-\Delta_\eta)^{\kappa}\tilde{f}|^2\,d\xi\,d\eta\right)^{\frac1{r}}
\left(\iint |\eta|^{2\beta s}\, |(1-\Delta_\eta)^{\kappa}\tilde{f}|^2\,d\xi\,d\eta\right)^{\frac1{s}}.\]
Applying the results for $\alpha=0$ and $\beta=0$ we deduce
\[ \iint |\xi|^{2\alpha}\,|\eta|^{2\beta} \, |(1-\Delta_\eta)^{\kappa} \tilde{f}|^2\,d\xi\,d\eta
\leq C \, \left( 1 + \iint |D_x^{\alpha'} f|^2 + \iint |D_v^{\beta'}f|^2\right)^{
1-\left(\frac{\theta_1}{r}+ \frac{\theta_2}{s}\right)},\]
where $C$ depends on high-order moments of $f$ and 
$\theta_1/r + \theta_2/s$ can be chosen arbitrarily close to
\[ \frac{\ov{\theta} (\alpha r,\alpha',0,\beta')}{r} + \frac{\ov{\theta}(0,\alpha',\beta s,\beta')}{s},\]
which by computation coincides with the exponent $\ov{\theta}$ appearing in \eqref{thetabar}.
\end{proof}

\begin{proof}[Proof of Theorem \ref{propnashinterp}]
Since $\int f=1$ it is equivalent to estimate by $\|f\|_{L^1_\sigma}$ or by $1+\|f\|_{L^1_\sigma}$, and the latter can be used without any restriction on $\int f$.

Introducing a smooth partition of unity $(\chi_i)_{1\leq i\leq N}$ to localize on $M$ via charts in open sets $U_i$; write $f= \sum f\chi_i$ and estimate each $f\chi_i$ separately;
\[ \|f\|_{H^{\alpha,\beta}_\kappa} \leq \sum_{1\leq i\leq N} \|f\chi_i\|_{H^{\alpha,\beta}_\kappa}.\]
If each $f\chi_i$ satisfies the desired inequality, then
\begin{align*}
\|f\|_{H^{\alpha,\beta}_\kappa} & \leq C \sum_{i=1}^N \|f\chi_i\|_{L^1_\sigma} \bigl( \|f\chi_i\|_{H^{\alpha',0}} + \|f\chi_i\|_{H^{0,\beta'}} \bigr)^{1-\theta} \\
& \leq C \|f\|_{L^1_\sigma} \, \sum_{i=1}^N \bigl( \|f\chi_i\|_{H^{\alpha',0}} + \|f\chi_i\|_{H^{0,\beta'}} \bigr)^{1-\theta}\\
& \leq C\, N\, \|f\|_{L^1_\sigma}\, \left[ \sum_{i=1}^N \bigl( \|f\chi_i\|_{H^{\alpha',0}} + \|f\chi_i\|_{H^{0,\beta'}} \bigr)^{1-\theta}\right]^{1-\theta}\\
& \leq C \bigl(N, (\chi_i)_{1\leq i\leq N}\bigr) \|f\|_{L^1_\sigma} \bigl( \|f_{H^{\alpha',0}} + \|f\|_{H^{0,\beta'}}\bigr).
\end{align*}
Thus it is sufficient to prove the inequality when $f$, read in a chart, is supported in $U\times\R^n$.

Working in $\R^n$, consider the derivations $\pa/\pa x^i$ and $\pa/\pa v_i$. The horizontal derivatives $Df/Dx^i = (\nabla_H)_i f$ and the usual partial derivatives $\pa f/\pa x^i$ can be bound in terms of one another as follows, using extra vertical derivatives:
\[ \left| \frac{Df}{Dx^i} \right| \leq
C\, \left( \left|\frac{\partial f}{\partial x^i}\right| + |v|\, |\nabla_v f|\right);
\qquad 
\left|\derpar{f}{x^i}\right|\leq C \left(
\left|\frac{Df}{Dx^i}\right| + |v|\,|\nabla_v f|\right);\]
here $|v|$ is the Euclidean norm of $v$ and $|\nabla_v f|$ is the Euclidean norm
of the vector $(\pa f/ \pa v_i)_{1\leq i\leq n}$. The ratio $|\nabla_vf|/|\nabla_Vf|$
is bounded from above and below, so in all estimates it does not matter whether we
work with the vertical derivatives or with the partial derivatives with respect to $v$.
Also similar formulas hold for higher-order derivatives.

Combining this with the Leibniz formula, one can estimate $\|f\|_{H^{\alpha,\beta}_\kappa}^2$: for instance the dominant part is
\begin{multline*} \iint |\nabla_H^\alpha \nabla_V^\beta f|^2\, (1+|v|^2)^\kappa \\
\leq C \sum_{\ov{\alpha}\leq \alpha,\, \ov{\beta}\leq\beta,\, \ov{\gamma}\leq\gamma}
\left( \iint |\nabla_x^{\ov{\alpha}} \nabla_v^{\ov{\beta}} f|^2\, 
(1+|v|^2)^{\ov{\gamma}}
+ \iint |\nabla_v^{\ov{\alpha} +\ov{\beta}} f|^2\, (1+|v|^2)^{\ov{\beta} + \ov{\gamma}}
\right).
\end{multline*}
Then all terms of the sum can be estimated separately.
For simplicity, let us consider only the term $\ov{\alpha}=\alpha$, $\ov{\beta}=\beta$, 
which leads to the worst estimate:
\begeq\label{worsen}
\iint |\nabla_x^{\alpha} \nabla_v^{\beta} f|^2\, (1+|v|^2)^{\gamma}
+ \iint |\nabla_v^{\alpha +\beta} f|^2\, (1+|v|^2)^{\beta + \gamma}.
\endeq
Thanks to Lemma \ref{lem2nash}, the first integral in \eqref{worsen} is bounded by 
\begin{multline} \label{worsen2}
C\left(\iint (1+|v|^\sigma) f \right)^{\theta} \left( 1+ \iint |\nabla_x^{\alpha'} f|^2 + \iint |\nabla_v^{\beta'}f|^2 + \iint f^2\right)^{1-\theta} \\
\leq C' \left(\iint (1+|v|^{\sigma'})^{\theta} f \right) \left( 1+ \iint |\nabla_H^{\alpha'}f|^2 + \iint |v|^{\alpha'}\,|\nabla_V^{\alpha'}f|^2
+ \iint |\nabla_V^{\beta'}f|^2 + \iint f^2\right)^{1-\theta},
\end{multline}
where $C$ and $C'$ are constants. By Lemma \ref{lem1nash}, since $\alpha'<\beta'$ we can bound \eqref{worsen2} by
the right-hand side of \eqref{ineqnashinterp}.

Finally, to control the second integral in \eqref{worsen}, apply again Lemma \ref{lem1nash} and use the inequality $\alpha+\beta<\beta'$. All in all, 
\[ \|f\|_{H^{\alpha,\beta}_\kappa}^2 \leq C \|f\|_{L^1_\sigma}\, \bigl( \|f\|_{H^{\alpha',0}}^2 +\|f\|_{H^{0,\beta'}}^2\bigr)\]
and the proof of Theorem \ref{propnashinterp} is complete.
\end{proof}

\bibnotes

John Nash used inequality \eqref{orignash} as a step towards his celebrated regularity result on solutions of parabolic equations \cite{nash:58}; it was provided to him by Elias Stein. It has been the object of many works in Riemannian manifolds \cite{BGL:book}. This inequality is a particular case of the more general family of {\bf Gagliardo--Nirenberg inequalities}, see e.g. Brezis--Mironescu \cite{brezismironescu:GN:18}.

I considered the variant with two variables ($x,v$) in \cite[Lemma A.25]{vill:hypoco}, there establishing \eqref{nashplat} for $\gamma=0$; the strategy was an expansion of Stein's Fourier-based original proof. The proof of Lemma \ref{lem2nash} is directly adapted from that source, and was used in \cite{DOV:preprint} to prove Theorem \ref{propnashinterp}.

\section{Lebeau's maximal hypoellipticity} \label{seclebeau}

In this section comes the first notable result of this program: I shall retrieve the maximal global hypoellipticity estimate associated with the stationary geometric kinetic Fokker--Planck operator, as well as its variant on vector fields (Bismutian). The result, first proven by Lebeau, is a gain of 2 vertical derivatives and 2/3 horizontal ones, in global Sobolev spaces with reference measure $\mu$. Convenient tools having already been established, this will be done without too much pain by adapting a method of Fran\c cois Bouchut.

\begin{Thm}[Maximal hypoellipticity for the geometric kinetic operator]\label{thmalabouchut} 
(i) If $M$ is a compact Riemannian manifold then there is a constant $C=C(M)$ such that for any $f:\TM\to\R$,
\begeq\label{maxhypo} \|f\|_{H^{2/3,0}(\mu)} + \|f\|_{H^{0,2}(\mu)} \leq C\, 
\bigl(\|Lf\|_{L^2(\mu)} + \|f\|_{L^2(\mu)}\bigr).
\endeq
Similarly for any integers $I,J$ there is $C=C(M,I,J)$ such that for any $X:\TM\to \TIJ$,
\[ \|X\|_{H^{2/3,0}(\mu)} + \|X\|_{H^{0,2}(\mu)} \leq C\, 
\bigl(\|LX\|_{L^2(\mu)} + \|X\|_{L^2(\mu)}\bigr);
\]

(ii) The same estimate holds true when ${\cal B}=\Xi-\Delta_V^\mu+ P_V$ is the Bismutian acting on vector fields on $\TM$, or tensors of vector fields:
\[ \|X\|_{H^{2/3,0}(\mu)} + \|X\|_{H^{0,2}(\mu)} \leq C\, 
\bigl(\|{\cal B} X\|_{L^2(\mu)} + \|X\|_{L^2(\mu)}\bigr).\]
\sm

(iii) For the Liouville measure, the following estimates holds: There are $\sigma=\sigma(n)>0$ and a constant $C=C(M)$ such that for any $f:\TM\to\R$,
\[ \|f\|_{H^{2/3,0}} + \|f\|_{H^{0,2}} \leq C\, \bigl(\|(\xi-\Delta_V)f\|_{L^2} + \|f\|_{L^2_\sigma}\bigr). \]
Likewise, if ${\cal L}=\xi - \Delta_V - v\cdot\nabla_V - n$ is the geometric Fokker--Planck acting on probability densities, then
\[ \|f\|_{H^{2/3,0}} + \|f\|_{H^{0,2}} \leq C\, \bigl(\|{\cal L}f\|_{L^2} + \|f\|_{L^2_\sigma}\bigr). \]
\sm

(iv) More generally, if $k\in\N$ and $\alpha,\beta\geq 0$ then the following estimates hold:
\[ \|f\|_{H^{\alpha + 2k/3,\beta}(\mu)} + \|f\|_{H^{\alpha,\beta+2k}(\mu)} \leq C\, \bigl(\|L^kf\|_{H^{\alpha,\beta}(\mu)} + \|f\|_{H^{\alpha,\beta}(\mu)}\bigr);\]
\[ \|X\|_{H^{\alpha+2k/3,\beta}(\mu)} + \|X\|_{H^{\alpha,\beta+2k}(\mu)} \leq C\, 
\bigl(\|{\cal B}^k X\|_{H^{\alpha,\beta}(\mu)} + \|X\|_{H^{\alpha,\beta}(\mu)}\bigr);\]
\[ \|f\|_{H^{\alpha+2k/3,\beta}} + \|f\|_{H^{\alpha,\beta+2k}} \leq C\, \bigl(\|{\cal L}^kf\|_{H^{\alpha,\beta}} + \|f\|_{H^{\alpha,\beta}_\sigma}\bigr). \]
\end{Thm}

\begin{Rks}  \label{rklebeau}
\begin{itemize}
\item[(i)] The short version of that theorem is that inverting the Fokker--Planck operator gains 2 vertical derivatives and 2/3 horizontal derivatives, whatever the geometry. Since $\xi = L+\Delta_V^\mu$ and $\Delta_V^\mu$ is controlled in $L^2$ by $L$, another way to see it is 
\[ \|\xi f \|_{L^2} + \|\Delta_V^\mu f\|_{L^2} \leq C \bigl ( \|(\xi-\Delta_V^\mu) f\|_{L^2} + \|f\|_{L^2}\bigr).\]
In other words, controlling $f$ along $\xi-\Delta_V^\mu$ is the same as controlling $f$ along $\xi$ and $\Delta_V^\mu$ separately, as if $\xi$ and $\Delta_V^\mu$ would be ``orthogonal''.
\sm

\item[(ii)] Even when one considers the generator $\xi$ acting on functions, it is perfectly possible to consider estimates for tensor fields, since $\xi$ acts covariantly on tensors. In the same way, the Bismutian acts on vector fields $(X_H,X_V) = ((X_H)_1,\ldots,(X_H)_n, (X_V)_1,\ldots, (X_V)_n)$ on $\TM$, but also acts covariantly on tensors of vector fields $((X_H)_{j i_1\ldots i_k}, (X_V)_{j i_1\ldots i_k})$. In the sequel I shall provide the proofs only when $f$ is a function and $X=(X_H,X_V)$ a vector field, but the proof goes on for tensors with no change at all.
\sm

\item[(iii)] Part (iii) of the Theorem expresses the fact that inverting the kinetic Fokker--Planck operator gains maximal regularity if one is ready to lose moments (worse kinetic weight). When the measure is Gaussian, one can absorb those moments in the regularity gain, but when the measure is Liouville, there is no reason for that to occur. Then a further localisation analysis is needed to use such estimates; this will be addressed in the next section, for the evolution equation. 
\sm

\item[(iv)] While there is full recovery of the vertical derivatives, there is a discrepancy between the number of horizontal derivatives which are controlled by, or which do control, the operator $L$. For instance, in $L^2(\mu)$, using also \eqref{propgradmom},
\begin{align*}
K \bigl( \|h\|_{H^{2/3},0} + \|h\|_{H^{0,2}}\bigr) & \leq \|Lh\|_{L^2} + \|h\|_{L^2} \\
& \leq C \bigl( \|h\|_{H^{1,0}_1} + \|h\|_{H^{0,2}}\bigr)\\
& \leq C \bigl( \|h\|_{H^{1,1}} + \|h\|_{H^{0,2}}\bigr).
\end{align*}
This loss by a factor $1/3$ in the horizontal estimates is unavoidable, recall the discussion in Section \ref{sechypo}. Another way to express that discrepancy is in terms of powers of the total laplacian $\Delta_H+\Delta_V$:
\begeq\label{discreptotal}
K \Bigl( \|f\|^2 + \bigl\|(-\Delta_H -\Delta_V)^{1/3} h\bigr\|^2 \Bigr) \leq \|f\|^2 + \|L h\|^2 \leq C \Bigl( \|f\|^2 + \bigl\|(-\Delta_H -\Delta_V) h\bigr\|^2 \Bigr),
\endeq
or equivalently
\begeq\label{discreptotalbis}
K \bigl( \Id -\Delta_H-\Delta_V\bigr)^{1/3} \leq \sqrt{ \Id + L^*L} \leq C \bigl( \Id - \Delta_H-\Delta_V\bigr),
\endeq
with no way to improve the exponent in either the upper or the lower bound.
\sm

\item[(v)] With respect to the flat case, the most importance difference, at least for the scalar case, lies in the commutator identity $[\nabla_H,\xi] = -\Sigma(v,v)\nabla_V$, \eqref{xiNablaHr}. To see beforehand that this does not threaten the maximal regularity, here is the heuristics. Count weight~1 for $\nabla_V$, hence 2 for its ``square'' $\Delta_V^\mu$, hence 2 also for $\xi$ which is the other dominant term in the expression of $L$, hence $1+2=3$ for $\nabla_H = [\nabla_V, \xi]$; and count (at most)~1 for multiplication by~$v$ in view of Proposition \ref{propgradmom}; then the weight of $\Sigma(v,v)\nabla_V$ is only~3, which is strictly less than the sum (5) of the operators involved in the commutator $[\xi,\nabla_H]$.
\sm

\item[(vi)] The maximal regularity makes it possible to perturb the operator: If $\|(L_\var-L)f\|$ is small $H^{0,2}+ H^{2/3,0}\to L^2$, then there is still the same estimate on the equation $L_\var f=h$ as on the equation $Lf=h$. This allows for quasilinear perturbations of the vertical Laplace operator, and zero-order perturbations of the geodesic operator. I did not try to go further down that road, which for elliptic estimates has been extensively explored.
\sm

\item[(vii)] As a consequence of Remark (iv) above, there is a ``skewed'' interpolation inequality of the form
\begeq\label{skewedinterp}
\|Lh\|_{L^2(\mu)} \leq C \bigl( \|L^2 h\|^{3/4}_{L^2(\mu)} \|h\|^{1/4}_{L^2(\mu)} + \|h\|_{L^2(\mu)}\bigr).
\endeq
(Compare with the usual interpolation inequality $\| Ah\| \leq \|A^2h\|^{1/2} \|h\|^{1/2}$ for symmetric operators.) Indeed,
\begin{align*}
\|Lh\| & \leq \bigl\| (\Id -\Delta_H -\Delta_V)^{-1} Lh\bigr\|^{1/4} \bigl\| (\Id-\Delta_H-\Delta_V)^{1/3} Lh\bigr\|^{3/4} \\
& \leq C \|h\|^{1/4} \bigl( \| (\Id + L) Lh\|\bigr)^{3/4}\\
& \leq C \bigl( \|h\|^{1/4} \|L^2h\|^{3/4} + \|h\|^{1/4} \|Lh\|^{3/4}\bigr)\\
& \leq C \bigl( \|h\|^{1/4} \|L^2h\|^{3/4} + \|h\| \bigr) + \var \|Lh\|,
\end{align*}
with $\var>0$ as small as desired, and the conclusion follows.
\sm

\item[(viii)] Combining the estimate \eqref{maxhypo} (maximal regularisation) and Proposition \ref{propinterp} (Sobolev interpolation),
\begin{align*} \|f\|_{H^{\frac13,1}(\mu)}
& \leq  C ( \|f\|_{H^{\frac23,0}(\mu)} + \|f\|_{H^{0,2}(\mu)})\\
& \leq C (\|Lf\|_{L^2(\mu)} + \|f\|_{L^2(\mu)}) \\
& \leq C (\|\Delta_V^\mu f\|_{L^2(\mu)} + \|\xi f\|_{L^2(\mu)} + \|f\|_{L^2(\mu)})\\
& \leq C (\|f\|_{H^{0,2}(\mu)} + \|\xi f\|_{L^2(\mu)}).
\end{align*}
One may conjecture that the same holds true by ``shifting the vertical regularity'', so that
\begeq\label{conjpoinc}
\|f\|_{H^{\alpha,0}(\mu)} \leq C \bigl( \|f\|_{H^{0,1}(\mu)} + \|\xi f\|_{H^{0,-1}(\mu)}\bigr),
\endeq
with $\alpha=1/3$, or maybe any $\alpha<1/3$ to be on the safer side.
If that is true, it will be an analogue, in this context, of the functional inequality introduced by Albritton--Armstrong--Mourrat--Novack under the name {\bf H\"ormander's inequality}, providing an alternative approach to global hypoelliptic estimates.
\sm

\item[(ix)] In addition to estimates \eqref{maxhypo} and \eqref{skewedinterp}, two further expressions of the global hypoellipticity property will be provided respectively in Sections \ref{secreg} (regularisation rates for the evolution equation) and \ref{secqualspectr} (localisation of large eigenvalues in the complex plane).
\end{itemize}
\end{Rks}

Part (iv) of the Theorem comes easily by iterating the estimate for $k=1$ and using the commutator identities between $L$ and the vertical and derivation operators; so let us focus on the proof of (i)--(ii)--(iii).

\begin{proof}[Proof of Theorem \ref{thmalabouchut}(i)]
Recall that
\[L =  \xi - \Delta_V^\mu = \xi - \Delta_V + v\cdot\nabla_V. \]
Let $D_V^\mu=\sqrt{-\Delta_V^\mu}$, $D_H = \sqrt{-\Delta_H}$, both unbounded operators in $L^2(\mu)$.
To alleviate notation, I will omit the measure $\mu$ and write just $L^2$ for $L^2(\mu)$, etc.
\sm

The proof is in three steps:
\sm

{\bf Step 1:} Commute the equation with $D_H^{1/3}$. Since $D_H$ and $\Delta_V^\mu$ commute, 
\[
(\xi-\Delta^\mu_V) D_H^{1/3} f = D_H^{1/3}(Lf) + [\xi,D_H^{1/3}]f.
\]
Then integrate against $D_H^{1/3}f$: Since $\<h, L h\>_{L^2} = \|\nabla_Vh\|_{L^2}^2$,
\begin{align} \label{sincehLh}
\|\nabla_V D_H^{1/3} f\|_{L^2}^2
& = \<D_H^{1/3} (L f), D_H^{1/3}f\>_{L^2} + \bigl\< [\xi, D_H^{1/3}]f, D_H^{1/3}f\bigr\>_{L^2}\\
& = \<D_H^{2/3} f, Lf\>_{L^2} + (\I), \nonumber
\end{align}
where
\begeq\label{I}
(\I) = \Bigl\< [\xi, D_H^{1/3}]f, D_H^{1/3}f\Bigr\>_{L^2}.
\endeq
By Proposition \ref{propxiboundnomom}(i),
\begin{align*} (\I)
& \leq \bigl\|(D_V^\mu)^{3/2} D_H^{-1/2} D_H^{1/3}f\bigr\|_{L^2}\,
\bigl\|(D_V^\mu)^{-3/2} D_H^{1/2} [\xi,D_H^{1/3}] f\bigr\|_{L^2}\\
& \leq \|f\|_{H^{-1/6,3/2}}\, \|[\xi,D_H^{1/3}]f\|_{H^{1/2,-3/2}}\\
& \leq C\, \|f\|_{H^{-1/6,3/2}}^2.
\end{align*}
Dropping the $1/6$ horizontal gain, just retain
\begeq\label{DVDH13}
\|\nabla_V D_H^{1/3} f\|_{L^2}^2 \leq 
C\, \Bigl( \|D_H^{2/3}f\|_{L^2}\,\|Lf\|_{L^2} + \|f\|^2_{H^{0,3/2}}\Bigr).
\endeq
\sm

{\bf Step 2:} Commute the equation with $\nabla_V$, using (i) and (iii) in Proposition \ref{propHV}:
\begeq\label{comeqnablaV} 
(\xi-\Delta^\mu_V) \nabla_V f = \nabla_V (Lf) + [\xi,\nabla_V]f + [v\cdot\nabla_V,\nabla_V]f
= \nabla_V(Lf) - \nabla_Hf - \nabla_Vf;
\endeq
then integrate against $\nabla_V f$:
\begin{align*}
\|\nabla_V^2 f\|^2_{L^2} & = \< \nabla_V(Lf), \nabla_V f\>_{L^2} - \<\nabla_Hf, \nabla_Vf\>_{L^2} - \|\nabla_V f\|^2_{L^2}\\
& = - \<Lf, \Delta^\mu_V f\>_{L^2} - \bigl\< \nabla_V D_H^{-2/3} \nabla_H f, D_H^{2/3}f\bigr\>_{L^2} - \|\nabla_V f\|^2_{L^2} \\
& \leq \|Lf\|_{L^2}\, \|\Delta^\mu_Vf\|_{L^2} + \|\nabla_V D_H^{1/3} f\|_{L^2}\, \|D_H^{2/3} f\|_{L^2} - \|\nabla_V f\|^2_{L^2}.
\end{align*}
Combining this with \eqref{DVDH13} and dropping the last (negative) term,
\begeq\label{afterdrop} 
\|\nabla_V^2 f\|_{L^2}^2 
\leq C\, \Bigl( \|Lf\|_{L^2}\, \|\nabla_V^2 f\|_{L^2} + \|D_H^{2/3}f\|_{L^2}^{3/2}\,\|Lf\|_{L^2}^{1/2}
+ \|D_H^{2/3}f\|_{L^2} \, \|f\|_{H^{0,3/2}}\Bigr),
\endeq
whence
\begeq\label{DVDH14}
\|\nabla_V^2 f\|_{L^2}^2 
\leq C\, \Bigl( \|Lf\|_{L^2}^2 + \|f\|^2_{H^{0,3/2}} \Bigr)
+ \var\, \|D_H^{2/3}f\|^2_{L^2}
\endeq
where $\var>0$ is arbitrarily small.
\sm

{\bf Step 3:} Again by Proposition~\ref{propHV}(i),
\begin{align}
\|D_H^{2/3} f\|^2_{L^2} & = \|D_H^{-1/3} \nabla_H f\|^2_{L^2} \nonumber \\
& = \<D_H^{-2/3} \nabla_H f, \nabla_Hf\>_{L^2} \nonumber \\
& = \bigl\<D_H^{-2/3} \nabla_Hf, \nabla_V \xi f\>_{L^2} - \<D_H^{-2/3} \nabla_H f, \xi\nabla_V f\bigr\>_{L^2} \nonumber \\
& = - \bigl\< \nabla_V D_H^{-2/3} \nabla_Hf, \xi f\bigr\>_{L^2} + \bigl\< \xi D_H^{-2/3} \nabla_H f, \nabla_V f\bigr\>_{L^2}\nonumber \\
& = - \bigl\< \nabla_V D_H^{-2/3} \nabla_H f, \xi f\bigr\>_{L^2}
+ \bigl\< D_H^{-2/3} \nabla_H \xi f, \nabla_V f\bigr\> + \bigl\< [\xi, D_H^{-2/3} \nabla_H]f, \nabla_V f\bigr\>_{L^2}\nonumber \\
& \leq 2\, \bigl\|D_V^\mu D_H^{1/3} f\bigr\|_{L^2}\, \|\xi f\|_{L^2} + (\II), \label{DVDH15}
\end{align} 
where
\begeq \label{II}
(\II) = \bigl\< [\xi, D_H^{-2/3} \nabla_H]f, \nabla_V f\bigr\>_{L^2}
\endeq

Then from Proposition \ref{propxiboundnomom}(i) again,
\begin{align*}
(\II) & \leq \bigl\| (D_V^\mu)^{-1} D_H^{1/3} [\xi,D_H^{-2/3}\nabla_H] f\bigr\|_{L^2}\,
\|D_H^{-1/3} D_V^\mu \nabla_V f\|_{L^2}\\
& \leq C\, \|f\|_{H^{-1/3,2}}\, \|D_H^{-1/3} (D_V^\mu)^2f\|^2_{L^2} \\
& \leq C \|f\|_{H^{-1/3,2}}^2.
\end{align*}

Dropping the $1/3$ horizontal derivative gain, only retain
\[ (\II) \leq C \|f\|_{H^{0.2}}^2.\]
Putting this back in \eqref{DVDH15}, using \eqref{DVDH13} and the obvious inequality
$\|\xi f\|\leq \|Lf\| + \|\nabla^2 f\|$,
\[ \|D_H^{2/3} f\|^2_{L^2} \leq C \Bigl( \|D_H^{2/3}f\|^{1/2}_{L^2}\, \|Lf\|^{1/2}_{L^2}
+ \|f\|_{H^{0,2/3}} \Bigr) \Bigl( \|Lf\|_{L^2} + \|\nabla_V^2f\|_{L^2} \Bigr) + C\, (\|\nabla_V^2 f\|^2_{L^2} + \|f\|^2_{L^2});
\]
hence
\begeq\label{DVDH16}
\|D_H^{2/3} f\|^2_{L^2} \leq C \,\Bigl( \|Lf\|^2_{L^2} + \|f\|_{H^{0,2/3}}^2 + \|\nabla_V^2 f\|_{L^2}^2\Bigr).
\endeq
Plugging this back in \eqref{DVDH14},
\begeq\label{DVDH17}
\|\nabla_V^2 f\|^2_{L^2} \leq C \, \bigl( \|Lf\|^2_{L^2} + \|f\|_{H^{0,2/3}}^2 \bigr).
\endeq
By interpolation (and Young),
\[ \|f\|_{H^{0,2/3}} \leq \var \|\nabla^2_V f\|_{L^2} + C \|f\|_{L^2},\]
where $\var>0$ is arbitrarily small and $C$ is a constant. So \eqref{DVDH17} becomes
\begeq\label{DVDH17'}
\|\nabla_V^2 f\|^2_{L^2} \leq C \, \bigl( \|Lf\|^2_{L^2} + \|f\|^2_{L^2} \bigr).
\endeq
Plugging this again in \eqref{DVDH16},
\[ \|D_H^{2/3} f\|^2_{L^2} \leq C \,\Bigl( \|Lf\|^2_{L^2} + \|f\|_{L^2}^2 \Bigr). \]
All in all,
\begeq\label{allinall} \|D_H^{2/3} f\|^2_{L^2} + \|\nabla_V^2 f\|^2_{L^2} \leq C \Bigl( \|Lf\|^2_{L^2} +\|f\|_{L^2}^2\Bigr),
\endeq
which concludes the proof.
\end{proof}

Before turning to the proof of part (ii), here are some comments. From the modification of $\Delta_V^\mu$ into $\Delta_V^\mu- P_V$ comes an extra term $-X_V$ (order~0) which does not change much. The replacement of $\xi$ by $\Xi$ is more significative. Commutator estimates with fractional horizontal derivatives are preserved and can be handled similarly for $\Xi$ as for $\xi$. But the estimate $\<X,\Xi X\>$ brings a horizontal contribution of the same order as the main term; to take care of this, it will be useful to separate estimates of the horizontal and vertical parts, using \eqref{XXiX} to leverage the better estimate for the vertical part.

\begin{proof}[Proof of Theorem \ref{thmalabouchut}(ii)]
Now the operator is
\[ {\cal B} = \Xi - \Delta_V^\mu + P_V = \Xi - \Delta_V + v\cdot\nabla_V + P_V \]
and $X=(X_H,X_V)$ is a vector field on $\TM$, with horizontal and vertical components. Apart from Propositions \ref{propHV} and \ref{propxiboundnomom}, I shall also use the following estimates deduced from \eqref{Xiexpl}, \eqref{Xicom} \eqref{XXiX} and \eqref{propgradmom}:
\begeq\label{+a}
\begin{cases} 
\<(\Xi X)_H, X_H\>_{L^2} = O \bigl( \|X_H\|_{H^{0,1}} \|X_V\|_{H^{0,1}}\bigr) \\[2mm]
\<(\Xi X)_V, X_V\>_{L^2} = O \bigl( \|X_H\|_{L^2} \|X_V\|_{L^2}\bigr)
\end{cases}
\endeq
\begeq\label{+b}
\Bigl\| \bigl( \bigl( [\nabla_V, \Xi] - \nabla_H \bigr) \Bigr\|_{L^2} = O ( \|X_V\|_{H^{0,1}})
\endeq
\begeq\label{+c}
\bigl\| (\Xi-\xi) X \bigr\|_{L^2} = O \bigl( \|X_V\|_{H^{0,2}} + \|X_H\|_{L^2}\bigr).
\endeq

Now come the same three steps as in the proof of part (i).

{\bf Step~1:} Commute with $D_H^{1/3}$:
\[
(\Xi-\Delta_V^\mu + P_V) D_H^{1/3}X = D_H^{1/3} ({\cal B}X) + [\Xi, D_H^{1/3}] X.\]
Take horizontal component and integrate against $D_H^{1/3}X_H$:
\begin{multline}\label{bisdvh}
\|\nabla_V D_H^{1/3}X_H\|_{L^2}^2
= \bigl\< D_H^{1/3}({\cal B} X)_H, D_H^{1/3}X_H\bigr\>_{L^2}
+ \bigl\< ([\Xi,D_H^{1/3}]X)_H, D_H^{1/3}X_H\bigr\>_{L^2} \\
 - \bigl\< (\Xi D_H^{1/3}X)_H, D_H^{1/3}X_H\bigr\>_{L^2}.
\end{multline}
Doing similarly with the vertical component,
\begin{multline}\label{bisdvv}
\|\nabla_V D_H^{1/3}X_V\|_{L^2}^2
= \bigl\< D_H^{1/3}({\cal B}X)_V, D_H^{1/3}X_V\bigr\>_{L^2}
+ \bigl\< ([\Xi,D_H^{1/3}]X)_V, D_H^{1/3}X_V\bigr\>_{L^2} \\
- \|D_H^{1/3}X_V\|_{L^2}^2 - \bigl\< (\Xi D_H^{1/3}X)_V, D_H^{1/3}X_V\bigr\>_{L^2}.
\end{multline}
Applying Proposition \ref{propxiboundnomom}, \eqref{+a} and \eqref{bisdvh},
\begin{multline} \label{bisdvh2}
\bigl\|\nabla_V D_H^{1/3} X_H \bigr\|_{L^2}^2 \leq 
C \Bigl( \|D_H^{2/3}X_H\|_{L^2} \|LX\|_{L^2} + \|X\|_{H^{0,3/2}}^2 \\
+ \|\nabla_V D_H^{1/3} X_V\|_{L^2}\| \nabla_VD_H^{1/3} X_H\|_{L^2} 
+ \|\nabla_V D_H^{1/3}X_V\|_{L^2} \|D_H^{1/3}X_H\|_{L^2} +  \|D_H^{1/3} X_V\|_{L^2} \|\nabla_V D_H^{1/3} X_H\|_{L^2} \Bigr).
\end{multline}
In the same way, \eqref{bisdvv} leads to
\begeq \label{bisdvv2}
\bigl\|\nabla_V D_H^{1/3} X_V \bigr\|_{L^2}^2 \leq 
C \Bigl( \|D_H^{2/3}X_V\|_{L^2} \|{\cal B} X\|_{L^2} + \|X\|_{H^{0,3/2}}^2 
+ \| D_H^{1/3} X_H\|_{L^2} \|D_H^{1/3} X_V\|_{L^2} \Bigr).
\endeq
(Note that the most serious error term $\|\nabla_V D_H^{1/3} X_V\|_{L^2}\| \nabla_VD_H^{1/3} X_H\|_{L^2}$, is of the same order as the quantity we are estimating, but appears only in the horizontal component.) Combining \eqref{bisdvh2} and \eqref{bisdvv2}, for $\delta>0$ small enough there is $C>0$ such that
\begeq \label{DVH13}
\delta \bigl\| \nabla_V D_H^{1/3} X_H\bigr\|_{L^2}^2 + \bigl\|\nabla_V D_H^{1/3} X_V \bigr\|_{L^2}^2 
\leq C \bigl( \|{\cal B}X\|_{L^2}^2 + \|X\|_{H^{0,3/2}}^2 + \|X\|_{H^{1/3,0}}^2\bigr).
\endeq
\sm

{\bf Step 2:} Commute the equation with $\nabla_V$:
\begin{align*}
{\cal B}\nabla_V X & = (\Xi-\Delta_V^\mu  + P_V)\nabla_V X \\
& = \nabla_V({\cal B}X) - [\nabla_V,\xi]X - [\nabla_V, v\cdot\nabla_V] X - \bigl( [\nabla_V,\Xi] - \nabla_H\bigr) X \\
& = \nabla_V ({\cal B}X) - \nabla_H X - \nabla_V X - \bigl( [\nabla_V,\Xi] -\nabla_V\bigr)X.
\end{align*}
Integrating against $\nabla_V X$,  separating horizontal and vertical components,
\begin{multline}\label{bisd2vh}
\|\nabla_V^2 X_H\|_{L^2}^2 
 = \bigl\< \nabla_V ({\cal B}X)_H, \nabla_V X_H \bigr\>_{L^2}
 - \bigl\< \nabla_H X_H, \nabla_V X_H\bigr\>_{L^2} \\
 - \Bigl\<\bigl( \bigl( [\nabla_V,\Xi] -\nabla_H\bigr)_H, \nabla_VX_H\Bigr\>_{L^2}
 - \bigl\< (\Xi\nabla_V X)_H, \nabla_V X_H\bigr\>_{L^2}
 \end{multline}
and
\begin{multline}\label{bisd2vv}
\|\nabla_V^2 X_V\|_{L^2}^2 
 = \bigl\< \nabla_V ({\cal B}X)_V, \nabla_V X_V \bigr\>_{L^2}
 - \bigl\< \nabla_H X_V, \nabla_V X_V\bigr\>_{L^2} 
 - \bigl\< (\Xi\nabla_V X)_V, \nabla_V X_V\bigr\>_{L^2}.
 \end{multline}
(There is no vertical component of $[\nabla_V,\Xi]-\nabla_H$.)
So by Proposition \ref{propxiboundnomom} and \eqref{+a}--\eqref{+b}, inequality \eqref{bisd2vh} leads to
\begin{multline}\label{bisdvvh2}
\|\nabla_V^2 X_H\|_{L^2}^2 
\leq C \Bigl( \|{\cal B}X\|_{L^2} \|\nabla_V^2 X_H\|_{L^2}
+ \|\nabla_V D_H^{1/3} X_H\|_{L^2} \|D_H^{1/3}X_H\|_{L^2}\\
+ \bigl( \|\nabla_VX_V\|_{L^2} + \|X_V\|_{L^2}\bigr) \bigl( \|\nabla_V X_H\|_{L^2} + \|X_H\|_{L^2}\bigr) \\
+ \bigl( \|\nabla_V^2 X_H\|_{L^2} + \|\nabla_V X_H\|_{L^2} \bigr)
\bigl( \|\nabla_V^2 X_V\|_{L^2} + \|\nabla_V X_V\|_{L^2} \bigr) \Bigr),
\end{multline}
while \eqref{bisd2vv} implies
\begin{multline}\label{bisdvvv2}
\|\nabla_V^2 X_V\|_{L^2}^2 
\leq C \Bigl( \|{\cal B}X\|_{L^2} \|\nabla_V^2 X_V\|_{L^2}
+ \|\nabla_V D_H^{1/3} X_V\|_{L^2} \|D_H^{1/3}X_V\|_{L^2}\\
+  \|\nabla_V X_H\|_{L^2} \|\nabla_VX_V\|_{L^2} \Bigr).
\end{multline}
Combining \eqref{bisdvvh2} and \eqref{bisdvvv2}, for $\delta>0$ small enough there is $C>0$ such that
\begin{multline}\label{DVV2}
\delta \|\nabla_V^2 X_H\|_{L^2}^2 + \|\nabla_V^2 X_V\|_{L^2}^2 
\leq C \Bigl( \|{\cal B}X \|_{L^2} \|\nabla_V^2 X\|_{L^2}
+ \|\nabla_V D_H^{1/3} X\|_{L^2} \|D_H^{2/3}X\|_{L^2} \\
+ \|\nabla_V X\|_{L^2}^2 + \|X\|_{L^2}^2 \Bigr).
\end{multline}
This leads, by interpolation and \eqref{DVH13}, to
\[ \|\nabla_V^2 X\|_{L^2}^2
\leq C \Bigl( \|{\cal B}X\|_{L^2} \bigl( \|\nabla_V^2 X\|_{L^2}+ \|D_H^{2/3}X\|_{L^2}\bigr)
+ \|X\|_{H^{1/3,0}} \|X\|_{H^{2/3,0}} + \|X\|_{H^{0,3/2}}^2.\]
(Keep in mind, to appreciate which are the dominant and secondary terms: each vertical derivative is worth~1, each horizontal derivative is worth~3.)
From there,
\begeq\label{DVV2X}
\|\nabla_V^2 X\|_{L^2}^2
\leq C_\var \bigl( \|{\cal B} X\|_{L^2}^2 + \|X\|_{L^2}^2\bigr) + \var \|X\|_{H^{2/3,0}}^2, 
\endeq
where $\var$ is as small as desired (and $C_\var$ depends on $\var$).
\sm

{\bf Step~3:} Estimate the horizontal regularity of $X$ using the commutator; for that, and just as in Step~3, arrive at
\begeq\label{HXfromcom}
\|D_H^{2/3} X\|_{L^2}^2 \leq C 
\Bigl( \|\nabla_V D_H^{1/3} X\|_{L^2} \|\xi X\|_{L^2} + \|X\|_{H^{0,2}}^2\Bigr)
\endeq
(there is no ${\cal B}$ at this level, so this is the usual geodesic generator $\xi$ as if acting on functions). So with \eqref{DVH13},
\begin{multline} \label{HXfromcom2}
\|D_H^{2/3}X\|_{L^2}^2 \\
\leq C \biggl[ \Bigl( \|D_H^{2/3}X\|_{L^2}^{1/2} \|{\cal B}X\|_{L^2}^{1/2} + \|X\|_{H^{0,3/2}} + \|D_H^{1/3}X\|_{L^2} \Bigr) 
\Bigl( \|LX\|_{L^2} + \| (\xi-\Xi)X\|_{L^2} + \|\Delta_V^\mu X\|_{L^2} + \|X\|_{L^2} \Bigr)\\
+ \|X\|_{H^{0,2}}^2 \biggr].
\end{multline}
Using \eqref{+c},
\begin{multline} \label{HXfromcom3}
\|D_H^{2/3}X\|_{L^2}^2 
\leq C \biggl[ \Bigl( \|D_H^{2/3}X\|_{L^2}^{1/2} \|{\cal B}X\|_{L^2}^{1/2} + \|X\|_{H^{0,3/2}} + \|D_H^{1/3}X\|_{L^2} \Bigr) \\
\bigl( \|{\cal B}X\|_{L^2} + \| X\|_{H^{0,2}}\bigr) + \|X\|_{H^{0,2}}^2 \biggr].
\end{multline}
So
\begeq\label{bisDH23}
\|D_H^{2/3} X\|_{L^2}^2 \leq C \bigl ( \|{\cal B}X\|_{L^2}^2 + \|X\|_{H^{0,2}}^2\bigr).
\endeq
(This estimate is parent to \eqref{DVV2X}, and both summarise relations between the three natural operators of weight~2: $\Delta_V^\mu$, $D_H^{2/3}$ and ${\cal B}$.) Now, combining \eqref{bisDH23} with \eqref{DVV2X}, and taking advantage of the smallness of $\var$ in \eqref{DVV2X}, eventually
\[ \|\nabla_V^2 X\|_{L^2}^2 + \|D_H^{2/3}X\|_{L^2}^2
\leq C\bigl( \|{\cal B}X\|_{L^2}^2 + \|X\|_{L^2}^2\bigr), \]
which concludes the proof.
\end{proof}

\begin{proof}[Proof of Theorem \ref{thmalabouchut}(iii)]
The proof follows exactly the same pattern, now with $D_V = \sqrt{-\Delta_V}$, and integration against the Liouville measure; but instead of using Proposition \ref{propxiboundnomom}, keep using moments to estimate extra terms coming from horizontal commutators. Then by interpolation reduce to $L^2$ moments, using Proposition \ref{propinterp}. This comes at the price of slightly deteriorating the suboptimal exponents appearing in the error terms: Given $\alpha,\beta,\kappa$, choose $\theta>0$ arbitrarily small, then $\alpha_1 = \alpha/(1-\theta)$, $\beta_1=\beta/(1-\theta)$, $\sigma=\kappa/\theta$: then \eqref{interpmoments} becomes
\[ \| f\|_{H^{\alpha,\beta}_\kappa} \leq C \, \|f\|_{H^{\alpha_1,\beta_1}}^{1-\theta} \|f\|_{L^2_\sigma}^\theta.\]
 (So the loss of regularity is arbitrarily small if moments are high enough.) About the impact of the term $v\cdot\nabla_V$, note that it is of order~0 in $L^2$-type estimates, because
\[ \iint (v\cdot\nabla_V f) f = \frac12 \iint v\cdot \nabla_V (f^2) = - \frac{n}2 \iint f^2.\]
Using this it is not difficult to adapt the proof of Part (i). 
\end{proof}

\bibnotes

For local estimates, maximal hypoellipticity is a classical topic ever since H\"ormander's work; but global estimates are much more rare. Lebeau \cite{lebeau:FP2:07} obtained the first global such result for Gaussian measure, through a subtle (global) microlocal analysis. A simpler method was developed by Nier--Sang--White \cite{NSW:25}, still using pseudo-differential operators.

The more elementary scheme of Bouchut \cite{bouchut:hypoell:02} was developed to prove maximal hypoellipticity for the Kolmogorov generator (and thus the simplest kinetic Fokker--Planck operator) in $L^2(\R^n)$. I adapted his method to the Riemannian context for the purpose of this course, with some extra terms coming from the commutators of $\xi$ or $\Xi$ with horizontal or vertical (fractional) differentiations. This provides, to my knowledge, the most elementary method so far to Lebeau's result. Part (iii) of Theorem \ref{thmalabouchut} is new as far as I~know.

Albritton--Armstrong--Mourrat--Novack~\cite{AAMN:KFP:24} prove the following inequality, in flat geometry and under appropriate boundary conditions, which I skip here:
\[ \|f\|_{H^{\alpha,0}} \leq C \bigl( \|\nabla_V h\|_{L^2} + \|\xi h\|_{H^{0,-1}}\bigr) \qquad \forall \alpha< \frac13. \]
It is the same as the plausible \eqref{conjpoinc}, except that the exponent $1/3$ is not included, and there is no need of $\|f\|_{L^2}$ term in the right hand side; actually these authors use the combination $\|\nabla_V h\|_{L^2} + \|\xi h\|_{H^{0,-1}}$ to establish both hypoellipticity and hypocoercivity.

In the context of elliptic equations, the use of maximal estimates to develop general theories via perturbation is addressed for instance in the classical treatise by Gilbarg--Trudinger \cite{GT:elliptic:book}.

\section{Localisation in velocity variable} \label{seclocal}

So far the analysis was centred on stationary operators; now it is time to start the study of time-evolution estimates.

As is classical in kinetic theory, the first concern is about the kinetic localisation of solutions: 
showing that they ``decay well'' at large velocities. For the $L^1$ problem, there is the conservation of mass, 
and one can establish weighted $L^1$ estimates: polynomial moments, exponential moments, or square exponential moments (obviously the best than one can hope for, since the stationary measure is Gaussian).
For the $L^2$ (scalar) problem, there is decay of $L^2$ norm, and one can establish weighted $L^2$ estimates: again, with weights that are polynomial or exponential or square exponential. For vector fields, at this stage, estimates will be more partial and just assert the domination of a weighted $L^2$ norm by a plain $L^2$ norm, ``as if large velocities did not matter''. 

The next statement covers the scalar case. Recall that $\<r\> = \sqrt{1+r^2}$.

\begin{Thm}[localisation for functions in $L^1$ and $L^2$] \label{thmloc}
Let $\vphi:\R_+\to\R_+$ be a smooth positive function such that $\vphi'(0)=0$, and
\begeq\label{condphiloc}
\liminf_{r\to\infty} \left(\frac{ r\vphi'(r) - (n-1)\vphi'(r) - \vphi''(r)}{\vphi(r)}\right) > 0.
\endeq
Such is the case in particular if $\vphi(r) = r^{2k}$ ($k\in\N$), or $\vphi(r) = e^{a\<r\>}$ ($a>0$), or $\vphi(r) = e^{\beta r^2/2}$ ($0<\beta<1$). 
Then

(i) if $f=f(t,x,v)$ solves $\pa_t f + {\cal L}f =0$ and $\int f(0,x,v)\,dx\,dv=1$, then $\int f(t,x,v)\,dx\,dv =1$ for all $t> 0$; moreover there is $A>0$ such that for all $t>0$
\begeq\label{boundL1}
\iint_{\TM} f(t,x,v)\,\vphi(|v|)\,dx\,dv \leq \max \left( \iint_{\TM} f(0,x,v)\,\vphi(|v|)\,dx\,dv , A \right).
\endeq

(ii) if $h=h(t,x,v)$ solves $\pa_t h + L h =0$ then $\int h(t,x,v)^2\,d\mu(x,v) \leq \int h(0,x,v)^2\,d\mu(x,v)$ for all $t>0$; moreover there is $A>0$ such that for all $t>0$
\begeq\label{boundL2}
\iint_{\TM} h(t,x,v)^2\,\vphi(|v|)\,d\mu(x,v) \leq \max \left( \iint_{\TM} h(0,x,v)^2\,\vphi(|v|)\,d\mu(x,v) , A \iint_{\TM} h(0,x,v)^2\,d\mu(x,v) \right).
\endeq
\end{Thm}

\begin{Rks} \label{rkL1loc}
\begin{itemize}
\item[(i)] The $L^1$ estimate works just the same when the initial condition is a finite measure, rather than a distribution function.

\item[(ii)] The $L^1$ condition reads $L \vphi \geq K \vphi - C$ for some constants $K,C>0$; this is the same as the classical Meyn--Tweedie localisation condition.
\end{itemize}
\end{Rks}

\begin{proof}[Proof of Theorem \ref{thmloc}]
The proof will be in the form of a priori estimates.
First, from the assumption, writing $\vphi(v)$ for $\vphi(|v|)$ and $r$ for $|v|$,
\begeq\label{DphiKC} \Delta_V^\mu\vphi = \Delta_V \vphi - v\cdot\nabla_V \vphi = \vphi''(r) + (n-1) \vphi'(r) - r\vphi'(r) \leq -K \vphi(r) + C, \endeq
for some constants $K,C>0$.

Then start with (i). From $\pa_t f + {\cal L}f=0$ follows
\begin{align*}
\frac{d}{dt} \iint f \vphi 
& = - \iint \bigl( \xi f - \Delta_V f - v\cdot\nabla_V f - n f\bigr)\vphi\\
& = \iint f \bigl( \xi\vphi + \Delta_V \vphi - v\cdot\nabla_V\vphi\bigr).
\end{align*}
From \eqref{basicinv}, $\xi\vphi =0$, so $(d/dt) \int f\vphi = \int f(\Delta_V^\mu\vphi)$. Of course this implies mass conservation. Also, inserting \eqref{DphiKC} and $\int f =1$,
\begeq
\frac{d}{dt} \iint f\vphi \leq -K \iint f\vphi + C
\endeq
so $(d/dt) \int f\vphi \leq 0$ if $\int f\vphi \geq C/K$. This proves (i).
\sm

Now for (ii). From $\pa_t h + Lh = 0$ follows
\[
\pa_t \left(\frac{h^2}{2}\right) + L \left(\frac{h^2}{2}\right) - \Gamma(h,h) = 0\]
where $\Gamma(h,h) = -(1/2) (L h^2- 2 h Lh) = |\nabla_V h|^2$. (This identity, expressing the commutator between the diffusion operator $L$ and the quadratic nonlinearity, is the {\bf Gamma formula for $L$}.) Upon integration against $\vphi$,
\[\frac{d}{dt} \iint h^2 \vphi\,d\mu = - \iint h^2 (L\vphi)\,d\mu- \iint |\nabla_V h|^2\,\vphi\,d\mu.\]
When $\vphi =1$ this of course implies that $\iint h^2\,d\mu$ is nonincreasing.
If on the other hand $\vphi$ satisfies the assumption of the theorem, 
\begin{align*} \frac{d}{dt} \iint h^2 \vphi\,d\mu
& \leq - \iint |\nabla_V h|^2\,\vphi\,d\mu -K \iint h^2\vphi + C \iint h^2 \\
& \leq - \iint |\nabla_V h|^2\,\vphi\,d\mu -K \iint h^2\vphi + C \iint h_0^2
\end{align*}
and (ii) follows from the same argument as (i).
\end{proof}

Now consider the vector case. Recall that ${\cal B} = \Xi - \Delta_V + v\cdot\nabla_V + P_V$.

\begin{Thm}[localisation for the vector diffusion] \label{thmlocalvect}
Let $X=X(t,x,v)$ solve $\pa_tX + {\cal B}X = 0$ with $X(0,\cdot) = X_0$; and let $\beta\in (0,1)$. Then there is $C>0$ such that for all $t>0$, 
\begeq\label{XexpCt} \iint_{\TM} |X(t,x,v)|^2\,d\mu(x,v) \leq e^{Ct} \iint_{\TM} |X_0|^2\,d\mu \endeq
and
\begin{multline}\label{XXvec}
\iint_{\TM} |X(t,x,v,)|^2 e^{\beta |v|^2}\,d\mu(x,v) \\ \leq \max \left( \iint_{\TM} |X_0 |^2\,e^{\beta|v|^2}\, d\mu, C \iint_{\TM} |X_0|^2\, d\mu, C \iint_{\TM} |X(t,x,v)|^2\,d\mu \right).
\end{multline}
\end{Thm}

\begin{Rk} The short version of \eqref{XXvec} is that the ratio of $\int |X|^2 e^{\beta |v|^2}\,d\mu$ to $\int |X|^2\,d\mu$ remains bounded with time, provided that it is initially finite. Of course, this combined with \eqref{XexpCt} implies that $\int |X|^2 e^{\beta |v|^2}\,d\mu$ grows at most exponentially in time, but as far as localisation is concerned, estimate \eqref{XXvec} gives more information.
\end{Rk}

\begin{proof}[Proof of Theorem \ref{thmlocalvect}]
Let $\vphi = e^{\beta |v|^2}$, viewed as a function of $v$ or $|v|=r$. Then
\[ L\vphi = \bigl( -\beta^2 r^2 - \beta - (n-1)\beta r + \beta r^2\bigr ) \vphi.\]
So for any $K\in (0,\beta (1-\beta))$, there is $C>0$ such that
\begeq\label{Lvphir2} L\vphi \geq (K r^2 - C)\vphi. \endeq

Repeat the $L^2$ computation of Theorem \ref{thmloc} (ii), separating the horizontal and vertical components, and integrating againt $\vphi$. This yields:
\begin{multline}\label{L2phi1}
\frac{d}{2dt} \iint |X_H|^2\,\vphi\,d\mu 
= - \iint |\nabla_V X_H|^2\,\vphi\,d\mu - \iint |X_H|^2 (L\vphi)\,d\mu -  \iint |X_H|^2\,\vphi\,d\mu\\
- \iint \bigl\< (\Xi-\xi) X_H,X_H\bigr\>\,\vphi\,d\mu 
\end{multline}
\begin{multline}\label{L2phi2}
\frac{d}{2dt} \iint |X_V|^2\,\vphi\,d\mu 
= - \iint |\nabla_V X_V|^2\,\vphi\,d\mu - \iint |X_V|^2 (L\vphi)\,d\mu -  \iint |X_V|^2\,\vphi\,d\mu\\
- \iint \bigl\< (\Xi-\xi) X_V,X_V\bigr\>\,\vphi\,d\mu 
\end{multline}
Recall from \eqref{Xiexpl} that
\[ \Bigl| \bigl\< (\Xi-\xi) X_H,X_H\bigr\> \Bigr| \leq C |v|^2 |X_V|\,|X_H|,\]
\[ \Bigl| \bigl\< (\Xi-\xi) X_V,X_V\bigr\> \Bigr| \leq  |X_V|\,|X_H|,\]
for some constant $C>0$. Insert those bounds as well as \eqref{Lvphir2} into \eqref{L2phi1}--\eqref{L2phi2}, to get
\begin{multline}\label{L2phi3}
\frac{d}{2dt} \iint |X_H|^2\,\vphi\,d\mu 
\leq - \iint |\nabla_V X_H|^2\,\vphi\,d\mu - K \iint |X_H|^2 |v|^2 \,\vphi \,d\mu + C \iint |X_V|\, |X_H|\, |v|^2\,\vphi\,d\mu\\
+ C \iint |X|^2\,\vphi\,d\mu
\end{multline}
\begeq\label{L2phi4}
\frac{d}{2dt} \iint |X_V|^2\,\vphi\,d\mu 
\leq - \iint |\nabla_V X_V|^2\,\vphi\,d\mu - K \iint |X_V|^2 |v|^2\,\vphi\,d\mu + C  \iint |X|^2\,\vphi\,d\mu.
\endeq

Replacing for a moment $\vphi$ by~1 (or $\beta$ by~0), this yields
\begeq\label{L211}
\frac{d}{2dt} \iint |X_H|^2\,d\mu 
\leq - \iint |\nabla_V X_H|^2\, d\mu  + C \iint |X_V|\, |X_H|\,d\mu
+ C \iint |X|^2\,d\mu
\endeq
\begeq\label{L212}
\frac{d}{2dt} \iint |X_V|^2\,d\mu
\leq - \iint |\nabla_V X_V|^2\,d\mu + C  \iint |X|^2\,d\mu.
\endeq
Using \eqref{ineqregv}, \eqref{L211} implies
\begin{align*}
\frac{d}{2dt} \iint |X_H|^2\,d\mu 
& \leq -\frac12 \iint |\nabla_V X_H|^2\, d\mu  + C \iint |X_H|\, |v|^2\,d\mu
+ C \iint |X|^2\,\vphi\,d\mu\\
& \leq -\frac12 \iint |\nabla_V X_H|^2\, d\mu  + C \iint |\nabla_V X_H|\, \,d\mu
+ C \iint |X|^2\,\vphi\,d\mu;
\end{align*}
combining this with \eqref{L212}, for $\delta>0$ small enough,
\begeq\label{L215}
\frac{d}{dt} \iint (\delta |X_H|^2 + |X_V|^2)\,d\mu \leq C \iint |X|^2\,d\mu.
\endeq
So the quantity
\begeq\label{calN}
{\cal N}(t) = \iint (\delta |X_H|^2 + |X_V|^2)\,d\mu 
\endeq
satisfies 
\begeq\label{dNdt}
\frac{d{\cal N}}{dt} \leq C{\cal N}
\endeq
for some constant $C>0$, and the bound \eqref{XexpCt} follows.

Now back to the estimate with $\vphi$, in the same way, for $\delta>0$ we have, from \eqref{L2phi3}--\eqref{L2phi4},
\begin{multline*}
\frac{d}{2dt} 
 \iint (\delta |X_H|^2 + |X_V|^2)\,\vphi\,d\mu 
\leq - \iint (\delta |\nabla_V X_H|^2 + |\nabla_V X_V|^2)\,d\mu \\
- \iint \bigl( \delta K |X_H|^2 - \delta R |X_V|\,|X_V| + |X_V|^2\bigr)|v|^2\,d\mu 
+C  \iint ( |X_H|^2 + |X_V|^2)\,\vphi\,d\mu.
\end{multline*}
Choosing $\delta>0$ small enough, letting
\begeq\label{calE}
{\cal E} = \iint (\delta |X_H|^2 + |X_V|^2)\,\vphi\,d\mu 
\endeq
we arrive at
\begeq\label{eqdiffcalE}
\frac{d{\cal E}}{dt} \leq -K \iint (\delta |X_H|^2 + |X_V|^2)\, |v|^2\,\vphi\,d\mu + C \iint (\delta |X_H|^2 + |X_V|^2) \,\vphi\,d\mu.
\endeq
if $K,C>0$ are well-chosen constants.
Next, for any $\var>0$,
\[ \vphi(v) \leq \var \vphi(v) |v|^2 + \vphi\left(\frac1{\sqrt{\var}}\right), \]
so upon changing $K,C$ \eqref{eqdiffcalE} becomes
\begeq
\frac{d{\cal E}}{dt} \leq -K \iint (\delta |X_H|^2 + |X_V|^2)\, |v|^2\,\vphi\,d\mu + C \iint (\delta |X_H|^2 + |X_V|^2) \,d\mu.
\endeq

By Jensen's inequality, applied with the probability measure $(\delta |X_H|^2 + |X_V|^2)\,\mu/Z$ ($Z$ the normalising constant) and the convex function $s\log s$,
\begin{multline*} \beta \iint (\delta |X_H|^2 + |X_V|^2)\, |v|^2\,e^{\beta |v|^2}\,d\mu \\ \geq 
\left( \iint (\delta |X_H|^2 + |X_V|^2)\, e^{\beta |v|^2}\,d\mu \right) \log \left(
\frac{\dps \iint (\delta |X_H|^2 + |X_V|^2)\, e^{\beta |v|^2}\,d\mu}{\dps \iint (\delta |X_H|^2 + |X_V|^2)\,d\mu}\right).
\end{multline*}
Thus recalling \eqref{calN},
\begeq\label{systEN}
\frac{d{\cal E}}{dt} \leq -K\, {\cal E} \log \left(\frac{\cal E}{\cal N}\right) + C {\cal N}.
\endeq
In particular,
\begin{align*}
 \frac{d}{dt} (\log {\cal E}) &\leq -K \log \left(\frac{{\cal E}}{{\cal N}}\right) + C\, \frac{{\cal N}}{{\cal E}} \\
& \leq -K \log \left(\frac{{\cal E}}{{\cal N}}\right) + C.
\end{align*}
Combining this with \eqref{dNdt}, in the form $|(d/dt)\log{\cal N}| \leq C$, 
\begeq\label{dlogE}
 \frac{d}{dt} \log \left(\frac{{\cal E}}{{\cal N}}\right) \leq -K \log \left(\frac{{\cal E}}{{\cal N}}\right) + C.
 \endeq
In particular, ${\cal E}/{\cal N}$ becomes decreasing if it ever goes above $\exp (C/K)$, which proves that it is bounded, uniformly in time and concludes the proof.
\end{proof}

The next observation is specific to the $L^1$ problem and has notable consequences: it shows that the square exponential moment property is propagated forward and backward according to a certain differential equation. I shall use the notation
\begeq\label{Mbeta}
M_\beta(f) = \iint_{\TM} f(x,v)\,e^{\beta \frac{|v|^2}2}\,dx\,dv.
\endeq

\begin{Prop}[Two-sided propagation of square exponential moments] \label{2sideexp}
Let $M$ be compact and let $f(t,v,x)$ be a time-dependent probability density of $\TM$ satisfying \eqref{eqf}, and let $f_0=f(0,\cdot)$. Let further $\beta_0\in (0,1)$ such that $M_{\beta_0}(f_0)<+\infty$. Let further $\beta(t)$ be the unique solution to the differential equation
\begeq\label{dotbeta} \dot{\beta}(t) = -2 \beta(t) (\beta(t)-1), \qquad \beta(0) = \beta_0. \endeq
Explicitly,
\begeq\label{betaexpl}
\frac1{\beta(t)} = 1 + \left(\frac1{\beta_0}-1\right)\, e^{-2t}.
\endeq
Then 
\[ e^{n\beta_0t} \leq \frac{M_{\beta(t)}(f(t,\cdot))}{M_{\beta_0}(f_0)} \leq e^{n t}. \]
In particular, 

(i) If $t\geq 0$ is given, then $f(t,\cdot)$ has a square exponential moment if and only if $f(0,\cdot)$ does;

(ii) If $M_{\beta_0}(f_0) <+\infty$ for some $\alpha>0$, then for any $\alpha<1$ one can find $t_0\geq 0$, only depending on $\beta_0$, such that $M_{\alpha}(f(t,\cdot))<+\infty$ for all $t\geq t_0$.
\end{Prop}

\begin{Cor}[Separation of the $L^1$ and $L^2$ regimes] \label{corL1L2}
Let $f_0\in L^1(dx\,dv)$ be such that $\int f(0,x,v) e^{\beta_0|v|^2/2}\,dx\,dv = +\infty$ for some $\beta_0<1$, and let $h(t,\cdot)$ be the density of $f(t,\cdot)\,dx\,dv$ with respect to $\mu$; then $h$ never belongs in $L^2(\mu)$.
\end{Cor}

\begin{Rk} Note that if $\beta$ solves \eqref{dotbeta}, then $\beta$ increases from $\beta_0$ to~1, never reaches~1, and the time it takes to reach any given $\alpha\in (0,1)$ diverges to $\infty$ as $\beta_0\to 0$.
\end{Rk}

\begin{proof}[Proof of Proposition \ref{2sideexp}]
Let $\beta=\beta(t)$ be a smooth positive function of $t$. Then as in the proof of Theorem \ref{thmloc},
\begin{multline}
\frac{d}{dt} \iint e^{\beta(t)\frac{|v|^2}2}\, f(t,x,v)\,dx\,dv \\
= \iint f \bigl( \Delta_V^\mu e^{\beta(t) \frac{|v|^2}2} + \dot{\beta}(t) \frac{|v|^2}2 e^{\beta\frac{|v|^2}2}\bigr)\,dx\,dv \\
= \iint \left(n\beta + \bigl[ 2\beta(\beta-1) + \dot{\beta}\bigr] \frac{|v|^2}{2} \right)\, e^{\beta\frac{|v|^2}2}\, f(t,x,v)\,dx\,dv.
\end{multline}
So if $\beta$ satisfies \eqref{dotbeta}, this simplifies to
\begeq\label{ddtMbeta}
\frac{d}{dt} \iint e^{\beta(t)\frac{|v|^2}2} f(t,x,v)\,dx\,dv = n\beta(t)  \iint e^{\beta(t)\frac{|v|^2}2} f(t,x,v)\,dx\,dv,
\endeq
or just
\[ \frac{d}{dt} \log\left( \iint e^{\beta(t)\frac{|v|^2}2} f(t,x,v)\,dx\,dv\right) =  n\beta (t) \in [n\beta_0, n].\]
The rest of the Proposition follows easily. (An explicit formula for $M_\beta$ can also be computed.)
\end{proof}

\begin{proof}[Proof of Corollary \ref{corL1L2}]
Let $f_0$ satisfy the assumption of the Corollary, and let $t>0$, so $\int f(t,x,v) e^{\beta_t |v|^2/2}\,dx\,dv = +\infty$. But by Cauchy--Schwarz, 
\[ \iint f(t,x,v)\, e^{\beta_t|v|^2/2}\,dx\,dv\leq \sqrt{ \iint f(t,x,v)^2\, e^{\frac{|v|^2}2}\,dx\,dv} \sqrt{\iint e^{(2\beta_t-1)|v|^2/2}\,dx\,dv},\]
and the final integral on the right-hand side is finite since $\beta_t<1$. So if the integral on the left hand side is infinite, the first integral on the right hand side has to be infinite too.
\end{proof}

\bibnotes

Moment estimates for evolution equations, both in $L^1$ and $L^2$, are classical in kinetic theory; see for instance \cite{DV:landau:1,TV:slow:00} for spatially homogeneous variants of the Fokker--Planck equation. The more subtle estimate for the Bismutian evolution equation (Theorem \ref{thmlocalvect}) was devised for the present notes.

Meyn and Tweedie developed and improved the methods for convergence of positive recurrent Markov processes in discrete or continuous phase space, discrete or continuous time~\cite{meyntweedie:stability:92,meyntweedie:stability:93,meyntweedie:geometric:94,meyntweedie:MCbook}. A review is provided by Ca\~{n}izo and Mischler \cite{canizomischler:harris:23}, including the history of the field and optimised proofs and results.

\section{Global regularisation} \label{secreg}

Now the focus will be on regularity for the evolution equation, globally in phase space and in time, also for nonsmooth initial data. Localisation is needed and various assumptions can be made in this respect. Without attempting at exhaustivity, I shall consider three such settings: for the equation in $L^2(\mu)$, the initial datum $h_0$ will be plainly assumed to belong in $L^2$; for the equation in $L^1$, the initial distribution $f_0$ will have all of its moments finite; for the equation on vector fields, the initial vector field $X_0$ will be assumed to have a finite square-exponential moment (which is true for instance if $X_0$ has at most polynomial or exponential growth in $v$).

\begin{Thm}[global regularisation for geometric kinetic Fokker--Planck] \label{thmregul}
Let $M$ be a smooth compact manifold.
\sm

(i) Let $h=h(t,x,v)$ solve $\pa_t h + Lh =0$ starting from $h_0\in L^2(\mu)$. Then $h$ is of class $C^\infty$ for positive times, and for any $\alpha>0$ there is $C>0$, depending only on $M$ and $\alpha$, such that for any $t>0$,
\begeq\label{Halpha}
\|h(t,\cdot)\|_{H^{0,\alpha}(\mu)} + \|h(t,\cdot)\|_{H^{\alpha/3},0(\mu)} \leq C\, \max \left(\frac1{t^{\alpha/2}},1\right) \, \|h_0\|_{L^2(\mu)}. 
\endeq
More generally, for any $\alpha,\beta>0$,
\begeq\label{Halphabeta}
\|h(t,\cdot)\|_{H^{\alpha,\beta}(\mu)} \leq C\, \max \left(\frac1{t^{(3\alpha+\beta)/2}},1\right) \, \|h_0\|_{L^2(\mu)}. 
\endeq
Moreover, if $h_0\in L^2_\kappa$ for all $\kappa>0$, then for all $\alpha,\beta,\kappa>0$,
\begeq\label{Halphabetakappa} \| h(t,\cdot)\|_{H^{\alpha,\beta}_\kappa} = O (1+t^{-r}) \qquad \forall r> (3\alpha+\beta)/2.\endeq
\sm

(ii) Let $X=X(t,x,v)$ solve $\pa_t X + {\cal B} X =0$, starting from $X_0 \in L^2(\mu^{1+\theta})$ for some $\theta>0$. Then $X$ is of class $C^\infty$ for positive times and for any $\alpha>0$ there is $C=C(M,\alpha,\theta)>0$, such that for any $t>0$,
\begeq\label{XHalpha}
\|X(t,\cdot)\|_{H^{0,\alpha}(\mu)} + \|X(t,\cdot)\|_{H^{\alpha/3,0}(\mu)} \leq Ce^{Ct}\, \max \left(\frac1{t^{\alpha/2}},1\right) \, \|X_0\|_{L^2(\mu^{1+\theta})}. 
\endeq
Likewise, for any $\alpha,\beta>0$ and $\kappa>0$, for any $r>(3\alpha+\beta)/2$, there is $C=C(\alpha,\beta,\kappa,\theta,r)$ such that
\begeq\label{XHalphakappa}
\|X(t,\cdot)\|_{H^{\alpha,\beta}_\kappa(\mu)} \leq \frac{Ce^{Ct}}{t^r}  \, \|X_0\|_{L^2(\mu^{1+\theta})}. 
\endeq
\sm

(iii) Let $f=f(t,x,v)$ solve $\pa_t f + {\cal L} f =0$, starting from a nonnegative distribution $f_0 \in L^1_\infty(dx\,dv)$. (That is, $f_0$ has finite moments of any order.) Then $f$ is of class $C^\infty$ for positive times, and for any $\alpha>0$ and $\gamma> \alpha+5n/2$ there are $\sigma=\sigma(\alpha,n,\gamma)>0$ and $C=C(M,\alpha,\gamma) >0$, such that for any $t>0$,
\begeq\label{fHalpha}
\|f(t,\cdot)\|_{H^{0,\alpha}(\mu)} + \|h(t,\cdot)\|_{H^{\alpha/3,0}(\mu)} \leq C\, \max \left(\frac1{t^{\gamma/2}},1\right) \, \|f_0\|_{L^1_\sigma}. 
\endeq
Moreover, if $f_0\in L^1_\infty$ then $f(t,\cdot)$ belongs in Schwartz class for all $t>0$, and for any $k,\ell,s\in\N$ there are $C,\sigma,r>0$ such that for all $t>0$ and $(x,v)\in \TM$
\begeq\label{schwartz} |\nabla^k_H \nabla^\ell_V f  (x,v) | \leq C (1+t^{-r}) \|f\|_{L^1_\sigma}.\endeq
\sm

(iv) With the same notation as in (iii), if moreover
\[ \iint f_0(x,v)\, e^{\beta_0\frac{|v|^2}{2}} \,dx\,dv < +\infty \]
for some $\beta_0 \in (0,1)$,
then $f(t,\cdot)$ satisfies a pointwise Gaussian upper bound for all $t>0$, and for any $\beta <\beta_0$, for any $k,\ell \in\N$ there are $C,r>0$ such that for all $t>0$ and $(x,v)\in \TM$
\begeq\label{fuppergaussian}
|\nabla^k_H \nabla^\ell_V f | \leq C (1+t^{-r}) \left( \iint f_0(x,v) e^{\beta \frac{|v|^2}{2}} \,dx\,dv \right) \, e^{-\beta\frac{|v|^2}{2}}.
\endeq
\end{Thm}

\begin{Rks}  \label{fromtregtomaxreg}
\begin{itemize}
\item[(i)] Let $S=I+ D_V^2 + D_H^{2/3}$, $L= -\Delta_V^\mu + \xi$;
Theorem \ref{thmregul} (i) implies
\begeq\label{SetL}
\|S\,e^{-tL}\|_{L^2(\mu)\to L^2(\mu)} = O(t^{-1})\qquad 0<t\leq 1.
\endeq
{\em If $L$ was symmetric}, this would amount to 
\begeq\label{te-tL} t e^{-tL}\leq C\,S^{-1}  \endeq
for some constant $C>0$. {\em If moreover $S$ and $L$ would commute},
one could codiagonalise $L$ and $S^{-1}$, and optimise on eigenvalues in
\eqref{te-tL} to get $L^{-1} \leq C\,S^{-1}$; in this way one would recover estimate \eqref{maxhypo} of Theorem \ref{thmalabouchut}.
Conversely, under the same assumptions of symmetry and commutativity, it is easy to see that
\eqref{maxhypo} would imply \eqref{SetL}. To summarize: If we had $L^*=L$ and $[L,S]=0$,
then the two bounds \eqref{SetL} and \eqref{maxhypo} would follow from one another,
without any further information on the structure of $L$ and $S$.
In the present situation both properties (the symmetry and the commutativity)
are ``very false'', so it is not clear whether there is a direct way
from \eqref{SetL} to \eqref{maxhypo}, or conversely.
\sm

\item[(ii)] Still a slightly diminished version of \eqref{maxhypo} can be obtained in this way.
Theorem~\ref{thmregul}(i) shows that $e^{-tL}$ maps $L^2(\mu)$ into $H^{0,\alpha}(\mu)\cap H^{\alpha/3,0}(\mu)$ like $O(t^{-\alpha/2})$, and for any $\alpha<2$ that estimate is integrable as a function of $t\in (0,1)$. 
Since $\int_0^1 e^{-t(1+L)}\,dt$ is a parametrix for $(I+L)^{-1}$ (that is, both coincide up to a $C^\infty$-smoothing operator), it follows
\begeq\label{sousopthypo}
\|h\|_{H^{0,\alpha}(\mu)}  + \|h\|_{H^{\alpha/3,0}(\mu)} \leq C \bigl(\|h\|_{L^2(\mu)} + \|Lh\|_{L^2(\mu)}\bigr), \qquad \forall \alpha<2.
\endeq
(One can also get the more precise variant in which the default of optimality is just a power of a logarithm of derivative.)
I don't know if a refined analysis allows to catch the optimal exponent $\alpha=2$, providing an alternative road to \eqref{maxhypo}. But this line and reasoning, added to Remark (i) above, demonstrate the close connection between the exponents in Theorem \ref{maxhypo} and those in Theorem \ref{thmregul}, in the $L^2$ setting. A similar reasoning applies for the $L^1$ equation with the Liouville reference measure, except that now constants may depend on high-order moments. 
\sm

\item[(iii)] Theorem~\ref{thmregul}(i) is only stated in positive regularity, but it may be that a precise tracking of estimates allows to recover regularisation for initial data in negative regularity, by decomposing such an initial datum into a weighted sum of smooth functions (\`a la Fourier, or \`a la Littlewood--Paley), applying regularisation for each term. One of the motivations to handle negative regularity would be to connect with Albritton--Armstrong--Mourrat--Novack's ``H\"ormander inequality'', via the same reasoning as in Remark (ii) above.
\sm

\item[(iv)] In Statement (ii) of Theorem \ref{thmregul}, the presumably optimal $\gamma$ should be $\alpha+2n$, rather than $\alpha+5n/2$. I~shall precisely establish $3m+2n+m/(2n)$ for $\alpha=3m$.
\end{itemize}
\end{Rks}

Now turn to the proof of Theorem \ref{thmregul}. Mixed estimates like \eqref{Halphabeta} follow from pure ones like \eqref{Halpha} upon use of Proposition~\ref{propmixinterp} with $\alpha/3+\beta = \beta'= \alpha'/3$. So it suffices to establish pure vertical and pure horizontal estimates. By interpolation, it suffices to treat integer values of $\alpha,\beta$, that is, to focus on classical derivatives of arbitrary order. 

The arguments presented below are somewhat involved but rather systematic, and they avoid any use of microlocal analysis, pseudodifferential estimates or even fractional derivatives. For pedagogical reasons, I shall first present first-order estimates in the $L^2$ setting, then higher order estimates in $L^2$, then turn to the vector-valued case; finally I shall consider high-order $L^1$ estimates.

\begin{proof}[Proof of Theorem \ref{thmregul}(i), first order estimates]
The equation is $\pa_t h + Lh=0$ and the goal is the a priori estimate
\begeq\label{apriorifirst}
\min(t^{1/2},1) \|\nabla_V h\|_{L^2(\mu)} + \min(t^{3/2},1) \|\nabla_H h\|_{L^2(\mu)} \leq C \, \|h_0\|_{L^2(\mu)}.
\endeq
So let
\begeq\label{Fth} {\cal F}(t,h) = \iint h^2 + at \iint |\nabla_V h|^2 + 2bt^2 \iint \<\nabla_V h, \nabla_H h\> + c t^3 \iint |\nabla_H h|^2,
\endeq
where the measure $\mu$ is implicit, and $a,b,c>0$, $b<\sqrt{ac}$, so that ${\cal F}(t,h)$ is bounded from above and below by multiples of $\int h^2 + t \int |\nabla_V h|^2 + t^3 \int |\nabla_H h|^2$. If it can be shown that, for appropriate $a,b,c>0$,
\begeq\label{ddtFth}
\frac{d}{dt} {\cal F}(t,h) \leq 0 \qquad 0\leq t\leq 1
\endeq
then for $0\leq t\leq 1$ this leads to $t\|\nabla_V h\|^2_{L^2(\mu)} + t^3 \|\nabla_H h\|^2_{L^2(\mu)} \leq C \, \|h_0\|^2_{L^2(\mu)}$, and for $t>1$, changing the origin of times, $\|\nabla_V h(t)\|^2_{L^2(\mu)} + \|\nabla_H h(t)\|^2_{L^2(\mu)} \leq C \|h(t-1)\|_{L^2(\mu)}^2\leq C\|h_0\|_{L^2(\mu)}^2$. So it all boils down to \eqref{ddtFth}. Here are the relevant equations:
\begeq\label{3eq}
\begin{cases}
\dps \pa_t h + \xi h = \Delta_V^\mu h\\
\dps \pa_t \nabla_V h + \xi \nabla_V h + \nabla_H h = \Delta_V^\mu \nabla_V h - \nabla_V h\\
\dps \pa_t\nabla_H h + \xi\nabla_H h - [\xi, \nabla_H]h = \Delta_V^\mu \nabla_H h.
\end{cases} \endeq

To alleviate notation the reference measure $\mu$ will be implicit as well as the $L^2$ norm. First consider the diagonal elements:
\begeq\label{diag1}
\frac{d}{2\,dt} \iint h^2 = - \iint |\nabla_Vh|^2;
\endeq
\begin{align}\label{diag2}
\frac{d}{2\,dt} \iint |\nabla_Vh|^2 & = - \iint |\nabla_V^2h|^2 - \iint |\nabla_V h|^2 - \iint \<\nabla_H h, \nabla_V h\>\\ \nonumber
& \leq - \|\nabla_V^2h\|^2  + \|\nabla_V h\|\, \|\nabla_H h\|
\end{align}
\begeq \label{diag3}
\frac{d}{2\,dt} \iint |\nabla_H h|^2
= - \iint |\nabla_V \nabla_H h|^2 + \iint \< [\xi,\nabla_H] h, \nabla_H h \>;
\endeq
from equation \eqref{xiNablaHr} and Proposition \ref{propgradmom},
\begin{align*}
\iint \< [\xi,\nabla_H] h, \nabla_H h \> & \leq C \|\nabla_V h\|_{L^2_1} \, \|\nabla_H. h\|_{L^2_1}\\
& \leq C \|\nabla_V h\|_{H^{0,1}}\, \|\nabla_H h\|_{H^{0,1}}.
\end{align*}
Inserting this in \eqref{diag3},
\begeq \label{diag3'}
\frac{d}{2\,dt} \iint |\nabla_H h|^2 \leq - \iint |\nabla_V \nabla_H h|^2 + C \bigl( \|\nabla_V h\| + \|\nabla_V^2 h\|\bigr) \bigl( \|\nabla_H h\|. + \|\nabla_V\nabla_H h\|\bigr).
\endeq
Using \eqref{diag1}, \eqref{diag2}, \eqref{diag3'} and Young's inequality repeatedly, under conditions $c\ll a\ll 1$,
\begin{multline} \label{bddtdiagFth}
\frac{d}{2\,dt} 
\left( \iint h^2 + a t \iint |\nabla_Vh|^2 + c t^3 \iint |\nabla_H h|^2 \right) \\
\leq -K \left( \iint |\nabla_V h|^2 + at \iint |\nabla_V^2h|^2 + ct^3 \iint |\nabla_V \nabla_H h|^2 \right)\\
+ C \max (c,a^2) \, t^2 \iint |\nabla_H h|^2,
\end{multline}
where $C$ is a positive constant only depending on $n$ and an upper bound on the norm of the curvature tensor of $M$.

Next from \eqref{3eq} again, taking the product of the second equation with $\nabla_Hh$ and the third equation with $\nabla_Vh$,
\begin{multline*} \pa_t \<\nabla_Vh, \nabla_Hh\> + \xi \<\nabla_V h, \nabla_H h\> + |\nabla_H h|^2 =
\<\Delta_V^\mu \nabla_Vh, \nabla_H h\> - \<\nabla_Vh,\nabla_Hh\>\\ 
+ \<[\xi,\nabla_H]h, \nabla_V h\> + \<\Delta_V^\mu\nabla_H h, \nabla_Vh\>,
\end{multline*}
so
\begin{multline}\label{ddtmix}
\frac{d}{dt} \iint \<\nabla_Vh,\nabla_Hh\>
= - \iint |\nabla_Vh|^2 - \iint \< \nabla_V^2 h\, \nabla_V \nabla_H h\> - \iint \<\nabla_Vh, \nabla_H h\> \\
+ \iint \<[\xi,\nabla_H] h,\nabla_V h\> - \iint \<\nabla_V \nabla_H h, \nabla_V^2h\>.
\end{multline}
Thus
\begin{multline*}
\frac{d}{dt} \left( bt^2 \iint \<\nabla_Vh,\nabla_Hh\> \right)
\leq 2 bt \|\nabla_V h\| \, \|\nabla_H h\|- bt^2 \|\nabla_H h\|^2 + bt^2 \Bigl( \|\nabla_V^2h\|\, \|\nabla_V\nabla_H h\| \\
+ \|\nabla_V h\|\,\|\nabla_Hh\| + (\|\nabla_V h\|^2 + \|\nabla_V^2 h\|^2 ) + \|\nabla_V \nabla_H h\| \, \|\nabla_V^2 h\| \Bigr).
\end{multline*}
From Young again,
\begin{multline*}
\frac{d}{dt} \left( bt^2 \iint \<\nabla_Vh,\nabla_Hh\>  \right)
\leq - K\, bt^2 \|\nabla_H h\|^2 + C_\var \, \left[ b \|\nabla_V h\|^2 + \max \left(\frac{b^2}{c}, b\right) t \| \nabla_V^2 h\|^2\right] \\
+ \var\, c t^3 \|\nabla_V\nabla_H \|^2,
\end{multline*}
where $\var>0$ is arbitrarily small. Assuming $b\ll a\ll 1$ and $b^2/c\ll a$,
\begeq\label{extradiag}
\frac{d}{dt} \left( bt^2 \iint \<\nabla_Vh,\nabla_Hh\>  \right)
\leq -K \, bt^2 \|\nabla_H h\|^2 + \var \left( \|\nabla_V h\|^2 + a t \|\nabla_V^2 h\|^2 + c t^3  \|\nabla_V\nabla_H h \|^2\right).
\endeq

Combining \eqref{bddtdiagFth} and \eqref{extradiag}, assuming that $\max(a^2,c)\ll b$,  results in
\[ \frac{d}{dt} {\cal F}(t,h(t)) \leq - K \left( \| \nabla_V h\|^2 + a t \| \nabla^2_Vh\|^2 + b t^2 \|\nabla_H h\|^2 + c t^3 \| \nabla_V\nabla_H h\|^2\right),\qquad 0\leq t\leq 1,\]
which concludes the argument. It has been assumed that
\begeq\label{assabc}
c\ll b\ll a \ll 1, \qquad b^2\ll a c, \qquad a^2 \ll b
\endeq
(with constants dictated only by the geometry of $M$). The feasibility of \eqref{assabc} is ensured by Lemma \ref{ll} below, and this concludes the proof of \eqref{Halphabeta}.

As for \eqref{Halphabetakappa}, it follows by interpolation with a uniform bound in $L^2_s$ for large $s$, which is guaranteed by Theorem \ref{thmloc}(ii).
\end{proof}

\begin{Lem}\label{ll} 
Let $\delta>0$ and $u_0>0$ be given. Then for any $N\in\N$ there are positive numbers $u_1, u_2, \ldots, u_N$ such that
\[ \begin{cases} \forall k\in \{0,\ldots, N-1\}, \qquad
u_{k+1} \leq \delta\, u_k; \\
\forall k\in \{1,\ldots, N-1\}, \qquad u_k^2 \leq \delta\, u_{k-1}\, u_{k+1}.
\end{cases}\]
\end{Lem}

\begin{proof} Without loss of generality, assume $u_0=1$.
Set $m_0=0, m_1=1$; by induction, pick up positive
numbers $m_k$ such that
\[ m_{k+1} \in (m_k, 2 m_k - m_{k-1}).\]
The resulting sequence will be increasing and satisfy $m_k > (m_{k-1}+ m_{k+1})/2$.
Next set $u_k = \var^{m_k}$; for $\var$ small enough, the desired inequalities will be satisfied.
\end{proof}

\begin{proof}[Proof of Theorem \ref{thmregul}(i), general estimates]
To handle derivatives of order $m\geq 2$, a natural generalisation of the functional \eqref{Fth} would be
\begin{multline*}
{\cal F}_m(t,h) = \sum_{k+\ell \leq m} a_{k,\ell}\, t^{k+3\ell} \iint |\nabla_V^k \nabla_H^\ell h|^2\,d\mu\\
+ 2 \,\sum_{k+\ell\leq m, \  \ell\geq 1}
a_{k+\frac12,\, \ell-\frac12}\, t^{k+3\ell -1}
\iint \bigl\< \nabla_V^k \nabla_H^\ell h, \nabla_V^{k+1} \nabla_H^{\ell-1} h\bigr\>\,d\mu.
\end{multline*}
But the following reduced version will also work:
\begin{multline}\label{F'th}
{\cal F}^\ast_m(t,h) = \iint h^2\,d\mu  + a  t^m \iint |\nabla_V^m h|^2\,d\mu + 2 bt^{3m-1} \iint \< \nabla_V^m h, \nabla_H \nabla_V^{m-1}h \>\,d\mu \\
+ c t^{3m} \iint |\nabla_H^m h|^2\,d\mu.
\end{multline}
Note that, by Young and Proposition \ref{propmixinterp},
\begin{align*} \left| bt^{3m-1} \iint \< \nabla_V^m h, \nabla_H \nabla_V^{m-1}h \>\,d\mu \right| 
& \leq bt^{3m-1} \| \nabla_V^m h\| \, \|\nabla_V^{m-1}h\| \\
& \leq \frac{c}{4} t^{3m} \|\nabla_H^mh\|^2 + \frac{b^2}{c}\, t^{3m-2} \|\nabla_V \nabla_H^{m-1}h\|^2 \\
& \leq \frac{c}{4} t^{3m} \|\nabla_H^mh\|^2 + \frac{b^2}{c}\, t^{3m-2} \|\nabla_V^mh\|^{\frac{2}{m}} \|\nabla_H^mh\|^{\frac{2(m-1)}{m}} \\
& \leq \frac{c}{2} t^{3m} \|\nabla_H^mh\|^2 + C \left( \frac{b^{2m}}{c^{2m-1}}\right) t^m \|\nabla_V^mh\|^2,
\end{align*}
where $C$ only depends on $m$. So if $a,b,c>0$ are chosen in such a way that 
\begeq\label{abcm}
b^{2m} \ll a\, c^{2m-1}
\endeq
(in the sense that the ratio $b^{2m}/(a c^{2m-1})$ is bounded above by a small number depending only on $m$) then ${\cal F}^\ast_m(t,h)$ is bounded from above and below by constant multiples of 
\[ \|h\|^2 + t^m \|\nabla_V^mh\|^2 + t^{3m} \|\nabla_H^mh\|^2. \]
Thus it suffices to prove
\[ \frac{d}{dt} {\cal F}^\ast_m (t,h(t)) \leq 0 \qquad (0< t\leq 1)\]
(or $(d/dt) {\cal F}^\ast_m(t,h(t)) \leq C {\cal F}^\ast_m$).

To establish regularity gain, negative terms are needed in the expression of the time-derivative. The $\Delta_V^\mu$ term acting on the symmetric terms ($\|h\|^2$, $\|\nabla_V^mh\|^2$, $\|\nabla_H^mh\|^2$) will provide negative terms like $-\|\nabla_Vh\|^2$ and $-\|\nabla_V^{m+1}h\|^2$ (the second is obviously stronger, but the first one comes with a better constant, so both will be kept), as well as $-\|\nabla_V \nabla_H^m h\|^2$. But the horizontal gain will come from $\xi$ acting on the mixed term $\< \nabla_V\nabla_H^{m-1}h,\nabla_H^m h\>$; the gain will not be in the form of an improved regularity but of a better time-exponent.

Let us proceed with the computation. First commute the equation with $\nabla_V^m$:
\begeq\label{patnablaVm} \pa_t\nabla_V^m h + \xi \nabla_V^m h - \Delta_V^\mu \nabla_V^m h 
= - [\nabla_V^m,\xi]h + [\nabla_V^m, v\cdot\nabla_V]h.
\endeq
The identity
\[ [A^m,B] = \sum_{k=0}^{m-1} A^k [A,B] A^{m-1-k} \]
yields
\begeq\label{nablaVmxi}
[\nabla_V^m,\xi] = \sum_{k=0}^{m-1} \nabla_V^k \nabla_H \nabla_V^{m-k-1}
\endeq
(all terms in the sum are equal except for the indices on which the horizontal derivative bears) and
\begeq\label{nablaVmnabla}
[\nabla_V^m,v\cdot\nabla_V] = m \nabla_V^m.
\endeq
Inserting \eqref{nablaVmxi} and \eqref{nablaVmnabla} in \eqref{patnablaVm}, integrating against $\nabla_V^m$, yields
\begeq\label{ddtnablavmh}
\frac{d}{2\,dt} \iint |\nabla_V^mh|^2 \leq
- \iint |\nabla_V^{m+1}h|^2 + m \iint \<\nabla_H\nabla_V^{m-1}h, \nabla_V^m h\> + m \iint |\nabla_V^m h|^2.
\endeq
The second term on the right-hand side can be treated through integration by parts, using Corollary \ref{divloss}, Sobolev estimates and Young,
\[ \left| \iint \<\nabla_H\nabla_V^{m-1}h, \nabla_V^m h\>\right|  \leq C \| \nabla_H \nabla_V^{m-2}h\|\, \|\nabla_V^{m+1}h\| \]
and from Proposition \ref{propmixinterp} 
\[ \|\nabla_H \nabla_V^{m-2} h\| \leq C \, \|\nabla_H^{m}h\|^{\frac1{m}}  \| \nabla_V h\|^{\frac2{m}-\frac1{m^2}} \|\nabla_V^{m+1}h\|^{1-\frac3{m} + \frac1{m^2}},\]
so
\begin{align} \label{prepa1}
\left| \iint \<\nabla_H\nabla_V^{m-1}h, \nabla_V^m h\>\right| 
& \leq C  \, \|\nabla_H^{m}h\|^{\frac1{m}}  \| \nabla_V h\|^{\frac2{m}-\frac1{m^2}} \|\nabla_V^{m+1}h\|^{2-\frac3{m}-\frac1{m^2}}\\
\nonumber
& \leq \var \|\nabla_V^{m+1}h\|^2 + C_\var \Bigl( \|\nabla_H^{m}h\|^{\frac1{m}} \|\nabla_V h\|^{\frac2{m}-\frac1{m^2}}  \Bigr)^{\frac{2m}{3-\frac1{m}}},
\end{align}
where $\var>0$ is arbitrarily small. As for the final term in \eqref{ddtnablavmh},
\begin{align} \label{prepa2}
\| \nabla_V^mh\| & \leq C \|\nabla_V^{m+1}h\|^{\frac{m-1}{m}} \|\nabla_Vh\|^{\frac{1}{m}} \\ \nonumber
&\leq \var \|\nabla_V^{m+1}h\|^2 + C_\var \|\nabla_Vh\|^2.
\end{align}

Likewise,
\begeq\label{nablaHm}
\pa_t\nabla_H^m h + \xi \nabla_H^m h - \Delta_V^\mu \nabla_H^m h = - [\nabla_H^m, \xi] h
\endeq
and
\[ [\nabla_H^m,\xi] = \sum_{k=0}^{m-1} \nabla_H^k [\xi, \nabla_H] \nabla_H^{m-k-1}.\]
From Proposition \ref{propxiboundnomom}, $[\xi,\nabla_H]$ loses at most 3 vertical derivatives and no horizontal derivative. More precisely (Proposition \ref{propcomH}), it takes the form of a differential operator of degree~1 in $v$, with coefficients that are polynomial of degree~2 in $v$. Distributing the velocity variables and applying Cauchy--Schwarz, then Proposition \ref{propgradmom} to trade moments for regularity,
\begin{align*} \left| \iint \< [\nabla_H^m,\xi] h, \nabla_H^m h \> \right|
&\leq C \|\nabla_V\nabla_H^{m-1}h\|_{L^2_1} \|\nabla_H^m h\|_{L^2_1} \\
& \leq C \bigl( \|\nabla_V^2 \nabla_H^{m-1} h\| + \|\nabla_V \nabla_H^{m-1}h\| \bigr)\, \bigl(
 \|\nabla_V \nabla_H^m h\| + \|\nabla_H^m h\|\bigr).
\end{align*}
Thus, upon multiplying \eqref{nablaHm} by $\nabla_H^mh$,
\begin{align} \label{prepa3}
\frac{d}{2\,dt} \iint |\nabla_H^m h|^2 & \leq - \|\nabla_V \nabla_H^m h\|^2 + C \bigl( \|\nabla_V^2 \nabla_H^{m-1} h\| + \|\nabla_V \nabla_H^{m-1}h\| \bigr)\, \bigl(
 \|\nabla_V \nabla_H^m h\| + \|\nabla_H^m h\|\bigr) \\
\nonumber
& \leq -K  \|\nabla_V \nabla_H^m h\|^2 + C \|\nabla_V^2 \nabla_H^{m-1}h\|^2 + C \bigl( \|\nabla_V\nabla_H^{m-1} h\|^2 + \|\nabla_H^mh\|^2\bigr)\\ \nonumber
& \leq -K \|\nabla_V \nabla_H^m h\|^2 + C \|\nabla_V \nabla_H^m h\|^{2(1-\frac1{m})} \|\nabla_V^{m+1} h\|^{\frac2{m}} + C \|\nabla_H^m h\|^2\\ \nonumber
& \leq -K \|\nabla_V \nabla_H^m h\|^2 + C \|\nabla_V^{m+1} h\|^2 + C \|\nabla_H^mh\|^2.
\end{align}

Also, still from Proposition \ref{propmixinterp},
\begeq\label{prepa4}
\|\nabla_V^m h\| \leq C \|\nabla_V^{m+1}h\|^{\frac{m-1}{m+1}} \|\nabla_Vh\|^{\frac1{m+1}}
\endeq

Combining \eqref{prepa1}, \eqref{prepa2}, \eqref{prepa3} and \eqref{prepa4}, for $0\leq t\leq 1$,
\begin{multline} \label{estddtF'}
\frac{d}{2\,dt} \left( \iint h^2 + a t^m \iint |\nabla_V^{m}h|^2 + c t^{3m} \iint |\nabla_H^{m}h\|^2 \right)\\ \leq
- K \Bigl ( \|\nabla_Vh\|^2 + a t^m \|\nabla_V^{m+1}h\|^2 + c t^{3m} \|\nabla_V \nabla_H^m h\|^2 \Bigr)\\
+ C \left( a t^{m-1} \|\nabla_V^{m+1}h\|^{\frac{2(m-1)}{m+1}} \|\nabla_V h\|^{\frac2{m+1}}  + c t^{3m-1} \|\nabla_H^mh\|^2
+ a t^m \|\nabla_Vh\|^2 \right.\\
\left. + a t^m \Bigl( \|\nabla_H^{m}h\|^{\frac1{m}} \|\nabla_V h\|^{\frac2{m}-\frac1{m^2}}  \Bigr)^{\frac{2m}{3-\frac1{m}}}
+ c t^{3m} \|\nabla_V^{m+1}h\|^2 \right).
\end{multline}

Use again Young's inequality to control the terms appearing on the right hand side:
\[ a t^{m-1} \|\nabla_V^{m+1}h\|^{\frac{2(m-1)}{m+1}} \|\nabla_V h\|^{\frac2{m+1}} 
 \leq \var \|\nabla_V h\|^2 + C_\var a^{\frac{1}{m-1}} (at^m) \|\nabla_V^{m+1}h\|^2
\]
(so the last term is arbitrarily small in front of $at^m \|\nabla_V^{m+1}h\|^2$ if $a\ll 1$),
\[  a t^m 
\Bigl( \|\nabla_H^{m}h\|^{\frac1{m}} \|\nabla_V h\|^{\frac2{m}-\frac1{m^2}}  \Bigr)^{\frac{2m}{3-\frac1{m}}}
\leq \var \|\nabla_V h\|^2 + C a^{3-\frac1{m}} t^{3m-1} \|\nabla_H^mh\|^2.\]
All in all, if $\var>0$ is properly chosen and $c<a\ll 1$, then \eqref{estddtF'} leads to
\begin{multline}\label{ddtFm}
\frac{d}{2\,dt} \left( \iint h^2 + a t^m \iint |\nabla_V^{m}h|^2 + c t^{3m} \iint |\nabla_H^{m}h\|^2  \right)\\ \leq
- K \Bigl ( \|\nabla_Vh\|^2 + a t^m \|\nabla_V^{m+1} h\|^2 + c t^{3m} \|\nabla_V \nabla_H^m h\|^2 \Bigr) \\ 
+ C \max\bigl(c,a^{3-\frac1{m}}\bigr) t^{3m-1} \|\nabla_H^{m}h\|^2.
\end{multline}
(of which \eqref{bddtdiagFth} is a particular case).
\sm

Next consider the mixed term with coefficient $b$.
Start again from \eqref{nablaHm}, replace $m$ by $m-1$ and apply $\nabla_V$:
\begin{multline*}
\pa_t \nabla_V\nabla_H^{m-1} h + \xi \nabla_H^{m-1} h + [\nabla_V,\xi] \nabla_H^{m-1} h 
-\Delta_V^\mu \nabla_V\nabla_H^{m-1}h + [\nabla_V, v\cdot\nabla_V] \nabla_H^{m-1}h \\
= -\nabla_V [\nabla_H^{m-1},\xi]h,
\end{multline*}
which, since $[\nabla_V,\xi]=\nabla_H$ and $[\nabla_V,v\cdot\nabla_V]=\nabla_V$, is the same as
\begin{multline}\label{1329'}
\pa_t\nabla_V\nabla_H^{m-1}h + \xi\nabla_V\nabla_H^{m-1}h - \Delta_V^\mu \nabla_V\nabla_H^{m-1}h + \nabla_H^m h 
= -\nabla_V \nabla_H^{m-1}h - \nabla_V [\nabla_H^{m-1},\xi]h.
\end{multline}
Multiply (in tensor sense and by contraction with $g$) \eqref{nablaHm} by $\nabla_V\nabla_H^{m-1}h$ and \eqref{1329'} by $\nabla_H^mh$, add up and use the derivation rule and $\Gamma$ formula for $\Delta_V^\mu$: it follows
\begin{multline}
\pa_t \<\nabla_V\nabla_H^{m-1}h, \nabla_H^m h\> + \xi \< \nabla_V\nabla_H^{m-1}h,\nabla_H^mh\>
- \Delta_V^\mu \< \nabla_V\nabla_H^{m-1}h, \nabla_H^mh\> + 2 \< \nabla_V\nabla_H^{m-1}h, \nabla_H^mh\>
+ |\nabla_H^mh|^2 \\
= - \<\nabla_V\nabla_H^{m-1}h,\nabla_H^mh\> - \<\nabla_H^mh,\nabla_V[\nabla_H^{m-1},\xi]h\> - \<\nabla_H^{m-1}\nabla_Vh, [\nabla_H^m,\xi] h\>
\end{multline}
(this is a pointwise identity, the scalar product stands for the metric). Now integrate against $\mu$, to find
\begin{multline} \label{integrmix}
\frac{d}{dt} \bigl\< \nabla_V\nabla_H^{m-1} h,\nabla_H^mh\bigr\>_{L^2} = 
- \|\nabla_H^mh\|_{L^2}^2 - 3 \<\nabla_V \nabla_H^{m-1}h,\nabla_H^mh\>_{L^2}\\
- \<\nabla_H^mh, \nabla_V [\nabla_H^{m-1},\xi] h\>_{L^2} - \<\nabla_H^{m-1}\nabla_V h, [\nabla_H^m,\xi]h\>_{L^2}.
\end{multline}
(From now on, I drop the $L^2$ subscripts to alleviate notation.) Thus there are positive constants $K,C>0$ only depending on $n$ (and $M$ later on), such that
\begin{multline}\label{integrmix'}
\frac{d}{dt} \bigl\< \nabla_V\nabla_H^{m-1} h,\nabla_H^mh\bigr\>_{L^2} \leq
- K \|\nabla_H^mh\|_{L^2}^2 \\
+ C \Bigl( \|\nabla_V\nabla_H^{m-1} h\|^2 + \bigl(\|\nabla_V \nabla_H^mh\| + \|\nabla_H^mh\| \bigr) \|[\nabla_H^{m-1},\xi] h\| \Bigr).
\end{multline}
Applying repeatedly H\"older, Young and $L^2$-interpolation,
\begeq\label{nvnhm-1} \|\nabla_V \nabla_H^{m-1}h\| \leq C\,\|\nabla_H^mh\|^{1-\frac1{m}}  \|\nabla_V^{m+1}h\|^{\frac1{m}(1-\frac1{m})} \|\nabla_Vh\|^{\frac1{m^2}}
\endeq
(exponents here are crucial, and doing the interpolation with $\|h\|$ rather than $\|\nabla_Vh\|$ in the final term would not suffice);
in particular
\[ \|\nabla_V\nabla_H^{m-1}h\|^2 \leq \var \|\nabla_H^mh\|^2 + C_\var \bigl( \|\nabla_V^{m+1}h\|^2 + \|\nabla_Vh\|^2\bigr)\]
(but it will also be useful to get back to \eqref{nvnhm-1} at some point);
\[ \| [\nabla_H^{m-1},\xi]h\| \leq C \bigl( \|\nabla_V^3 \nabla_H^{m-2} h\| + \|\nabla_V \nabla_H^{m-2}h\|\bigr)\]
(using the horizontal derivation commutator formulae and Proposition \ref{propgradmom});
and the latter term is bounded similarly as $\|\nabla_V\nabla_H^{m-2}h\|^2$ above. The other term involving $[\nabla_H^m,\xi]$ is handled in a similar way. At this stage,
\[ \frac{d}{dt} \<\nabla_V\nabla_H^{m-1}h,\nabla_H^mh\>
\leq -K \|\nabla_H^mh\|^2 + C \bigl( \|\nabla_V^{m+1}h\|^2 + \|\nabla_Vh\|^2\bigr)
+ C\, \|\nabla_V^3 \nabla_H^{m-2}h\|\,\|\nabla_V\nabla_H^mh\|.\]
It is for the last term, with the derivatives of order $m+1$, that one should be most careful: by interpolation again,
\[ \|\nabla_V^3 \nabla_H^{m-2}h\| \leq C\, \|\nabla_V\nabla_H^mh\|^{1-\frac2{m}} \|\nabla_V^{m+1}h\|^{\frac2{m}}.\]
The temporary conclusion is
\begeq\label{tempmixed}
\frac{d}{dt}
\<\nabla_V\nabla_H^{m-1}h,\nabla_H^mh\> 
\leq -K \|\nabla_H^mh\|^2 + C \Bigl( \|\nabla_V^{m+1}h\|^2 + \|\nabla_Vh\|^2 + \|\nabla_V^{m+1}h\|^{\frac2{m}} \|\nabla_V\nabla_H^mh\|^{2-\frac2{m}}\Bigr).
\endeq

Thus, using again \eqref{nvnhm-1},
\begin{multline*}
\frac{d}{dt} 
\left( bt^{3m-1} \< \nabla_V\nabla_H^{m-1}h,\nabla_H^mh\> \right) \\
\leq -K bt^{3m-1} \|\nabla_H^mh\|^2
+ C \Bigl( bt^{3m-2} \|\nabla_H^mh\|^{2-\frac1{m}} \|\nabla_V^{m+1}h\|^{\frac1{m}-\frac1{m^2}} \|\nabla_Vh\|^{\frac1{m^2}} \\
+ bt^{3m-1} \|\nabla_V^{m+1}h\|^2 + bt^{3m-1} \|\nabla_V h\|^2  
 + bt^{3m-1} \|\nabla_V^{m+1}h\|^{\frac2{m}} \|\nabla_V\nabla_H^mh\|^{2-\frac2{m}} \Bigr).
\end{multline*}
The goal is to control the positive terms in the right hand side using a small fraction of $\|\nabla_Vh\|^2$, $at^m\|\nabla_V^{m+1}h\|^2$, $bt^{3m-1}\|\nabla_H^mh\|^2$ and $ct^{3m}\|\nabla_V\nabla_H^mh\|^2$. This is achieved again by repeated use of Young:
\begin{multline*} bt^{3m-2}\|\nabla_H^mh\|^{2-\frac1{m}} \|\nabla_V^{m+1}h\|^{\frac1{m}-\frac1{m^2}} \|\nabla_Vh\|^{\frac1{m^2}}\\
\leq \var bt^{3m-1} \|\nabla_H^mh\|^2 + \var at^m \|\nabla_V^{m+1}h\|^2 + C_\var a \left(\frac{b}{a}\right)^m \|\nabla_Vh\|^2,
\end{multline*}
(and the last coefficient is small if $b<a\ll 1$); and
\begin{multline}\label{fincestep}
bt^{3m-1} \|\nabla_V^{m+1}h\|^{\frac2{m}} \|\nabla_V\nabla_H^mh\|^{2-\frac2{m}}\\
\leq \var at^m \|\nabla_V^{m+1}h\|^2 + C_\var \left(\frac{b^m}{a}\right)^{\frac1{m-1}} 
t^{(3m-2)\frac{m}{m-1}} \|\nabla_V\nabla_H^mh\|^2.
\end{multline}
Note that the last coefficient in \eqref{fincestep} is small as soon as
\[ b^m\ll ac^{m-1}, \]
a condition which is weaker than \eqref{abcm} if $c<a$; and that the power of $t$ satisfies
\[ (3m-2) \left(\frac{m}{m-1}\right) > 3m.\]
To summarise this step: If $b^m\ll ac^{m-1}$ then
\begeq\label{endstepmix}
\frac{d}{dt} \Bigl( bt^{3m-1} \<\nabla_V\nabla_H^{m-1}h, \nabla_H^mh\>_{L^2} \Bigr)
\leq -K bt^{3m-1} \|\nabla_H^mh\|^2
+ \var \Bigl( \|\nabla_Vh\|^2 + at^m \|\nabla_V^{m+1}h\|^2 + ct^{3m} \|\nabla_V\nabla_H^m h\|^2\Bigr),
\endeq
where $\var$ is as small as desired, and $K,C$ are positive constants.

Choosing $\var>0$ small enough, \eqref{endstepmix} combines with \eqref{ddtFm} to conclude, at last, that
\begeq\label{ddtFtfinal} \frac{d}{dt} {\cal F}^\ast_m (t,h(t)) \leq -K \Bigl( \|\nabla_Vh^2\|^2 + at^m \|\nabla_V^{m+1}h\|^2 + bt^{3m-1} \|\nabla_H^mh\|^2 + ct^{3m} \|\nabla_V\nabla_H^mh\|^2\Bigr) \leq 0
\endeq
provided that
\[ a^{3m-1} \ll b, \qquad b^{2m} \ll ac^{2m-1}, \qquad c\ll b\ll a.\]
To see the compatibility of those equations, put $a= \var^{\alpha}$, $b=\var^{\beta}$, $c=\var^{\gamma}$, with $0<\var<1$, then for $\var$ small enough this is satisfied if
\[ \alpha<\beta<\gamma,\qquad \beta < \left(3-\frac1{m}\right)\alpha,\qquad (2m-1)\gamma+\alpha< 2m\beta,\]
for this first fix $\alpha>0$, then $\beta\in (\alpha, (3-1/m)\alpha)$, then $\gamma\in (\beta,(2m)/(2m-1)\beta - \alpha/(2m-1))$, which is nontrivial since $\beta>\alpha$. Thus the proof of \eqref{ddtFtfinal} is complete, and the whole argument as well.
\end{proof}

\begin{proof}[Proof of Theorem \ref{thmregul}(ii)]
As for (i), it suffices to consider the case when $\alpha$ is an integer $m\in\N$; and \eqref{XHalphakappa} will follow from the case $\kappa=0$ by interpolation with the bound of Theorem \ref{thmloc}. So let us focus on \eqref{XHalpha} with $\alpha=m\in\N$.

From the equations \eqref{BXexpl}, in the form
\begeq\label{eqxhv}
\begin{cases} \dps \pa_t X_H + \xi X_H - \Sigma(v,v) X_V = \Delta_V^\mu X_H \\[2mm]
\dps \pa_t X_V + \xi X_V + X_H = \Delta_V^\mu X_V - X_V,
\end{cases}
\endeq
follows
\begeq\label{ddtl2X}
\begin{cases}\dps \frac{d}{dt} \|X_H\|_{L^2}^2 \leq -K \|\nabla_V X_H\|^2_{L^2} + C \bigl( \|X\|_{L^2}^2 + \|\nabla_V X_V\|_{L^2}^2\bigr) \\[3mm]
\dps \frac{d}{dt} \|X_V\|_{L^2}^2 \leq -K \|\nabla_V X_V\|^2_{L^2} + C \| X\|_{L^2}^2.
\end{cases}
\endeq
In particular, for $A>0$ large enough, $(d/dt) (A\|X_V\|^2 + \|X_H\|^2)\leq C \|X\|^2$, which implies that $\|X\|^2$ grows at most exponentially fast.

Then from \eqref{eqxhv},
\begin{multline}\label{eqderBH}
\pa_t \nabla_H^\ell \nabla_V^k X_H + \xi \nabla_H^\ell \nabla_V^k X_H + [\nabla_H^\ell, \xi] \nabla_V^k X_H + k \nabla_H^{\ell+1}\nabla_V^{k-1} X_H \\
- \nabla_H^\ell \nabla_V^k \bigl( \Sigma(v,v)X_V\bigr) = \Delta_V^\mu \nabla_H^\ell  \nabla_V^k X_H - k \nabla_H^\ell \nabla_V^k X_H
\end{multline}
and
\begin{multline} \label{eqderBV}
\pa_t\nabla_H^\ell \nabla_V^k X_V + \xi \nabla_H^\ell \nabla_V^k X_V + [\nabla_H^\ell, \xi] \nabla_V^k X_V + k \nabla_H^{\ell+1}\nabla_V^{k-1} X_V\\
+ \nabla_H^\ell \nabla_V^k X_H = \Delta_V^\mu \nabla_H^\ell \nabla_V^k X_V - (k+1) \nabla_H^\ell \nabla_V^k X_V.
\end{multline}

Choosing $k=0,\ell=m$ and multiplying the first equation by $\nabla_H^mX_H$,
\begin{multline*}
\pa_t \frac{|\nabla_H^m X_H|^2}{2} + \xi \frac{|\nabla_H^m X_H|^2}{2} + \bigl\< [\nabla_H^m,\xi] X_H, \nabla_H^m X_H \bigr\> 
- \bigl\< \nabla_H^m (\Sigma(v,v)X_V), \nabla_H^m X_H \bigr\> \\
= \Delta_V^\mu \frac{|\nabla_H^m X_H|^2}{2} - |\nabla_V \nabla_H^m X_H |^2.
\end{multline*}
Upon using Proposition \ref{propgradmom} and Cauchy--Schwarz,
\begin{multline*}
\frac{d}{2dt} \| \nabla_H^m X_H\|^2 \leq - K \|\nabla_V \nabla_H^m X_H\|^2 \\
+ C \bigl( \|\nabla_V^2 \nabla_H^{m-1} X_H\|^2 + \|\nabla_H^{m-1} X_H\|^2 + \|\nabla_H^m\nabla_V X_V\|^2 + \|\nabla_H^m X_V\|^2 \bigr).
\end{multline*}
By Proposition \ref{propinterp},
\[ \|\nabla_V^2 \nabla_H^{m-1} X_H\|^2 \leq C (\|X\|^2 + \|\nabla_H^m X\|^2 + \|\nabla_V^{3m}X\|^2.\]
All in all,
\begeq\label{eqBsyst1}
\frac{d}{dt} \|\nabla_H^m X_H\|^2 \leq -K \|\nabla_V \nabla_H^m X_H\|^2 
+ C \|\nabla_H^m \nabla_V X_V\|^2 + C \bigl( \|X\|^2 + \|\nabla_H^m X\|^2 + \|\nabla_V^{3m}X\|^2\bigr).
\endeq
(Note the $ \|\nabla_H^m \nabla_V X_V\|$ which comes from $\Sigma(v,v)$.)

A similar reasoning from the second equation in \eqref{eqderBV} yields
\begeq\label{eqBsyst2}
\frac{d}{dt} \|\nabla_H^m X_V\|^2 \leq -K \|\nabla_V \nabla_H^{m+1} X_V\|^2 + C  \bigl( \|X\|^2 + \|\nabla_H^m X\|^2 + \|\nabla_V^{3m}X\|^2\bigr).
\endeq

The case $k=3m, \ell=0$ similarly yields
\begin{multline}\label{eqBsyst3}
\frac{d}{dt} \|\nabla_V^{3m} X_H\|^2 \leq - K \| \nabla_V^{3m+1} X_H\|^2 + C \|\nabla_V^{3m+1} X_H\|^2
+ C \|\nabla_H \nabla_V^{3m-2} X_H\|^2 \\
+ C \bigl( \|X\|^2 + \|\nabla_H^m X\|^2 + \|\nabla_V^{3m}X\|^2\bigr)
\end{multline}
and
\begeq\label{eqBsyst4}
\frac{d}{dt} \|\nabla_V^{3m} X_V\|^2 \leq - K \| \nabla_V^{3m+1} X_V\|^2 + C \|\nabla_H \nabla_V^{3m-2} X_V\|^2 
+ C \bigl( \|X\|^2 + \|\nabla_H^m X\|^2 + \|\nabla_V^{3m}X\|^2\bigr).
\endeq

More generally, for $1\leq \ell\leq m-1$ and $k=3(m-\ell)$,
\begin{multline} \label{eqBsyst5}
\frac{d}{dt} \|\nabla_H^\ell\nabla_V^{3(m-\ell)} X_H\|^2 \leq - K \| \nabla_H^\ell \nabla_V^{3(m-\ell)+1} X_H\|^2 + C \|\nabla_H^\ell \nabla_V^{3(m-\ell)+1} X_V\|^2 \\
+ C \|\nabla^{\ell+1}_H \nabla_V^{3(m-\ell)+1} X_H\|^2 
+ C \bigl( \|X\|^2 + \|\nabla_H^m X\|^2 + \|\nabla_V^{3m}X\|^2\bigr)
\end{multline}
and
\begin{multline} \label{eqBsyst6}
\frac{d}{dt} \|\nabla_H^\ell\nabla_V^{3(m-\ell)} X_V\|^2 \leq - K \| \nabla_H^\ell \nabla_V^{3(m-\ell)+1} X_V\|^2 
+ C \|\nabla^{\ell+1}_H \nabla_V^{3(m-\ell)+1} X_V|^2 \\ + C \bigl( \|X\|^2 + \|\nabla_H^m X\|^2 + \|\nabla_V^{3m}X\|^2\bigr).
\end{multline}

From \eqref{eqBsyst1}--\eqref{eqBsyst6}, recalling also \eqref{ddtl2X}, one can choose constants $a_0, a'_0,\ldots, a_m,a'_m$ in such a way that $a_{\ell+1}\ll a'_\ell \ll a_\ell$ and
\begeq\label{Enew}
{\cal E} = A \|X_V\|^2 + \|X_H\|^2 + \sum_{\ell=0}^m \Bigl( a_\ell \|\nabla_H^\ell \nabla_V^{3(m-\ell)} X_V\|^2 + a'_\ell  \|\nabla_H^\ell \nabla_V^{3(m-\ell)} X_H\|^2
\endeq
satisfies
\begeq\label{newsyst1} \frac{d{\cal E}}{dt}\leq -K \|\nabla_V^{3m+1} X\|^2 + C {\cal E}.
\endeq
(keeping only the pure vertical derivative negative term).
\sm

Next turn to the mixed derivative terms: Starting again from \eqref{eqderBH} and \eqref{eqderBV},
\begin{multline} \label{ddtmixedB}
\frac{d}{dt} \<\nabla_H^{m-1}\nabla_V X_V,\nabla_H^m X_V\> 
\leq - K \|\nabla_H^m X_V\|^2 \\
+ C \Bigl( \|\nabla_H^{m-1}\nabla_V^3 X_V\|\, \|\nabla_V\nabla_H^m X_V\| + \|\nabla_H^m\nabla_V X_H\|\, \|\nabla_H^m X_H\| + {\rm l.o.t.} 
\end{multline}
(lower order terms), and likewise for $(d/dt) \<\nabla_H^{m-1}\nabla_V X_H, \nabla_H^m X_H\>$.
Eventually
\begeq\label{ddtmixedtotal}
\frac{d}{dt} \<\nabla_H^{m-1}\nabla_V X,\nabla_H^m X\> 
\leq - K \|\nabla_H^m X\|^2 + C \Bigl( \|X\|^2 + \|\nabla_V^{3m+1}X\|^2\Bigr).
\endeq

Then let
\begeq\label{defBXYM} {\cal X} = \|\nabla_V^{3m}X\|^2;\qquad {\cal Y} = \|X\|^2 + \|\nabla_V^{3m}X\|^2;\qquad
{\cal M} = \<\nabla_H^{m-1}\nabla_V X,\nabla_H^m X\>.\endeq
So \eqref{syst1} and \eqref{syst5} from Lemma \ref{lemEDO} are satisfied if $\|X_0\| = O(1)$ and
\[ {\cal Z} = \|\nabla_V^{3m+1}X\|^2.\]
On the other hand, by Proposition \ref{propinterp},
\begeq\label{Mleq}
|{\cal M}| \leq C ({\cal X} + {\cal Y})^{1-\frac1{3m}}
\endeq
and 
\begeq\label{Yleq}
{\cal Y} \leq C {\cal Z}^{1-\frac1{3m+1}}.
\endeq
So Lemma \ref{lemEDO} yields
\[ {\cal E}(t) = O( t^{-3m}) \]
(here $\delta = \theta/(1-\theta)= 1/3m$).
This concludes the argument towards \eqref{XHalpha}, and the remaining part of (ii) is straightforward.
\end{proof}

\begin{proof}[Proof of Theorem \ref{thmregul}(iii)]
As in (ii) I shall use a differential inequality relating the time derivatives of symmetric functionals like $\int |\nabla_H^\ell\nabla_V^k f|^2\,dv\,dx$ and a mixed functional like $\int \nabla_V\nabla_H^{m-1}f\cdot \nabla_H^mf\,dx\,dv$.

For a start, let us establish the equations on higher-order derivatives of $f$. Starting from
\[ \pa_t f + \xi f = \Delta_V f + v\cdot\nabla_V f + n f,\]
apply $\nabla_V^k$ for $k\in\N_0$, to get
\[ \pa_t \nabla_V^k f + \xi \nabla_V^k f + [\nabla_V^k,\xi] f
= \Delta_V \nabla_V^k f + (v\cdot\nabla_V) \nabla_V^kf.\]
Upon using Proposition \ref{propHV}, this is the same as
\[ \pa_t \nabla_V^k f + \xi \nabla_V^k f + k \nabla_H \nabla_V^{k-1}f
= \Delta_V \nabla_V^k f + (v\cdot\nabla_V) \nabla_V^kf + (k+n)\nabla_V^k f.\]
(Here the notation is just a bit sloppy since there is some permutation in the order of indices in the various terms $\nabla_H\nabla_V^{k-1}$ but this has no consequence.) Then apply $\nabla_H^\ell$ for $\ell\in\N_0$, to get
\begin{multline}\label{eqdhdvf}
\pa_t \nabla_H^\ell \nabla_V^k f + \xi \nabla_H^\ell \nabla_V^k f + [\nabla_H^\ell,\xi] \nabla_V^k f 
+ k \nabla_H^{\ell+1}\nabla_V^{k+1}f \\
= \Delta_V \nabla_H^\ell \nabla_V^k f + (v\cdot\nabla_V) \nabla_H^\ell \nabla_V^k f + (m+n) \nabla_H^\ell \nabla_V^k f.
\end{multline}
\sm 

{\bf First step: The symmetric part.} 
Multiply \eqref{eqdhdvf} by $\nabla_H^\ell \nabla_V^m f$ and apply the derivation formula (and $\Gamma$ formula for $\Delta_V$):
\begin{multline} \label{symparthv}
\pa_t \frac{|\nabla_H^\ell\nabla_V^k f|^2}{2} + \xi \frac{|\nabla_H^\ell \nabla_V^m f|^2}{2}
+ \bigl\< [\nabla_H^\ell,\xi]\nabla_V^k f, \nabla_H^\ell \nabla_V^k f\bigr\> 
+ k \bigl\< \nabla_H^{\ell+1} \nabla_V^{k-1} f, \nabla_H^\ell \nabla_V^k f\bigr\> \\
= \Delta_V \frac{|\nabla_H^\ell \nabla_V^kf|^2}{2} - |\nabla_H^\ell \nabla_V^{k+1}f|^2
+ (v\cdot\nabla_V) \frac{|\nabla_H^\ell\nabla_V^kf|^2}{2} + (m+n) |\nabla_H^\ell \nabla_V^kf|^2
\end{multline}
So far these are local inequalities and the goal is to integrate against $dx\,dv$.
The third, fourth and seventh terms in \eqref{symparthv} deserve some comment.

First, according to Proposition \ref{propcomH},
\begin{align*}
 & \bigl\<  [\nabla_H^\ell,\xi]\nabla_V^k f, \nabla_H^\ell \nabla_V^k f\bigr\>_{L^2}  \\
 & = \sum_{0\leq j\leq \ell-1} \bigl\< \nabla_H^j \Sigma_x(v,v) \nabla_V \nabla_H^{\ell-1-j} \nabla_V^k, \nabla_H^\ell \nabla_V^k f \bigr\> 
 - \sum_{0\leq j\leq \ell-1} \bigl\< \nabla_H^j \rho_x(v) \nabla_H^{\ell-1-j} \nabla_V^k, \nabla_H^\ell \nabla_V^k f \bigr\> \\
 & = - \sum_{0\leq j\leq \ell-1} \bigl\< \nabla_H^j \Sigma_x(v,v) \nabla_H^{\ell-1-j} \nabla_V^k, \nabla_V\nabla_H^\ell \nabla_V^k f \bigr\> 
 - \sum_{0\leq j\leq \ell-1} \bigl\< \nabla_H^j (\rho_x+A_x)(v) \nabla_H^{\ell-1-j} \nabla_V^k, \nabla_H^\ell \nabla_V^k f \bigr\>,
 \end{align*}
 where $\Sigma_x$ is a tensor-valued quadratic operator in $v$, and $\rho_x$ is a tensor-valued linear operator in $v$, both with $x$-dependent coefficients, and $A_x = \nabla_v\cdot (\Sigma_x(v,v))$. All those are smooth, with coefficients bounded by a constant multiple of curvature bounds. Then we can estimate those by Cauchy--Schwarz, using the weighted $L^2$ Sobolev spaces, getting
 \begin{align}\label{3rdterm}
 \left| \bigl\<  [\nabla_H^\ell,\xi]\nabla_V^k f, \nabla_H^\ell \nabla_V^k f\bigr\>_{L^2} \right| 
&  \leq C \Bigl( \|\nabla_H^\ell \nabla_V^kf\|_{L^2} \, \|\nabla_H^{\ell-1}\nabla_V^k f\|_{L^2_2}
+ \|\nabla_H^\ell \nabla_V^{k+1} f\|_{L^2} \, \|\nabla_H^{\ell-1} \nabla_V^k f\|_{L^2_1}\Bigr) \\ \nonumber
& \leq \var \bigl( \|\nabla_H^\ell \nabla_V^{k+1}f\|^2_{L^2} + \|\nabla_H^\ell \nabla_V^k f\|^2_{L^2}\bigr)
+ C \|\nabla_H^{\ell-1} \nabla_V^kf\|^2_{L^2_2}.
\end{align}

Next for the fourth term in \eqref{symparthv}: by integration by parts and Cauchy--Schwarz,
\begin{align} \label{5thterm}
\left| \bigl\< \nabla_H^{\ell+1}\nabla_V^{k-1}f , \nabla_H^\ell \nabla_V^k f\bigr\> \right|
& = \left| \bigl\< \nabla_H^{\ell+1}\nabla_V^{k-2} f, \nabla_H^\ell \nabla_V^{k+1} f\bigr\> \right| \\ \nonumber
& \leq \var \|\nabla_H^\ell \nabla_V^{k+1} f\|^2 + C \|\nabla_H^{\ell+1}\nabla_V^{k-2} f\|^2.
\end{align}

As for the seventh term in \eqref{symparthv}: using $\nabla_V\cdot v = n$,
\begeq\label{7thterm} \iint (v\cdot\nabla_V) |\nabla_H^\ell\nabla_V^k f|^2 \,dx\,dv
= - n \iint |\nabla_H^\ell \nabla_V^k f|^2\,dx\,dv. 
\endeq

Inserting estimates \eqref{3rdterm}, \eqref{5thterm} and \eqref{7thterm} in \eqref{symparthv} yields, upon integration,
\begin{multline}\label{symparthvv}
\frac{d}{2dt} \|\nabla_H^\ell\nabla_V^kf\|_{L^2}^2 
\leq - K \|\nabla_H^\ell \nabla_V^{k+1} f\|_{L^2}^2 \\
+ C \Bigl( \|\nabla_H^\ell \nabla_V^kf\|_{L^2}^2 + \|\nabla_H^{\ell+1}\nabla_V^{k-2}f\|_{L^2}^2 
+ \|\nabla_H^{\ell-1}\nabla_V^k f\|_{L^2}^2\Bigr).
\end{multline}
(The term in $k-2$ is only if $k\geq 2$, the term in $\ell-1$ only if $\ell\geq 1$.)
Let us fix $m\in\N$ and combine all estimates \eqref{symparthvv} for $0\leq \ell\leq m$ and $k=3(m-\ell)$; this provides an estimate on the time-derivative of
\begeq\label{Em}
{\cal E}(t) = \|f\|^2 + \sum_{0\leq \ell\leq m, \ k+3\ell=3m} a_\ell\, \|\nabla_H^\ell \nabla_V^k f\|_{L^2}^2.
\endeq
If the $a_\ell$ are chosen recursively in such a way that $Ka_\ell \geq 2 C a_{\ell+1}$ (where $K$ and $C$ are respectively minimum and maximum for those constants $K$ and $C$ appearing in \eqref{symparthvv} for those choices of indices), all terms $\|\nabla_H^{\ell+1}\nabla_V^{k-2}f\|_{L^2}^2$ in the right hand side will be telescopically erased (remember that for $\ell=m$, i.e. $k=0$, this term is not present). The conclusion is
\[
\frac{d{\cal E}}{dt} \leq -K \sum_\ell \|\nabla_H^\ell \nabla_V^{3(m-\ell)+1} f\|^2 + C {\cal E} + C \sum_\ell \|\nabla_H^{\ell-1} \nabla_V^{3(m-\ell)} f\|_{L^2}^2.\]

At this stage let us apply the anisotropic Nash inequality, Theorem \ref{propnashinterp}: For $\sigma$ large enough,
\[ \|\nabla_H^{\ell-1}\nabla_V^{3(m-\ell)}f\|_{L^2_2} 
\leq C\, \|f\|_{L^1_\sigma} 
\bigl( \|\nabla_H^mf\| + \|\nabla_V^{3m}f\| \bigr)^{1-\theta},\]
for $\theta$ arbitrarily close to
\[ \ov{\theta} = \frac{1/m}{1+ 2n/(3m)} = \frac{3}{3m+2}. \]
(At this stage the estimate of $\theta$ will not be useful.)
For the moment, I will only retain
\begeq\label{conclsymparthv}
\frac{d{\cal E}}{dt} \leq -K \|\nabla_V^{3m+1} f\|^2 + C\, {\cal E},
\endeq
where $C$ depends on $f$ only through a global bound on $\|f\|_{L^1_\sigma}$.
\sm

{\bf Second step: The mixed part.} Start again from
\begeq\label{againhm}
\pa_t\nabla_H^m f + \xi \nabla_H^m f + [\nabla_H^m,\xi] f = \Delta_V \nabla_H^m f + (v\cdot\nabla_V) \nabla_H^m f + n \nabla_H^m f
\endeq
and
\begin{multline}\label{againhdm}
\pa_t \nabla_H^{m-1}\nabla_V f + \xi \nabla_H^{m-1} \nabla_V f + \nabla_H^m f+ [\nabla_H^{m-1},\xi] \nabla_V f \\
= \Delta_V \nabla_H^{m-1} \nabla_V f + (v\cdot\nabla_V)\nabla_H^{m-1} \nabla_V f + (n+1) \nabla_H^{m-1} \nabla_V f.
\end{multline}
Multiply \eqref{againhm} by $\nabla_H^{m-1}\nabla_V f$ and \eqref{againhdm} by $\nabla_H^mf$, add and use the derivation rule:
\begin{multline*}
\pa_t \< \nabla_H^mf, \nabla_H^{m-1}f\> + \xi \<\nabla_H^mf,\nabla_H^{m-1}\nabla_V f\>
+ \<\nabla_H^{m-1}\nabla_Vf, [\nabla_H^m,\xi] f\> + |\nabla_H^mf|^2 + \<[\nabla_H^{m-1},\nabla_H^m f\>\\
= \Delta_V \<\nabla_H^{m-1}\nabla_V f, \nabla_H^m f\>
- 2 \< \nabla_H^{m-1}\nabla_V^2 f, \nabla_H^m\nabla_V f\> + (v\cdot\nabla_V) \<\nabla)H^mf,\nabla_H^{m-1}\nabla_Vf\>\\
+ (2n+1) \<\nabla_H^mf, \nabla_H^{m-1}\nabla_Vf\>.
\end{multline*}
Integrating yields
\begin{multline*}
\frac{d}{dt} \< \nabla_H^mf,\nabla_H^{m-1}\nabla_Vf\>_{L^2} + \|\nabla_H^m f\|_{L^2}^2 \\
= - \<[\nabla_H^{m-1},\xi]\nabla_V f,\nabla_H^mf\>_{L^2} 
- \<\nabla_H^{m-1}\nabla_V^2 f,\nabla_H^m\nabla_Vf\>_{L^2}
+ (n+1) \<\nabla_H^mf, \nabla_H^{m-1}\nabla_Vf\>.
\end{multline*}
Both terms $\<[\nabla_H^{m-1},\xi]\nabla_V f,\nabla_H^mf\>$ and $\<\nabla_H^{m-1}\nabla_V^2 f,\nabla_H^m\nabla_Vf\>$ involve degree~2 polynomial multiplication by $v$, $2m-2$ horizontal and 2 vertical derivations. Since
\[ \frac{2m-2}{m} + \frac2{3m+1} < 2,\]
by using Cauchy-Schwarz and Theorem \ref{propnashinterp}, they are bounded by
\[ C \|f\|_{L^1_\sigma} \|\nabla_H^m f\|^{2(m-1)/m} \|\nabla_V^{3m+1} f\|^{2/(3m+1)}
\leq \var \|\nabla_H^mf\|^2 + C_\var \|f\|_{L^1_\sigma}^\alpha \|\nabla_V^{3m+1}\|^2\]
for some $\alpha>0$.
Likewise,
\begin{align*} \left| \<\nabla_H^{m-1}\nabla_V^2 f, \nabla_H^m \nabla_V f\>_{L^2} \right| 
& \leq \|\nabla_H^m f\|\, \|\nabla_H^{m-1}\nabla_V^3 f\|
\\ &  \leq C\|f\|_{L^1_\sigma} \|\nabla_H^m f\|^{2-\frac1{m}} \|\nabla_V^{3m+1} f\|^{\frac{3}{3m+1}}\\
& \leq \var \|\nabla_H^m f\|^2 + C \|f\|_{L^1_\sigma}^\alpha \|\nabla_V^{3m+1} f\|^2
\end{align*}
\sm

{\bf Step~3: Differential system.}
Let
\[ {\cal X}(t) = \|\nabla_H^m f\|_{L^2}^2, \qquad {\cal Y}(t) = \|\nabla_V^{3m} f\|_{L^2}^2, \qquad {\cal Z}(t) = \|\nabla_V^{3m+1}f\|_{L^2}^2.\]
From the expression of ${\cal E}$ in \eqref{Em} and Theorem \ref{propnashinterp},
\begeq\label{EXY} K({\cal X}+{\cal Y}) \leq {\cal E} \leq C ({\cal X}+{\cal Y}), \endeq
where $K,C>0$ are constants depending only on the geometry, and in the case of $C$ proportional to $\|f\|_{L^1_\sigma}$ for some $\sigma>0$. Steps~1 and~2 can be summarised as
\begeq\label{ddtEZ} 
\begin{cases} \dps \frac{d}{dt} {\cal E} \leq - K {\cal Z} + C {\cal E}, \\[3mm]
\dps \frac{d}{dt} {\cal M} \leq -K {\cal X} + C ({\cal Y}+{\cal Z}).
\end{cases}
\endeq
Next use again Theorem \ref{propnashinterp} to get
\[\|\nabla_V^{3m} f\|_{L^2}^2 \leq C \|f\|_{L^1_\sigma} \bigl(\|\nabla_H^mf\|^2 + \|\nabla_V^{3m+1} f\|^2\bigr)^{1-\theta}, \]
or
\begeq\label{Ydelta}
{\cal Y} \leq C ({\cal X}+{\cal Z})^{1-\theta}
\endeq
for $\theta$ arbitrarily close to
\[ \ov{\theta} = \frac{1-\frac{3m}{3m+1}}{1+\frac{n}{2} \left( \frac1{m}+\frac1{3m+1}\right)} = \frac1{3m+2n+1+\frac{n}{2m}}. \]
Also
\begin{align*} \bigl| \<\nabla_H^{m-1} \nabla_V f, \nabla_H^m f\> \bigr| & \leq \|\nabla_H^m\|\,\|\nabla_H^{m-1} \nabla_V f\|\\
& \leq C \|f\|_{L^1_\sigma} \|\nabla_H^mf\| \bigl( \|\nabla_H^mf\| + \|\nabla_V^{3m} f\| \bigr)^{1-2\delta}\\
& \leq C \bigl( \|\nabla_H^mf\| + \|\nabla_V^{3m} f\| \bigr)^{1-\delta}
\end{align*}
or
\begeq\label{Mtheta}
|{\cal M}| \leq C {\cal E}^{1-\delta}
\endeq
for $\delta$ arbitrarily close to
\[ \ov{\delta} = \frac{\frac1{3m}}{1+\frac{2n}{3m}} = \frac1{2n+3m}. \]
By combining \eqref{EXY}, \eqref{ddtEZ}, \eqref{Ydelta}, \eqref{Mtheta} and Lemma \ref{lemEDO} below,
\[ {\cal E}(t) \leq \frac{C}{t^{-1/\kappa}}, \]
for $\kappa = \min (\delta, \theta/(1-\theta))$, which is arbitrarily close to $(3m+2n+n/(2m))^{-1}$. Note that for large $m$ this approaches the presumably optimal rate $t^{-(3m+2n)}$; one could certainly self-improve that estimate by interpolation if one had access to the full range of $\theta$ in Theorem \ref{propnashinterp}.

Here the constant is polynomial in $\|f\|_{L^1_\sigma}$, but this can always be reduced to just one moment, since by convexity and $\int f=1$,
\[ \|f\|_{L^1_\sigma}^r \leq C \|f\|_{L^1_{\sigma r}}\qquad r\geq 1.\]
This concludes the proof of \eqref{fHalpha}. From there \eqref{schwartz} follows easily by interpolation and Sobolev embedding (Proposition \ref{propsobemb}.
\end{proof}

\begin{proof}[Proof of Theorem \ref{thmregul} (iv)]
The proof is similar to that of part (iii), working on 
\[ F(t,x,v) = f(t,x,v)\, e^{\beta\frac{|v|^2}{2}},\]
where $\beta$ is fixed, $0<\beta<\beta_0<1$.
Indeed, it was already shown in Theorem \ref{thmloc} that $F$ is uniformly bounded in $L^1(dx\,dv)$, and further, that for all $\gamma<\beta_0-\beta$, $\int F e^{\gamma|v|^2/2}$ remains bounded; in particular $F$ has bounded moments of all order, uniformly in $t\geq 0$. Moreover, $F$ satisfies
\begeq\label{eqF}
\pa_t F + \xi F = \Delta_V F + (1-2\beta)v\cdot\nabla_V F + (1-\beta) \bigl( n - \beta |v|^2 \bigr) F.
\endeq
This and the moment bounds allow to repeat the whole scheme of proof for (ii), and show that $F$ is smooth and rapidly decaying (Schwartz class) on $\TM$, uniformly in $t>0$. But successive derivatives of $f= F e^{-\beta |v|^2/2}$ take the form of $F$, multiplied by a polynomial in $v$, multiplied by $e^{-\beta|v|^2/2}$; so it is a bounded multiple of $e^{-\beta|v|^2/2}$.
\end{proof}

The following lemma was used in the proof of Theorem \ref{thmregul} (ii)--(iii) (and implicitly (iv)).

\begin{Lem} \label{lemEDO}
Let ${\cal E}$, ${\cal X}$, ${\cal Y}$, ${\cal Z}$ and ${\cal M}$ be continuous functions of $t\in [0,1]$,
with ${\cal E}, {\cal X},{\cal Y},{\cal Z}\geq 0$, such that
\begeq\label{syst1}
K ({\cal X}+{\cal Y}) \leq {\cal E} \leq C ({\cal X}+{\cal Y}),
\endeq
\begeq\label{syst2}
|{\cal M}| \leq C {\cal E}^{1-\delta},
\endeq
\begeq\label{syst3}
\frac{d{\cal E}}{dt} \leq - K {\cal Z} + C {\cal E},
\endeq
\begeq\label{syst4}
{\cal Y} \leq C ({\cal X}+{\cal Z})^{1-\theta},
\endeq
\begeq\label{syst5}
\frac{d{\cal M}}{dt} \leq -K{\cal X} + C ({\cal Y}+{\cal Z}),
\endeq
where $C,K$ are positive constants, and $\delta,\theta$ are
real numbers lying in $(0,1)$. Then
\[ {\cal E}(t) \leq \frac{\ov{C}}{t^{1/\kappa}},\qquad
\kappa= \min \left(\delta, \, \frac{\theta}{1-\theta}\right),\]
where $\ov{C}$ is an explicit constant which only depends on
$C,K,\theta,\delta$.
\end{Lem}

A proof, from an earlier memoir of mine, is reproduced here for convenience.

\begin{proof}[Proof of Lemma~\ref{lemEDO}]
Let $\tilde{\cal E}(t)=e^{-Ct}{\cal E}(t)$; then
$\tilde{\cal E}$ satisfies estimates similar to ${\cal E}$,
except that equation~\eqref{syst3} becomes $d\tilde{\cal E}/dt\leq -K {\cal Z}$.
In the sequel I~shall keep the notation ${\cal E}$ for $\tilde{\cal E}$,
so this just amounts to replacing~\eqref{syst3} by
\begeq\label{syst3'}
\frac{d{\cal E}}{dt} \leq - K {\cal Z}.
\endeq
In particular, ${\cal E}$ is nonincreasing.

Now let $E>0$, and let $I\subset [0,1]$ be the time-interval where 
$(E/2) \leq {\cal E}(t)\leq E$. The goal is to show that
the length $|I|$ of $I$ is bounded like $O(E^{-\kappa})$
for some $\kappa>0$. If that is the case, then the conclusion
follows. Indeed, let $E_0>0$ be given, and let $T$ be the first
time $t$ such that ${\cal E}(t)\leq E_0$, then
\[ T \leq C' \sum_{n\geq 1} E_0^{-n\kappa} \leq C'' E_0^{-\kappa};\]
so $E_0\leq T^{-1/\kappa}$. (Here as in the sequel,
$C$, $C'$, $C''$ stand for various constants that only depend on
the constants $C$ and $K$ appearing in the statement of the lemma.)

If $E\leq 1$ then the conclusion obviously holds true.
So we might assume that $E\geq 1$.

It follows by integration of~\eqref{syst3'} over $I$ that
\begeq\label{intIZE}
\int_I {\cal Z}(t)\,dt \leq E -\frac{E}2 = \frac{E}2.
\endeq

By integrating~\eqref{syst4}, we find
\begin{align*}
\int_I {\cal Y}(t)\,dt & \leq C \int_I \bigl[ {\cal X}(t) + {\cal Z}(t)\bigr]^{1-\theta}\,dt\\
& \leq C' \left( \int_I {\cal X}(t)^{1-\theta}\,dt + 
   \int_I {\cal Z}(t)^{1-\theta}\,dt\right)\\
& \leq C' \left( |I|\: \Bigl[\sup_I {\cal X}(t)^{1-\theta}\Bigr]
   + \left(\int_I {\cal Z}(t)\,dt \right)^{1-\theta} |I|^\theta\right).
\end{align*}
To estimate the first term inside parentheses, note that 
${\cal X}\leq C {\cal E}\leq CE$; to bound the second term, use~\eqref{intIZE}.
The result is
\begeq\label{intIY} \int_I {\cal Y}(t)\,dt \leq C \Bigl( |I| E^{1-\theta} + 
E^{1-\theta} |I|^\theta\Bigr) \leq C' |I|^\theta E^{1-\theta},
\endeq
where the last inequality follows from $|I|\leq |I|^\theta$.
(Note indeed that $|I|\leq 1$ and $\theta<1$.)

Next, integrate inequality~\eqref{syst5} over $I=[t_1,t_2]$, to get
\begin{align}
K \int_I {\cal X}(t)\,dt & \leq |{\cal M}(t_1)| + |{\cal M}(t_2)|
+ C \int_I [{\cal Y}(t)+{\cal Z}(t)]\,dt \noindent \\
& \leq 2 \sup_{t\in I} |{\cal M}(t)| + C\left(\int_I {\cal Y}(t)\,dt +
\int_I {\cal Z}(t)\,dt\right).
\label{intIX} 
\end{align}

Also, since ${\cal E}\geq E/2$ on $I$, we have
\begeq\label{TE2}
\frac{|I|\,E}{2} \leq \int_I {\cal E}(t)\,dt
\leq C \left( \int_I {\cal X}(t)\,dt + \int_I {\cal Y}(t)\,dt\right),
\endeq
where the last inequality follows from~\eqref{syst1}.

The combination of~\eqref{intIX} and~\eqref{TE2} implies
\[ \frac{|I|\,E}2 \leq C \Bigl ( \sup_{t\in I} |{\cal M}(t)| + 
\int_I {\cal Y}(t)\,dt + \int_I {\cal Z}(t)\,dt\Bigr).\]
To estimate the first term inside the brackets, use~\eqref{syst4};
to estimate the second one, use~\eqref{intIY}; to estimate
the third one, use~\eqref{intIZE}. The result is
\begeq\label{TE22} 
|I|\,E \leq C (E^{1-\delta} + |I|^\theta E^{1-\theta} + E).
\endeq

Now we can conclude, separating three cases according to which one of
the three terms in the right-hand side of~\eqref{TE22} is largest:

- If it is $E^{1-\delta}$, then
$|I|\,E \leq 3 C E^{1-\delta}$, so $|I|\leq 3 C E^{-\delta}$;

- If it is $|I|^\theta E^{1-\theta}$, then
$|I|\,E \leq 3 C |I|^\theta E^{1-\theta}$, so
$|I| \leq (3C)^{\frac1{1-\theta}} E^{-\frac{\theta}{1-\theta}}$;

- If it is $E$, then $|I|\leq 3C$.
\sm

In any case, there is an estimate like $|I|\leq \ov{C} E^{-\kappa}$,
where $\kappa$ is as in the statement of the lemma.
So the proof is complete.
\end{proof}

\bibnotes

This chapter is largely taken from \cite{DOV:preprint}. The overall strategy is taken from \cite{vill:hypoco} where global regularisation in Euclidean space was established for $L^2$ and measure initial data. In particular, Remark \ref{fromtregtomaxreg} is similar to \cite[Remark A.10]{vill:hypoco} in the flat case; Lemmas \ref{ll} and \ref{lemEDO} are also taken from \cite{vill:hypoco}.

The proof by H\'erau is from \cite{herau:FP:07}, with a slight twist as in \cite[Theorem A.12]{vill:hypoco}.

\section{Strict positivity} \label{secpos}

Besides regularity, positivity is the most important property of diffusion equations acting on densities. 
Since regularity and localisation have already been established, local strict positivity will follow easily.

\begin{Prop}[Local strict positivity] \label{proplocalpos}
Let $f=f(t,x,v)$ solve $\pa_t f + {\cal L}f =0$ with initial datum $f_0$ such that 
\begeq\label{f0posit} \iint_{\TM} f_0 =1,\qquad E_0= \frac12 \iint_{\TM} f_0(x,v)\,|v|^2\,dx\,dv <+\infty. \endeq
Then for any $R>0$ there is $K>0$ such that for all $(x,v)\in\TM$,
\[ |v|\leq R \Longrightarrow\qquad f(t,x,v) \geq K/t^r,\]
for some universal $r>0$ depending only on the dimension.
\end{Prop}

\begin{proof}[Proof of Proposition \ref{proplocalpos}]
By Tchebyshev's inequality, a portion at least $1/2$ of the mass of $f_0$ lies within $|v|\leq 2\sqrt{E_0}=:V$. Since the equation is order-preserving, it suffices to consider the case in which $f_0$ is compactly supported in $\{|v|\leq V\}$. Then the solution is uniformly smooth, by Theorem \ref{thmregul}. By Theorem \ref{thmloc}, $\int f |v|^2\,dx\,dv$ remains uniformly bounded, so there is always a lower bound on the mass of $f$ in $|v|\leq V$. By smoothness, there are $K>0$, $r>0$, $\tau>0$, such that for all $t_0>0$ there is some ball $B$ with center $(x_0,v_0)$ and radius $r$, with $x_0\in M$, $|v_0|\leq V$, such that $f\geq K$ on $B$ for all $t\in [t_0,t_0+\tau)$.

If $r$ is small enough, then the partial differential equation $\pa_t f + {\cal L}f=0$ can be rewritten in $B((x_0,v_0),2r)$ as a hypoelliptic equation in a region of $\R^n\times\R^n$, and from classical estimates the mass of $f$ on any ball of radius $r$, whose center lies within the boundary of $B$, is uniformly bounded below for $t_0<t<t_0+\tau$. (This is the spreading of the mass.) A finite number of steps allows to cover the whole $|v|\leq V$ in $\TM$. All of this is done with estimates that are terrible, but not worse than negative powers of $t$.
\end{proof}

\begin{Rk} \label{rkMTconv} Remark \ref{rkL1loc} and Proposition \ref{proplocalpos} together make it possible to apply the Meyn--Tweedie theory and get exponential rates of equilibration. This provides a convergence rate that depends on the constants in Proposition \ref{thmloc}(i) (rather good) and in Proposition \ref{proplocalpos} (very bad).
\end{Rk}

Now comes a first global positivity result.

\begin{Thm}[Uniform Gaussian lower bound] \label{thmlowerbound}
Let $f=f(t,x,v)$ solve $\pa_t f + {\cal L}f =0$ with initial datum $f_0$ satisfying \eqref{f0posit}. Then there are a positive function $\alpha=\alpha(t)>0$ and $K>0$ such that for all $t>0$ and $(x,v)\in \TM$,
\[ f(t,x,v) \geq K e^{-\alpha(t) \frac{|v|^2}{2}}, \]
where 
\[ \alpha(t) = O(1+t^{-r}),\qquad \alpha(t) \xrightarrow[t\to\infty]{} 1. \]
\end{Thm}

\begin{proof}[Proof of Theorem \ref{thmlowerbound}]
Fix $V_0>0$. From the local positivity, there are $r_0,m=m(t_0)>0$ such that $f(t,x,v)\geq m$ as soon as $t\geq t_0$ and $|v|\leq V_0$, and $m$ depends polynomially on $t_0$ for $t_0\leq 1$. 
Let then
\[ \vphi(t,x,v) = m\, e^{-\alpha(t) |\frac{|v|^2}{2}}, \qquad t\geq t_0 \]
 where $\alpha$ will be chosen later on, $\alpha(t_0) = +\infty$. The goal is to show, by maximum principle, that $f\geq\vphi$ for $|v|>V_0$. By construction, $\vphi\leq f$ for $|v|=V_0$ and $t=t_0$.
By computation,
\[
\pa_t\vphi + \xi\vphi - \Delta_V\vphi - v\cdot\nabla_V \vphi - n \vphi \\
= \Bigl[ (-\dot{\alpha} + \alpha^2 - \alpha) |v|^2 - n \Bigr] \vphi. 
\]
So if 
\begeq \label{edoalpha}
\dot{\alpha} \geq \alpha^2 - \alpha - \frac{n}{V_0^2}
\endeq
then $\pa_t\vphi + {\cal L}\vphi\leq 0$ for all $|v|>V_0$, and the global lower bound will follow.

Then for \eqref{edoalpha} to be satisfied, it suffices to let $\alpha(t) = A + B/(t-t_0)$, with $A,B>0$ well chosen.
\end{proof}

\bibnotes

Regularity and positivity go hand in hand all throughout the classical theory of elliptic and parabolic equations for densities. For instance, Gaussian upper and lower bounds on the fundamental solution are a key step in Nash's theory of regularity for nonsmooth equations in divergence form \cite{nash:58}; they are also called Aronson's estimates \cite{aronson:fundamental:67} and related to Moser's version of Harnack inequality \cite{moser:harnack:64}. See Fabes--Stroock \cite{fabesstroock:86} and Bass \cite[chapter~7]{bass:diffusionbook} for a survey of these inequalities and their relation.

In the context of hypoelliptic equations, I am not aware of such a neat set of relations, but there is also an abundance of works relating regularisation and appearance of strict positivity, especially in the probabilistic context, see for instance Stroock \cite{stroock:path:book}.

The strict positivity regularisation in $\R^n$ which I used in the proof of Proposition \ref{proplocalpos} is from \cite[Corollary A.26]{vill:hypoco}, based on maximum principle. In the context of kinetic theory, Desvillettes and I~used such a method to prove the strict positivity for the spatially homogeneous Landau equation \cite{DV:landau:1}. For the spatially homogeneous Boltzmann equation a large number of works have been devoted to the quantitative strict positivity of the solution, ever since its founding by Carleman \cite{carleman}, in particular by Pulvirenti--Wennberg \cite{pulviwenn:CMPlower:97} and Imbert--Mouhot--Silvestre \cite{IMS:boltzmann:20}. It would be nice to have more intrinsic and more sharp estimates for the geometric kinetic Fokker--Planck equation.

Remark \ref{rkMTconv} is obtained by combining Proposition \ref{thmloc} (i), Remark \ref{rkL1loc} and Proposition \ref{proplocalpos} with the Meyn--Tweedie convergence result from Ca\~{n}izo--Mischler \cite[Theorem 5.2]{canizomischler:harris:23}.

\section{Coercivity inequalities in $\TM$} \label{secineq}

This section is about functional inequalities by which a dissipation functional, expressed in terms of integrals of differential quantities, control the distance to equilibrium. This topic is extremely classical when there is just one (position) variable, and the goal is now to extend this to the tangent bundle phase space. Three situations will be considered:
\sm

\bul In the real-valued $L^2$ setting, gradients control the distance to constant functions: This is the {\bf Poincar\'e inequality}.
\sm

\bul In the vector-valued $L^2$ setting, gradients control the distance to {\bf parallel vector fields}, which are the very special vector fields $Y$ on $M$ satisfying $\nabla Y =0$ (the covariant derivative along any vector field vanishes). Such vector fields exist only under special conditions, namely if $M=\T^1\times M'$ (in the sense of warped products). Of course if $M=\T^n$ there are $n$ independent parallel vector fields. I am not aware of any specific name for the associated inequality and will denote it by just the {\bf Poincar\'e inequality for vector fields}.
\sm

\bul In the $L^1$ (density) setting, the equilibrium is the thermodynamical equilibrium, and the natural estimate of equilibration come from information theory: the Boltzmann--Shannon information $\int f\log f$ and the Fisher information $\int |\nabla f|^2/f$, or their relative counterparts where $f$ is replaced by the density with respect to the equilibrium probability measure. The domination of the Boltzmann information by the Fisher information is called the {\bf logarithmic Sobolev inequality}.

\begin{Def}[Poincar\'e inequality] \label{defpoinc}
The manifold $M$ satisfies a Poincar\'e inequality ${\rm P}(K)$ if, for all functions $u:M\to\R$ such that $\int_M u =0$, one has
\[ \|u \|^2_{L^2(M)} \leq \frac1{K} \|\nabla u\|_{L^2(M)}^2,\]
where the measure is the volume measure on $M$.
\end{Def}

\begin{Def}[Poincar\'e inequality for vector fields] \label{defpoincvf}
The manifold $M$ satisfies a Poincar\'e inequality for vector fields, ${\rm P}_{\cal V}X(K)$, if, for all vector fields $X:M\to\TM$ which are orthogonal to parallel vector fields, one has
\[ \|X \|^2_{L^2(M)} \leq \frac1{K} \|\nabla X\|_{L^2(M)}^2,\]
where the measure is the volume measure on $M$.
\end{Def}

\begin{Def}[log Sobolev inequality] \label{defLSI}
The manifold $M$ satisfies a logarithmic Sobolev inequality $\LSI(K)$ if, for all functions $f:M\to\R_+$ such that $\int_M f\,d\nu =1$, one has
\[ \int_M f \log f\,d\nu \leq \frac1{2K} \int_M \frac{|\nabla f|^2}{f}\,d\nu,\] 
where $\nu$ is the normalised volume measure on $M$.
\end{Def}

\begin{Rks}
\begin{itemize} 
\item[(i)] The normalisation in Definition \ref{defLSI} is to ensure that $\LSI(K)$ implies ${\rm P}(K)$ in the limit $f  = 1+\var u$.

\item[(ii)] A compact manifold $M$ always satisfies the three inequalities above for some positive constants.
\end{itemize}
\end{Rks}

Now come the generalisations of these three tools to the phase space $\TM$. I will denote by $\<h\> = \int h\,d\mu$ the average of the real-valued function $h$ against $\mu$, and by ${\cal P}$ the closed vector space of parallel horizontal vector fields, that is, those vector fields on $\TM$ which take the particular form $Y(x,v) = (Y(x),0)$ with $Y$ parallel.

\begin{Prop}[Poincar\'e inequality in the tangent bundle] \label{poincare}
There is a constant $K_P>0$, only depending on $M$, such that for any $h\in H^1(\TM;\R)$,
\begeq\label{ineqpoinc}
\iint_{\TM} \left| h - \< h\> \right|^2\, d\mu \leq
\frac1{K_P} \iint_{\TM} \Bigl( |\nabla_H h|^2 + |\nabla_V h|^2\Bigr)\,d\mu.
\endeq
\end{Prop}

\begin{Prop}[Poincar\'e for vector fields in the tangent bundle] \label{poincarevect}
There is a constant $K_{\cal V}>0$, only depending on $M$, such that for any $X\in H^1(\TM;\Tzeroun)$,
\begeq\label{ineqpoincX} X\bot {\cal P} \Longrightarrow \qquad \iint_{\TM} |X|^2\,d\mu \leq \frac1{K_{\cal V}} \iint_{\TM} \Bigl( |X_V|^2 + |\nabla_H X|^2 + |\nabla_V X|^2 \Bigr)\,d\mu.
\endeq
\end{Prop}

\begin{Prop}[Log Sobolev inequality in the tangent bundle] \label{logsob}
There is a constant $K_L>0$, only depending on $M$, such that for any probability density $h=h(x,v)$ on $(\TM,\mu)$,
\begeq\label{ineqls}
\iint_{\TM} h\log h\,d\mu \leq \frac1{2 K_L} \iint_{\TM} \left( \frac{|\nabla_H h|^2 + |\nabla_V h|^2}{h}\right)\,d\mu.
\endeq
\end{Prop}

\begin{Rk} The functional appearing in the right hand side of \eqref{ineqls} is the {\bf Fisher information} of $h\mu$ with respect to $\mu$, evaluated with the Sasaki metric, that is, giving equal importance to the horizontal and vertical gradients:
\begeq\label{FI}
I_\mu(h\mu) =  \iint_{\TM} \left( \frac{|\nabla_H h|^2 + |\nabla_V h|^2}{h}\right)\,d\mu.
\endeq
It plays an important role in the theories of logarithmic Sobolev inequalities, concentration inequalities, large deviation principles, optimal transport theory, and has gained in influence in kinetic theory over the last years.
\end{Rk}

\begin{proof}[Proof of Proposition \ref{poincare}]
The proof is based on a classical two-step argument. Without loss of generality assume
that the volume of $M$ is normalized to~1, and as in \eqref{mu}, one has $\mu(dx\,dv) = \mu_x(dv)\,\vol(dx)$.

Each measure $\mu_x$ is a Gaussian measure in the fiber $T_xM\simeq \R^n$, with covariance
matrix $g_x$ bounded from above and below. So $\mu_x(dv)$ satisfies a Poincar\'e inequality, uniformly in $x$, with respect to the vertical gradient:
\begeq\label{p1}
\int_{T_xM} \left | h - \int h\,d\mu_x \right|^2\,d\mu_x
\leq C \int_{T_xM} |\nabla_V h|^2\,d\mu_x,
\endeq

Moreover, since $M$ is compact, the volume measure also satisfies a Poincar\'e inequality :
for any function $u$ on $M$ 
\begeq\label{p2}
\int_M \left| u - \int u\,d\vol \right|^2\,d\vol \leq
C \int |\nabla u|^2\,d\vol;
\endeq

Let $\<h\>_x = \int h\,d\mu_x$; then
\begin{align} 
& \int \left| h - \int h\,d\mu \right|^2\,d\mu 
= \int h^2\,d\mu - \left(\int h\,d\mu\right)^2 \nonumber \\
& = \left[\int h^2\,d\mu - \int \left(\int h\,d\mu_x\right)^2\,\vol(dx)\right]
+ \left[\int \left(\int h\,d\mu_x\right)^2\,\vol(dx)
- \left(\int h\,d\mu\right)^2 \right] \nonumber\\
& = \int_M \left[ \int_{Tx M} h^2\,\mu_x(dv) - \left(\int h\,d\mu_x\right)^2\right]\,\vol(dx)
+ \left[ \int_M \<h\>_x^2\,\vol(dx) - \left(\int \<h\>_x\,\vol(dx)\right)^2\right] \nonumber\\
& \leq C \int_M \int_{T_xM} |\nabla_V h|^2\,\mu_x(dv)\,\vol(dx)
+ C \int_M \bigl| \nabla_x\<h\> \bigr|^2\,d\vol, \label{ppp}
\end{align}
thanks to \eqref{p1} and \eqref{p2}. The first term in the right-hand side
of \eqref{ppp} is nothing but $\int |\nabla_V h|^2\,d\mu$. As for the second term,
by~\eqref{intdh} it is also
\begin{align*}
\int_M \left| \int_{T_xM} \nabla_H h\,d\mu_x \right|^2\,\vol(dx)
& \leq \int_M \int_{T_xM} |\nabla_Hh|^2\,d\mu_x\,\vol(dx) = \int |\nabla_Hh|^2\,d\mu.
\end{align*}
This establishes \eqref{ineqpoinc}.
\end{proof}

\begin{proof}[Proof of Proposition \ref{poincarevect}]
In each fiber $T_xM$, the Poincar\'e inequality for Gaussian measure, applied componentwise to $X(x,\cdot)$, yields
\[ \|X - \<X\>_x \|_{L^2_1(\mu_x)} \leq C\, \|\nabla_V X\|_{L^2(\mu_x)}.\]
(Here this is a weighted Poincar\'e inequality as in Proposition \ref{propgradmom}.) Upon $x$-integration,
\begeq\label{X-avX} \|X - \<X\>_x \|_{L^2_1} \leq C\, \|\nabla_V X\|_{L^2(\mu)}.
\endeq
Here $\<X\>_x$ is a vector field on $M$, valued in $\TM$, and using Proposition \ref{propdiv}(iv), 
\begeq\label{XavXcpt}
\|\<X\>_x\|_{L^2} \leq \|X\|_{L^2},\qquad \|\nabla \<X\>\|\leq \|\nabla_H X\|.
\endeq

Now consider the minimising problem
\begeq\label{minpvect}
\inf \Bigl\{ \|X_V\|_{L^2}^2 + \|\nabla X\|_{L^2}^2; \ \|X\|_{L^2}=1, \ X\bot {\cal P} = 0 \Bigr\}. 
\endeq
If $(X^k)_{k\in\N}$ be a minimising sequence for \eqref{minpvect}, then from \eqref{X-avX} and \eqref{XavXcpt}, $\<X\>_x$ and $X-\<X\>_x$ are bounded in $L^2_1\cap H^1$, hence compact in $L^2(\mu)$; so without loss of generality $X^k$ converges strongly in $L^2(\mu)$ to some vector field $Y=Y(x,v)$, with $\|Y\|_{L^2} = 1$ and $Y\bot{\cal P}$. If the minimum in \eqref{minpvect} is~0, then by lower semi-continuity, $Y_V=0$, $\nabla Y = 0$. This implies that $Y$ is a parallel horizontal vector field, a contradiction. Hence the minimum is positive and \eqref{ineqpoincX} follows.
\end{proof}

\begin{proof}[Proof of Proposition \ref{logsob}]
First start with the Stam--Gross logarithmic Sobolev inequality, in each fiber: For any probability density $h$ on $(\T_xM,\mu_x)$,
\begeq\label{ls1}
\int_{T_xM} h \log h\,d\mu_x 
\leq C \int_{T_xM} \frac{|\nabla_V h|^2}{h}\,d\mu_x,
\endeq
where $C$ is a uniform constant.

On the other hand, there is $C>0$ such that for any probability density $\rho$ on $(M,\vol)$,
\begeq\label{ls2}
\int_M \rho\log\rho\,d\vol \leq 
C \int \frac{|\nabla\rho|^2}{\rho}\,d\vol.
\endeq

Then
\begin{multline} \label{decflogf}
\int h\log h\,d\mu =
\int_M \left[ \left(\int_{T_xM} h\log h\,d\mu_x\right)
- \left(\int_{T_xM} h\,d\mu_x\right) \log \left(\int_{T_xM} h\,d\mu_x\right)
\right]\,\vol(dx) \\
+ \int_M \<h\>_x\log\<h\>_x\,\vol(dx).
\end{multline}
Then apply \eqref{ls1} and (since $\int\<h\>_x\,\vol(dx)=1$) \eqref{ls2}
to bound \eqref{decflogf} by
\[ C \int_M \left(\int_{T_xM} \frac{|\nabla_Vh|^2}{f}\,d\mu_x \right)\,\vol(dx)
+ C \int_M \frac{|\nabla_x\<h\>_x|^2}{\<h\>_x}\,\vol(dx).\]
The first integral is just $\int \frac{|\nabla_V h|^2}{h}\,d\mu$. As for the second one,
use~\eqref{intdh} again to rewrite it as
\[ \int_M \frac{|\nabla_x\<h\>_x|^2}{\<h\>_x}\,\vol(dx)
= \int_M \frac{\bigl| \<\nabla_H h\>_x \bigr|^2}{\<h\>_x}\,\vol(dx)
\leq \int_M \int_{T_xM} \frac{|\nabla_Hh|^2}{h}\,\mu_x(dv)\,\vol(dx),\]
where the last inequality follows from Jensen's inequality, applied to the convex
function $C(a,b)=|a|^2/b$ and the probability measure $\mu_x$.
This concludes the proof of Proposition \ref{logsob}.
\end{proof}

In the sequel of this section I will discuss the possibility to keep the coercivity inequalities while regularising or taming the horizontal gradient. This will be most useful later on.

\begin{Prop}[Coercivity in the tangent bundle with regularised horizontal gradient] \label{propcoerreg}
If $M$ is a compact manifold and $r\geq 0$, then there are $\tilde{K}_P>0$, $\tilde{K}_{\cal V}>0$ such that 

(i) For all functions $h\in L^2(\TM)$,
\[ \|\nabla_V h\|_{L^2(\mu)}^2 + \bigl\| (\Id + (-\Delta_V)^r - \Delta_H)^{-1/2} \nabla_H h \bigr\|_{L^2(\mu)}^2 \geq \tilde{K}_P \|h-\<h\>\|^2_{L^2(\mu)},\]

(ii) For all vector fields $X \in L^2(\TM)$,
\[ X\bot {\cal P} \Longrightarrow\qquad  \|\nabla_V X\|_{L^2(\mu)}^2 + \bigl\| (\Id + (-\Delta_V)^r - \Delta_H)^{-1/2} \nabla_H X \bigr\|_{L^2(\mu)}^2 \geq \tilde{K}_{\cal V} \|X\|^2_{L^2(\mu)}.\]
\end{Prop}

\begin{Rk} An analogous procedure can be realised in terms of Boltzmann and Fisher informations; this more subtle construction is deferred to another work. \end{Rk}

\begin{proof}[Proof of Proposition \ref{propcoerreg}]
Let us prove for instance (i). It amounts to establish a coercivity bound on $-\Delta_V + (\Id + (-\Delta_V)^r - \Delta_H)^{-1} (-\Delta_H)$. All operators involved commute, and the eigenspaces ${\cal E}_s$ for $-\Delta_V$ are invariant and orthogonal. So it suffices to prove the coercivity separately on each ${\cal E}_s$. For $s\geq 1$ this is clear since the operator is bounded below by $-\Delta_V$, itself bounded below by $s\, \Id$. Then ${\cal E}_0/\R$ is made of functions depending only on $x$, and there the operator reduces to $(\Id-\Delta_x)^{-1}(-\Delta_x)$, whose eigenvalues are $\lambda_i/(1+\lambda_i)\geq \lambda_i/2$, if $(\lambda_i)_{i\geq 1}$ are the nonzero eigenvalues of $-\Delta_x$. The conclusion follows by the spectral gap property of $-\Delta_x$.
\end{proof}

As mentioned before, another variant consists in taming the horizontal gradient through a kinetic multiplier. 

\begin{Prop}[Coercivity in the tangent bundle with tamed horizontal gradient] \label{propcoertame}
Let $a=a(|v|)$ such that $\<a\> = \int a(|v|)\,\mu(dv) \neq 0$ and $\|a\|_{L^2}^2 = \int a(|v|)^2\,\mu(dv)<\infty$. Then 

(i) There is $\tilde{K}_P(a)>0$ such that for all $h\in L^2(\mu)$,
\begeq\label{qtroncH}
\|\nabla_V h\|_{L^2(\mu)}^2 + \bigl\| a \nabla_H h \bigr\|_{L^2(\mu)}^2 \geq \tilde{K}_P(a) \|h-\<h\>\|^2_{L^2(\mu)};
\endeq

(ii) There is $\tilde{K}_{\cal V}(a)>0$ such that for all vector fields $X\in L^2(\mu)$,
\begeq\label{qtroncHX}
X\bot {\cal P} \Longrightarrow \qquad \|\nabla_V X\|_{L^2(\mu)}^2 + \bigl\| a \nabla_H X \bigr\|_{L^2(\mu)}^2 \geq \tilde{K}_{\cal V}(a) \|X\|^2_{L^2(\mu)};
\endeq

(iii) If in addition $a$ is nonnegative and bounded, and $h=h(x,v)$ is a probability density on $(\TM,\mu)$ such that 
\[ \rho(x) = \int_{T_xM} h(x,v)\,\mu_x(dv) \]
is bounded above and below by positive constants, then there is $\tilde{K}_L(a,h)>0$, only depending on $a$ and on the abovementioned bounds, such that
\begeq\label{qtroncL}
\iint_{\TM} \frac{|\nabla_V h|^2}{h}\,d\mu + \iint_{\TM} \frac{|\nabla_H h|^2}{h}\,a(|v|)\,d\mu \geq
\tilde{K}_L(a,h)\, \iint_{\TM} h \log h\,d\mu.
\endeq
\end{Prop}

\begin{proof}[Proof of Proposition \ref{propcoertame} (i)--(ii)]
Consider for instance the proof of (i), since (ii) is similar. Since $a$ commutes with $\Delta_H$, one can write $\|a\nabla_H h\|^2 = \<a(-\Delta_H)a h, h\>$. Let
\[ A = -\Delta_V + a (-\Delta_H)a;\]
the goal is to show that if $h$ satisfies $\<h,1\>_{L^2(\mu)} = 0$ then $\<Ah,h\> \geq K \|h\|^2$ for some $K>0$.

Whenever $h$ and $k$ are two functions on $\TM$, write
\[ \<h,k\>_x = \int_{T_xM} h(x,v) k(x,v)\,\mu_x(dv), \qquad \<h\>_x = \int_{T_xM} h(x,v)\,\mu_x(dv) = \<h,1\>_x. \] 
For each $s\in\N_0$, let $\lambda_{1,s}\geq 0$ be the minimum eigenvalue of $\Delta_H$ restricted to ${\cal E}_s$, the $s$-eigenspace for $-\Delta_V$. Then
\[ -\Delta_H \geq \sum_{s\in\N_0} \lambda_{1,s} P_s \geq \lambda_{1,0} P_0,\]
so for any $\theta\in [0,1]$,
\begeq\label{Ageq} A \geq \sum_{s\in\N} s P_s + (\theta \lambda_1) a P_0 a, \endeq
where $\lambda_1=\lambda_{1,0}$ is the spectral gap for $-\Delta_x$ on $M$. Note that $P_0h = \<h\>_x$.

Then let $h$ be such that its mean on $\TM$ is 0, or equivalently $\int_M \<h\>_x\,\vol(dx) = 0$. So
\begin{align*} \<aP_0 ah,h\>_{L^2(\mu)} & = \int_M \<ah\>_x^2\,\vol(dx)\\
& = \int_M \<a,h\>_x^2\,\vol(dx) \\
& \geq \int_M \Bigl[ \frac12 \<a\>^2 \<h\>_x^2 - 2 \<a, h-\<h\>_x\>^2_x \Bigr]\,\vol(dx) \\
& \geq \frac12 \<a\>^2 \left(\int_M \<h\>_x^2\,\vol(dx) \right) - 2 \|a\|_{L^2}^2 \left(\int_M \|h-\<h\>_x\|_{L^2(\mu_x)}^2\,\vol(dx)\right)\\
& \geq \frac12\<a\>^2 \int_M \<h\>_x^2\,\vol(dx) - 2 \|a\|_{L^2}^2 \iint_{\TM} (h-\<h\>_x)^2\,\mu(dx\,dv).
\end{align*}
Inserting this back in \eqref{Ageq},
\begin{align*}  \<Ah,h\>_{L^2(\mu)} & \geq \sum_{s\in\N} \|P_s  h\|^2+ (\theta \lambda_1) \<a P_0 a h, h\>_{L^2(\mu)} \\
& \geq \|h-\<h\>_x\|_{L^2(\mu)}^2 + \frac{\theta \<a\>^2}{2} \int_M \<h\>_x^2\,\vol(dx) - 2\theta \|a\|_{L^2}^2 \|h-\<h\>_x\|_{L^2(\mu)}^2\\
& \geq \max\left( 1- 2\theta \|a\|_{L^2}^2, \frac{\theta \<a\>^2}{2}\right) \Bigl( \|h-\<h\>_x\|_{L^2(\mu)}^2 + \|\<h\>_x\|_{L^2(\mu)}^2\Bigr) \\
& = \max\left( 1- 2\theta \|a\|_{L^2}^2, \frac{\theta \<a\>^2}{2}\right) \|h\|_{L^2(\mu)}^2,
\end{align*}
and the proof is complete by choosing $\theta = (4 \|a\|_{L^2}^2)^{-1}$.
\end{proof}

The information theoretical version will turn out to be a bit more intricate:

\begin{proof}[Proof of Proposition \ref{propcoertame}(iii)]
Recall that $\<h\>_x = \int h\,d\mu_x$, $\<ah\>_x = \int ah\,d\mu_x$, and let $\<\<ah\>\>= \vol (M)^{-1}\int ah\,d\mu$. Note that $\<\<h\>\> = 1$. 

Start with the logarithmic Sobolev inequality in $(T_xM,\mu_x)$:
\begin{align*}
\int_{T_xM} \frac{|\nabla_V h|^2}{h}\,d\mu_x
& = \<h\>_x \int_{T_xM} \frac{|\nabla_V (h/\<h\>_x)|^2}{h/\<h\>_x}\, d\mu_x\\
& \geq 2 \<h\>_x \int_{T_xM} \frac{h}{\<h\>_x} \log \frac{h}{\<h\>_x}\,d\mu_x \\
& = 2 \int_{T_xM} h \log \frac{h}{\<h\>_x}\,d\mu_x,
\end{align*}
so upon integrating against $\vol(dx)$,
\begeq\label{step1qL}
\iint_{\TM} \frac{|\nabla_Vh|^2}{h}\,d\mu \geq
2 \iint_{\TM} h \log \frac{h}{\<h\>_x}\,d\mu.
\endeq

Next, using $\nabla_H a=0$ and the convexity of $(X,Y)\longmapsto |X|^2/Y$, 
\begin{multline*}
\int_{T_xM} \frac{|\nabla_H h|^2}{h}\, a(|v|)\,\mu_x(dv)  = \int_{T_xM} \frac{|\nabla_H (ah)|^2}{ah}\,d\mu_x
\geq \frac{\left| \dps \int \nabla_H(ah)\, d\mu_x \right|^2}{\dps \int (ah)\,d\mu_x} = \frac{|\nabla_x \<ah\>_x|^2}{\<ah\>_x}.
\end{multline*}
Integrating over $x$ and using the Poincar\'e inequality ${\rm P}(K)$,
\begin{align} \label{step2qL}
\iint_{\TM} \frac{|\nabla_H h|^2}{h} a(|v|)\,d\mu 
& \geq \int_M \frac{|\nabla_x \<ah\>_x|^2}{\<ah\>_x}\,\vol(dx) \\ \nonumber
& \geq \frac1{\sup\<ah\>_x} \int_M |\nabla_x \<ah\>_x |^2\,\vol(dx)\\ \nonumber
& \geq \frac{K}{\|a\|_{L^\infty} \sup \<h\>_x} \int_M \bigl( \<ah\>_x - \<\<ah\>\>\bigr)^2\,\vol(dx).
\end{align}
Combining this with \eqref{step1qL},
\begin{multline} \label{step3qL}
\iint_{\TM} \frac{|\nabla_Vh|^2}{h}\,d\mu + \iint_{\TM} \frac{|\nabla_Hh|^2}{h}\,d\mu  \\
\geq \min \left(2,  \frac{K}{\|a\|_{L^\infty} \sup \<h\>_x}\right) \left( \iint_{\TM} h \log \frac{h}{\<h\>_x}\,d\mu + 
\iint_M \bigl( \<ah\>_x - \<\<ah\>\>\bigr)^2\,\vol(dx) \right).
\end{multline}

For any $x\in M$,
\begin{align*}
\<ah\>_x - \<\<ah\>\> 
& = \int_{T_xM} a(|v|) (h-\<h\>_x)\,\mu_x(dv) \\
& + \<a\> \bigl( \<h\>_x -1\bigr)\\
& -  \vol(M)^{-1} \iint_{\TM} a \bigl( h-\<h\>_x\bigr)\,\mu(dx\,dv).
\end{align*}
Thus
\begin{multline} \label{ah-ah}
\bigl( \<ah\>_x - \<\<ah\>\>\bigr)^2 
\geq \frac{\<a\>^2}{2} \bigl(\<h\>_x-1\bigr)^2 \\
- 4 \left( \int_{T_xM} a(|v|) (h-\<h\>_x)\,\mu_x(dv)\right)^2
- 4 \left( \vol(M)^{-1} \iint_{\TM} a(|v|) (h-\<h\>_x)\,\mu(dx\,dv) \right)^2.
\end{multline}
Let us estimate $\int a(h-\<h\>)\,d\mu_x$:
\begin{align} \label{ah-ah1}
\left| \int_{T_xM} a(|v|) (h-\<h\>_x)\,\mu_x(dv) \right|
& \leq \|a\|_{L^\infty}\, \bigl\|h-\<h\>_x\bigr\|_{L^1(\mu_x)} \\ \nonumber
& = \|a\|_{L^\infty} \<h\>_x \left\| \frac{h}{\<h\>_x} - 1 \right\|_{L^1(\mu_x)}\\ \nonumber
& \leq \sqrt{2} \|a\|_{L^\infty} \<h\>_x \left( \int_{T_xM} \frac{h}{\<h\>_x} \log \frac{h}{\<h\>_x}\,d\mu_x\right)^{1/2}\\ \nonumber
& \leq \sqrt{2} \|a\|_{L^\infty} \sqrt{\<h\>_x} \left( \int_{T_xM} h \log \frac{h}{\<h\>_x}\,d\mu_x\right)^{1/2},
\end{align}
where Pinsker's inequality was used to control total variation by Boltzmann's information.
Likewise,
\begin{align} \label{ah-ah2}
\vol(M)^{-1} & \left|  \iint_{\TM} a(||v|) (h-\<h\>_x)\,\mu(dx\,dv) \right| \\ \nonumber
& \leq \sqrt{2} \|a\|_{L^\infty} \sqrt{\<h\>_x} \int_M\left( \int_{T_xM} \frac{h}{\<h\>_x} \log \frac{h}{\<h\>_x}\,d\mu_x\right)^{1/2}\,\frac{\vol(dx)}{\vol(M)}\\ \nonumber
& \leq \sqrt{2} \|a\|_{L^\infty} \sqrt{\<h\>_x} \left(\iint_{\TM} \frac{h}{\<h\>_x} \log \frac{h}{\<h\>_x}\,d\mu\right)^{1/2}.
\end{align}

Inserting \eqref{ah-ah1} and \eqref{ah-ah2} in \eqref{ah-ah},
\begin{multline} \label{ah-ah3}
\int_M \bigl( \<ah\>_x -\<\<ah\>\> \bigr)^2\,\vol(dx) 
 \geq \frac{\<a\>^2}2 \int_M \bigl(\<h\>_x - 1\bigr)^2\,\vol(dx)\\
- 16 \|a\|_{L^\infty}^2 \|\<h\>_x\|_{L^\infty} \iint_{\TM} h \log \frac{h}{\<h\>_x}\,d\mu.
\end{multline}
By Taylor expansion,
\[ \<h\>_x \log \<h\>_x - \<h\>_x + 1 \leq \frac1{2 \inf \<h\>_x} (\<h\>_x-1)^2,\]
so
\begin{align*} \int_M \<h\>_x\log \<h\>_x\,\vol(dx)  & = \int_M \bigl(\<h\>_x \log \<h\>_x - \<h\>_x + 1\bigr)\,\vol(dx) \\
& \leq \frac1{2 \inf \<h\>_x} \int_M (\<h\>_x-1)^2\,\vol(dx).
\end{align*}
Plugging back in \eqref{ah-ah3} and then in \eqref{step3qL}, for any $\theta\in [0,1]$,
\begin{multline*}
\iint_{\TM} \frac{|\nabla_Vh|^2}{h}\,d\mu + \iint_{\TM} \frac{|\nabla_Hh|^2}{h}\,a\,d\mu 
\\ \geq 
\min \left(2, \frac{K}{\|a\|_{L^\infty} \sup \<h\>_x}\right)
\min \Bigl( 1- 16 \theta \|a\|_{L^\infty}^2 (\sup_x \<h\>_x), \ \theta \<a\>^2 (\inf_x \<h\>_x) \Bigr) \\
\left( \iint_{\TM} h \log \frac{h}{\<h\>_x}\,\mu(dx\, dv) + \iint_{\TM} \<h\>_x \log \<h\>_x\,\mu(dx\,dv) \right).
\end{multline*}
The last term in brackets is no else than the full information, $\int h \log h\,d\mu$. Choosing $\theta$ appropriately finally leads to the conclusion
\begin{multline}
\iint_{\TM} \frac{|\nabla_Vh|^2}{h}\,d\mu + \iint_{\TM} \frac{|\nabla_Hh|^2}{h}\,a\,d\mu \\ \geq
\min \left(2, \frac{K}{\|a\|_{L^\infty} \dps\sup_x\, \<h\>_x}\right)
\left( \frac{\<a\>^2\, \dps\inf_x\, \<h\>_x}{ 32 \|a\|_{L^\infty}^2\, \dps \sup_x\, \<h\>_x} \right) \iint_{\TM} h\log h\,d\mu.
\end{multline}
\end{proof}

\bibnotes

The logarithmic Sobolev inequality for Gaussian measure has a long history going back at least to Stam \cite{stam:59}, and made popular by Gross \cite{gross:log:75}; see e.g.~\cite{BGL:book} for much more on the history, context and developments. As for the Poincar\'e inequality for the Gaussian measure, it is even much older, and follows anyway from Proposition \ref{propgradmom}. 

Rothaus \cite{rothaus:massgap:81} proved that any compact Riemannian manifold satisfies a log Sobolev inequality, a result which is stronger than the Poincar\'e inequality. The tensorisation argument is classical to establish those inequalities in product spaces, both for Poincar\'e and log Sobolev inequalities (again, see e.g.~\cite{BGL:book}). As seen above it does not present any difficulty to adapt to the tangent bundle.

I am not aware of any definitive reference for the vector-valued Poincar\'e inequality, but the question is so natural that it has necessarily been considered beforee. When the manifold has positive Ricci curvature there is a well-known more precise quantitative upper bound, as a consequence of Bochner's formula for vector fields:
\[ \<X,\Delta X\> - \Delta \frac{|X|^2}{2} = \|\nabla X\|^2 + \Ric(X,X)\]
(often stated for 1-forms; see \cite{GHL:Riemann:book}, or \cite[Chapter~14]{vill:oldnew}), which implies
\[ \int |\nabla X|^2 \geq K_{\cal V} \int |X|^2\]
if $K_{\cal V}$ is the infimum of the eigenvalues of the Ricci tensor.

The terminology ``Poincar\'e inequality for vector fields'' should not be mistaken with another homonymous problem, namely controlling the deviation of a scalar function from its mean, by the derivation along a collection of vector fields ($\|u-\<u\>\|\leq C \sum \|X_i u\|$), as in Jerison \cite{jerison:poincarehormander:86} or Albritton--Armstrong--Mourrat--Novack~\cite{AAMN:KFP:24}. In particular, the latter authors prove, in flat geometry and under various assumptions, an inequality of the form
\[ \|f -\<f\> \|_{L^2} \leq C \bigl( \|\nabla_V h\|_{L^2} + \|\xi h\|_{H^{0,-1}}\bigr). \]
This strategy was vastly expanded through a series of papers by Mouhot and his co-authors Anceschi, Dietert, Guerand, Loher, Niebel, Rebucci and Zacher into a kinetic version of De Giorgi's method for controlling oscillations, in low regularity; see Mouhot \cite[Section~4]{mouhot:Festum:24} for a review.

Fisher information was introduced by Ronald Fisher \cite{fisher:25} in 1925 for his theory of efficient statistics; my course \cite{vill:fisher-Festum:25} reviews its use in kinetic theory and points references to other fields of mathematics and physics.

The Pinsker (or Csisz\'ar--Kullback--Pinsker) inequality relates entropy and $L^1$ norm (or total variation): If $\nu$ is any reference measure and $f, \tilde{f}$ are any two probability densities, then 
\begeq\label{CKP} \int f \log \left(\frac{f}{\tilde{f}}\right)\,d\nu \geq \frac12 \bigl\|f-\tilde{f}\bigr\|_{L^1(\nu)}^2.\endeq
It is proven in many references, such as \cite{csi:inf:67,pinsker:book:64,vill:int}.

The conditions for existence of parallel vector fields were established by Welsh \cite{welsh:manifoldparallel:86}.

Proposition \ref{propcoerreg} is natural to ask for in view of the strategy of Dolbeault--Mouhot--Schmeiser \cite{DMS:hypo:15} and will be used in Section \ref{sechypoco}.

Proposition \ref{propcoertame}(i) in the particular case $M=\T^n$ follows from \cite[Theorem~A.3]{vill:hypoco}. The proof here goes along the same general lines as that reference. Proposition \ref{propcoertame}(iii) is new as far as I know and will serve in Section \ref{secfisher}.


\section{Hypocoercivity in $L^2$ and Sobolev norms} \label{sechypoco}

It is now time to consider the convergence to equilibrium for the geometric kinetic Fokker--Planck equation. I shall distinguish three main methods for equilibration: from more specialised to more general,
\sm

\bul Harris-type estimates, developed by Meyn--Tweedie for continuous time diffusions, rest on positivity and localisation; they are usually limited to linear models;
\sm

\bul Hilbertian techniques, generalising spectral gap estimates, also cover linearised equations;
\sm

\bul Information estimates, \`a la Boltzmann, are usually required to handle statistical equations far from equilibrium.
\sm

\begin{Rk} \label{rkmoments} For the spatially homogeneous Fokker--Planck equation in $\R^n_v$, there are other available options. One is to diagonalise the operator and evolution in a basis of Hermite polynomials. Another one is to establish closed differential equations on the moments $\int P(v) f(t,v)\,dv$, for polynomials $P$, showing each of them converges to the equilibrium value. Both approaches disappear when $x$ varies in a general manifold $M$. Even moments do not satisfy close equations then, among other things because $\int P(v) f(t,x,v)\,dv\,dx$ does not even make sense if $P$ is nonradial.
\end{Rk}

In this section only $L^2$ problem will be considered, either for functions or vectors. Hypocoercivity in $L^2$ will first be obtained by combining the short-time regularisation with a hypocoercivity in Sobolev norm. Recall that $\<h\> = \int h\,d\mu$.

\begin{Thm}[Hypocoercivity in $L^2$] \label{thmhypocoL2}
Let $L=\xi-\Delta_V^\mu=\xi - \Delta_V + v\cdot\nabla_V$ acting on $L^2(\mu)$.
Then there are constants $A=A(M)$, $\alpha=\alpha(M)>0$ such that for any $h_0\in L^2(\mu)$,
\begeq\label{cvgceL2}
\bigl\| e^{-tL} h_0 - \<h_0\> \bigr\|_{L^2(\mu)}
\leq A\, e^{-\alpha t}\, \|h_0\|_{L^2(\mu)}.\qquad
\endeq
\end{Thm}

For vector fields the convergence will only be proven when the diffusion is strong enough (which can also be recast, via rescaling, as a smallness assumption on the curvature). Recall that $R$ is the curvature tensor and that ${\cal P}$ stands for the space of horizontal parallel vector fields.

\begin{Thm}[Hypocoercivity for vector fields in $L^2$] \label{thmhypocoL2vect}
For $\lambda>0$, let ${\cal B}_\lambda=\Xi-\lambda (\Delta_V^\mu- P_V)$, acting on $L^2(\mu;\Tzeroun)$.
Then there are constants $A=A(M), \Lambda>0$, $\alpha=\alpha(M)>0$ such that if $\lambda \geq \Lambda (1+\|R\|_{C^1})^4$ then for any $X_0\bot{\cal P}$,
\begeq\label{cvgceL2vect}
\bigl\| e^{-t\cal B} X_0  \bigr\|_{L^2(\mu)}
\leq A\, e^{-\alpha t}\, \|X_0\|_{L^2(\mu)}.\qquad
\endeq
\end{Thm}

\begin{Rk} When $\lambda\to\infty$ the underlying flow is governed in some sense by the heat equation, recall the discussion leading to \eqref{dd}, so it is not expected that $\alpha$ becomes large as $\lambda\to \infty$. Instead, the limiting value should be determined from the spectrum of the heat equation on vector fields.
\end{Rk}

\begin{proof}[Proof of Theorem \ref{thmhypocoL2}]
Let $h(t)=e^{-tL}h$.
Without loss of generality $\<h_0\> =0$, and then $\<h(t)\>=0$ for all $t\geq 0$. By Theorem \ref{thmloc}(ii), $\|h(t)\|_{L^2}\leq \|h_0\|$, so it is sufficient to prove the theorem when $t\geq 1$. By Theorem \ref{thmregul}(i), $\| h(1)\|_{H^1} \leq C \|h_0\|_{L^2}$, so we may assume that the initial datum lies in $H^1$.

Let 
\begeq \label{2parentheses}
((h,h)) = \iint h^2\,d\mu + a \iint |\nabla_V h|^2\,d\mu + 2b \iint \<\nabla_V h,\nabla_H h\>\,d\mu + c \iint |\nabla_H h|^2\,d\mu,
\endeq
where $a,b,c>0$, to be chosen later, satisfy $b^2< ac$; then there is $C>0$ such that
\[ C^{-1} \|h\|^2_{H^1} \leq ((h,h)) \leq C \|h\|^2_{H^1}. \]

By the same computation as in the proof of Theorem \ref{thmregul}(i), more precisely \eqref{diag1}, \eqref{diag2}, \eqref{diag3'},
\begin{align*}
\frac{d}{dt} & \left( \iint h^2 + a \iint |\nabla_Vh|^2 + \iint |\nabla_H h|^2 \right)\\
& \leq -K \Bigl( \|\nabla_Vh\|^2 + a \|\nabla_V^2 h\|^2 + c \|\nabla_V\nabla_H h\|^2 \Bigr)\\
& \qquad \qquad + C \Bigl( a \|\nabla_Vh\|\, \|\nabla_Hh\| + c (\|\nabla_V h\| + \|\nabla_V^2 h\| ) (\|\nabla_Hh\| + \|\nabla_V\nabla_H h\|)\Bigr) \\
& \leq -K \bigl( \|\nabla_V h\|^2 + a \|\nabla_V^2 h\|^2 + c\| \nabla_V\nabla_H h\|^2 \bigr) + C\, \max \left(a^2, \frac{c^2}{a}\right) \|\nabla_H h\|^2.
\end{align*}

On the other hand, as in \eqref{ddtmix},
\begin{align*} \frac{d}{dt} \iint \< \nabla_V h , \nabla_H h\> 
& \leq - \|\nabla_H h\|^2 + C \Bigl( \|\nabla_V^2h\| \|\nabla_V\nabla_H h\| + \|\nabla_V h\| \|\nabla_H h\| + \|\nabla_V h\|^2\Bigr )\\
& \leq -K \|\nabla_Hh\|^2 + C \bigl(b \|\nabla_Vh\|^2 + b \|\nabla_V^2 h\| \|\nabla_V\nabla_H h\| \bigr).
\end{align*}

All in all,
\begeq\label{ddthh}
\frac{d}{dt} ((h,h)) \leq -K \Bigl( a \|\nabla_Vh\|^2 + b \|\nabla_H h\|^2 + c \|\nabla_V\nabla_Hh\|^2\Bigr)
\endeq
provided that
\[ a^2 \ll b\qquad \frac{c^2}{a} \ll b \qquad b \ll \sqrt{ac}. \]
Those conditions boil down to just two: $a^2\ll b$ and $b^2\ll ac$, which is a particular case of Lemma \ref{ll}.

Combining \eqref{ddthh} with the Poincar\'e inequality in tangent space, Proposition \ref{poincare}, there is $\alpha>0$ such that
\begeq\label{ddthh2}
\frac{d}{dt} ((h,h)) \leq -2\alpha\, ((h,h)),
\endeq
which implies the exponential convergence of $((h,h))$ to~0 and the claim since $((h,h))\geq \|h\|^2$.
\end{proof}

\begin{Rk} A key commutation property used was
\begin{align*} \left. \frac{d}{dt}\right|_{-\xi} \<\nabla_V h, \nabla_H h\> & = 
- \iint_{\TM} |\nabla_H h|^2 \,d\mu + \iint_{\TM} \<\Sigma(v,v) \nabla_V h, \nabla_V h\>\\\
& = - \iint_{\TM} |\nabla_H h|^2\,d\mu + \iint_{\TM} \sigma_x(v,\nabla_Vh) |v\wedge \nabla_V h|^2\, d\mu,
\end{align*}
where $\sigma_x(a,b)$ is the sectional curvature in the plane generated by $a$ and $b$, and $|a\wedge b|^2 = |a|^2 |b|^2 - \<a,b\>^2$. So the commutator is more favorable, for the equilibration, when the sectional curvature is negative. This is reminiscent of Anosov theory, where negative sectional curvature is associated with the mixing of the geodesic flow. In contrast, for the heat equation, estimates are all the better as curvature is positive, and such equilibration estimates are usually expressed in terms of Ricci rather than sectional curvature.
\end{Rk}

\begin{proof}[Proof of Theorem \ref{thmhypocoL2vect}]
Now the equation is
\begeq\label{Blambda}
\begin{cases}
\dps \pa_t X_H + \xi X_H - \Sigma(v,v) X_V = \lambda \Delta_V^\mu X_H \\[2mm]
\dps \pa_t X_V + \xi X_V + X_H = \lambda \bigl( \Delta_V^\mu X_V - X_V\bigr).
\end{cases}
\endeq

Let us assume $X_0\bot {\cal P}$; then $X(t)\bot {\cal P}$ for all $t\geq 0$, since ${\cal P}$ is left invariant by the semigroup and its adjoint. 
By Theorem \ref{thmloc}(iii), $\|X(t)\|_{L^2}\leq C \|X_0\|$, so it suffices to prove the theorem when $t\geq 1$. By Theorem \ref{thmregul}(i), $\| e^{-L}X_0\|_{H^1} \leq C \|X_0\|_{L^2}$, so one may assume that the initial datum lies in $H^1$.

It follows from \eqref{Blambda} that
\begin{align}\label{ddtL2vect}
\frac{d}{2\,dt} & (\|X_H\|^2 + \|X_V\|^2)  \\ \nonumber
&  \leq -\lambda \bigl( \|X_V\|^2 + \|\nabla_V X_H\|^2 + \|\nabla_V X_V\|^2\bigr)\\ \nonumber
& \qquad\qquad\qquad
+ C (1+ \|R\|_{C^1}) \bigl (\|\nabla_V X_V\| + \|\nabla_V^2 X_V\| ) (\|X_H\| + \|\nabla_V X_H\|) \\
\nonumber
& \leq -\lambda \bigl( \|X_V\|^2 + \|\nabla_V X_H\|^2 + \|\nabla_V X_V\|^2\bigr)
+ \var \|X_H\|^2 + C (1+\|R\|_{C^1})^2 \|\nabla_V X_V\|^2,
\end{align}
where
\[ \lambda'= \lambda - \frac{C}{\var^2} (1+\|R\|_C^1)^2.\]

Now consider the first order terms:
\begeq \label{firstordervectlambda}
\begin{cases}
\dps \pa_t \nabla_H X_H + \xi \nabla_H X_H + [\nabla_H,\xi] X_H - \nabla_H (\Sigma(v,v)X_V) = \lambda \Delta_V^\mu \nabla_H X_H \\[2mm]
\dps \pa_t \nabla_V X_V + \xi \nabla_V X_H + \nabla_H X_H - \nabla_V (\Sigma(v,v)X_V) = \lambda \Delta_V^\mu \nabla_V X_H - \lambda \nabla_V X_H \\[2mm]
\dps \pa_t \nabla_H X_V + \xi \nabla_H X_V + [\nabla_H,\xi] X_V + \nabla_H X_H = \lambda \Delta_V^\mu \nabla_H X_V - \lambda \nabla_H X_V \\[2mm]
\dps \pa_t \nabla_V X_V + \xi \nabla_V X_V + \nabla_H X_V + \nabla_V X_H = \lambda \Delta_V^\mu \nabla_V X_V - 2\lambda \nabla_V X_V.
\end{cases}
\endeq
From there and the same computations as before, it follows that if $a,a',b,b',c,c'$ are well chosen (with $a\gg a'\gg b,b'\gg c\gg c'>0$ and  $b^2\ll ac$, $(b')^2\ll a'c'$, as in the proof of Theorem \ref{thmregul}(ii)),
\begin{align}\label{ddtsymvect}
\frac{d}{dt} & \Bigl( a \frac{\|\nabla_V X_V\|^2}{2} + a'\frac{\|\nabla_VX_H\|^2}2 + b \<\nabla_V X_V,\nabla_HX_V\> + b'\<\nabla_V X_H, \nabla_H X_H \>  \\ \nonumber
& \qquad\qquad\qquad\qquad\qquad\qquad\qquad\qquad\qquad\qquad\qquad + c\frac{\|\nabla_HX_V\|^2}2 + c'\frac{\|\nabla_HX_H\|^2}2 \Bigr) \\ \nonumber
& \leq - \lambda' \Bigl( \|\nabla_VX_V\|^2 + \|\nabla_H X_V\|^2 + \|\nabla_V^2 X_V\|^2 + \|\nabla_V^2 X_H\|^2 + \|\nabla_V\nabla_H X_V\|^2 + \|\nabla_V\nabla_H X_H\|^2 \Bigr) \\ \nonumber
& \qquad - K (\|\nabla_HX_H\|^2 + \|\nabla_H X_V\|^2) + C \|\nabla_VX_H\|^2,
\end{align}
where now $\lambda' = \lambda - C (1+\|R\|_{C^1}^4)$.
Adding \eqref{ddtL2vect} and \eqref{ddtsymvect} yields, for $\lambda$ large enough,
\begeq
\frac{d}{dt} ((X,X)) \leq - 2\alpha ((X,X)),
\endeq
where
\begin{multline} \label{((XX))}
((X,X)) = \|X\|^2 + a \frac{\|\nabla_V X_V\|^2}{2} + a'\frac{\|\nabla_VX_H\|^2}2 + b \<\nabla_V X_V,\nabla_HX_V\> + b'\<\nabla_V X_H, \nabla_H X_H \> \\
+ c\frac{\|\nabla_HX_V\|^2}2 + c'\frac{\|\nabla_HX_H\|^2}2.
\end{multline}
From the choice of coefficients, this is equivalent to the squared $H^1$ norm of $X$, and the result follows.
\end{proof}

I will conclude this section with the adaptation of the Dolbeault--Mouhot--Schmeiser strategy, providing an alternative proof of Theorems \ref{thmhypocoL2} and \ref{thmhypocoL2vect}. I will consider for instance Theorem \ref{thmhypocoL2}, the strategy is similar for Theorem \ref{thmhypocoL2vect}.

\begin{proof}[Alternative proof of Theorem \ref{thmhypocoL2}]
For any $h$ with zero average, define
\[ {\cal F}(h) = \|h\|_{L^2(\mu)}^2 + \delta \Bigl\< \nabla_V h, (\Id-\Delta_V-\Delta_H)^{-1} \nabla_H h\Bigr\>_{L^2(\mu)}. \]
In the sequel, the measure $\mu$ will be implicit. 
First note that 
\begin{align*}
\Bigl| \Bigl\< \nabla_V h, (\Id-\Delta_V^\mu - \Delta_H)^{-1}  \nabla_H h \Bigr\>_{L^2} \Bigr| 
& \leq \Bigl\| (\Id-\Delta_V^\mu -\Delta_H)^{-1/2} \nabla_V h\Bigr\|_{L^2} \\
& \leq \|h\|_{L^2}^2,
\end{align*}
so that
\begeq\label{1-deltaL2}
(1-\delta) \|h\|_{L^2}^2 \leq {\cal F}(h) \leq (1+\delta) \|h\|_{L^2}^2.
\endeq

Applying $\pa_t h = -\xi h + \Delta_V^\mu h$, 
\begin{align} \label{dFdt}
\frac{d{\cal F}(h(t))}{dt} 
= & - 2 \|\nabla_Vh\|_{L^2}^2 \\ \nonumber
& + \delta \Bigl\< \nabla_V \Delta_V^\mu h, (\Id-\Delta_V^\mu-\Delta_H)^{-1} \nabla_H h \Bigr\>_{L^2}
+ \delta \Bigl\< \nabla_Vh, (\Id-\Delta_V^\mu-\Delta_H)^{-1} \nabla_H \Delta_V^\mu h \Bigr\>_{L^2}\\ \nonumber
& - \delta \Bigl\< \nabla_V\xi h, (\Id-\Delta_V^\mu -\Delta_H)^{-1} \nabla_H h\Bigr\>_{L^2} 
- \delta \Bigl\< \nabla_V h, (\Id-\Delta_V^\mu -\Delta_H)^{-1} \nabla_V \xi h\Bigr\>_{L^2}.
\end{align}
Using the commutation of $\Delta_V^\mu$ with $\Delta_H$ and the commutator properties of the geodesic field $\xi$ with $\nabla_H$ and $\nabla_V$ (Propositions \ref{propHV} and \ref{propcomH}),
\begin{align} \label{dFdt2}
\frac{d{\cal F}(h(t))}{dt} 
= & - 2 \|\nabla_Vh\|_{L^2}^2 \\ \nonumber
& + 2\delta \Bigl\< \nabla_V h, (\Id-\Delta_V^\mu -\Delta_H)^{-1}\Delta_V^\mu \nabla_H h\Bigr\> \\ \nonumber
& - \delta \Bigl( \bigl\< \xi\nabla_V h, (\Id-\Delta_V^\mu -\Delta_H)^{-1} \nabla_H h\bigr\> 
+ \bigl\< \nabla_Vh, \xi (\Id-\Delta_V^\mu-\Delta_H)^{-1} \nabla_H h\bigr\> \Bigr)\\ \nonumber
& - \delta \bigl\<\nabla_H h, (\Id-\Delta_V^\mu-\Delta_H)^{-1} \nabla_H h\bigr\> \\ \nonumber
& -\delta \Bigl( \bigl\< \nabla_Vh, (\Id-\Delta_V^\mu-\Delta_H)^ {-1} \Sigma(v,v) \nabla_V h\bigr\>
+ \bigl\< \nabla_V h, \bigl[ (\Id-\Delta_V^\mu-\Delta_H)^{-1}, \xi\bigr] \nabla_H h\bigr\> \Bigr).
\end{align}

To bound the second line in \eqref{dFdt2}, note that
\begin{align*}
\Bigl|  \Bigl\< \nabla_V h, (\Id-\Delta_V^\mu -\Delta_H)^{-1}\Delta_V^\mu \nabla_H h\Bigr\> \Bigr| 
& \leq \|\nabla_V h\|\, \bigl\| (\Id-\Delta_V^\mu-\Delta_H)^{-1/2} \Delta_V^\mu h\bigr\| \\
& \leq \|\nabla_V h\|^2.
\end{align*}

Next, the third line in \eqref{dFdt2} vanishes by the derivation property for $\xi$.

The fourth line is equal to
\[ - \delta \Bigl\| (\Id-\Delta_V^\mu-\Delta_H)^{-1/2} \nabla_H h\Bigr\|^2.\]

As for the fifth line in \eqref{dFdt2}, the first part is bounded by
\begin{align*}  \|\nabla_V h\|\,  \bigl\| (\Id-\Delta_V^\mu-\Delta_H)^{-1} & (\Sigma(v,v) \nabla_V h) \bigr\| \\
\leq & \|\nabla_V h\|\, \|\Sigma(v,v)\nabla_V h\|_{H^{0,-2}} \\
\leq & C \|\nabla_V h\|^2,
\end{align*}
in view of Proposition \ref{propgradmom}. To evaluate the second part, first note that
\begeq\label{Id-Delta-Delta}
[(\Id-\Delta_V^\mu -\Delta_H)^{-1}, \xi ] = (\Id-\Delta_V^\mu -\Delta_H)^{-1} [ \Id-\Delta_V^\mu -\Delta_H, \xi]\, (\Id-\Delta_V^\mu -\Delta_H)^{-1}.
\endeq
Then by Propositions \ref{propxiboundnomom}(i) and \ref{propboundsxi}(v),
\[ \Bigl\| [ \Id-\Delta_V^\mu -\Delta_H, \xi] h \Bigr\|_{H^{\alpha,\beta}} \leq C \| h\|_{H^{\alpha+1,\beta+3}}, \]
combining this with $\|(\Id-\Delta_V^\mu -\Delta_H)^{-1} h\|_{H^{\alpha,\beta}} \leq \|h\|_{H^{\alpha-\theta,\beta-(2-\theta)}}$ yields
\[ \Bigl \| (\Id-\Delta_V^\mu -\Delta_H)^{-1} [ \Id-\Delta_V^\mu -\Delta_H, \xi]\, (\Id-\Delta_V^\mu -\Delta_H)^{-1} \Bigr\|_{L^2} 
\leq C \|h\|_{L^2}. \]
All in all, the fifth line in \eqref{dFdt2} is bounded by $C\|\nabla_Vh\|\, \|h\|$.

In conclusion,
\begin{align*}
\frac{d{\cal F}}{dt} & \leq - 2(1-C\delta) \|\nabla_V h\|^2 + C \delta\, \|\nabla_V h\|\, \|h\| - \delta \Bigl\| (\Id-\Delta_V^\mu-\Delta_H)^{-1/2} \nabla_H h\Bigr\|^2\\
& \leq -2 \left(1-C\delta - \frac12\right) \|\nabla_V h\|^2 + C \delta^2 \|h\|^2 - \delta \Bigl\| (\Id-\Delta_V^\mu-\Delta_H)^{-1/2} \nabla_H h\Bigr\|^2.
\end{align*}
If $\delta<(4\max (C,1))^{-1}$, then
\[ \frac{d{\cal F}}{dt} \leq - \delta \Bigl( \|\nabla_V h\|^2 + \bigl\| (\Id-\Delta_V^\mu-\Delta_H)^{-1/2} \nabla_H h\bigr\|^2 \Bigr) + C \delta^2 \|h\|^2.\]
Applying Proposition \ref{propcoerreg},
\[ \frac{d{\cal F}}{dt} \leq - \delta \tilde{K}_P \|h\|^2 + C \delta^2 \|h\|^2.\]
Choosing $\delta$ small enough yields
\[ \frac{d{\cal F}}{dt} \leq - \frac{\delta \tilde{K}_P}2 \|h\|^2.\]
Combining this with \eqref{1-deltaL2}, there is $\kappa>0$ such that
\[ \frac{d{\cal F}}{dt} \leq - 2\kappa {\cal F},\]
hence ${\cal F}$ converges exponentially fast to~0, and by \eqref{1-deltaL2} again, $h$ converges also exponentially fast to~0 in $L^2$. Thus the proof is complete.
\end{proof}

\bibnotes

An overview of the equilibration problem for nonlinear statistical equations can be found in \cite{MV:companion,vill:handbook:02,vill:hypoco}. As for Remark \ref{rkmoments} mentioning the moment method for spatially homogeneous equation, it also applies to the spatially homogeneous Boltzmann equation \cite{iktru:56}, to which the Fokker--Planck equation is closely connected \cite{bob:theory:88}.

The core method in this section is taken from \cite[Part I]{vill:hypoco}, which was expanded in \cite{DOV:preprint}.

The behaviour of the bottom of the spectrum of ${\cal B}_\lambda$ when $\lambda\to\infty$ is considered by various authors \cite{bismut:hypoelliptic:05,bismutlebeau:hypo:book,NSW:25}. It would be good to see how far the methods in these notes can be pushed in this direction.

To establish hypocoercivity, Albritton--Armstrong--Mourrat--Novack~\cite{AAMN:KFP:24} prove (in flat geometry and under some assumptions) a Poincar\'e inequality of the form
\[ \|f -\<f\> \|_{L^2} \leq C \bigl( \|\nabla_V h\|_{L^2} + \|\xi h\|_{H^{0,-1}}\bigr), \]
suggesting that the quantity on the right hand side is an intrinsic and natural norm to work with. This method involving negative regularity is in contrast with the one that I presented here, only based on ``natural'' derivatives. They also ask whether there would be a plain coercivity inequality in that norm.

There is a large literature on mixing properties of the geodesic flow in negative sectional curvature, starting with Anosov \cite{anosov:geodesic:67}.

The Dolbeault--Mouhot--Schmeiser method is exposed in \cite{DMS:hypo:15}; the adaptation to the geometric kinetic Fokker--Planck equation was devised for these notes.

\section{Qualitative spectral description} \label{secqualspectr}

Throughout the notes I relied very little on spectral theory, and in particular did not use any spectral information to prove convergence to equilibrium. However, both for the intuition and for connection with other parts of the literature, it is enlightening to evoke the structure of the spectrum. I will do this only in $L^2(\mu)$, for $L$ (on functions) and ${\cal B}$ (vector fields), or more precisely ${\cal B}_\lambda$ as in \eqref{thmhypocoL2vect}. I will denote by $z$ the potential eigenvalue of the operator.

For a start, Section \ref{sechypo} provides information about the bottom of the spectrum, near the line $\Re z =0$ in complex plane:

\bul For functions, there is an isolated eigenvalue at~0 for $L$, corresponding to constant functions, and there is a spectral gap: all the rest of the spectrum lies in $\{\Re z \geq K\}$ for some $K>0$;

\bul For vector fields, there may be an eigenvalue at~0 for ${\cal B}_\lambda$, corresponding to the vector field of parallel horizontal vector fields if they exist. If~$\lambda$ is large enough then that eigenvalue is isolated, and there is a spectral gap: all the rest of the spectrum lies in $\{\Re z \geq K \}$ for some $K>0$.
\sm

On the other hand, Section \ref{seclebeau} on the maximal hypoellipticity (or its counterpart in terms of semigroup, Section \ref{secreg}) provides information about the top of the spectrum, meaning large eigenvalues. In general the spectrum will not be contained in any angular sector of the form $\{|\Re z| \geq K |\Im z|\}$ (thus $L$ will be called a {\bf nonsectorial operator}), but still the real part will grow to infinity as a fractional power of the imaginary part, at least:

\begin{Thm}[Localisation of the spectrum] \label{thmlocalspectrum}
There are constants $A,K>0$ such that for any $z$ in $\sigma(L)$, or in $\sigma({\cal B})$,
\begeq\label{localspectrum}
|z| \geq A \Longrightarrow \qquad |\Re z| \geq K |\Im z|^{1/2}.
\endeq
\end{Thm}

This fractional growth applies similarly for $L$ and ${\cal B}_\lambda$, whatever $\lambda>0$. The constant $K$ itself may depend on $\lambda$, but can be chosen uniformly as $\lambda\to\infty$. Such estimates were first obtained by Lebeau and will be retrieved here with more elementary methods.

\begin{figure}
\def\svgwidth{0.5\textwidth}
\begingroup%
  \makeatletter%
  \providecommand\color[2][]{%
    \errmessage{(Inkscape) Color is used for the text in Inkscape, but the package 'color.sty' is not loaded}%
    \renewcommand\color[2][]{}%
  }%
  \providecommand\transparent[1]{%
    \errmessage{(Inkscape) Transparency is used (non-zero) for the text in Inkscape, but the package 'transparent.sty' is not loaded}%
    \renewcommand\transparent[1]{}%
  }%
  \providecommand\rotatebox[2]{#2}%
  \newcommand*\fsize{\dimexpr\f@size pt\relax}%
  \newcommand*\lineheight[1]{\fontsize{\fsize}{#1\fsize}\selectfont}%
  \ifx\svgwidth\undefined%
    \setlength{\unitlength}{595.27559055bp}%
    \ifx\svgscale\undefined%
      \relax%
    \else%
      \setlength{\unitlength}{\unitlength * \real{\svgscale}}%
    \fi%
  \else%
    \setlength{\unitlength}{\svgwidth}%
  \fi%
  \global\let\svgwidth\undefined%
  \global\let\svgscale\undefined%
  \makeatother%
  \begin{picture}(1,1.41428571)%
    \lineheight{1}%
    \setlength\tabcolsep{0pt}%
    \put(0,0){\includegraphics[width=\unitlength,page=1]{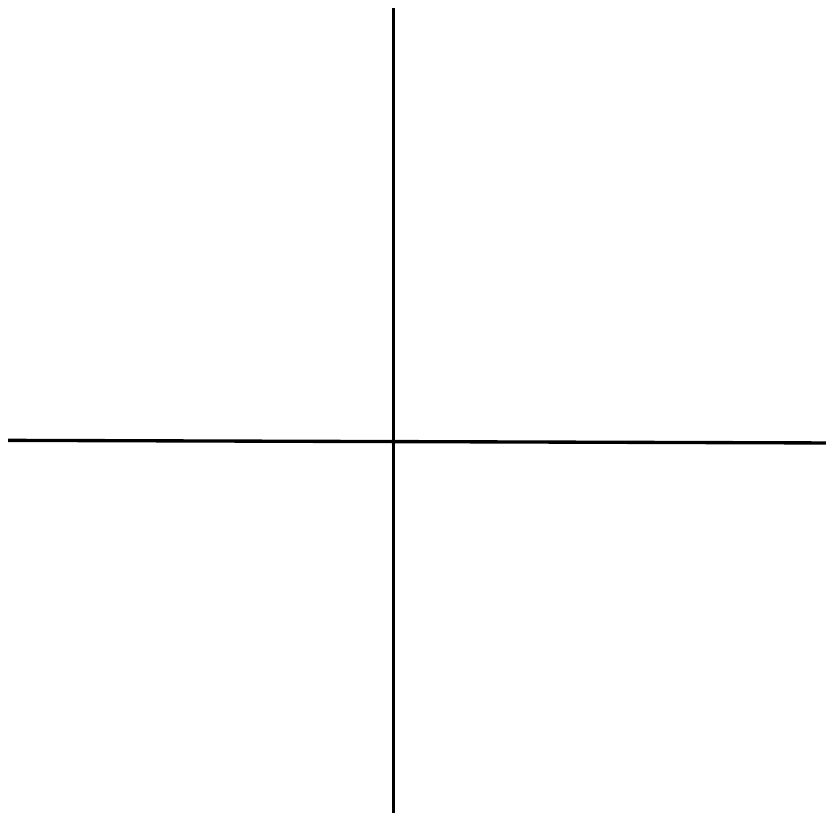}}%
    \put(0.78989766,0.89570232){\color[rgb]{0,0,0}\makebox(0,0)[lt]{\lineheight{1.75375009}\smash{\begin{tabular}[t]{l}$\Re$\end{tabular}}}}%
    \put(0.38010977,1.30298285){\color[rgb]{0,0,0}\makebox(0,0)[lt]{\lineheight{1.75374711}\smash{\begin{tabular}[t]{l}$\Im$\end{tabular}}}}%
    \put(0,0){\includegraphics[width=\unitlength,page=2]{spectrum.pdf}}%
  \end{picture}%
\endgroup%

\vspace*{-45mm}

\caption{Localisation of the spectrum of $L$ on $(\ker L)^\bot$, or the spectrum of ${\cal B}_\lambda$ on $(\ker {\cal B}_\lambda)^\bot$ for large $\lambda$, loosely sketched: the spectrum cannot approach the imaginary axis (vertical dashed line) and is bounded within an ``antiparabola'', $|\Im z| = O (|\Re z|^2).$}
\label{figspectrum}

\end{figure}

Before proving estimate \eqref{localspectrum}, let us first rewrite the eigenvalue equation. Turn to complex values, write $h=F+iG$ for the function, with $F,G$ real-valued, and $z=\alpha+i\beta$ for the eigenvalue, with $\alpha,\beta\in\R$. The eigenvalue equation
\[ L(F+iG) = (\alpha+i\beta) (F+iG)\]
becomes
\[ 
\begin{cases} L F =\alpha F -\beta G \\
LG = \alpha G + \beta F,
\end{cases}
\]
thus
\begin{align*}
L^2 F & = \alpha LF - \beta LG \\
& = \alpha LF - \beta (\alpha G +\beta F) \\
& = \alpha LF -\alpha (\alpha F - LF)- \beta^2 F,
\end{align*}
or
\begeq\label{eigeneq}
L^2 F - 2\alpha LF + (\alpha^2 +\beta^2) F = 0.
\endeq
Since $L+L^*\geq 0$, necessarily $\alpha\geq 0$; also without loss of generality we may assume $\beta\geq 0$; the goal is to get a lower bound on $\alpha$ as $\beta\to\infty$.

What to expect from \eqref{eigeneq}? Obviously,
\[ \bigl\| L^2 F + (\alpha^2 +\beta^2) F \bigr\|^2 = 4 \alpha^2 \|LF\|^2\]
so
\begeq\label{soL2F}
 \|L^2F\|^2 + 2 (\alpha^2 +\beta^2) \<L^2F,F\> + (\alpha^2 +\beta^2)^2 \|F\|^2 = 4 \alpha^2 \|LF\|^2.
 \endeq

For heuristic purposes, let us cheat by pretending that the cross-product with $\<L^2 F, F\>$ in the left hand side is negligible; so \eqref{soL2F} implies
\begeq\label{soL2F'}
 \|L^2F\|^2 + (\alpha^2 +\beta^2)^2 \|F\|^2 \leq C \alpha^2 \|LF\|^2.
 \endeq
Combining this with the skewed interpolation inequality \eqref{skewedinterp}, with $\theta=1/4$,
\begin{align*}
\|L^2F\|^2 + (\alpha^2 +\beta^2)^2 \|F\|^2 
& \leq C \alpha^2 \bigl( \|L^2F\|^{2(1-\theta)} \|F\|^{2\theta} + \|F\|^2 \bigr)\\
& \leq C \var \|L^2F\|^2 + C (\alpha^2 + \alpha^{\frac2{\theta}}) \|F\|^2.
\end{align*}
When $\var$ is small enough and $\alpha\geq 1$, this implies
\[ \alpha\geq K \beta^{2\theta} = K \beta^{1/2}\]
for some $K>0$, as announced.
That heuristic argument shows that the nonsectorial nature of the operator is closely related to the exponent $\theta$ in the skewed interpolation inequality \eqref{skewedinterp}, itself related to the nonsymmetry of $L$.

Turning the heuristics into a rigorous calculation is not straightforward, because $\<L^2F,F\> = \|\Delta_V F\|^2 - \|\xi F\|^2$ is not really small. Taking product of \eqref{eigeneq} with $F,LF,L^2F$, etc. leads to a tangled collection of estimates. However, as I shall now show, there is an elementary way using once again the distorted $H^1$ norm which was useful in Sections \ref{secreg} and \ref{sechypoco}.

\begin{proof}[Proof of Theorem \ref{thmlocalspectrum}]
I shall only consider the estimate for $L$, since the reasoning for ${\cal B}_\lambda$ is similar, using the same kind of estimates as in the proof of Theorem \ref{thmregul}(ii). (Of course, ${\cal B}_\lambda$ may have eigenvalues with negative real part if $\lambda$ is not large enough, but it was shown in Section \ref{sechypoco} that the real part is bounded below, so upon replacing ${\cal B}_\lambda$ by ${\cal B}_\lambda + C \Id$ for $C$ large enough, the estimates will be similar and the real parts positive.)

Beforehand, consider the equation
\begeq\label{eqsurj} (L-\lambda) f = h,\endeq
or equivalently
\[ (I+L) f = (\lambda+1)f + h.\]
As $(I+L)^{-1}$ is regularising (Theorem \ref{thmalabouchut}), the operator $f\longmapsto (I+L)^{-1} ((\lambda+1)f + h)$ is compact, and the existence of a fixed point will follow from Schauder's theorem. This is to say that \eqref{eqsurj} always has solutions in $L^2(\mu)$, so $L-\lambda\,\Id$ is always surjective, and the spectrum of $L$ is made of eigenvalues.
\sm

So let us start again from \eqref{eigeneq}. The steps will be the same as in the proofs of regularity or hypocoercivity.
\sm

1. Take scalar product of \eqref{eigeneq} with $LF$:
\[ \<L^2 F, LF\>_{L^2(\mu)} + (\alpha^2 + \beta^2) \<F,LF\>_{L^2(\mu)} = 2\alpha \|LF\|_{L^2(\mu)}^2,\]
which is the same as
\begeq\label{localspectr1}
\|\nabla_V LF\|_{L^2(\mu)} + (\alpha^2+\beta^2) \|\nabla_V F\|_{L^2(\mu)}^2 = 2\alpha \|LF\|_{L^2(\mu)}^2.
\endeq
(From now on the measure $\mu$ and the $L^2$ subscript will be implicit.)
\sm

2. Apply $\nabla_V$ to \eqref{eigeneq} and take scalar product with $\nabla_VLF = L\nabla_VF + \nabla_H F$ (using that decomposition only where useful): this yields successively
\[ \nabla_V L^2F + (\alpha^2 + \beta^2) \nabla_V F = 2\alpha \nabla_V LF,\]
\[ L\nabla_V LF + \nabla_H LF + (\alpha^2+\beta^2) \nabla_V F = 2\alpha \nabla_V LF,\]
\begin{multline*}
\<L\nabla_V LF, \nabla_V LF\> + \<\nabla_H LF, \nabla_V LF\> + (\alpha^2+\beta^2) \<\nabla_VF, L\nabla_V F\>\\
+ (\alpha^2+\beta^2) \<\nabla_VF,\nabla_HF\>  = 2\alpha \|\nabla_V LF\|^2
\end{multline*}
and thus
\begin{multline} \label{localspectr2}
\|\nabla_V^2 LF \|^2 + (\alpha^2+\beta^2) \|\nabla_V^2 F\|^2
+ (\alpha^2+\beta^2) \<\nabla_VF,\nabla_HF\> \\
\leq 2\alpha \|\nabla_V LF\|^2 + C\, \|\nabla_HLF\|\, \|\nabla_V LF\|.
\end{multline}

3. Apply $\nabla_H$ to \eqref{eigeneq} and take scalar product with $\nabla_H LF = L \nabla_H F - \Sigma(v,v) \nabla_VF$:
\[ \nabla_H L^2F + (\alpha^2 +\beta^2) \nabla_HF = 2\alpha \nabla_H LF,\]
\[ L \nabla_H LF - \Sigma(v,v) \nabla_V LF + (\alpha^2+\beta^2) \nabla_HF = 2\alpha \nabla_H LF,\]
\begin{multline*}
\<L\nabla_H LF, \nabla_H LF\> - \<\Sigma(v,v) \nabla_V LF, \nabla_H LF\> + (\alpha^2+\beta^2) \bigl( \<\nabla_HF, L\nabla_HF\> - \<\nabla_HF, \Sigma(v,v)\nabla_VF\>\bigr)\\
= 2\alpha \|\nabla_H LF\|^2,
\end{multline*}
thus, using Proposition \ref{propgradmom},
\begin{multline}\label{localspectr3}
\|\nabla_V\nabla_H LF\|^2 + (\alpha^2+\beta^2) \|\nabla_V\nabla_H F\|^2 \leq 2\alpha \|\nabla_HLF\|^2\\
+ C \bigl( \|\nabla_V^2 LF\| + \|\nabla_VLF\|\bigr) \bigl( \|\nabla_V\nabla_H LF\| + \|\nabla_H LF\|\bigr)\\
+ C (\alpha^2+\beta^2) \bigl( \|\nabla_V\nabla_H F\| + \|\nabla_HF\| \bigr) \bigl( \|\nabla_V^2F\| + \|\nabla_VF\|\bigr).
\end{multline}

4. Apply $\nabla_V$ to \eqref{eigeneq} and multiply by $\nabla_H LF$, then apply $\nabla_H$ to \eqref{eigeneq} and multiply by $\nabla_V LF$, sum both:
\begin{multline*}
\<\nabla_V L^2 F, \nabla_H LF\> + \<\nabla_H L^2F,\nabla_V LF\> 
+ (\alpha^2+\beta^2) \bigl[ \<\nabla_V LF,\nabla_H LF\> + \<\nabla_HF, \nabla_V LF\> \bigr] \\
= 4\alpha \<\nabla_V LF, \nabla_H LF\>.
\end{multline*}
Playing commutators,
\begin{multline*}
\<L\nabla_V LF,\nabla_H LF\> + \<L\nabla_H LF, \nabla_V LF\> + \|\nabla_H LF\|^2
- \<\Sigma(v,v) \nabla_VLF,\nabla_VLF\> \\
+ (\alpha^2+\beta^2) \bigl( \<\nabla_VF,L\nabla_HF\> + \<\nabla_HF, L\nabla_VF\> \bigr) 
- (\alpha^2+\beta^2) \bigl( \<\nabla_VF, \Sigma(v,v)\nabla_VF\> + \|\nabla_HF\|^2\bigr)\\
= 4\alpha \<\nabla_VLF,\nabla_HLF\>.
\end{multline*}
Applying the $\Gamma$ formula for $L$,
\begin{multline*}
2 \<\nabla_V^2 LF, \nabla_V\nabla_H LF\> + \|\nabla_HF\|^2 - \<\Sigma(v,v) \nabla_V LF,\nabla_V LF\> \\
+ (\alpha^2+\beta^2) \bigl( \<\nabla_V^2 F,\nabla_V\nabla_H F\> - \<\nabla_VF, \Sigma(v,v)\nabla_VF\> + \|\nabla_HF\|^2\bigr)\\
= 4 \alpha \<\nabla_VLF,\nabla_HLF\>.
\end{multline*}
So
\begin{multline} \label{localspectr4}
\|\nabla_HLF\|^2 + (\alpha^2+\beta^2) \|\nabla_HF\|^2 \leq \\
C \Bigl[ \|\nabla_V^2 LF\|\, \|\nabla_V\nabla_H LF\| + \bigl( \|\nabla_V^2 LF\| + \|\nabla_V LF\| \bigr)^2 \Bigr]
+ C \alpha \|\nabla_V LF\|\,\|\nabla_H LF\| \\
+ C (\alpha^2+\beta^2) \Bigl[ \|\nabla_V^2F\| \|\nabla_V\nabla_H F\| + \bigl( \|\nabla_V^2F\|+\|\nabla_VF\|\bigr)^2 \Bigr].
\end{multline}
\sm

5. Combining \eqref{localspectr1}, \eqref{localspectr2}, \eqref{localspectr3}, \eqref{localspectr4}, with respective positive weights $1,a,b,c$, it results
\begin{align} \label{localspectr5}
\|\nabla_V LF\|^2 & + a \|\nabla_V^2 LF\|^2 + b \|\nabla_H LF\|^2 + c \|\nabla_V\nabla_H LF\|^2 \\ \nonumber
& + (\alpha^2+\beta^2) \Bigl( \|\nabla_V F\|^2 + a \|\nabla_V^2 F\|^2 + b \|\nabla_HF\|^2 + c \|\nabla_V\nabla_HF\|^2 \\ \nonumber
& \leq C \alpha \Bigl( \|LF\|^2 + a \|\nabla_VLF\|^2 + b \|\nabla_V LF\| \|\nabla_H LF\| + c \|\nabla_H LF\|^2 \Bigr)\\ \nonumber
& \quad + C \Bigl[ a \|\nabla_H LF\| \|\nabla_V LF\| 
+ b \Bigl( \|\nabla_V^2 LF \| \|\nabla_V\nabla_H LF\| + \bigl( \|\nabla_V^2 LF\| + \|\nabla_V LF\| \bigr)^2 \Bigr) \\ \nonumber
& \qquad + c \bigl( \|\nabla_V^2 LF\| + \|\nabla_V LF\| \bigr) \bigl( \|\nabla_V\nabla_H LF\| + \|\nabla_V LF\|\bigr)^2 \Bigr) \\ \nonumber
& \quad + C (\alpha^2+\beta^2) \Bigl[ b \Bigl( \|\nabla_V^2 F\| \|\nabla_V\nabla_H F\| + \bigl( \|\nabla_V^2 F\| + \|\nabla_VF\|\bigr)^2 \Bigr)\\ \nonumber
& \qquad + c (\|\nabla_V\nabla_HF\| + \|\nabla_HF\| ) (\|\nabla_V^2 F\| + \|\nabla_V F\|) \Bigr].
\end{align}

Let us choose $a,b,c$ in such a way that all terms in $\|\nabla_VF\|$, $\|\nabla_V^2F\|$, $\|\nabla_HF\|$, $\|\nabla_V\nabla_HF\|$, $\|\nabla_VLF\|$, $\|\nabla_V^2LF\|$, $\|\nabla_HLF\|$, $\|\nabla_V\nabla_HLF\|$ in the right hand side are controlled by half of those occurring in the left hand side. This is true as soon as
\begin{multline*}
a\alpha \ll 1, \qquad. b\alpha \ll \sqrt{b}, \qquad c\alpha \ll b ,  \qquad a \ll \sqrt{b}, \\
b \ll \min (\sqrt{ac},a,1), \qquad c \ll \sqrt{\min(a,1)\, \min(b,c)} ,
\end{multline*}
conditions which all boil down to
\begeq \label{condspectrabc}
a\ll \sqrt{b}, \qquad b \ll \sqrt{ac}, \qquad c\alpha^3 \ll b\alpha^2\ll a\alpha \ll 1.
\endeq
Those can be realised with $\min (a\alpha, b\alpha^2, c\alpha^3)\geq K>0$ as $\alpha\to \infty$. Then \eqref{localspectr5} reduces to
\begin{multline} \label{localspectr6}
\|\nabla_VLF\|^2 + \alpha^{-1} \|\nabla_V^2 LF\|^2 + \alpha^{-2} \|\nabla_HLF\|^2 + \alpha^{-3} \|\nabla_V\nabla_H LF\|^2\\
+ (\alpha^2+\beta^2)  \bigl( \|\nabla_V F\|^2  + \alpha^{-1} \|\nabla_V^2 F\|^2 +\alpha^{-2} \|\nabla_HF\|^2 + \alpha^{-3} \|\nabla_V\nabla_HF\|^2 \bigr) \leq C \alpha \|LF\|^2.
\end{multline}
On the other hand, recalling Corollaries \ref{divloss} and \ref{L2geod},
\begeq\label{LFleq}
\|LF\|^2 \leq C \bigl( \|\nabla_V^2F\|^2 + \|\nabla_HF\|^2 + \|\nabla_V\nabla_HF\|^2\bigr).
\endeq
Combining \eqref{localspectr6} and \eqref{LFleq},
\[ \frac{\beta^2}{\alpha^3} \bigl( \|\nabla_V^2 F\|^2 + \|\nabla_HF\|^2 + \|\nabla_V\nabla_HF\|^2 \bigr)
\leq C \alpha \bigl( \|\nabla_V^2F\|^2 + \|\nabla_HF\|^2 + \|\nabla_V\nabla_HF\|^2\bigr).\]
If $\|\nabla_V^2 F\|^2 + \|\nabla_HF\|^2 + \|\nabla_V\nabla_HF\|^2=0$ then $F$ is a constant, which is impossible for $\alpha>0$. So $\beta^2 \leq C \alpha^4$, which concludes the proof of \eqref{localspectrum}.
\end{proof}

\begin{Rk} A more synthetic presentation would be to estimate $L$ in terms of the unbounded operator 
\[ A:h \longmapsto (h, \nabla_Vh,\nabla_Hh),\] 
valued in $L^2(\mu; \R\times T(TM))$, and the scalar product 
\[ \<(x,Y,Z), (x',Y',Z')\> = xx'+ a\<Y,Y'\> + b (\<Y,Z'\>+\<Z,Y'\>) + c \<Z,Z'\>.\]
For the moment the pedestrian computation above suffices.
\end{Rk}

\bibnotes

Sectorial operators are popular in spectral analysis and numerical analysis of linear operators \cite{haase:sectorial:06,yagi:sectorial:10}, for the properties of the analytic semigroup they generate and the quality of their approximation.

Estimates \eqref{localspectrum} were first obtained by Lebeau \cite{lebeau:FP1:05,lebeau:FP2:07} using pseudo-differential formalism.

Bismut \cite{bismut:hypoLie:08} has computed the spectrum of ${\cal B}_\lambda$ when $M$ is the sphere. It behaves nicely as $\lambda\to\infty$, but nothing remarkable seems to emerge for $\lambda\to 0$ (private communication by Bismut).

\section{$L^1$ convergence through $L^2$ approximation} \label{secL1conv}

This section addresses the study of the long-time convergence for the $L^1$ problem, eq.~\eqref{eqf}, starting from an initial datum $f_0\in L^1(dx\,dv)$. Regularisation will drive the solution in a regime where $f$ is automatically $L^2(dx\,dv)$, and even uniformly smooth and rapidly decaying, but not in the $L^2(\mu)$ regime of Section \ref{sechypoco}. Even under the strong (still not unreasonable) assumption
\[ \iint f_0(x,v)\, e^{\alpha |v|^2/2}\,dv\,dx < +\infty \qquad \text{for some $\alpha>0$},\]
the best that one can hope is an estimate of the form
\[ f(t,x,v) = O (e^{-\beta|v|^2/2});\]
but in general $f$ will never enter the $L^2(\mu)$ regime, recall Corollary \ref{corL1L2}. However, combining the improvement of square exponential bounds (Proposition \ref{2sideexp}) and the strong regularisation bounds from Section \ref{secreg} (Theorem \ref{thmregul}(iv)), will make it possible to approximate the $L^1$ problem by the $L^2$ problem, well enough to get the exponential convergence. This is the content of the next theorem.

\begin{Thm}[Exponential equilibration in $L^1$, first result] \label{L1exp}
Let $f_0$ be a probability distribution on $\TM$ satisfy
\[ M_{\beta_0}(f_0) = \iint_{\TM} e^{\beta_0 \frac{|v|^2}2} f_0(x,v)\,dx\,dv < +\infty \]
for some $\beta_0>0$; and let $f=f(t,x,v)$ be the solution of \eqref{eqf} with initial datum $f_0$. Further let $f_\infty(x,v)$ be the normalised Gaussian distribution function. Then there are $\nu>0$, depending only on $M$, and $C>0$, depending only on $M$ and $\beta_0$, such that for all $t>0$,
\begeq\label{f-f8} \|f(t,\cdot) - f_\infty\|_{L^1(M)} \leq C\, M_{\beta_0}(f_0)\, e^{-\nu t},\endeq
Convergence is also exponential in all weighted Sobolev spaces $H^s_\kappa$; likewise, for all integers $k,\ell,r$, and any $\nu'<\nu$, there is $C">0$, only depending on $M$, $\beta_0$, $\nu'-\nu$, such that
\begeq\label{ffSob} (1+|v|^r)\, \bigl |\nabla_V^k \nabla_H^\ell (f - f_\infty) \bigr| \leq C\, M_{\beta_0}(f_0)\, e^{-\nu' t}. \endeq
\end{Thm}

\begin{Rk} The theorem applies exactly the same if $f_0$ is replaced by a probability measure $\mu_0$ such that
\[ M_{\beta_0}(\mu_0) = \iint_{\TM} e^{\beta_0 \frac{|v|^2}2}\,\mu_0(dx\,dv) < +\infty. \]
\end{Rk}

\begin{proof}[Proof of Theorem \ref{L1exp}]
The proof follows a natural scheme.
If $f$ is a probability distribution on $\TM$, then for any $R>0$ one may write
\[  f = f\, 1_{|v|\leq R} + f\,1_{|v|>R} = f^{\leq R} + f^{>R}. \]
Let now $f_t = f(t,\cdot)$ solve \eqref{eqf} and satisfy the square exponential moment condition of Theorem \ref{L1exp}. Without loss of generality, $\beta_0<1$. Let $t_\ast>0$. From the localisation estimates in Section \ref{seclocal}, $M_{\beta_{\ast}}(f_{t_\ast}) \leq C(\beta_0,\beta_\ast) M_{\beta_0}(f_0)$, where $\beta_\ast$ is the solution at time $t_\ast$ of the differential equation \eqref{dotbeta}. Moreover, from the regularisation estimates in  Section \ref{secreg},  $f_{t_0}\in L^\infty$, with bounds depending on $t_\ast$ and proportional to a power of a moment of $f_0$, hence controlled above by a multiple of $M_{\beta_0}(f_0)$. Thus
\begin{align*}
\iint_{\TM} f_{t_\ast}^{\leq R}(x,v)^2\, e^{\frac{|v|^2}2} \,dx\,dv & \leq \|f_{t_0}\|_{L^\infty} M_{\beta_0}(f_0)\, e^{(1-\beta_\ast)\frac{R^2}2} \\
& \leq C_{\beta_0,\beta_\ast}\, M_{\beta_0}(f_0)^2\, e^{(1-\beta_\ast)\frac{R^2}2},
\end{align*}
or equivalently, if $h_t = f_t^{\leq R}/f_\infty$,
\begeq\label{intfinf}
\|h_{t_0}\|_{L^2(\mu)} \leq C_{\beta_0,\beta_\ast}\, M_{\beta_0}(f_0)\, e^{(1-\beta_\ast)\frac{R^2}4},
\endeq
where as in the sequel, $C_{\beta_0}$ stands for any constant depending only on $\beta_0$ and $M$.
On the other hand,
\begeq\label{intfsup}
\| f_{t_0}^{>R}\|_{L^1} \leq M_{\beta_\ast}(f_0)\, e^{-\beta_\ast \frac{R^2}{2}}.
\endeq
In the sequel I shall abbreviate $M_{\beta_0}(f_0)$ into just $M_{\beta_0}$.
Let
\[ m_{t_\ast}(R) = \iint f_{t_\ast}^{\leq R} = 1 - O(M_{\beta_\ast}\,  e^{-\beta_\ast R^2/2}),\]
then for $t\geq t_0$,
\begin{align*}
\|f_t- f_\infty\|_{L^1} & 
\leq \Bigl\| e^{-(t-t_\ast){\cal L}} (f_{t_\ast}^{\leq R}) - m(R) f_\infty \Bigr\|_{L^1} 
+ \|e^{-(t-t_\ast){\cal L}} f^{>R}_{t_\ast}\|_{L^1} + (1-m(R))\\
& = \bigl\| e^{-(t-t_\ast)L} h_{t_\ast} - \<h_{t_\ast}\> \|_{L^1(\mu)} + C_{\beta_0,\beta_\ast}\, M_{\beta_0}\, e^{-\beta_\ast R^2/2}\\
& \leq \bigl\| e^{-(t-t_\ast)L} h_{t_\ast} - \<h_{t_\ast}\> \|_{L^2(\mu)} + C_{\beta_0,\beta_\ast}\, M_{\beta_0}\,  e^{-\beta_\ast R^2/2} \\
& \leq C\, e^{-\lambda (t-t_\ast)} \|h_{t_\ast}\|_{L^2(\mu)} + C_{\beta_\ast}\, M_{\beta_\ast}\,  e^{-\beta_\ast R^2/2}  \\
& \leq C_{\beta_0,\beta_\ast}\, e^{\lambda t_\ast}\, M_{\beta_0}(f_0)\, \bigl[ e^{(1-\beta_\ast)\frac{R^2}4} e^{-\lambda t} + e^{-\beta_\ast R^2/2} \bigr].
\end{align*}
Now choose $t_\ast\geq 1/2$ such that $\beta_\ast \geq \max (\beta_0,1/2)$. Time $t_\ast$ only depends on $\beta_0$, so all in all
\[ \|f_t- f_\infty\|_{L^1} \leq C_{\beta_0} \, M_{\beta_0} \bigl[ e^{(1-\beta_\ast)\frac{R^2}4} e^{-\lambda t} + e^{-\beta_\ast R^2/2} \bigr].\]
To optimise, choose
\[ \lambda t = (1+\beta_\ast)\frac{R^2}4; \]
then
\[ \|f_t- f_\infty\|_{L^1} \leq C_{\beta_0} \, M_{\beta_0} e^{ - 2\frac{\beta_\ast}{1+\beta_\ast} \lambda t}.\]
Since $\beta_\ast\geq 1/2$ that estimate yields
\[ \|f_t- f_\infty\|_{L^1} \leq C_{\beta_0} \, M_{\beta_0} e^{ - \frac23 \lambda t},\]
concluding the proof of the $L^1$ exponential equilibration. By interpolation with the square exponential bound, exponential convergence also holds in all weighted $L^1_\kappa$ spaces, with arbitrarily small degradation of the convergence rate if high enough moments are used. From the uniform regularity estimates of Theorem \ref{thmregul} and the Nash-type interpolation inequality of Theorem \ref{propnashinterp}, exponential convergence holds in all weighted Sobolev spaces, at the price of an arbitrarily small degradation of the convergence rate. Finally, from Sobolev inequality (Proposition \ref{propsobemb}) this also implies pointwise results. The proof is complete.
\end{proof}

\begin{Rk} With respect to \eqref{rkMTconv} the estimate relies on a more demanding square exponential estimate on the initial distribution, but may provide a better rate; it would be good to proceed to a precise comparison. In the next section another strategy will be explored, which also provides a decent exponential rate, without needing the square exponential estimate.
\end{Rk}

\bibnotes

In my initial work with Debbasch and Ollivier \cite{DOV:preprint} we only proved $O(t^{-\infty})$ convergence by using the nonlinear method from my memoir \cite{vill:hypoco}. The refinement to exponential convergence, Theorem \ref{L1exp}, was proven for these notes, making better use of our estimates.

\section{$L^1$ convergence through Fisher information hypocoercivity} \label{secfisher}

This section explores a more direct approach to the exponential convergence in $L^1$. It will not only improve the generality of the $L^1$ convergence, but also complete the picture in relation to other key issues in kinetic theory and diffusion processes, as well as fundamental thermodynamic concepts.

Working with equilibration for densities, a natural strategy is to resort to the entropy identity, or $H$-Theorem, or second law of thermodynamics, or decrease of the free energy, which for the kinetic Fokker--Planck model reads
\begeq\label{freedecay}
\frac{d}{dt} \left( \iint_{\TM} f \log f\,dv\,dx + \iint_{\TM} f \, \frac{|v|^2}2 \,dx\,dv\right) = - \iint_{\TM} |\nabla_V \log f + v|^2 \,f\,dx\,dv.
\endeq
(The free energy is the kinetic energy minus the entropy, the equilibrium temperature being set to unity.)
Writing $h= d(f\,dx\,dv)/d\mu$, so that $\pa_t h + Lh=0$, $L=\xi-\Delta_V^\mu=\xi-\Delta_V+v\cdot\nabla_V$, identity \eqref{freedecay} can be recast as
\begeq\label{ddthlogh}
\frac{d}{dt} \iint h\log h\,d\mu = - \iint \frac{|\nabla_Vh|^2}{h}\,d\mu
\endeq
which can be seen as a consequence of the classical identity
\[ L(h\log h) - (\log h +1) Lh = \frac{|\nabla_Vh|^2}{h}. \]
The functional appearing in the right-hand side of \eqref{ddthlogh} is only one fraction of the Fisher information \eqref{FI}, namely the one depending on vertical variations.

It is useful to note that, by chain rule,
\begeq\label{notefisher}
\frac{|\nabla_Vh|^2}{h} = 4 \bigl|\nabla_V\sqrt{h}\bigr|^2 = h |\nabla_V \log h|^2.
\endeq

Of course \eqref{freedecay} in itself does not lead to quantifiable decay because only the velocity (vertical) gradient appears on the right hand side. In the spirit of Sections~\ref{secreg} and~\ref{sechypoco}, let us define
\begeq\label{Hcal}
{\cal H} (h\mu) = \iint h\log h\,d\mu + a \iint \frac{|\nabla_Vh|^2}{h}\,d\mu + 2b \iint \frac{\<\nabla_V h,\nabla_H h\>}{h}\,d\mu 
+ c \iint \frac{|\nabla_H h|^2}{h}\,d\mu.
\endeq
where $a,b,c>0$ will be chosen later on. The goal is to show that this can be done so that
\[ \frac{d{\cal H}}{dt} \leq -2\kappa\,{\cal H}\]
for some $\kappa>0$ which can be interpreted as an infinitesimal rate of exponential convergence. (The nonessential factor~2 here is because ${\cal H}$, like Boltzmann's $H$ function, controls a squared distance to equilibrium, as shown by Pinsker's inequality \eqref{CKP}.) The functional \eqref{Hcal} is a (twisted) combination of Boltzmann and Fisher informations, which quantify respectively the rarity of the distribution function and the difficulty for an observer to evaluate it; so it does have some meaning from the point of view of statistical physics.

With respect to the $L^2$ computations of Section \ref{sechypoco}, there are two extra difficulties. The first one is the $h\log h$ nonlinearity, leading to much more subtle computation than the quadratic functionals. However there is a long tradition of handling such computations, going back at least to Henry P. McKean in the sixties, and famously used by Dominique Bakry and Michel \'Emery in the eighties for their ``$\Gamma_2$ calculus'' in the theory of logarithmic Sobolev inequalities. Inspired by those works, Lemmas \ref{lemlog1} and \ref{lemlog2} below will show that the terms arising in the time-derivative of \eqref{Hcal} along the kinetic Fokker--Planck equation are, algebraically speaking, {\em exactly the same} as those arising in the time-derivative of the quadratic functional \eqref{2parentheses}. More precisely, for any bilinear term $\int B(h,h)\,d\mu$ occurring there, there will be a corresponding term $\int B(\log h,\log h)\, h\,d\mu$ here. Arriving painlessly at that nontrivial conclusion will require to get to the core of the corresponding nonlinear calculus.

The second difficulty is that there is no logarithmic analogue of Proposition \ref{propgradmom}, which was crucially used in Section \ref{sechypoco} to ``close the estimates''. Even when there is just the velocity variable, there is a gap here, because of the failure of the following tentative inequality for positive functions $h$ on $\R^n$:
\begeq\label{ILS2}
\int_{\R^n} |v|^2 \, |\nabla_v\log h|^2 \, h\,d\mu 
\leq C \int_{\R^n} \bigl( |\nabla_v \log h|^2 + |\nabla_v^2 \log h|^2 \bigr)\, h\,d\mu.
\endeq
(To see that \eqref{ILS2} fails, choose $h(v) = \exp (a\cdot v)$; then the right hand side grows like $|a|^4$ as $|a|\to\infty$, and the left hand side like $|a|^2$.)

I don't know if \eqref{ILS2} can be ensured under a moment (or exponential moment) condition on $h$; even that would not directly solve the problem because there is also need to handle horizontal derivation. In any case, I will show that the second difficulty is solved as soon as the density of $f$ satisfies the Gaussian tail Poincar\'e inequality. The proof of this estimate, which is of interest in itself, is postponed to a complementary study. But already note that, by Proposition \ref{2sideexp} the Gaussian tail Poincar\'e inequality can be true only if $f_0$ satisfies a square-exponential bound, and that has to involve a regularisation argument as well.

To summarise, the same fundamental ingredients appear in both strategies, but woven differently. In the argument of Section \ref{secL1conv}, square exponential bounds were used together with regularisation in Sobolev spaces, to approximate the $L^1$ problem by the $L^2$ problem, and take advantage of the hypocoercivity in $L^2$ and $H^1$, which itself made strong use of the Gaussian tail Poincar\'e for $\mu$. In the scheme now considered, square exponential bounds lead to the Gaussian tail Poincar\'e estimate for the time-dependent density, and this in turn allows to prove hypocoercivity in the sense of entropy and Fisher information.

Now there is also a third approach which bypasses that Gaussian-type control and still does grasp the exponential rate of convergence for the entropy. The price to pay will be a kinetic taming of the gradients in the Fisher information:
\begin{multline}\label{tildeHcal}
\tilde{\cal H} (h) = \iint_{\TM} h\log h\,d\mu + a \iint_{\TM} \frac{|\nabla_Vh|^2}{h}\,\frac{d\mu}{(1+|v|^2)^{r/2}} + 2b \iint_{\TM} \frac{\<\nabla_V h,\nabla_H h\>}{h}\,
\frac{d\mu}{(1+|v|^2)^r} \\
+ c \iint_{\TM} \frac{|\nabla_H h|^2}{h}\,\frac{d\mu}{(1+|v|^2)^{3r/2}}.
\end{multline}
It turns out that when $r\geq 1$ one can prove a differential inequality of linear type for $\tilde{\cal H}$ (exponential convergence) even without the Gaussian tail Poincar\'e inequality. What makes this possible is, on the one hand, the possibility to use taming by negative powers of $|v|$ to cut the quadratic weights arising from the commutator $[\nabla_H,\xi]$; and on the other hand, the inequality of logarithmic Sobolev type, Proposition \ref{propcoertame}(iii), which allows to control the entropy even when horizontal gradients are tamed. Eventually the conclusion will be the exponential convergence of $h$ in entropy sense.
\sm

Let us go on with this program. Write
\[ \left.\frac{d}{dt} \right|_O J(h) = \left.\frac{d}{dt}\right|_{t=0} J(e^{tO}h). \]

\begin{Lem}\label{lemlog1}
If $D,D'$ are any two smooth derivation operators on $\TM$, then
\begeq\label{eqlemlog1}
-\left. \frac{d}{dt} \right|_{\xi} \iint  h \< D\log h, D'\log h\> \,d\mu = 
\iint h \< D\log h, [D',\xi]\log h \>\,d\mu + \iint h \< [D,\xi]\log h, D'\log h\>\,d\mu.
\endeq
More generally, if $a=a(|v|)$ then
\begeq\label{eqlemlog1a}
-\left. \frac{d}{dt} \right|_{\xi} \iint  h \< D\log h, D'\log h\> \,a\,d\mu = 
\iint h \< D\log h, [D',\xi]\log h \>\,a\,d\mu + \iint h \< [D,\xi]\log h, D'\log h\>\,a\,d\mu.
\endeq
\end{Lem}

\begin{proof}[Proof of Lemma \ref{lemlog1}]
The proof of \eqref{eqlemlog1} uses only the derivation property for $\xi$, $D$ and $D'$. By chain rule,
\begeq\label{chainruleI}
\iint h \<D\log h, D'\log h\> \,d\mu = 4 \iint \<D\sqrt{h},D'\sqrt{h}\>\,d\mu
\endeq
and 
\begeq\label{evolsqrt} \pa_t h + \xi h =0 \quad \Longrightarrow \quad \pa_t\sqrt{h} + \xi\sqrt{h} =0. \endeq
So the time-derivative of \eqref{chainruleI} under \eqref{evolsqrt} is, with $\mu$ as reference measure,
\begin{multline*}
4 \iint \<D\xi\sqrt{h}, D'\sqrt{h} \>  + 4 \iint \<D\sqrt{h}, D'\xi\sqrt{h}\> 
= 4 \left( \iint \<[D,\xi]\sqrt{h}, D'\sqrt{h}\> + \iint \<D\sqrt{h},[D',\xi]\sqrt{h}\> \right) \\
+ 4 \left( \iint \<\xi D\sqrt{h}, D'\sqrt{h} \> + \iint \<D\sqrt{h},\xi D'\sqrt{h}\> \right),
\end{multline*}
which using again the derivation rule is the same as
\begin{multline*}
4 \iint \<D\xi\sqrt{h}, D'\sqrt{h} \>  + 4 \iint \<D\sqrt{h}, D'\xi\sqrt{h}\> 
= 4 \left( \iint \<[D,\xi]\sqrt{h}, D'\sqrt{h}\> + \iint \<D\sqrt{h},[D',\xi]\sqrt{h}\> \right) \\
+ 4 \iint \xi \<D\sqrt{h}, D'\sqrt{h} \>,
\end{multline*}
and the last integral vanishes by divergence formula. The remaining part is the same as the right hand side of \eqref{eqlemlog1} by chain-rule again.

To generalise to \eqref{eqlemlog2}, just replace $D$ by $aD$, which is still a derivation, and note that $[aD,\xi] = a[D,\xi]$ because $a\xi = \xi a$.
\end{proof}

\begin{Lem}\label{lemlog2}
If $D,D'$ are two smooth derivation operators on $\TM$, commuting with $\nabla_V$, then 
\begin{multline}\label{eqlemlog2}
\left.\frac{d}{dt} \right|_{\Delta_V^\mu} \iint h \<D\log h, D'\log h\>\,d\mu 
= - 2 \iint h \<\nabla_V D \log h, \nabla_V D \log h \>\, d\mu\\
- \left( \iint h \< [\Delta_V^\mu, D] \log h, D'\log h\> \, d\mu
+ \iint h \< D\log h, [\Delta_V^\mu, D']\log h\> \,d\mu\right).
\end{multline}
More generally, if $a=a(x,v)$ is a smooth function, then
\begin{multline}\label{eqlemlog2a}
\left.\frac{d}{dt} \right|_{\Delta_V^\mu} \iint h \<D\log h, D'\log h\>\,a\,d\mu 
= - 2 \iint h \<\nabla_V D \log h, \nabla_V D \log h \>\, a\, d\mu\\
- \left( \iint h \< [\Delta_V^\mu, D] \log h, D'\log h\> \, a\, d\mu
+ \iint h \< D\log h, [\Delta_V^\mu, D']\log h\> \,a\, d\mu\right)\\
+ \iint h \<D\log h, D'\log h\>\,(\Delta_V^\mu a)\,d\mu.
\end{multline}
\end{Lem}

\begin{Rks} \begin{itemize}
\item[(i)] The first term in the right hand side of \eqref{eqlemlog2} is, more explicitly,
\[ -2 \iint h\, g^{ij}g^{k\ell} (\nabla_V)_i D_k \log h (\nabla_V)_j D_\ell \log h \, d\mu.\]

\item[(ii)] This lemma will be applied with $D,D'$ being equal either to $(\nabla_V)_i$ or to $(\nabla_H)_i$, for any $i\in \{1,\ldots, n\}$. More general formulas, involving extra commutators, exist when $D,D'$ fail to commute with $\nabla_V$.

\item[(iii)] Later on, \eqref{eqlemlog2a} will be used when $a$ is a function of $|v|$, in which case (with a slight abuse of notation)
\[ \Delta_V^\mu a = a''(|v|) + \left(\frac{(n-1)}{|v|} - |v| \right) a'(|v|).\]
\end{itemize}
\end{Rks}

Before proving Lemma \ref{lemlog2} I shall recast the core of the computation in the simplified setting when there is only the velocity variable and instead of the derivations $C,C'$ there is the usual gradient. So $\nabla_V$ is just $\nabla_v$, $\Delta_V =\Delta_v$, and $\Delta_v^\mu = \Delta_v - v\cdot\nabla_v$. The goal is the following {\bf second-order entropic identity} (formula for the second time-derivative of the free energy along the Fokker--Planck equation):
\begeq\label{BEid}
\left. \frac{d}{dt}\right|_{\Delta_v^\mu} \int \frac{|\nabla_vh|^2}{h}\,d\mu 
= - 2 \left( \int h |\nabla_v^2 \log h|^2\,d\mu + \int h |\nabla_v\log h|^2\,d\mu \right).
\endeq
This specific identity has been proven explicitly or implicitly several times, including by McKean, Stam, Bakry--\'Emery, Toscani.
While there are several ways to arrive at \eqref{BEid}, the organisation of computations below is the one which, I think, will most fluently lead to the generalisation of Lemma \ref{lemlog2}.

\begin{proof}[Proof of identity \eqref{BEid}]
It is convenient to systematically commute nonlinearities and operators. For a start, let
\begeq\label{Gamma0} \Gamma_0(f,f) = \frac{f^2}{2} \endeq
\begin{align} \label{Gamma1} 
\Gamma_1(f,f) & = \frac12 \bigl( \Delta_v^\mu f^2 - 2 f \Delta_v^\mu f\bigr) \\
\nonumber & = |\nabla_v f|^2
\end{align}
(By design, the derivation part of $\Delta_v^\mu$ disappears in the commutation.)
\begin{align} \label{Gamma2}
\Gamma_2(f,f) &  = \frac12 \Bigl[ \Delta_v^\mu \Gamma_1(f,f) - 2 \Gamma_1 (f,\Delta_v^\mu f) \Bigr] \\
\nonumber & = |\nabla^2_v f|^2 + |\nabla_vf|^2.
\end{align}
To see \eqref{Gamma2}, first note that 
\[ \Delta \frac{|\nabla f|^2}{2} - \nabla f\cdot\nabla \Delta f = |\nabla^2 f|^2\]
(Bochner's identity in flat space),
and further, using $[\nabla_v,v\cdot\nabla_v] = \nabla_v$, that
\begin{align*}
- v\cdot\nabla_v \frac{|\nabla_v f|^2}{2} + \nabla_v f\cdot \nabla_v(v\cdot\nabla_v f) 
& = -v\cdot\nabla_v \frac{|\nabla_vf|^2}{2} + \nabla_v f \cdot\bigl(v\cdot\nabla_v (\nabla_vf)\bigr) + |\nabla_v f|^2 \\
& = |\nabla_v f|^2.
\end{align*}

Now turn to \eqref{BEid}. By analogy, define
\[ I(h) = \int \frac{|\nabla_v h|^2}{h}\, d\mu,\]
\[ \Gamma_{I,0} (h) = \frac{|\nabla_v h|^2}{h} = h |\nabla_v \log h|^2,\]
and
\begeq\label{GammaI1}
\Gamma_{I,1}(h) = \Delta_v^\mu (h |\nabla_v \log h|^2) - \Bigl( \Delta_v^\mu h |\nabla_v \log h|^2 + 2 h \nabla_v \log h \cdot \nabla_v \frac{\Delta_v^\mu h}{h} \Bigr),
\endeq
so that
\[ \int_{\R^n} \Gamma_{I,0} \,d\mu = I(h), \qquad \int_{\R^n} \Gamma_{I,1}\,d\mu = - \left.\frac{d}{dt}\right|_{\Delta_v^\mu} I(h), \]
and the goal is to evaluate $\Gamma_{I,1}$.
To rewrite the term involving $\Delta f/ f$ in terms of $\log f$, there are the identities
\begeq\label{usefulDff}
\nabla_v f = f \nabla_v \log f,\qquad \frac{\Delta_v^\mu h}{h} = \Delta_v^\mu \log h + |\nabla_v \log h|^2.
\endeq
Inserting this into \eqref{GammaI1},
\begin{align*} 
\Gamma_{I,1}(h) & = h \Delta_v^\mu |\nabla_v \log h|^2 + 2 h\cdot \nabla_v \log h \cdot \nabla_v |\nabla_v\log h|^2
- 2 h \nabla_v \log h\cdot\nabla_v (\Delta_v \log h + |\nabla_v\log h|^2) \\
& = 2 h \Bigl( \Delta_v \frac{|\nabla_v \log h|^2}2 - \nabla_v \log h \cdot\nabla_v \Delta_v\log f\Bigr) \\
& = 2 h \, \Gamma_2(\log h,\log h)\\
& = 2 h \bigl( |\nabla_v \log h|^2 + |\nabla_v^2 \log h|^2\bigr).
\end{align*}
Integration on $\R^n$ yields \eqref{BEid}.
\end{proof}

Now comes the adaptation to the more general situation of interest for this section.

\begin{proof}[Proof of Lemma \ref{lemlog2}]
Let 
\[ J(h) = \iint h \<D\log h, D'\log h\>\,d\mu,\]
\begeq \Gamma_{J,0} (h) = h \< D\log h, D'\log h\>, \endeq
\begin{multline}
\Gamma_{J,1}(h) = \Delta_V^\mu \bigl( h \<D\log h, D'\log h\> \bigr) \\ - \Bigl[ (\Delta_V^\mu h) \<D\log h, D'\log h\> 
+ h \left\< D \left( \frac{\Delta_V^\mu h}{h}\right), D'\log h\right\> 
+ h \left\< D\log h, D'\left(\frac{\Delta_V^\mu h}{h} \right) \right\> \Bigr],
\end{multline}
so that
\begeq\label{JJ} J(h) = \iint \Gamma_{J,0}(h)\,d\mu  \qquad \left.\frac{d}{dt} \right|_{\Delta_V^\mu} J(h) = - \iint \Gamma_{J,1}(h)\,d\mu.
\endeq
Using the basic rules of calculus,
\begin{align*}
\Gamma_{J,1} (h) 
& = 2 \nabla_V h\cdot \nabla_V \<D\log h, D'\log h\> + h \Delta_V^\mu \<D\log h, D'\log h\> \\
& \qquad - h \bigl\< D (\Delta_V^\mu \log h + |\nabla_V \log h|^2 ), D'\log h\bigr\>
- h \bigl\< D\log h, D'(\Delta_V^\mu \log h + |\nabla_V\log h|^2)\bigr\> \\
& = 2 h \biggl[ \bigl\< \nabla_V\log h, (\nabla_V D\log h) D'\log h\bigr\> + \bigl\< \nabla_V \log h, (\nabla_V D'\log h) D\log h\bigr\>  \\
& \qquad 
- \bigl\< (D\nabla_V\log h)\nabla_V\log h, D'\log h\bigr\> - \bigl\< (D'\nabla_V \log h) \nabla_V\log h, D\log h\bigr\> \biggr]\\
& \quad + h \Bigl[ \Delta_V^\mu \<D\log h, D'\log h\> - \< D\Delta_V^\mu \log h, D'\log h\> - \<D\log h, D'\Delta_V^\mu \log h\> \Bigr].
\end{align*}
Here the first four terms inside the square brackets cancel out by symmetry. For instance,
\begin{align*} & \bigl\< \nabla_V\log h, (\nabla_V D\log h) D'\log h\bigr\> - \bigl\< (D\nabla_V\log h)\nabla_V\log h, D'\log h\bigr\>\\
& \qquad = g^{ij} (\nabla_V)_i \log h, (\nabla_V)_i D_j \log h\, D'_j \log h - g^{ij} D_j(\nabla_V)_i \log h (\nabla_V)_i \log h\, D'_j (\log h)\\
& \qquad =0.
\end{align*}
It remains only
\begin{align*}
\Gamma_{J,1}(h) 
& = h \Bigl[  \Delta_V^\mu \<D\log h, D'\log h\> - \< D\Delta_V^\mu \log h, D'\log h\> - \<D\log h, D'\Delta_V^\mu \log h\> \Bigr]\\
& = h \Bigl[ \Delta_V^\mu \<D\log h, D'\log h\> - \<\Delta_V^\mu D\log h, D'\log h\> - \<D\log h, \Delta_V^\mu D'\log h\>\Bigr] \\
& \qquad + h \Bigl( \bigl\< [\Delta_V^\mu ,D] , D'\log h \bigr\> + \bigl\< D\log h, [\Delta_V^\mu, D'] \log h\bigr\> \Bigr) \\
& = 2 h\,\Gamma_1 (D\log h, D'\log h) + h \Bigl( \bigl\< [\Delta_V^\mu ,D] , D'\log h \bigr\> + \bigl\< D\log h, [\Delta_V^\mu, D'] \log h\bigr\> \Bigr).
\end{align*}
Equation \eqref{eqlemlog2} follows upon integration.

Now for the generalisation to \eqref{eqlemlog2a}: Define
\begeq\label{Ja} J_a(h) = \iint h \<D\log h, D'\log h\>\,a\,d\mu, \endeq
so that \eqref{JJ} becomes
\begeq\label{JJa} J_a(h) = \iint \Gamma_{J,0}(h)\,a\,d\mu, \endeq
\begin{align*}
\left.\frac{d}{dt} \right|_{\Delta_V^\mu} J_a(h) & = - \int \bigl( \Gamma_{J,1}(h) - \Delta_V^\mu \Gamma_{J,0} \bigr) \,a\,d\mu\\
& = - \iint\Gamma_{J,1}(h)\,a\,d\mu + \iint \Gamma_{J,0} (\Delta_V^\mu a)\,d\mu\\
& = - \iint\Gamma_{J,1}(h)\,a\,d\mu + J_{\Delta_V^\mu a}(h).
\end{align*}
Then the proof is completed by using the previously computed expression for $\Gamma_{J,1}$.
\end{proof}

Now come the main results of this section.

\begin{Thm}[Differential exponential decay in distorted Fisher information] \label{thmfisherdecay}
Let $M$ be a compact manifold and let $f=f(x,v)\,dx\,dv$ be a probability density on $\TM$.
\sm

(i) If $f$ satisfies the Gaussian tail Poincar\'e inequality of Definition~\ref{defGTP}, then there are constants $a,b,c,\kappa>0$, depending only on $n$, on the constant $K$ in Definition \ref{defGTP}, on an upper bound on the curvature tensor and on the logarithmic Sobolev inequality satisfied by $M$ (Definition \ref{defLSI}), such that the functional ${\cal H}$ defined by \eqref{Hcal} satisfies
\begeq\label{ddtkappaH}
\left.\frac{d{\cal H}(f)}{dt}\right|_{-{\cal L}} \leq -2\kappa {\cal H}(f).
\endeq
Such is the case in particular if $f$ satisfies the strict log concavity condition \eqref{condlogconc}.
\sm

(ii) Alternatively, for any $r\geq 1$, if $f$ is such that $\rho (x) = \int_{T_xM} f(x,v)\,dv$ is bounded above and below by positive constants, then there are constants $a,b,c,\kappa>0$, depending only on $n$, $r$, on $\inf \rho$, $\sup\rho$, on an upper bound on the curvature tensor and on the Poincar\'e inequality satisfied by $M$ (Definition \ref{defpoinc}), such that the functional $\tilde{\cal H}$ defined by \eqref{tildeHcal} satisfies
\begeq\label{ddtkappatildeH}
\left.\frac{d\tilde{\cal H}(f)}{dt}\right|_{-{\cal L}} \leq -2\kappa \tilde{\cal H}(f).
\endeq
Such is the case in particular if $f$ satisfies a local positivity estimate and is bounded in $L^\infty_s$ for some $s>n$.
\end{Thm}

\begin{Cor}[Exponential decay in $L^1$, improved results] \label{corexpdecayL1}
Let $M$ be a compact manifold and $f=f(t,x,v)$ be a time-dependent probability density (or probability measure) solving \eqref{eqf} and starting from an initial datum $f_0$. If $f_0$ has finite moment of order $s= s(n)$ large enough, then $f(t,\cdot)$ converges exponentially fast to the equilibrium $f_\infty(x,v) = \vol(M)^{-1} (2\pi)^{-n/2}\, e^{-|v|^2/2}$ as $t\to\infty$, in the sense of entropy:
\[ \iint_{\TM} f \log \frac{f}{f_\infty}\,dx\,dv  \leq C\, e^{-2\kappa t} \qquad \text{for $t\geq 1$}, \]
for some $\kappa>0$ depending only on $M$, and $C$ depending only on $M$ and on $\|f_0\|_{L^1_s}$. If $f_0$ has moments of all orders, then the convergence is also exponential in all weighted Sobolev spaces $H^s_\kappa$; likewise, for all integers $k,\ell,r$,
\[ (1+|v|^r) \bigl|\nabla_V^k \nabla_H^\ell (f - f_\infty) \bigr| = O(e^{-\lambda t}) \]
for any $\lambda<\kappa$. 

If in addition there is $t_0\geq 0$ such that for $t\geq t_0$, $f$ satisfies \eqref{condlogconc} for some positive constants $\alpha,A>0$, uniformly as $t\to\infty$, then the convergence is also exponential in Fisher information:
\[\iint_{\TM} f \left|\nabla_V \log \frac{f}{f_\infty}\right|^2\,dx\,dv + \iint_{\TM} f \left|\nabla_H \log \frac{f}{f_\infty}\right|^2\,dx\,dv = O (e^{-2\kappa t}).\]
\end{Cor}

\begin{proof}[Proof of Theorem \ref{thmfisherdecay}(i)]
By Cauchy--Schwarz inequality in $L^2(h\,d\mu)$,
\[ \left| \iint h\<\nabla_V \log h, \nabla_H \log h \>\,d\mu \right| 
\leq \sqrt{\iint h|\nabla_V\log h|^2\,d\mu} \sqrt{\iint h|\nabla_H\log h|^2\,d\mu}, \]
so that if $c<a$ and $b\ll \sqrt{ac}$, then 
\[ H_\mu (h\mu) + \frac{c}{2} I_\mu(h\mu) \leq {\cal H}(h) \leq H_\mu(h\mu) + 2a I_\mu(h\mu).\]
In other words, ${\cal H}$ is equivalent to the sum $H+I$, in the sense that the ratio between the two functionals is bounded from above and below.

Next, apply Lemmas \ref{lemlog1} and \ref{lemlog2}, with $C$ and $C'$ being either $(\nabla_H)_i$ or $(\nabla_V)_i$, to compute the time-derivative of \eqref{Hcal}. Using 
\[ [\Delta_V^\mu,\nabla_H] = 0,  \qquad [\Delta_V^\mu, \nabla_V] = \nabla_V, \]
one obtains, if $f$ evolves according to \eqref{eqf},
\begin{multline}
\frac{d{\cal H}}{dt} = - \iint h |\nabla_V\log h|^2\,d\mu\\ 
- 2 a \iint h \<\nabla_V\log h,\nabla_H\log h\> \,d\mu - 2 a \iint |\nabla_V^2 \log h|^2\,d\mu 
- 2 a \iint h |\nabla_V\log h|^2\,d\mu \\
- 2 b \iint h |\nabla_H \log h|^2\,d\mu + 2 b\iint h \bigl\< \nabla_V\log h, \Sigma(v,v) \nabla_V\log h \bigr\>\,d\mu \\
+ 4 b \iint h \<\nabla_V\nabla_H \log h, \nabla_V^2 \log h \>\,d\mu 
- 4 b \iint h \<\nabla_H\log h,\nabla_V\log h\>\,d\mu \\
+ 2 c \iint h \bigl\< \nabla_H \log h, \Sigma(v,v) \nabla_V \log h \bigr\> \,d\mu
-2 c \iint h |\nabla_V\nabla_H \log h|^2\,d\mu.
\end{multline}
At this stage apply again Cauchy--Schwarz inequality in $L^2(h\,d\mu)$, and the pointwise quadratic bound on $\Sigma(v,v)$, to systematically reduce to quadratic expressions:
\[ \left| \iint h \bigl\< \nabla_H \log h, \Sigma(v,v) \nabla_V\log h\bigr\> \, d\mu \right|
\leq C \sqrt{\iint h |v|^2 |\nabla_H \log h|^2\,d\mu} \sqrt{\iint h|v|^2 |\nabla_V\log h|^2\,d\mu}, \]
etc. Then apply inequality \eqref{GTPineq} to trade weights for vertical derivatives:
\[ \iint h |v|^2 |\nabla_H \log h|^2\,d\mu \leq C \left( \iint h |\nabla_V\nabla_H \log h|^2\,d\mu + \iint h |\nabla_H \log h|^2\,d\mu\right), \]
etc. Then use Young inequality and reason as in the proof of Theorem \ref{thmhypocoL2} to conclude that if
\[ c\ll b \ll a \ll 1, \qquad a\ll \sqrt{b},\qquad b \ll \sqrt{ac},\]
then
\begin{multline*} 
\frac{d{\cal H}(h)}{dt} \leq -K 
\left( \iint h |\nabla_V \log h|^2\,h\,d\mu + \iint h |\nabla_H\log h|^2\,d\mu \right. \\
\left. + \iint h |\nabla_V^2 \log h|^2\,d\mu + \iint h |\nabla_H\log h|^2\,d\mu \right).
\end{multline*}

From Proposition \ref{logsob}, the reference measure $\mu$ on $\TM$ satisfies a log Sobolev inequality; hence, reducing $K$ if necessary,
\begin{align*} \frac{d{\cal H}(h)}{dt} 
& \leq -K 
\left( \int h \log h\,d\mu+ \iint h |\nabla_V \log h|^2\,h\,d\mu + \iint h |\nabla_H\log h|^2\,d\mu \right. \\
& \qquad\qquad\qquad\qquad \left. + \iint h |\nabla_V^2 \log h|^2\,d\mu + \iint h |\nabla_H\log h|^2\,d\mu \right) \\
& \leq -K {\cal H},
\end{align*}
and the proof of \eqref{ddtkappaH} is complete.
\end{proof}

\begin{Rk} If $M$ has negative sectional curvature $\sigma<0$ then, recalling \eqref{Rvuuv},
\[ \iint h \bigl\<\nabla_V\log h, \Sigma(v,v) \nabla_V\log h\bigr\>\, d\mu =  \iint h \bigl| v\wedge \nabla_V\log h\bigr|^2\,\sigma_x\,d\mu <0 \]
goes in the right direction for the purpose of the theorem, which might help. If the curvature is positive then nothing is gained at that level, in contrast with the theory of logarithmic Sobolev inequalities and spectral gap inequalities, where positive curvature almost always helps. (Thanks to Fr\'ed\'eric Rousset and Georgios Moschidis for asking.)
\end{Rk}

\begin{proof}[Proof of Theorem \ref{thmfisherdecay}(ii)]
Now the relevant functional is $\tilde{\cal H}$. I shall use the notation $\<v\> = \sqrt{1+|v|^2}$, not to be confused with the average. Compute again the time-derivative using Lemmas \ref{lemlog1} and \ref{lemlog2}, taking into account the extra terms caused by the kinetic weights.
Note that
\begin{align*} \Delta_V^\mu (\<v\>^{-r}) & = - \frac{nr}{\<v\>^{r+2}} + \frac{r(r+2)|v|^2}{\<v\>^{r+4}} + \frac{r|v|^2}{\<v\>^{r+2}} \\
& = r \bigl( 1 + O (\<v\>^{-2}) \bigr) \<v\>^{-r}.
\end{align*}
Thus
\begin{multline}\label{dcalHdt}
\frac{d\tilde{\cal H}}{dt} = - \iint h |\nabla_V\log h|^2\,d\mu\\ 
- 2 a \iint h \<\nabla_V\log h,\nabla_H\log h\> \,\frac{d\mu}{\<v\>^r} - 2 a \iint |\nabla_V^2 \log h|^2\,\frac{d\mu}{\<v\>^r}
- 2 a \iint h |\nabla_V\log h|^2\,\frac{d\mu}{\<v\>^r} \\
+ 2 a r \iint h |\nabla_V \log h|^2 \, \bigl( 1 + O (\<v\>^{-2}) \bigr) \,\frac{d\mu}{\<v\>^r}\\
- 2 b \iint h |\nabla_H \log h|^2\,\frac{d\mu}{\<v\>^{2r}} + 2 b\iint h \bigl\< \nabla_V\log h, \Sigma(v,v) \nabla_V\log h \bigr\>\,\frac{d\mu}{\<v\>^{2r}} \\
+ 4 b \iint h \<\nabla_V\nabla_H \log h, \nabla_V^2 \log h \>\,\frac{d\mu}{\<v\>^{2r}}
- 4 b \iint h \<\nabla_H\log h,\nabla_V\log h\>\,\frac{d\mu}{\<v\>^{2r}} \\
+ 4b r \iint h \<\nabla_H\log h,\nabla_V\log h\>\,\bigl(1+O(\<v\>^{-2})\bigr)\,\frac{d\mu}{\<v\>^{2r}} \\
+ 2 c \iint h \bigl\< \nabla_H \log h, \Sigma(v,v) \nabla_V \log h \bigr\> \,\frac{d\mu}{\<v\>^{3r}}
-2 c \iint h |\nabla_V\nabla_H \log h|^2\,\frac{d\mu}{\<v\>^{3r}} \\
+ 2 c r \iint h |\nabla_H \log h|^2\, \bigl( 1+O(\<v\>^{-2})\bigr)\, \frac{d\mu}{\<v\>^{3r}}.
\end{multline}

There are four nice negative terms in there:
\begin{multline} \label{fournice}
\frac{d\tilde{\cal H}}{dt} = - \left( \iint h |\nabla_V\log h|^2\,d\mu 
 + 2 a \iint h |\nabla_V^2 \log h|^2\,\frac{d\mu}{\<v\>^r}
+ 2 b \iint h |\nabla_H \log h|^2\,\frac{d\mu}{\<v\>^{2r}} \right. \\ \left. + 
2 c \iint h |\nabla_V\nabla_H \log h|^2\,\frac{d\mu}{\<v\>^{3r}}\right) \quad + (R),
\end{multline}
where $(R)$ stands for all the other terms. Note that the third integral in the right-hand side of \eqref{dcalHdt} has the right sign but is dominated by the fourth one, as soon as $r>1$, so it will not be useful here. All other remaining terms making up $(R)$ will be controlled one by one:
\[
\left| \iint  h \<\nabla_V\log h,\nabla_H\log h\> \,\frac{d\mu}{\<v\>^r} \right| \\ \leq
\sqrt{\iint h |\nabla_V\log h|^2\,d\mu} \sqrt{\iint h |\nabla_H \log h|^2\,\frac{d\mu}{\<v\>^{2r}}};
\]
\[
\left| \iint h |\nabla_V \log h|^2 \, \bigl( 1 + O (\<v\>^{-2}) \bigr) \,\frac{d\mu}{\<v\>^r} \right| \leq
C \iint h |\nabla_V\log h|^2\,d\mu; 
\]
\[ \left|\iint h \bigl\< \nabla_V\log h, \Sigma(v,v) \nabla_V\log h \bigr\>\,\frac{d\mu}{\<v\>^{2r}} \right| 
\leq C \iint h |\nabla_V\log h|^2\,d\mu \qquad \text{if $r\geq 1$}; \]
\begin{multline*} \left|  \iint h \<\nabla_V\nabla_H \log h, \nabla_V^2 \log h \>\,\frac{d\mu}{\<v\>^{2r}} \right| \\
\leq \sqrt{\iint h |\nabla_V^2 \log h|^2\,\frac{d\mu}{\<v\>^r}} \sqrt{\iint h |\nabla_V\nabla_H \log h|^2\,\frac{d\mu}{\<v\>^{3r}}};
\end{multline*} 
\begin{multline*} \left|  \iint h \<\nabla_V \log h, \nabla_V \log h \>\,\bigl[ 1+ O(\<v\>^{-2}) \bigr] \frac{d\mu}{\<v\>^{2r}} \right| \\
\leq C \sqrt{\iint h |\nabla_V\log h|^2\,d\mu} \sqrt{\iint h |\nabla_H \log h|^2\,\frac{d\mu}{\<v\>^{2r}}};
\end{multline*}
\begin{multline*} \left| \iint h \bigl\< \nabla_H \log h, \Sigma(v,v) \nabla_V \log h \bigr\> \,\frac{d\mu}{\<v\>^{3r}}  \right| \\
\leq C \sqrt{\iint h |\nabla_V\log h|^2\,d\mu} \sqrt{\iint h |\nabla_H \log h|^2\,\frac{d\mu}{\<v\>^{2r}}} \qquad
\text{if $r\geq 1$};
\end{multline*}
\[  \iint h |\nabla_H \log h|^2\, \bigl( 1+O(\<v\>^{-2})\bigr)\, \frac{d\mu}{\<v\>^{3r}} 
\leq C \iint h |\nabla_H \log h|^2\,\frac{d\mu}{\<v\>^{2r}}.\]
Inserting all these bounds into \eqref{fournice} and playing systematically Young's inequality, shows that
\begin{multline} \label{fournice'}
\frac{d\tilde{\cal H}}{dt} \leq - \frac12 \left( \iint h |\nabla_V\log h|^2\,d\mu 
 + 2 a \iint h |\nabla_V^2 \log h|^2\,\frac{d\mu}{\<v\>^r}
+ 2 b \iint h |\nabla_H \log h|^2\,\frac{d\mu}{\<v\>^{2r}} \right. \\ \left. + 
2 c \iint h |\nabla_V\nabla_H \log h|^2\,\frac{d\mu}{\<v\>^{3r}}\right),
\end{multline}
as soon as
\[ c\ll b\ll a\ll 1, \qquad a \ll \sqrt{b},\qquad b \ll \sqrt{ac}.\]

When these conditions are imposed, it remains to control $\tilde{\cal H}$ by the negative of the right-hand side of \eqref{fournice'}. For the Fisher-type terms this is obvious, since
\begin{multline*} a \iint h |\nabla_V\log h|^2\,\frac{d\mu}{\<v\>^{r}} + c \iint h |\nabla_H \log h|^2\,\frac{d\mu}{\<v\>^{3r}} \ll 
\iint  h |\nabla_V\log h|^2\,d\mu  \\ + b \iint h |\nabla_H \log h|^2\,\frac{d\mu}{\<v\>^{2r}}.
\end{multline*}
For the entropic term $H_\mu(h\mu) = \int h\log h\, d\mu$ this comes from Proposition \ref{propcoertame}(iii), rewriting \eqref{qtroncL} in the form
\[
\iint_{\TM} h\log h\,d\mu \leq \frac1{\tilde{K}_L\bigl(\<v\>^{-{2r}}, h\bigr)} \left( \iint_{\TM} \frac{|\nabla_V h|^2}{h}\,d\mu + \iint_{\TM} \frac{|\nabla_H h|^2}{h}\,\frac{d\mu}{\<v\>^{2r}}\right).\]
This concludes the proof of \eqref{ddtkappatildeH}.
\end{proof}

\begin{proof}[Proof of Corollary \ref{corexpdecayL1}]
First, by regularisation and moments estimates (Theorems \ref{thmloc}(i) and \ref{thmregul}(iii)), $f$ has finite entropy for $t>0$. It also satisfies estimates in weighted Sobolev spaces $H^2_s$, where is $s$ is arbitrarily large, then by a result of P.-L.~Lions and myself, $\sqrt{f}$ is bounded in Sobolev space $H^1$, and thus the Fisher information $I (h\mu) = \int |\nabla f|^2/f$ is bounded too. From that and moment estimates, also $I_\mu(h\mu) = \int (|\nabla h|^2/h)\,d\mu$ is finite. Thus ${\cal H}(h)$ is finite for positive times, as well as $\tilde{\cal H}(h)$. 

Moreover, still by the regularisation theorem \ref{thmregul}(iii), $f(t,\cdot)$ is bounded in $L^\infty_s$ (the proof shows that this only involves moments of a certain order for the initial datum), thus by $v$-integration $\rho$ is bounded above, uniformly in time. Also, $f$ satisfies a uniform local lower bound for positive times, Theorem \ref{proplocalpos}. Then all conditions are fulfilled to apply Theorem \ref{thmfisherdecay}(ii). 

The differential inequality \eqref{ddtkappaH} implies exponential convergence of ${\cal H}$. From the choice of $a,b,c$, ${\cal H}(f)\geq H_\mu(f)$, so the relative Boltzmann information does converge exponentially fast to~0. By Pinsker's inequality, this also implies exponential $L^1$ convergence of $f$ like $O(e^{-\kappa t})$.

Again, the regularisation Theorem \ref{thmregul} implies that $f$ is uniformly smooth in weighted Sobolev spaces of arbitrarily large regularity indices and weights. By localisation estimates, arbitrarily large moments of $f$ are uniformly bounded, and thus by weighted moment interpolation, all moments of $f$ converge exponentially fast. Then by anisotropic Nash-type interpolation, all Sobolev norms converge exponentially fast. By Sobolev embedding, this is true also for all weighted $C^k$ norms. If the interpolation is done with weights and derivations of high enough order, the resulting rate of exponential decay will be as close as desired to $\kappa$. 

At this stage $\kappa$ may depend on $f$ through the estimates on $\sup \rho$ and $\inf \rho$. But the convergence in weighted Sobolev spaces implies that those quantities both converge to~1. So by restarting the differential inequality at some large enough time, the obtained rate will be independent of the regularity and positivity estimates. (The prefactor of the exponential will be affected, though.)

Finally if the Gaussian tail Poincar\'e inequality for the time-dependent density is satisfied, then one can apply Theorem \ref{thmfisherdecay}(i) to show that ${\cal H}$ converges exponentially fast to~0; by the choice of coefficients ${\cal H}(f) \geq H_\mu(f) + (c/2) I_\mu(f)$, so not only the relative Boltzmann functional, but also the Fisher information, converge exponentially fast to~0. This concludes the proof.
\end{proof}

\begin{Rks} 
\begin{itemize}
\item[(i)] It is not obvious to me that the exponential decay in Fisher information can be obtained from exponential convergence in entropy and Sobolev regularity bounds. If such is the case, the use of the Gaussian tail Poincar\'e can be dispended with.

\item[(ii)] At this stage the exponential convergence of Fisher information is conditional to the proof of the Gaussian tail Poincar\'e inequality, but the exponential convergence in entropy sense is fully proved, as well as in weighted Lebesgue and Sobolev spaces.
\end{itemize}
\end{Rks}

\bibnotes

Boltzmann's $H$ Theorem is one of the most fundamental results in mathematical physics; its proof and implications are discussed in a number of sources, including my general review articles on Boltzmann equation, like the older \cite{vill:handbook:02} or the more recent and concise \cite{MV:companion} with Mouhot; or the more focused proceedings of the 2004 International Congress of Mathematical Physics \cite{ICMP}.

As already mentioned, the Pinsker (or Csisz\'ar--Kullback--Pinsker) inequality is proven in many references, such as \cite{csi:inf:67,pinsker:book:64,vill:int}. As already mentioned, Fisher information was introduced by Ronald Fisher \cite{fisher:25} in 1925 for his theory of efficient statistics, then imported to kinetic theory by Henry P.~McKean \cite{mck:kac:65}, and recently benefited from a remarkable revival \cite{vill:fisher-Festum:25}.

The earliest proofs of \eqref{BEid} known to me are hiddden in Stam \cite{stam:59} (see Carlen's account \cite{carlen:logsob:91}) and McKean \cite{mck:kac:65} (see Toscani's rewriting \cite{tosc:fplcras:97,tosc:fpl:99}). The Bakry--\'Emery calculus is in \cite{bakry:lnm:94,bakem:hyperc:85}. Generalisations of those manipulations, in velocity space, appear in my course on Fisher information for Boltzmann equation \cite{vill:fisher-Festum:25}. Variants appear in \cite[Lemma~32]{vill:hypoco}, which extends Lemmas \ref{lemlog1} and \ref{lemlog2} to more general structure assumptions, albeit only in flat geometry and for $a\equiv 1$, and in a more painful way.

The marriage of Bakry--\'Emery calculus with hypocoercivity appeared in \cite{DOV:preprint} and was generalised by Baudoin \cite{baudoin:BE+V}. The strategy of hypocoercivity in Fisher information sense underlying the proof of Theorem \ref{thmfisherdecay} appeared in \cite[Theorem~28]{vill:hypoco}. In \cite{DOV:preprint} as well as in other attempts, we first stumbled on the problem of controlling weighted quadratic expressions of log derivatives.

The control of $\sqrt{f}$ in $H^1$ from the estimate of $f$ in $H^2_\kappa$ with $\kappa$ large enough is in Lions--Villani~\cite{PLLV:95} and Lions--Toscani~\cite[Lemma~1]{TV:slow:00}. Using charts it is easy to go from the flat estimate there to a curved one, as I did for the proof of Proposition~\ref{propsobemb}. It would also be interesting to have a direct intrinsic proof by working on $\TM$ with the Sasaki metric.

The failure of inequality \eqref{ILS2} was pointed to me by Ledoux.

The problem of the Gaussian tail Poincar\'e inequality for the time-dependent measure will be examined in a forthcoming work; it will involve pointwise regularisation methods, in contrast with the Hilbertian approach (via Sobolev embeddings) pursued here.

\section{Conclusions and final comments}

All along the way, I have shown that the kinetic Fokker--Planck equation accommodates quite well with Riemannian geometry, through the notions of covariant and tensor calculus, horizontal and vertical derivations. The fundamental properties of localisation, regularisation, positivity and equilibration can all be treated with functional methods similar in spirit to those in Euclidean space, based on divergence formulas, commutators, interpolation and coercivity inequalities. Moreover this can be done in a consistent way for the three main variants of the equation: $L^1$, $L^2$, and $L^2$ for vector fields. This does a lot to reinforce the coherence of the framework from my older memoir on Hypocoercivity.

In this geometric programme, the most notable departure from the flat case is due to the commutator between the horizontal derivation and geodesic flow, which yields a vertical derivation with quadratic coefficients:
\[ [\nabla_H, \xi] = O(|v|^2) \nabla_V. \]
This stretches all estimates, but still falls within the admissible range. Eventually, localisation, regularisation, positivity and equilibration results are formally similar to those in flat geometry, in spite of some notable additional difficulties. For the hypocoercivity problem in $L^2$, also the lower-regularity approach by Dolbeault, Mouhot and Schmeiser works out just fine. The worst complications are found in the $L^1$-equilibration problem, which have led to tweak the information functional, not only by twisting the Fisher information, but also by taming the gradients by inverse powers of $|v|$; and to establish appropriate variants of the log Sobolev and $\Gamma_2$ calculus in that context.

A condition which emerges consistently is a square exponential moment bound, or Gaussian tail moment condition,
\begeq\label{condexp0} \iint e^{\beta_0\frac{|v|^2}2}\, f(x,v)\,dx\,dv < +\infty, \endeq
which (a) leads to propagation and even improvement of the exponent $\beta_0$, (b) provides pointwise Gaussian bounds for positive times, (c) allows to approximate the $L^1$ convergence problem by the $L^2$ one, (d) suggests a neat linear differential inequality for a combination of entropy and twisted Fisher information. That programme can be fully completed if a first-order integral reinforcement of \eqref{condexp0} is proven for the density at positive times, namely the Gaussian tail Poincar\'e inequality (Definition \ref{defGTP}). At this point it is worth recalling that some of the most famous results in kinetic theory, like Lanford's derivation of the Boltzmann equation, crucially use a pointwise Gaussian bound on the density.

At the level of the functional toolbox, a key role was played by the anisotropic Nash inequality from Theorem \ref{propnashinterp}: an interpolation inequality involving $L^1$ moments and $L^2$ horizontal and vertical Sobolev norms, which was a cornerstone of the $L^1$ regularisation. Among functional tools, this is certainly where the treatment is most incomplete, with unnatural restrictions on the range of exponents, and a rather unsatisfactory extrinsic proof. Even though that version is sufficient for this study, a better treatment is natural to ask for.

Probably the other place in this memoir in which the treatment was not intrinsic enough, is the local positivity estimate from Section \ref{secpos}. This estimate also is also of tantamount importance, in itself and because of its link to $L^1$ equilibration.

As far as global hypoelliptic regularisation is concerned in $L^2(\mu)$, for either the scalar-valued or the vector-valued problem, a consistent picture was presented along the way, extending the pioneering work of Lebeau. Eventually there are four strongly related inequalities with optimal exponents, all proven with more elementary tools than before:

\bul the fractional regularisation in the stationary equation (Theorem \ref{thmalabouchut});

\bul the regularisation time scales in the evolution equation (Theorem \ref{thmregul});

\bul the skewed interpolation inequality for the Fokker--Planck operator (inequality \eqref{skewedinterp});

\bul the algebraic localisation of the spectrum (estimate \eqref{localspectrum}).

Applied to the problem of $L^1$ equilibration, Fisher information has proven once again its versatility, able to incorporate geometric information, with all nonlinear identities going through, fit for direct proof of exponential decay in natural information-theoretic functionals or their variants.

All throughout, the burden of large velocities has been very heavy, at times concentrating most of the technical difficulties. This reinforces the idea, already explored by various authors, that unit-speed diffusion models are quite more tractable. However, such models do not have natural nonlinear counterparts, unlike the original Fokker--Planck equation.

Some further questions left aside from this work are:

\bul Asymptotics of equilibration rates when the friction is either very low (and the rate should decay as the inverse of the friction) or very high (and then the rate should approach a fixed value determined by the heat equation in this geometry). Especially the high-friction case has been studied by Bailleul, Baudoin, Bismut, Helffer, H\'erau, Lebeau, Nier, Ren, Tao, Sang, Tardif, White under various assumptions; a recent work by Ren--Tao establishes neat results for the $L^2(\mu)$ problem with unit speed. It is natural to for a treatment of that problem along the lines of this memoir; also the $L^1$ theory has not been touched in those works as far as I know.

\bul Unbounded geometries with a confining potential, possibly with restrictions on the possible growth of the curvature tensors.

\bul A Riemannian adaptation of the popular coupling techniques which have already proven their efficiency for the study of equilibration of the kinetic Fokker--Planck equation in the Euclidean setting (works by Baudoin, Bolley, Guillin, Le Bris, Malrieu, Monmarch\'e).

\bul The nonlinear interaction problem, in which an interaction potential is added (Vlasov--Fokker--Planck equation): \eqref{eqf} should be replaced by
\begeq\label{NLV}
\derpar{f}{t} + \xi f - \nabla_x \left(\int_{\TM}  W(x,y) f(y,w)\,dx\,dy \right) \cdot\nabla_v f(x,v)  = \nabla_V\cdot (\nabla_V f + fv), 
\endeq
where $W:M\times M\to \R$ is the interaction potential (felt at $x$ from the action of a particle located at $y$). It is in this regime, possibly far from equilibrium, that the interest of information methods should be most apparent. While smooth potentials are certainly within range of the present techniques, for more realistic (singular) potentials, the adaptation of the flat theory might be a substantial work. (By the way, the Vlasov--Fokker--Planck equation was much less studied by mathematicians than the Vlasov equation.) As for the equilibration of \eqref{NLV}, it can be expected to be tricky even when the potential is smooth, as already the case of the torus (considered in my memoir on Hypocoercivity) has subtleties, with phase transition arising for large enough variations of $W$.

To this list, of course, one may add nonlinear diffusive and collisional relatives of the kinetic Fokker--Planck equation, connecting to the extraordinarily rich landscape of collisional kinetic theory.

\bibliographystyle{acm}

\bibliography{../../Biblio/biblio, ../../Biblio/mybiblio}

\def\cprime{$'$} \def\cprime{$'$} \def\cprime{$'$} \def\cprime{$'$}
\begin{thebibliography}{100}

\bibitem{AAMN:hypo:25}
{\sc Achleitner, F., Arnold, A., Mehrmann, V., and Nigsch, E.}
\newblock Hypocoercivity in {H}ilbert spaces.
\newblock {\em J. Funct. Anal. 288}, 2 (2025), 110691.

\bibitem{AAMN:KFP:24}
{\sc Albritton, D., Armstrong, S., Mourrat, J.-C., and Novack, M.}
\newblock Variational methods for the kinetic fokker-planck equation.
\newblock {\em Anal. PDE 17}, 6 (2024), 1953--2010.

\bibitem{AV:boltz:02}
{\sc Alexandre, R., and Villani, C.}
\newblock On the {B}oltzmann equation for long-range interactions.
\newblock {\em Comm. Pure Appl. Math. 55}, 1 (2002), 30--70.

\bibitem{andreasson:einsteinvlasov:11}
{\sc Andr{\'e}asson, H.}
\newblock The {E}instein--{V}lasov system / kinetic theory.
\newblock {\em Living Rev. Rel. 14}, 4 (2011).

\bibitem{ABT:brownian:15}
{\sc Angst, J., Bailleul, I., and Tardif, C.}
\newblock Kinetic {B}rownian motion on {R}iemannian manifolds.
\newblock {\em Electron. J. Probab. 20\/} (2015), 1--40.

\bibitem{anosov:geodesic:67}
{\sc Anosov, D.}
\newblock Geodesic flows on closed {R}iemannian manifolds with negative
  curvature.
\newblock {\em Proc. Steklov Inst. Mathematics 90\/} (1967), 3--210.

\bibitem{AMTU:FP:01}
{\sc Arnold, A., Markowich, P., Toscani, G., and Unterreiter, A.}
\newblock On logarithmic {S}obolev inequalities and the rate of convergence to
  equilibrium for {F}okker--{P}lanck type equations.
\newblock {\em Comm. Partial Differential Equations 26}, 1--2 (2001), 43--100.

\bibitem{arnoldtoshpulatov:25}
{\sc Arnold, A., and Toshpulatov, G.}
\newblock Trend to equilibrium and hypoelliptic regularity for the relativistic
  {F}okker--{P}lanck equation.
\newblock {\em SIAM J. Math. Anal. 57}, 3 (2025), 10.1137/24M1671244.

\bibitem{arnold:EDO:mir}
{\sc Arnol{\cprime}d, V.}
\newblock {\em \'{E}quations diff\'erentielles ordinaires}, fourth~ed.
\newblock Traduit du Russe: Math\'ematiques. [Translations of Russian Works:
  Mathematics]. ``Mir'', Moscow, 1988.
\newblock Translated from the Russian by Djilali Embarek.

\bibitem{aronson:fundamental:67}
{\sc Aronson, D.}
\newblock Bounds for the fundamental solution of a parabolic equation.
\newblock {\em Bull. Amer. Math. Soc. 73\/} (1967), 890--896.

\bibitem{bakry:lnm:94}
{\sc Bakry, D.}
\newblock L'hypercontractivit{\'e} et son utilisation en th{\'e}orie des
  semigroupes.
\newblock In {\em {\'E}cole d'{\'e}t{\'e} de Probabilit{\'e}s de
  {S}aint-{F}lour}, no.~1581 in Lecture Notes in Math. Springer, 1994.

\bibitem{bakem:hyperc:85}
{\sc Bakry, D., and {\'E}mery, M.}
\newblock Diffusions hypercontractives.
\newblock In {\em S{\'e}m. Proba. {XIX}}, no.~1123 in Lecture Notes in
  Mathematics. Springer, 1985, pp.~177--206.

\bibitem{BGL:book}
{\sc Bakry, D., Gentil, I., and Ledoux, M.}
\newblock {\em Analysis and Geometry of Markov Diffusion Operators}, vol.~348
  of {\em Grundlehren der mathematischen Wissenschaften}.
\newblock Springer-Verlag, New York, Berlin, 2014.

\bibitem{BDR:relatK:01}
{\sc Barbachoux, C., Debbasch, F., and Rivet, J.}
\newblock Covariant {K}olmogorov equation and entropy current for the
  relativistic {O}rnstein--{U}hlenbeck process.
\newblock {\em Eur. Phys. J. B 20\/} (2001), 487--496.

\bibitem{BDR:relat:01}
{\sc Barbachoux, C., Debbasch, F., and Rivet, J.}
\newblock The spatially one-dimensional relativistic {O}rnstein--{U}hlenbeck
  process in arbitrary inertial frame.
\newblock {\em Eur. Phys. J. B 19\/} (2001), 37--47.

\bibitem{bass:diffusionbook}
{\sc Bass, R.~F.}
\newblock {\em Dffusions and Elliptic Operators}.
\newblock Probability and its Applications. Springer-Verlag, 1998.

\bibitem{baudoin:wasshypo:16}
{\sc Baudoin, F.}
\newblock Wasserstein contraction properties for hypoelliptic diffusions.
\newblock ArXiV Preprint, 1602.0417, 2016.

\bibitem{baudoin:BE+V}
{\sc Baudoin, F.}
\newblock Bakry--{{\'E}}mery meet {V}illani.
\newblock {\em J. Funct. Anal. 273}, 7 (2017), 2275--2291.

\bibitem{BT:foliations:18}
{\sc Baudoin, F., and Tardif, C.}
\newblock Hypocoercive estimates on foliations and velocity spherical
  {B}rownian motion.
\newblock {\em Kinet. Relat. Models 11}, 1 (2018), 1--23.

\bibitem{berger:panoramic:book}
{\sc Berger, M.}
\newblock {\em A Panoramic View of {R}iemannian Geometry}.
\newblock Springer, 2003.

\bibitem{BFLS:hypo:22}
{\sc Bernard, {\'E}., Fathi, M., L{\'e}vitt, A., and Stoltz, B.}
\newblock Hypocoercivity with {S}chur complements.
\newblock {\em Annales Henri Lebesgue 5\/} (2022), 523--557.

\bibitem{bismut:surveyhypo:08}
{\sc Bismut, J.}
\newblock A survey of the hypoelliptic {L}aplacian.
\newblock {\em Ast{\'e}risque 322\/} (2008), 39--69.

\bibitem{bismut:hypoelliptic:05}
{\sc Bismut, J.-M.}
\newblock The hypoelliptic {L}aplacian on the cotangent bundle.
\newblock {\em J. Amer. Math. Soc. 18}, 2 (2005), 379--476.

\bibitem{bismut:hypoLie:08}
{\sc Bismut, J.-M.}
\newblock The hypoelliptic {L}aplacian on a compact {L}ie group.
\newblock {\em J. Funct. Anal. 255\/} (2008), 2190--2232.

\bibitem{bismutlebeau:hypo:book}
{\sc Bismut, J.-M., and Lebeau, G.}
\newblock {\em The Hypoelliptic {L}aplacian and {R}ay--{S}inger Metrics}.
\newblock Annals of Mathematics Studies. Princeton University Press, Princeton,
  2008.

\bibitem{bob:theory:88}
{\sc Bobylev, A.~V.}
\newblock The theory of the nonlinear, spatially uniform {B}oltzmann equation
  for {M}axwellian molecules.
\newblock {\em Sov. Sci. Rev. C. Math. Phys. 7\/} (1988), 111--233.

\bibitem{BGM:VFP:10}
{\sc Bolley, F., Guillin, A., and Malrieu, F.}
\newblock Trend to equilibrium and particle approximation for a weakly
  selfconsistent {V}lasov--{F}okker--{P}lanck equation.
\newblock {\em ESAIM Math. Model. Numer. Anal. 44}, 5 (2010), 867--884.

\bibitem{boltz:weitere:72}
{\sc Boltzmann, L.}
\newblock Weitere {S}tudien {\"{u}}ber das {W}{\"{a}}rmegleichgewicht unter
  {G}asmolek{\"u}len.
\newblock {\em Sitzungsberichte der Akademie der Wissenschaften 66\/} (1872),
  275--370.
\newblock {Translation : Further studies on the thermal equilibrium of gas
  molecules, in {\em Kinetic Theory 2}, 88--174, Ed. S.G. Brush, Pergamon,
  Oxford (1966).}

\bibitem{boltz:book}
{\sc Boltzmann, L.}
\newblock {\em Lectures on gas theory}.
\newblock University of California Press, Berkeley, 1964.
\newblock Translated by Stephen G. Brush. Reprint of the 1896--1898 Edition by
  Dover Publications, 1995.

\bibitem{bouchut:hypoell:02}
{\sc Bouchut, F.}
\newblock Hypoelliptic regularity in kinetic equations.
\newblock {\em J. Math. Pures Appl. (9) 81}, 11 (2002), 1135--1159.

\bibitem{BDL:frachypo:22}
{\sc Bouin, E., Dolbeault, J., and Lafl\`eche, L.}
\newblock Fractional hypocoercivity.
\newblock {\em Comm. Math. Phys. 390}, 3 (2022), 1369--1411.

\bibitem{BDLS:hyposub:21}
{\sc Bouin, E., Dolbeault, J., Lafl\`eche, L., and Schmeiser, C.}
\newblock Hypocoercivity and sub-exponential local equilibria.
\newblock {\em Monatsh. Math. 194}, 1 (2021), 41--65.

\bibitem{BDS:hypoweak:20}
{\sc Bouin, E., Dolbeault, J., and Schmeiser, C.}
\newblock Diffusion limit for kinetic {F}okker--{P}lanck equation with heavy
  tails equilibria: the critical case.
\newblock {\em Kinet. Relat. Models 13}, 2 (2020), 345--371.

\bibitem{BM:hypofluid:22}
{\sc Bouin, E., and Mouhot, C.}
\newblock Quantitative fluid approximation in transport theory: A unified
  approach.
\newblock {\em Probability and Mathematical Physics 3}, 3 (2022), 491--542.

\bibitem{bramanti:invitation:book}
{\sc Bramanti, M.}
\newblock {\em An invitation to hypoelliptic operators and H{\"o}rmander's
  vector fields}.
\newblock Springer, Cham, 2014.

\bibitem{bramantibrandolini:hormander:book}
{\sc Bramanti, M., and Brandolini, L.}
\newblock {\em H{\"o}rmander operators}.
\newblock World Scientific, Singapore, 2023.

\bibitem{brezis:AF}
{\sc Br{\'e}zis, H.}
\newblock {\em Analyse fonctionnelle. {T}h{\'e}orie et applications.}
\newblock Masson, Paris, 1983.

\bibitem{brezismironescu:GN:18}
{\sc Brezis, H., and Mironescu, P.}
\newblock {G}agliardo–-{N}irenberg inequalities and non-inequalities: The
  full story.
\newblock {\em Ann. Inst. H. Poincar{\'e} C Anal. non lin{\'e}aire 35}, 5
  (2018), 1355--1376.

\bibitem{mouhot:Festum:24}
{\sc Brigatti, G., and Mouhot, C.}
\newblock Introduction to quantitative {D}e {G}iorgi methods.
\newblock Proceedings of the 2024 Festum Pi festival, Lect. Notes in Math.,
  2025.

\bibitem{CCEY:harris:20}
{\sc Ca\~{n}izo, J., Cao, C., Evans, J., and Yolda\c{s}, H.}
\newblock Hypocoercivity of linear kinetic equations via {H}arris's theorem.
\newblock {\em Kinet. Relat. Models 13}, 1 (2020), 97--128.

\bibitem{canizomischler:harris:23}
{\sc Ca\~{n}izo, J., and Mischler, S.}
\newblock Harris-type results on geometric and subgeometric convergence to
  equilibrium for stochastic semigroups.
\newblock {\em J. Funct. Anal. 284}, 7 (2023), 109830.

\bibitem{cao:kFP:21}
{\sc Cao, C.}
\newblock The kinetic {F}okker--{P}lanck equation with general force.
\newblock {\em J. Evol. Equ. 21}, 2 (2021), 2293--2337.

\bibitem{carleman}
{\sc Carleman, T.}
\newblock Sur la solution de l'\'equation int\'egrodiff\'erentielle de
  {B}oltzmann.
\newblock {\em Acta. Math. 60}, 3 (1933), 91--146.

\bibitem{carlen:logsob:91}
{\sc Carlen, E.}
\newblock Superadditivity of {F}isher's information and logarithmic {S}obolev
  inequalities.
\newblock {\em J. Funct. Anal. 101}, 1 (1991), 194--211.

\bibitem{CNP:heavycrit:19}
{\sc Cattiaux, P., Nasreddine, E., and Puel, M.}
\newblock Diffusion limit for kinetic {F}okker--{P}lanck equation with heavy
  tails equilibria: the critical case.
\newblock {\em Kinet. Relat. Models 12}, 4 (2019), 727--748.

\bibitem{chandr:43}
{\sc Chandrasekhar, S.}
\newblock Stochastic problems in physics and astronomy.
\newblock {\em Rev. Modern Phys. 15}, 1 (1943), 1--89.

\bibitem{CD:relat:08}
{\sc Chevalier, C., and Debbasch, F.}
\newblock Relativistic diffusions: A unifying approach.
\newblock {\em J. Math. Phys. 49}, 4 (2008), 043303.

\bibitem{CD:relatH:08}
{\sc Chevalier, C., and Debbasch, F.}
\newblock A unifying approach to relativistic diffusions and {$H$}-theorems.
\newblock {\em Mod. Phys. Lett. B 22}, 5 (2008), 383--392.

\bibitem{csi:inf:67}
{\sc Csisz{\'a}r, I.}
\newblock Information-type measures of difference of probability distributions
  and indirect observations.
\newblock {\em Stud. Sci. Math. Hung. 2\/} (1967), 299--318.

\bibitem{davison:book}
{\sc Davison, B.}
\newblock {\em Neutron Transport Theory}.
\newblock Clarendon Press, 1958.

\bibitem{debbasch:diffcurved:04}
{\sc Debbasch, F.}
\newblock A diffusion process in curved space-time.
\newblock {\em J. Math. Phys. 45}, 7 (2004), 2744--2760.

\bibitem{debbaschchevalier:relativistic:07}
{\sc Debbasch, F., and Chevalier, C.}
\newblock Relativistic stochastic processes: A review.
\newblock In {\em Medyfinol 2006: {XV} Conference on Nonequilibrium Statistical
  Mechanics and Nonlinear Physics\/} (2007), O.~Descalzi, O.~Rosso, and
  H.~Larrondo, Eds., vol.~913 of {\em AIP Conference Proceedings}, pp.~42--48.

\bibitem{DMR:relativistic:97}
{\sc Debbasch, F., Mallick, K., and Rivet, J.}
\newblock Relativistic {O}rnstein--{U}hlenbeck process.
\newblock {\em J. Stat. Phys. 88\/} (1997), 945--966.

\bibitem{DOV:preprint}
{\sc Debbasch, F., Ollivier, Y., and Villani, C.}
\newblock On the kinetic {F}okker-–{P}lanck equation in {R}iemannian
  geometry.
\newblock Preprint, 2008.

\bibitem{DR:relat:98}
{\sc Debbasch, F., and Rivet, J.}
\newblock A diffusion equation from the relativistic {O}rnstein--{U}hlenbeck
  process.
\newblock {\em J. Stat. Phys. 90\/} (1998), 1179--1199.

\bibitem{delcbers:book:94}
{\sc Delcroix, J.-L., and Bers, A.}
\newblock {\em Physique des Plasmas}.
\newblock InterEditions/ CNRS Editions, 1994.

\bibitem{DV:landau:1}
{\sc Desvillettes, L., and Villani, C.}
\newblock On the spatially homogeneous {L}andau equation for hard potentials.
  {I}. {E}xistence, uniqueness and smoothness.
\newblock {\em Comm. Partial Differential Equations 25}, 1-2 (2000), 179--259.

\bibitem{DV:FP:01}
{\sc Desvillettes, L., and Villani, C.}
\newblock On the trend to global equilibrium in spatially inhomogeneous
  entropy-dissipating systems: the linear {F}okker--{P}lanck equation.
\newblock {\em Comm. Pure Appl. Math. 54}, 1 (2001), 1--42.

\bibitem{DV:boltz:05}
{\sc Desvillettes, L., and Villani, C.}
\newblock On the trend to global equilibrium for spatially inhomogeneous
  kinetic systems: the {B}oltzmann equation.
\newblock {\em Invent. Math. 159}, 2 (2005), 245--316.

\bibitem{DPLM:average:91}
{\sc DiPerna, R.~J., Lions, P.-L., and Meyer, Y.}
\newblock {$L^p$} regularity of velocity averages.
\newblock {\em Ann. IHP 8}, 3--4 (1991), 271--287.

\bibitem{docarmo:Riemann:book}
{\sc do~Carmo, M.~P.}
\newblock {\em Riemannian geometry}.
\newblock Mathematics: Theory \& Applications. Birkh\"auser Boston Inc.,
  Boston, MA, 1992.
\newblock Translated from the second Portuguese edition by Francis Flaherty.

\bibitem{DMS:hypo:15}
{\sc Dolbeault, J., Mouhot, C., and Schmeiser, C.}
\newblock Hypocoercivity for linear kinetic equations conserving mass.
\newblock {\em Trans. Amer. Math. Soc. 367}, 6 (2015), 3807--3828.

\bibitem{dunkelhanggi:review:09}
{\sc Dunkel, J., and H{\"a}nggi, P.}
\newblock Relativistic {B}rownian motion.
\newblock {\em Phys. Rep. 471}, 1 (2009), 1--73.

\bibitem{eckmannhairer:uniqueness:01}
{\sc Eckmann, J.-P., and Hairer, M.}
\newblock Uniqueness of the invariant measure for a stochastic {PDE} driven by
  degenerate noise.
\newblock {\em Comm. Math. Phys. 219}, 3 (2001), 523--565.

\bibitem{eckmannhairer:hypo:03}
{\sc Eckmann, J.-P., and Hairer, M.}
\newblock Spectral properties of hypoelliptic operators.
\newblock {\em Comm. Math. Phys. 235}, 2 (2003), 233--253.

\bibitem{fabesstroock:86}
{\sc Fabes, E., and Stroock, D.}
\newblock A new proof of {M}oser's parabolic harnack inequality using the old
  ideas of {N}ash.
\newblock {\em Arch. Rational Mech. Anal. 96\/} (1986), 327--338.

\bibitem{fisher:25}
{\sc Fisher, R.~A.}
\newblock Theory of statistical estimation.
\newblock {\em Math. Proc. Cambridge Philos. Soc. 22\/} (1925), 700--725.

\bibitem{folland:75}
{\sc Folland, G.}
\newblock Subelliptic estimates and function spaces on nilpotent {L}ie groups.
\newblock {\em Ark. Mat. 13}, 2 (1975), 161--207.

\bibitem{FJW:book}
{\sc Frazier, M., Jawerth, B., and Weiss, G.}
\newblock {\em {L}ittlewood--{P}aley Theory and the Study of Function Spaces},
  vol.~79 of {\em CBMS}.
\newblock American Mathematical Society, 1991.

\bibitem{GHL:Riemann:book}
{\sc Gallot, S., Hulin, D., and Lafontaine, J.}
\newblock {\em Riemannian geometry}, second~ed.
\newblock Universitext. Springer-Verlag, Berlin, 1990.

\bibitem{GT:elliptic:book}
{\sc Gilbarg, D., and Trudinger, N.~S.}
\newblock {\em Elliptic partial differential equations of second order}.
\newblock Classics in Mathematics. Springer-Verlag, Berlin, 2001.
\newblock Reprint of the 1998 edition.

\bibitem{gross:log:75}
{\sc Gross, L.}
\newblock Logarithmic {S}obolev inequalities.
\newblock {\em Amer. J. Math. 97\/} (1975), 1061--1083.

\bibitem{GMM:nonsym:17}
{\sc Gualdani, M., Mischler, S., and Mouhot, C.}
\newblock {\em Factorization of non-symmetric operators and exponential
  {$H$}-Theorem}, vol.~153 of {\em M\'emoires de la SMF}.
\newblock Soci\'et\'e Math\'ematique de France, 2017.

\bibitem{GLBM:VFP:22}
{\sc Guillin, A., Le~Bris, P., and Monmarch{\'e}, P.}
\newblock Convergence rates for the {V}lasov--{F}okker--{P}lanck equation and
  uniform in time propagation of chaos in nonconvex cases.
\newblock {\em Electron. J. Probab. 27\/} (2022), 1--44.

\bibitem{GLWZ:kFP:21}
{\sc Guillin, A., Liu, W., Wu, L., and Zhang, C.}
\newblock The kinetic {F}okker--{P}lanck equation with mean field interaction.
\newblock {\em J. Math. Pures Appl. 150\/} (2021), 1--23.

\bibitem{guo:landau:02}
{\sc Guo, Y.}
\newblock The {L}andau equation in a periodic box.
\newblock {\em Comm. Math. Phys. 231}, 3 (2002), 391--434.

\bibitem{haase:sectorial:06}
{\sc Haase, M.}
\newblock {\em The {F}unctional {C}alculus for {S}ectorial {O}perators}.
\newblock No.~169 in Operator Theory: Advances and Applications.
  Birkh{\"a}user, Basel, 2006.

\bibitem{HM:SGW:08}
{\sc Hairer, M., and Mattingly, J.}
\newblock Spectral gaps in {W}asserstein distances and the {2D} stochastic
  {N}avier--{S}tokes equations.
\newblock {\em Ann. Probab. 36}, 6 (2008), 2050--2091.

\bibitem{HM:yetanother:11}
{\sc Hairer, M., and Mattingly, J.}
\newblock Yet another look at {H}arris' ergodic theorem for {M}arkov chains.
\newblock In {\em Seminar on {S}tochastic {A}nalysis, Random Fields and
  Applications {VI}\/} (2011), R.~Dalang, M.~Dozzi, and F.~Russo, Eds.,
  Springer, Basel, pp.~109--117.

\bibitem{hebey:manifolds:96}
{\sc Hebey, E.}
\newblock {\em {S}obolev spaces on manifolds}.
\newblock No.~1635 in Lecture Notes in Mathematics. Springer, 1996.

\bibitem{helfnier:witten:05}
{\sc Helffer, B., and Nier, F.}
\newblock {\em Hypoellipticity and spectral theory for {F}okker--{P}lanck
  operators and {W}itten {L}aplacians}, vol.~1862 of {\em Lecture Notes in
  Math.}
\newblock Springer, Berlin, 2005.

\bibitem{herau:linboltz:06}
{\sc H{\'e}rau, F.}
\newblock Hypocoercivity and exponential time decay for the linear
  inhomogeneous relaxation {B}oltzmann.
\newblock {\em Asympt. Anal. 46}, 3-4 (2006), 349--359.

\bibitem{herau:FP:07}
{\sc H{\'e}rau, F.}
\newblock Short and long time behavior of the {F}okker-{P}lanck equation in a
  confining potential and applications.
\newblock {\em J. Funct. Anal. 244}, 1 (2007), 95--118.

\bibitem{heraunier:FP:04}
{\sc H{\'e}rau, F., and Nier, F.}
\newblock Isotropic hypoellipticity and trend to equilibrium for the
  {F}okker--{P}lanck equation with a high-degree potential.
\newblock {\em Arch. Ration. Mech. Anal. 171}, 2 (2004), 151--218.

\bibitem{horm:hypo:67}
{\sc H{\"o}rmander, L.}
\newblock Hypoelliptic second order differential equations.
\newblock {\em Acta Math. 119\/} (1967), 147--171.

\bibitem{hormander:LPDO:3}
{\sc H{\"o}rmander, L.}
\newblock {\em The {A}nalysis of {L}inear {P}artial {D}ifferential {O}perators.
  {V}ol. {III}: {P}seudodifferential {O}perators}.
\newblock Grundlehren der Mathematischen Wissenschaften. Springer-Verlag,
  Berlin, 1985.

\bibitem{iktru:56}
{\sc Ikenberry, E., and Truesdell, C.}
\newblock On the pressures and the flux of energy in a gas according to
  {M}axwell's kinetic theory. {I}.
\newblock {\em J. Rational Mech. Anal. 5\/} (1956), 1--54.

\bibitem{IMS:boltzmann:20}
{\sc Imbert, C., Mouhot, C., and Silvestre, L.}
\newblock Decay estimates for large velocities in the {B}oltzmann equation
  without cut-of.
\newblock {\em J. {\'E}c. Polytech. Math. 7\/} (2020), 143--184.

\bibitem{jerison:poincarehormander:86}
{\sc Jerison, D.}
\newblock The {P}oincar{\'e} inequality for vector fields satisfying
  {H}{\"o}rmander’s condition.
\newblock {\em Duke Math. J. 53}, 2 (1986), 503--523.

\bibitem{KMN:FP:21}
{\sc Kavian, O., Mischler, S., and Ndao, M.}
\newblock The {F}okker--{P}lanck equation with subcritical confinement force.
\newblock {\em J. Math. Pures Appl. 151\/} (2021), 171--211.

\bibitem{kohn:hypo:69}
{\sc Kohn, J.}
\newblock Pseudo-differential operators and hypoellipticity.
\newblock In {\em Proc. Symp. Pure Math.\/} (1969), vol.~23, AMS Providence,
  RI, pp.~61--69.

\bibitem{landau:Coulomb:36}
{\sc Landau, L.}
\newblock Die kinetische {G}leichung f{\"u}r den {F}all {C}oulombscher
  {W}echselwirkung.
\newblock {\em Phys. Z. Sowjet. 10\/} (1936), 154.
\newblock Translation : The transport equation in the case of {C}oulomb
  interactions, in D. ter Haar, ed., {\em Collected papers of L.D. Landau},
  pp.~163--170. Pergamon Press, Oxford, 1981.

\bibitem{lebeau:FP1:05}
{\sc Lebeau, G.}
\newblock Geometric {F}okker-{P}lanck equations.
\newblock {\em Port. Math. (N.S.) 62}, 4 (2005), 469--530.

\bibitem{lebeau:FP2:07}
{\sc Lebeau, G.}
\newblock \'{E}quations de {F}okker-{P}lanck g\'eom\'etriques. {II}.
  {E}stimations hypoelliptiques maximales.
\newblock {\em Ann. Inst. Fourier (Grenoble) 57}, 4 (2007), 1285--1314.

\bibitem{LP:heavy:19}
{\sc Lebeau, G., and Puel, M.}
\newblock Diffusion approximation for {F}okker--{P}lanck with heavy tail
  equilibria: a spectral method in dimension 1.
\newblock {\em Commun. Math. Phys. 366}, 2 (2019), 709--735.

\bibitem{lions:landau:94}
{\sc Lions, P.-L.}
\newblock On {Boltzmann} and {Landau} equations.
\newblock {\em Philos. Trans. Roy. Soc. London Ser. A A}, 346 (1994), 191--204.

\bibitem{PLLV:95}
{\sc Lions, P.-L., and Villani, C.}
\newblock R\'egularit\'e optimale de racines carr\'ees.
\newblock {\em C. R. Acad. Sci. Paris S\'er. I Math. 321}, 12 (1995),
  1537--1541.

\bibitem{lunardi:interpolation}
{\sc Lunardi, A.}
\newblock {\em Interpolation theory}, third ed.~ed.
\newblock Ed. Scuola Normale Superiore, 2018.

\bibitem{maxw:67}
{\sc Maxwell, J.~C.}
\newblock On the dynamical theory of gases.
\newblock {\em Philos. Trans. Roy. Soc. London Ser. A 157\/} (1867), 49--88.

\bibitem{mck:kac:65}
{\sc McKean, H.~J.}
\newblock Speed of approach to equilibrium for {K}ac's caricature of a
  {M}axwellian gas.
\newblock {\em Arch. Rational Mech. Anal. 21\/} (1966), 343--367.

\bibitem{menegaki:anharmonic:20}
{\sc Menegaki, A.}
\newblock Quantitative rates of convergence to non-equilibrium steady state for
  a weakly anharmonic chain of oscillators.
\newblock {\em J. Stat. Phys. 181\/} (2020), 53--94.

\bibitem{meyntweedie:stability:92}
{\sc Meyn, S.~P., and Tweedie, R.~L.}
\newblock Stability of {M}arkovian processes. {I}. {C}riteria for discrete-time
  chains.
\newblock {\em Adv. in Appl. Probab. 24}, 3 (1992), 542--574.

\bibitem{meyntweedie:stability:93}
{\sc Meyn, S.~P., and Tweedie, R.~L.}
\newblock Stability of {M}arkovian processes. {II}. {C}ontinuous-time processes
  and sampled chains.
\newblock {\em Adv. in Appl. Probab. 25}, 3 (1993), 487--517.

\bibitem{meyntweedie:geometric:94}
{\sc Meyn, S.~P., and Tweedie, R.~L.}
\newblock Computable bounds for geometric convergence rates of {M}arkov
  {C}hains.
\newblock {\em Annal in Appl. Probab. 4}, 4 (1994), 981--1011.

\bibitem{meyntweedie:MCbook}
{\sc Meyn, S.~P., and Tweedie, R.~L.}
\newblock {\em Markov chains and stochastic stability}.
\newblock Cambridge University Press, Cambridge, 2010.

\bibitem{MM:expslow:16}
{\sc Mischler, S., and Mouhot, C.}
\newblock Exponential stability of slowly decaying solutions to the kinetic
  {F}okker--{P}lanck equation.
\newblock {\em Arch. Ration. Mech. Anal. 221}, 2 (2016), 677--723.

\bibitem{MQT:NN:16}
{\sc Mischler, S., Qui\~{n}inao, C., and Touboul, G.}
\newblock On a kinetic {F}itzhugh--{N}agumo model of neuronal network.
\newblock {\em Comm. Math. Phys. 342}, 3 (2016), 1001--1042.

\bibitem{MT:FP:17}
{\sc Mischler, S., and Tristani, I.}
\newblock Uniform semigroup spectral analysis of the discrete, fractional and
  classical {F}okker--{P}lanck equations.
\newblock {\em J. Ec. Polytech. Math. 4\/} (2017), 389--433.

\bibitem{moser:harnack:64}
{\sc Moser, J.}
\newblock A {H}arnack inequality for parabolic differential equations.
\newblock {\em Comm. Pure Appl. Math. 17\/} (1964), 101--134.
\newblock Erratum in {\em Comm. Pure Appl. Math. 20} (1967), 231--236.

\bibitem{mouhotneumann:06}
{\sc Mouhot, C., and Neumann, L.}
\newblock Quantitative perturbative study of convergence to equilibrium for
  collisional kinetic models in the torus.
\newblock {\em Nonlinearity 19}, 4 (2006), 969--998.

\bibitem{MV:landau}
{\sc Mouhot, C., and Villani, C.}
\newblock On landau damping.
\newblock {\em Acta Math. 207}, 1 (2011), 29--201.

\bibitem{MV:companion}
{\sc Mouhot, C., and Villani, C.}
\newblock Kinetic theory.
\newblock In {\em Princeton Companion for Applied Mathematics}, N.~Higham, Ed.
  Princeton University Press, 2015, pp.~428--446.

\bibitem{nash:58}
{\sc Nash, J.}
\newblock Continuity of solutions of parabolic and elliptic equations.
\newblock {\em Amer. J. Math. 80\/} (1958), 931--954.

\bibitem{NP:heavy:15}
{\sc Nasreddine, E., and Puel, M.}
\newblock Diffusion limit of {F}okker--{P}lanck equation with heavy tail
  equilibria.
\newblock {\em ESAIM Math. Model. Numer. Anal. 49}, 1 (2015), 1--17.

\bibitem{nier:pekin:06}
{\sc Nier, F.}
\newblock Hypoellipticity for {F}okker--{P}lanck operators and {W}itten
  {L}aplacians.
\newblock Preprint IRMAR 06-34, University of Rennes, Lecture Notes from a
  course in Morningside Center, Beijing, 2006.

\bibitem{nier:boundary:18}
{\sc Nier, F.}
\newblock Boundary conditions and subelliptic estimates for geometric
  {K}ramers--{F}okker--{P}lanck opeartors on manifolds with boundaries.
\newblock {\em Mem. Amer. Math. Soc. 1200\/} (2018).

\bibitem{NSW:25}
{\sc Nier, F., Sang, X., and White, F.}
\newblock Global subelliptic estimates for geometric
  {K}ramers--{F}okker--{P}lanck operators on closed manifolds.
\newblock arXiv Preprint 2402.07511v2, 2025.

\bibitem{NSW:grushin:25}
{\sc Nier, F., Sang, X., and White, F.}
\newblock A {G}rushin problem for {B}ismut's hypoelliptic {L}aplacian.
\newblock HAL Preprint, 04572421, 2025.

\bibitem{niki:book}
{\sc Nikiforov, A., and Ouvarov, V.}
\newblock {\em Éléments de la théorie des fonctions spéciales}.
\newblock MIR, 1976.

\bibitem{oleinikradkevic:book}
{\sc Ole\u{i}nik, O., and Radkevi\v{c}, E.}
\newblock {\em Second order equations with nonnegative characteristic forms}.
\newblock Plenum Press, New York-London, New York -- London, 1973.
\newblock Translated from the 1969 Russian original.

\bibitem{pinsker:book:64}
{\sc Pinsker, M.~S.}
\newblock {\em Information and Information Stability of Random Variables and
  Processes}.
\newblock Holden-Day, San Francisco, 1964.

\bibitem{pulviwenn:CMPlower:97}
{\sc Pulvirenti, A., and Wennberg, B.}
\newblock A {M}axwellian lower bound for solutions to the {B}oltzmann equation.
\newblock {\em Comm. Math. Phys. 183\/} (1997), 145--160.

\bibitem{reedsimon:MMMP}
{\sc Reed, M., and Simon, B.}
\newblock {\em Methods of Modern Mathematical Physics, {I--IV}}.
\newblock Academic Press, 1980.

\bibitem{RT:spectral:23}
{\sc Ren, Q., and Tao, Z.}
\newblock Spectral asymptotics for kinetic {B}rownian motion on {R}iemannian
  manifolds.
\newblock ArXiV Preprint, 2212.053994, 2023.

\bibitem{reybelletthomas:anharmonic:00}
{\sc Rey-Bellet, L., and Thomas, L.~E.}
\newblock Asymptotic behavior of thermal nonequilibrium steady states for a
  driven chain of anharmonic oscillators.
\newblock {\em Comm. Math. Phys. 215}, 1 (2000), 1--24.

\bibitem{reybelletthomas:exp:02}
{\sc Rey-Bellet, L., and Thomas, L.~E.}
\newblock Exponential convergence to non-equilibrium stationary states in
  classical statistical mechanics.
\newblock {\em Comm. Math. Phys. 225}, 2 (2002), 305--329.

\bibitem{risken:book}
{\sc Risken, H.}
\newblock {\em The {F}okker--{P}lanck equation}, second~ed., vol.~18 of {\em
  Springer Series in Synergetics}.
\newblock Springer-Verlag, Berlin, 1989.
\newblock Methods of solution and applications.

\bibitem{rothaus:massgap:81}
{\sc Rothaus, O.~S.}
\newblock Diffusion on compact {R}iemannian manifolds and logarithmic {S}obolev
  inequalities.
\newblock {\em J. Funct. Anal. 42\/} (1981), 102--109.

\bibitem{rothschildstein:nilp:76}
{\sc Rothschild, L.~P., and Stein, E.~M.}
\newblock Hypoelliptic differential operators and nilpotent groups.
\newblock {\em Acta Math. 137}, 3-4 (1976), 247--320.

\bibitem{sogge:laplacian:book}
{\sc Sogge, C.}
\newblock {\em Hangzhou Lectures on Eigenfunctions of the {L}aplacian},
  vol.~188 of {\em Annals of Mathematics Studies}.
\newblock Princeton University Press, Princeton, Oxford, 2014.

\bibitem{stam:59}
{\sc Stam, A.}
\newblock Some inequalities satisfied by the quantities of information of
  {F}isher and {S}hannon.
\newblock {\em Inform. Control 2\/} (1959), 101--112.

\bibitem{stroock:markovbook}
{\sc Stroock, D.}
\newblock {\em An Introduction to {M}arkov Processes}, vol.~230 of {\em
  Graduate Texts in Mathematics}.
\newblock Springer-Verlag, 2005.

\bibitem{stroock:path:book}
{\sc Stroock, D.~W.}
\newblock {\em An introduction to the analysis of paths on a {R}iemannian
  manifold}, vol.~74 of {\em Mathematical Surveys and Monographs}.
\newblock American Mathematical Society, Providence, RI, 2000.

\bibitem{talay:stochHamilt:02}
{\sc Talay, D.}
\newblock Stochastic {H}amiltonian systems: exponential convergence to the
  invariant measure, and discretization by the implicit {E}uler scheme.
\newblock {\em Markov Process. Related Fields 8}, 2 (2002), 163--198.
\newblock Inhomogeneous random systems (Cergy-Pontoise, 2001).

\bibitem{tosc:fplcras:97}
{\sc Toscani, G.}
\newblock Sur l'in{\'e}galit{\'e} logarithmique de {S}obolev.
\newblock {\em C.R. Acad. Sci. Paris, S{\'e}rie I, 324\/} (1997), 689--694.

\bibitem{tosc:fpl:99}
{\sc Toscani, G.}
\newblock Entropy production and the rate of convergence to equilibrium for the
  {F}okker--{P}lanck equation.
\newblock {\em Quart. Appl. Math. 57}, 3 (1999), 521--541.

\bibitem{TV:slow:00}
{\sc Toscani, G., and Villani, C.}
\newblock On the trend to equilibrium for some dissipative systems with slowly
  increasing a priori bounds.
\newblock {\em J. Statist. Phys. 98}, 5-6 (2000), 1279--1309.

\bibitem{tristani:fractional:15}
{\sc Tristani, I.}
\newblock Fractional {F}okker--{P}lanck equation.
\newblock {\em Commun. Math. Sci 13}, 5 (2015), 1243--1260.

\bibitem{varopoulos:heatk:89}
{\sc Varopoulos, N.}
\newblock Small time gaussian estimates of heat diffusion kernels.
\newblock {\em Bull. Sci. Math. 113\/} (1989), 253--277.

\bibitem{vill:int}
{\sc Villani, C.}
\newblock Int{\'e}gration et analyse de {F}ourier (in french).
\newblock Incomplete lecture notes, available upon request.

\bibitem{vill:handbook:02}
{\sc Villani, C.}
\newblock A review of mathematical topics in collisional kinetic theory.
\newblock In {\em Handbook of mathematical fluid dynamics, Vol. I}.
  North-Holland, Amsterdam, 2002, pp.~71--305.

\bibitem{ICMP}
{\sc Villani, C.}
\newblock Entropy production and convergence to equilibrium for the {B}oltzmann
  equation.
\newblock In {\em XIVth International Congress on Mathematical Physics}. World
  Sci. Publ., Hackensack, NJ, 2005, pp.~130--144.

\bibitem{vill:icm}
{\sc Villani, C.}
\newblock Hypocoercive diffusion operators.
\newblock In {\em Proceedings of the {\em International Congress of
  Mathematicians} (Madrid, 2006)\/} (2006), vol.~III, European Mathematical
  Society, pp.~473--498.

\bibitem{vill:hypoco}
{\sc Villani, C.}
\newblock {\em Hypocoercivity}.
\newblock Memoirs of the American Mathematical Society. AMS, 2009.

\bibitem{vill:oldnew}
{\sc Villani, C.}
\newblock {\em Optimal transport, old and new}, vol.~338 of {\em Grundlehren
  des mathematischen Wissenschaften}.
\newblock Springer-Verlag, Berlin, New York, 2009.

\bibitem{vill:fisher-Festum:25}
{\sc Villani, C.}
\newblock Fisher information in kinetic theory.
\newblock Preprint, ArXiv 2501.00925 . To appear in the Proceedings of the 2024
  Festum Pi festival, Lect. Notes in Math., 2025.

\bibitem{welsh:manifoldparallel:86}
{\sc Welsh, D.}
\newblock Manifolds that admit parallel vector fields.
\newblock {\em Illinois J. Math. 30}, 1 (1986), 9--18.

\bibitem{yagi:sectorial:10}
{\sc Yagi, A.}
\newblock Sectorial {O}perators.
\newblock In {\em Abstract Parabolic Evolution Equations and Their
  Applications}. Springer, Berlin, 2010, pp.~55--116.

\bibitem{zelditch:laplacian:book}
{\sc Zelditch, S.}
\newblock {\em Eigenfunctions of the {L}aplacian on a {R}iemannian {M}anifold},
  vol.~125 of {\em CBMS}.
\newblock American mathematical Society, Providence, 2017.

\end{thebibliography}


\bigskip

\signcv

\end{document}